\documentclass[12pt]{article}
\usepackage{amssymb,amsmath,amsthm,tikz,multirow,wasysym,mathtools}
\usepackage{subcaption}% FOR SUBFIGURES
\usetikzlibrary{calc,arrows, arrows.meta, math} % LIBRARY NEEDED FOR PICTURE

\title{Edge-to-edge Tilings of the Sphere by Angle Congruent Pentagons}
\author{
Robert Barish\thanks{Research was supported in part by MEXT/JSPS Kakenhi grant 25K21149.} \\
Institute of Medical Science, University of Tokyo \\
Hoi Ping Luk\thanks{Research was supported in part by the funding of Academic Career in Pilsen 2024 \& 2025 under Plze\v{n}sk\'y kraj a Z\'apado\v{c}esk\'a univerzita v Plzni.} \\ 
Katedra matematiky, Z\'apado\v{c}esk\'a univerzita v Plzni \\
Min Yan\thanks{Research was supported by NSFC-RGC Joint Research Scheme N-HKUST607/23 and Hong Kong RGC General Research Fund 16310925.} \\ 
Hong Kong University of Science and Technology}

\usepackage[hidelinks, hyperindex]{hyperref}

\hyperref[sec:function]{}

\newcommand{\ssq}{\tikz \draw (0.05,0.05) -- ++(-0.1,0) -- ++(0,-0.1) -- ++(0.1,0) -- cycle;}

\newcommand*\circled[1]{\tikz[baseline=(char.base)]{
		\node[shape=circle,draw,inner sep=0.5pt] (char) {#1};}}

\newcommand{\mc}{\mathcal}

\newtheorem{theorem}{Theorem}

\theoremstyle{definition}

\newtheorem*{definition*}{Definition}

\theoremstyle{remark}
\newtheorem*{subcase*}{Subcase}
\newtheorem*{family*}{Family}
\newtheorem*{case*}{Case}

\numberwithin{equation}{section}

\colorlet{HBlue}{teal!75!blue!95!white}
\colorlet{HOrange}{orange!80!black}

\begin{document}
\maketitle

\begin{abstract}
Congruent polygons are congruent in angles as well as in edge lengths. We concentrate on the angle aspect, and show how tilings of the sphere by congruent pentagons can be determined by the angle information only. We also show how the features of tilings are changed under reductions, i.e., by ignoring the difference among the angles.
\end{abstract}

\section{Introduction}

Two polygons on a surface of constant curvature are {\em geometrically congruent} if there exists a distance preserving transformation between them (i.e., an \textit{isometry}) that respects the metric of the space. In \cite{so,ua}, Sommerville, Ueno and Agaoka classified edge-to-edge tilings of the sphere by geometrically congruent triangles. In \cite{awy,ay,cly1,cly2,gsy,wy1,wy2}, Akama, Luk, Wang, Yan, et. al. completely classified edge-to-edge tilings of the sphere by geometrically congruent polygons.

If we ignore the edge length information in geometric congruence, then we get angle congruence. Two polygons are {\em angle congruent} if they have the same angle values, and the angles are arranged in the same way. We may similarly define {\em edge congruent} polygons. In this context, the geometric congruence can be regarded as the combination of angle congruence and edge congruence. Indeed, in the first paper on the edge-to-edge tilings of the sphere by geometrically congruent pentagons \cite{gsy}, Gao, Shi and Yan separately classified the angle congruent tilings and edge congruent tilings with the smallest number of tiles, i.e., the dodecahedron. They then obtained geometrically congruent tilings by combining the two classifications.

%{\color{red}

In this work, we investigate tilings based on the angle information alone. In other words, we consider tilings where we neglect constraints based on edge lengths, requiring only that tiles be angle congruent. 

To get a sense of angle congruence, we may consider spherical ``rectangles'' that are quadrilaterals on the left of Figure \ref{rectangle}, with all four angles being $\theta$. Dividing the rectangle by the diagonal, we get two congruent triangles. The angles of the triangle are $\theta,\rho,\theta-\rho$. For all $\rho$ in some range, we always have the triangle and the corresponding rectangle. All these rectangles are angle congruent but not congruent.

\begin{figure}[h!]
\centering
\begin{tikzpicture}[>=latex]

\foreach \x in {1,-1}
\foreach \y in {1,-1}
\draw[xscale=\x, yscale=\y]
	(0,1.2) to[out=0, in=160] (1.2,1) to[out=-70, in=90] (1.4,0);

\draw[dashed]
	(1.2,1) -- (-1.2,-1);
	
\node at (-1.1,0.85) {\scriptsize $\theta$};
\node at (1.1,-0.85) {\scriptsize $\theta$};
\node at (0.9,0.9) {\scriptsize $\rho$};
\node[rotate=40] at (-0.95,-0.6) {\scriptsize $\theta\!-\!\rho$};

\begin{scope}[xshift=4cm]

\draw
	(-1.2,1) to[out=20, in=180] (0,1.2) to[out=0, in=160] (1.2,1) to[out=-70, in=70] (0.6,-1) to[out=200, in=-20] (-1.2,-1) to[out=110, in=250] (-1.2,1);

\node at (-1.1,0.85) {\scriptsize $\theta$};
\node at (-1.1,-0.85) {\scriptsize $\theta$};
\node at (1.1,0.85) {\scriptsize $\theta$};
\node at (0.5,-0.85) {\scriptsize $\theta$};

\end{scope}

\end{tikzpicture}
\caption{Angle congruent rectangles.}
\label{rectangle}
\end{figure}

In the discussion above, we implicitly assume all edges are straight (i.e., great arc in the sphere). In fact, since we are only concerned with angles, we will no longer assume edges to be straight in this paper. This means we allow the rectangle on the right of Figure \ref{rectangle}, which is angle congruent to the rectangle on the left, and has non-straight edges. In fact, in the 3d renderings of tilings in this paper, we always make the angles to be faithful. However, it is often impossible to have all the edges to be straight in the renderings.

Another problem we investigate in this paper is the change of tilings when some distinct angles are regarded as the same. We already observed such phenomenon in the classification of tilings by geometrically congruent polygons. For example, in \cite{wy2}, Wang and Yan showed that, if the angles satisfy certain linear equalities, then the tiling can be modified by flipping parts of the pentagonal subdivision tilings. To properly address the phenomenon, we introduce an even higher level of abstraction, where we assign labels to the corners of the tiles, and require that all tiles are {\em corner congruent} in the sense of having the same arrangements of labels for angles. In this setting, two corners of the same angle value may have different label, and two corners of different angle value may have the same label. %}

There are three overlapping themes for the current paper. Each concerns our ability to determine pentagonal tilings of the sphere under more stringent constraints from tilings obtained under either seemingly weaker constraints or more limited information about nature of individual tiles (e.g., where we may not have explicit angle values).

The first theme of this paper concerns the extent to which geometrically congruent tilings are determined by angle congruence information alone. Here, somewhat surprisingly, our Theorem \ref{5Athm} says that the geometrically congruent tilings obtained by subdividing the Platonic solids -- the pentagonal subdivision tilings first introduced in \cite{wy1} -- are completely determined by angle congruence constraints.

To elaborate, the angle information in a tiling is given by the {\em anglewise vertex combination} (AVC), which is all the possible angle combinations at vertices in the tiling. The angle congruent tilings are constructed from these AVCs. For example, the angle congruent tilings in Theorem \ref{5Athm} are constructed from the following AVCs that first appeared for geometrically congruent tilings in \cite{wy1}
\begin{align*}
\text{AVC(5A24)}
&=\{
24\alpha\beta\gamma\delta\epsilon \colon 
24\alpha\beta\gamma,\,
8\delta^3,\,
6\epsilon^4 \};  \\
\text{AVC(5A60)}
&=\{
60\alpha\beta\gamma\delta\epsilon \colon 
60\alpha\beta\gamma,\,
20\delta^3,\,
12\epsilon^5 \}. 
\end{align*}
We also prove in Theorem \ref{5Athm} that there is no angle congruent tiling for the AVC
\[
\text{AVC(5A36)}
=\{
36\alpha\beta\gamma\delta\epsilon \colon 
36\alpha\beta\gamma,\,
8\delta^3,\,
12\delta\epsilon^3 \}.
\]

In Theorem \ref{EMTthm}, we prove the angle congruent tilings for the following AVC are the earth map tilings and their rotation modifications
\[
\text{AVC(EMT)}
=\{f\alpha\beta\gamma\delta\epsilon \colon f\alpha\beta\gamma,
\tfrac{1}{2}(f-4+2y_2)\delta\epsilon^2,
(4-2y_2)\delta^{\frac{f+4}{8}}\epsilon,
y_2\delta^{\frac{f}{4}}\}
\]
The geometrically congruent tilings for the AVC appeared in \cite[Proposition 30]{cly2}, and our theorem shows that the tilings are also completely determined by angle congruence constraints.

The choice of the four specific AVCs in the two theorems is related to the second theme of the paper, which concerns understanding the specific angle information required to determine tilings. 

For the four AVCs in Theorems \ref{5Athm} and \ref{EMTthm}, the five angles $\alpha,\beta,\gamma,\delta,\epsilon$ are assumed to be distinct (say they have distinct values). When some angles become equal, we get {\em reductions} of the AVCs. The tilings in the two theorems still give tilings for the reduced AVCs. The interesting question is whether there are additional tilings for the reduced AVCs. If there are no additional tilings, then even less information about the angles -- ignoring the difference between the angles that become equal -- still determines the tiling.

Declaring angles becoming equal introduces equalities. This usually requires the original AVC to allow flexibility in the angle values. In fact, the four AVCs only require $\alpha,\beta,\gamma$ to satisfy $\alpha+\beta+\gamma=2\pi$. Therefore, the AVCs allow continuous and free choice of two angle values. In the language of \cite{rao}, these are AVCs of dimension 2. Luk and Yan \cite{ly1} studied all the possible AVCs for edge-to-edge tilings of the sphere by angle congruent pentagons, and they showed in \cite[Proposition 5.1]{ly1} that there are no AVCs of dimension $>2$, and the four AVCs are all AVCs of dimension 2 that have at least 16 tiles. Therefore, we only study reductions of AVC(5A24), AVC(5A36), AVC(5A60). We will have another paper about earth map tilings. 

Throughout this work, we consider a comprehensive list of reductions of AVC(5A24) and AVC(5A60), that are related by reductions in Figure \ref{Fig-reduction}. We find almost all the tilings for these reductions. The comparison of these tilings shows what features of tilings are related to the distinction among angles.

Since Theorem \ref{5Athm} shows there is no tiling for AVC(5A36), we consider several special reductions of the AVC, and find that there are tilings for the reduced AVCs. This pinpoints the specific reason for no tiling for AVC(5A36).

Finally, the construction of tilings in this paper only uses the angle combinations in the AVCs, and never uses the actual angle values. This leads to the third theme of the paper, that angle congruent pentagonal tilings of the sphere -- and by our first theme, geometrically congruent pentagonal tilings of subdivided Platonic solids -- can be understood though a purely combinatorial lens, where we require only that the individual tiles are corner congruent.

\medskip

\noindent{\bf Acknowledgement}: We would like to thank Dr. Yixi Liao for his help in the 3d renderings of the tilings.

\section{Tilings with Two Free Angles}
\label{2free}

For brevity in the future statements, by {\em tilings} we mean edge-to-edge tilings of the sphere by angle congruent pentagons.

In the first theorem (Theorem \ref{5Athm}), we will show that tilings for AVC(5A24) and AVC(5A60) are pentagonal subdivision tilings. Such tilings, which are by geometrically congruent pentagons, first appeared in \cite{wy1}. If we remove the edge length information, then we get the pentagonal subdivision tilings in this paper, in which all tiles are angle congruent. The study of the purely combinatorial aspect of the pentagonal subdivision appeared in \cite{yan2}.

The pentagonal subdivision (see \cite[Section 3]{wy1}) is a general construction that converts any edge-to-edge tiling of an orientable surface into a pentagonal tiling. If the construction is applied to Platonic solids in a uniform way, then we get tilings of the sphere by congruent pentagons. The pentagonal subdivision of the tetrahedron $P_4$ is the dodecahedron $PP_4=P_{12}$. The pentagonal subdivision $PP_6$ of the cube $P_6$ in Figure \ref{Subfig-psubdivision-tiling}, and the pentagonal subdivision $PP_8$ of the octahedron $P_8$ (dual of $P_6$) in Figure \ref{Subfig-5AfigB-f24}, are the same. The pentagonal subdivision $PP_{12}$ of the dodecahedron $P_{12}$, and the pentagonal subdivision $PP_{20}$ of the icosahedron $P_{20}$ (dual of $P_{12}$) in Figure \ref{Subfig-5AfigB-f60}, are also the same.

\begin{figure}[h!]
\begin{subfigure}[b]{0.5\linewidth} %% f=24
\centering
\begin{tikzpicture}[>=latex,scale=1]

\foreach \a in {0,...,3}
{
\begin{scope}[rotate=90*\a, thick]
	
\draw
	(67.5:0.65) -- (90:0.4)	
	(15:1) -- (45:1)
	(60:1.3) -- (90:1.3);

\draw[blue]
	(0,0) -- (0:0.4)
	(-22.5:0.65) -- (-15:1) -- (15:1)
	(-45:1) -- (-15:1) -- (0:1.3)
	(60:1.3) -- (60:1.6);

\draw[red]
	(0:0.4) -- (22.5:0.65) -- (67.5:0.65)
	(15:1) -- (22.5:0.65)
	(0:1.3) -- (30:1.3) -- (60:1.3)
	(30:1.3) -- (45:1);
	
\fill
	(22.5:0.65) circle (0.05)
	(30:1.3) circle (0.05);
	
\end{scope}
}

\foreach \a in {0,...,3}
\filldraw[fill=white, rotate=90*\a]
	(0,0) circle (0.05)
	(-15:1) circle (0.05)
	(60:1.6) circle (0.05);

\end{tikzpicture}
\caption{$PP_6=PP_8$}
\label{Subfig-5AfigB-f24}
\end{subfigure}
\begin{subfigure}[b]{0.4\linewidth} %% f=60
\centering
\begin{tikzpicture}[>=latex,scale=1]

\foreach \a in {0,...,4}
{
\begin{scope}[rotate=72*\a, thick]

\draw
	(72:0.7) -- (90:0.45) 
	(30:1.1) -- (54:1.1)
	(-24:1.6) -- (-9:1.4) 
	(6:1.9) -- (-18:1.9) 
	(21:1.4) -- (12:1.6)
	(-54:2.2) -- (-30:2.2);

\draw[blue]
	(0,0) -- (18:0.45)
	(0:0.7) -- (6:1.1) 
	(-18:1.1) -- (6:1.1) -- (30:1.1)
	(-9:1.4) -- (6:1.1) -- (21:1.4)
	(12:1.6) -- (30:1.9) -- (48:1.6)
	(54:1.9) -- (30:1.9) -- (6:1.9)
	(42:2.2) -- (30:1.9)
	(18:2.2) -- (18:2.5);

\draw[red]
	(18:0.45) -- (36:0.7) -- (72:0.7)
	(30:1.1) -- (36:0.7)
	(21:1.4) -- (39:1.4) -- (54:1.1)
	(39:1.4) -- (48:1.6)
	(-18:1.9) -- (0:1.6) -- (12:1.6)
	(-9:1.4) -- (0:1.6)
	(-30:2.2) -- (-6:2.2) -- (18:2.2)
	(-6:2.2) -- (6:1.9);	
	
\fill
	(36:0.7) circle (0.05)
	(39:1.4) circle (0.05)
	(0:1.6) circle (0.05)
	(-6:2.2) circle (0.05);
	
\end{scope}
}

\foreach \a in {0,...,4}
\filldraw[fill=white, rotate=72*\a]
	(0,0) circle (0.05)
	(6:1.1) circle (0.05)
	(30:1.9) circle (0.05)
	(18:2.5) circle (0.05);
	
\end{tikzpicture}
\caption{$PP_{12}=PP_{20}$}
\label{Subfig-5AfigB-f60}
\end{subfigure}
\caption{Pentagonal subdivision tilings $PP_6,PP_8,PP_{12},PP_{20}$.}
\label{5AfigB}
\end{figure}

Figure \ref{5AfigD} gives 3d renderings of the pentagonal subdivision tilings.

\begin{figure}[h!]
\centering
\begin{subfigure}[b]{0.3\linewidth}
\centering
\begin{tikzpicture}[>=latex,scale=1]

\pgftext{
	\includegraphics[scale=0.07]{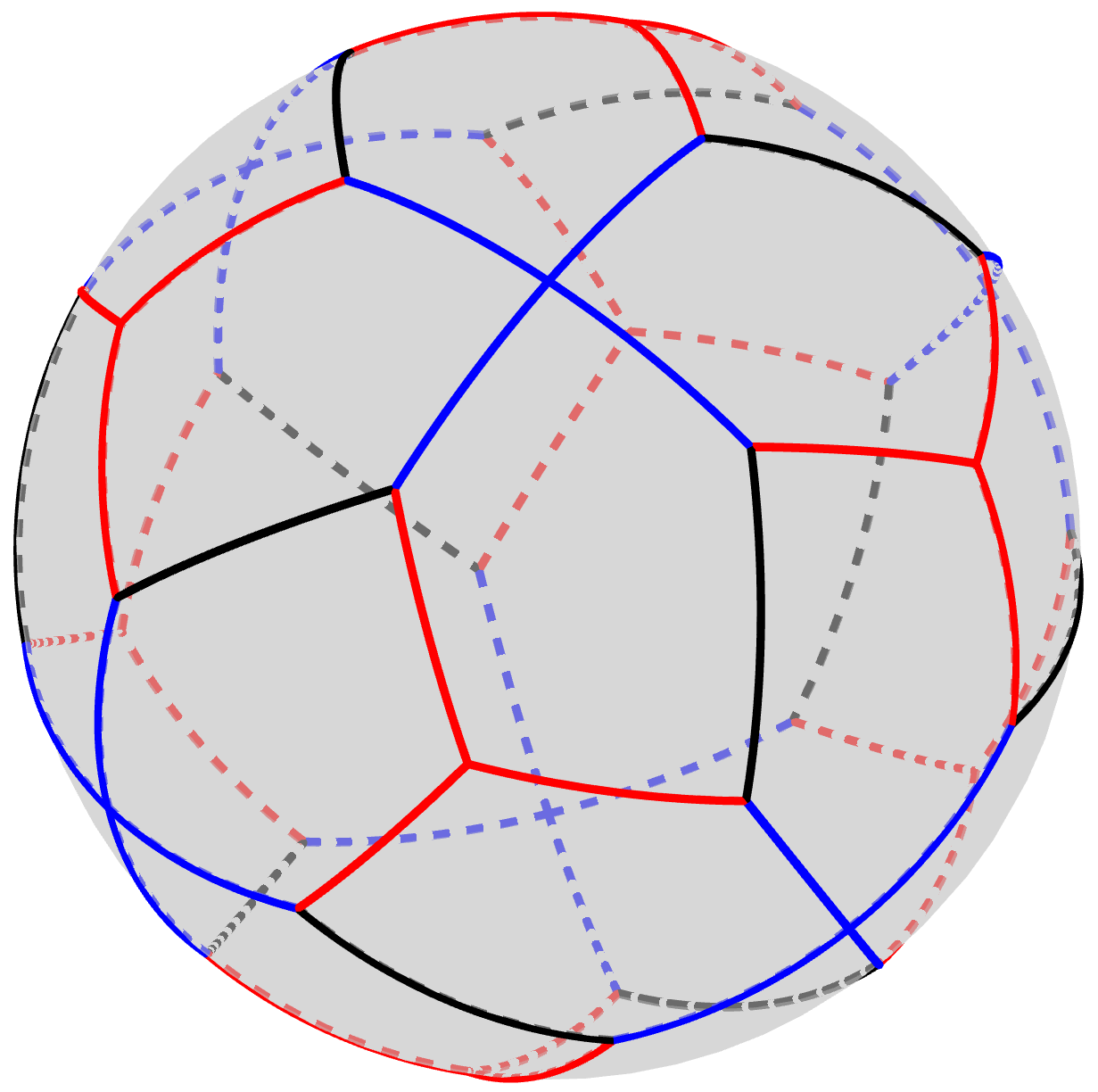}
	};

\end{tikzpicture}	
\caption{$PP_6=PP_8$}
\label{5AfigDa}
\end{subfigure}
\begin{subfigure}[b]{0.3\linewidth}
\centering
\begin{tikzpicture}[>=latex,scale=1]

\pgftext{
	\includegraphics[scale=0.072]{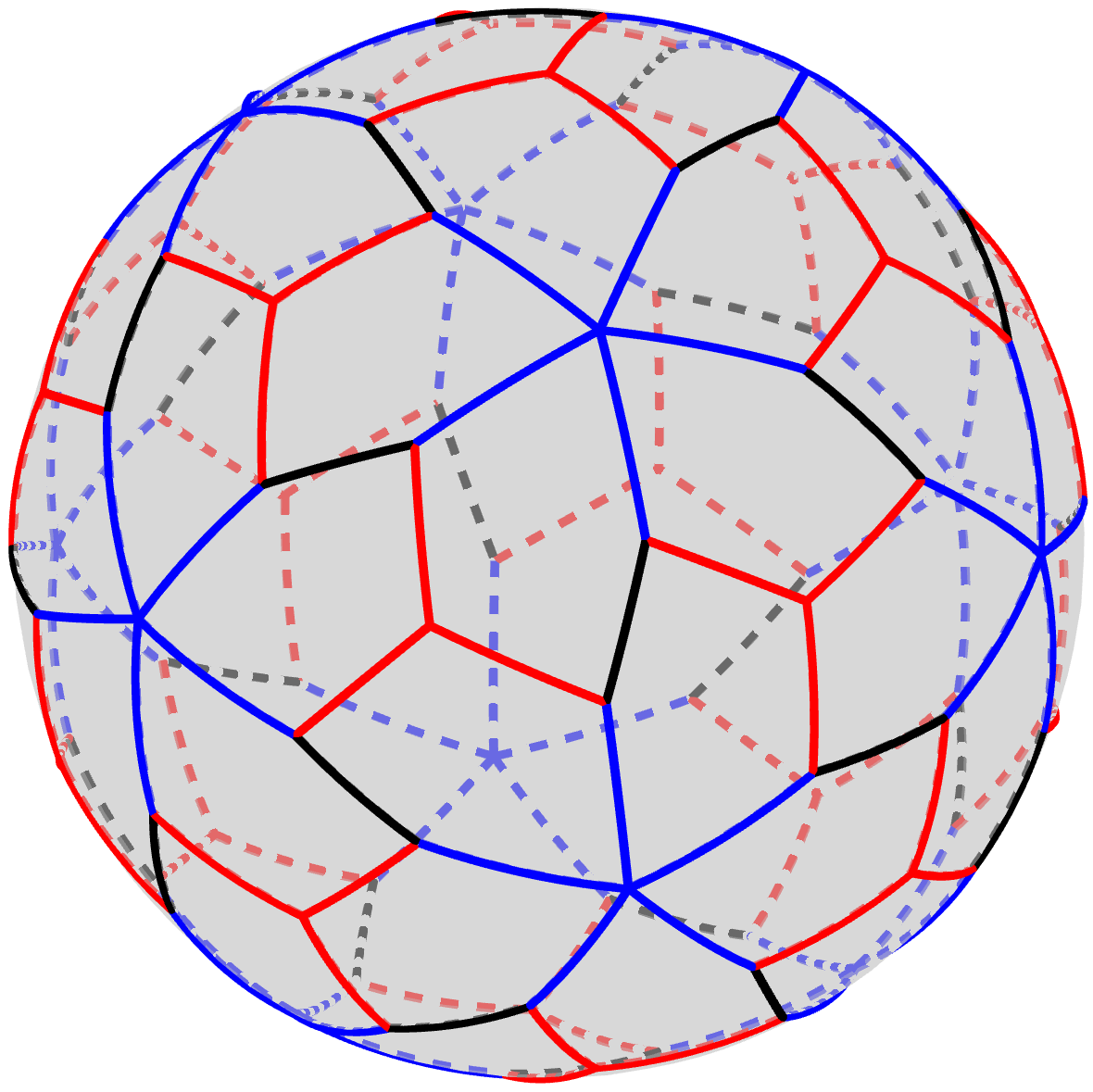}
	};

\end{tikzpicture}
\caption{$PP_{12}=PP_{20}$}
\label{5AfigDb}
\end{subfigure}
\caption{3d renderings of pentagonal subdivision tilings.}
\label{5AfigD}
\end{figure}

\begin{theorem}\label{5Athm} 
The tilings for {\rm AVC(5A24)} and {\rm AVC(5A60)} are the pentagonal subdivisions of the octahedron $PP_8$ and the icosahedron $PP_{20}$. Moreover, there is no tiling for {\rm AVC(5A36)}.
\end{theorem}

\begin{proof}
The AVCs are symmetric with respect to the permutations of $\alpha, \beta, \gamma$. Up to permutations, the angles can be arranged in the pentagon in two ways, given by Figure \ref{Subfig-5AfigA-T1} ($\delta, \epsilon$ adjacent) and Figure \ref{Subfig-5AfigA-T2} ($\delta, \epsilon$ non-adjacent).

\begin{figure}[h!]
\centering
\begin{subfigure}[b]{0.12\linewidth}
\centering

\begin{tikzpicture}[>=latex,scale=1]
\raisebox{1.75 ex}{
\foreach \a in {0,...,4}
\draw[rotate=72*\a]
	(18:0.6) -- (90:0.6)
;
	
\node at (90:0.4) {\scriptsize $\alpha$};
\node at (162:0.4) {\scriptsize $\beta$};
\node at (18:0.4) {\scriptsize $\gamma$};
\node at (234:0.4) {\scriptsize $\delta$};
\node at (-54:0.4) {\scriptsize $\epsilon$};
}
\end{tikzpicture}

\caption{}
\label{Subfig-5AfigA-T1}
\end{subfigure}
\begin{subfigure}[b]{0.12\linewidth}
\centering
\begin{tikzpicture}[>=latex,scale=1]
\raisebox{1.75 ex}{
\foreach \a in {0,...,4}
\draw[rotate=72*\a]
	(18:0.6) -- (90:0.6)
;

\node at (90:0.4) {\scriptsize $\alpha$};
\node at (162:0.4) {\scriptsize $\delta$};
\node at (18:0.4) {\scriptsize $\epsilon$};
\node at (234:0.4) {\scriptsize $\beta$};
\node at (-54:0.4) {\scriptsize $\gamma$};
}
\end{tikzpicture}
\caption{}
\label{Subfig-5AfigA-T2}
\end{subfigure}
\begin{subfigure}[b]{0.5\linewidth}
\centering
\begin{tikzpicture}[>=latex,scale=1]
\begin{scope}[shift={(4.3cm,-0.2cm)}]

\foreach \a in {1,-1}
\foreach \b in {0,1,2}
\draw[xshift=2.2*\b cm, xscale=\a]
	(0,0) -- (0,0.5) -- (0.4,0.8) -- (0.8,0.5) -- (0.8,0) -- (0,0);

\foreach \b in {0,1,2}
\fill (2.5*\b,0) circle (0.05);

\foreach \b in {0,2}
{
\begin{scope}[xshift=2.2*\b cm]

\node at (-0.65,0.15) {\scriptsize $\beta$};
\node at (-0.15,0.15) {\scriptsize $\delta$};
\node at (-0.65,0.45) {\scriptsize $\alpha$};
\node at (-0.15,0.45) {\scriptsize $\epsilon$};
\node at (-0.4,0.6) {\scriptsize $\gamma$};

\end{scope}
}

\foreach \b in {0,1}
{
\begin{scope}[xshift=2.2*\b cm]

\node at (0.65,0.15) {\scriptsize $\epsilon$};
\node at (0.15,0.15) {\scriptsize $\delta$};
\node at (0.65,0.45) {\scriptsize $\gamma$};
\node at (0.15,0.45) {\scriptsize $\beta$};
\node at (0.4,0.6) {\scriptsize $\alpha$};

\end{scope}
}

\begin{scope}[xshift=2.2 cm]

\node at (-0.65,0.15) {\scriptsize $\epsilon$};
\node at (-0.15,0.15) {\scriptsize $\delta$};
\node at (-0.65,0.45) {\scriptsize $\gamma$};
\node at (-0.15,0.45) {\scriptsize $\beta$};
\node at (-0.4,0.6) {\scriptsize $\alpha$};

\end{scope}

\begin{scope}[xshift=4.4 cm]

\node at (0.65,0.15) {\scriptsize $\beta$};
\node at (0.15,0.15) {\scriptsize $\delta$};
\node at (0.65,0.45) {\scriptsize $\alpha$};
\node at (0.15,0.45) {\scriptsize $\epsilon$};
\node at (0.4,0.6) {\scriptsize $\gamma$};

\end{scope}

\node at (0,-0.35) {\footnotesize $|^{\beta}\delta^{\epsilon}|^{\beta}\delta^{\epsilon}|$};

\node at (2.2,-0.35) {\footnotesize $|^{\epsilon}\delta^{\beta}|^{\beta}\delta^{\epsilon}|$};

\node at (4.4,-0.35) {\footnotesize $|^{\beta}\delta^{\epsilon}|^{\epsilon}\delta^{\beta}|$};

\end{scope}
\end{tikzpicture}
\caption{}
\label{Subfig-5AfigA-dede}
\end{subfigure}
\begin{subfigure}[b]{0.22\linewidth}
\centering
\begin{tikzpicture}[>=latex,scale=1]

\foreach \a in {0,1,2}
{
\begin{scope}[rotate=120*\a, thick]

\draw
	(0:0.693) -- (60:0.693);

\draw[blue]
	(-30:1.2) -- (0.693,0)
	(210:1.2) -- (-0.693,0)
	(90:1.2) -- ++(30:0.3)
	(-30:1.2) -- ++(30:0.3);

\draw[red]
	(0,0) -- (0.693,0)
	(-0.693,0) -- ++(150:0.3);

\fill
	(0,0) circle (0.05);

\node at (60:0.18) {\scriptsize $\delta$};
\node at (90:0.95) {\scriptsize $\epsilon$};
\node at (185:0.56) {\scriptsize $\gamma$};
\node at (107:0.65) {\scriptsize $\alpha$};
\node at (138:0.47) {\scriptsize $\beta$};

\node[gray] at (75:1.15) {\scriptsize $\delta/\epsilon$};
\node at (100:1.15) {\scriptsize $\epsilon$};
\node at (5:0.82) {\scriptsize $\gamma$};
\node at (63:0.88) {\scriptsize $\alpha$};
\node at (45:0.78) {\scriptsize $\beta$};

\end{scope}
}

\foreach \a in {0,1,2}
\filldraw[fill=white, rotate=120*\a]
	(90:1.2) circle (0.05)
;

\end{tikzpicture}	
\caption{$N(\delta^3)$}
\label{Subfig-5AfigB-N}
\end{subfigure}
\caption{Angle arrangements for $\alpha\beta\gamma\delta\epsilon$, AADs for $|\delta|\delta|$ and neighborhood $N(\delta^3)$ of the vertex $\delta^3$.}
\label{5AfigA}
\end{figure}

We use the notations for {\em adjacent angle deduction} (or AAD) introduced in \cite{wy1}. Consider two consecutive $\delta$ angles, denoted $|\delta|\delta|$, at the vertex $\delta^3$. If the pentagon is the one in Figure \ref{Subfig-5AfigA-T1}, then Figure \ref{Subfig-5AfigA-dede} shows all the ways the angles in the two tiles containing the two $\delta$ angles can be arranged. In Figure \ref{Subfig-5AfigA-dede}, the AAD notations $|^{\beta}\delta^{\epsilon}|^{\beta}\delta^{\epsilon}|$, $|^{\epsilon}\delta^{\beta}|^{\beta}\delta^{\epsilon}|$, $|^{\beta}\delta^{\epsilon}|^{\epsilon}\delta^{\beta}|$ indicate the positions of the angles $\beta$, $\epsilon$ adjacent to $\delta$. Both the pictures and the AADs show respectively that $\beta\epsilon\cdots$, $\beta^2\cdots$, or $\epsilon^2\cdots$ must be a vertex in the tiling. Since AVC(5A24), AVC(5A36) and AVC(5A60) do not contain such vertices, we get a contradiction.

Therefore, the pentagon must be the one in Figure \ref{Subfig-5AfigA-T2}, where $\delta, \epsilon$ are non-adjacent. Again since $\alpha^2\cdots, \beta^2\cdots$ are not in the AVCs, the AAD of $\delta^3$ is $|^{\alpha}\delta^{\beta}|^{\alpha}\delta^{\beta}|^{\alpha}\delta^{\beta}|$. This determines a subtiling $N(\delta^3)$ consisting of the three tiles around $\delta^3$ (denoted by $\bullet$) in Figure \ref{Subfig-5AfigB-N} (the color matches the color in Figure \ref{5AfigB}). In general, if $X$ is part of a tiling, then we will use $N(X)$ to denote all the tiles touching $X$.

By the AVC, there are three vertices $\alpha\beta\cdots=\alpha\beta\gamma$, and three other vertices $\gamma\cdots=\alpha\beta\gamma$ along the boundary of $N(\delta^3)$. Since $\alpha, \gamma$ are non-adjacent in the pentagon, we also know the locations of the angles just outside $N(\delta^3)$. 

Since $\delta\cdots=\delta^3$, any tiling is a union of $N(\delta^3)$. Moreover, in AVC(5A24) and AVC(5A60), by $\epsilon\cdots=\epsilon^4 / \epsilon^5$, the angle $\delta/\epsilon$ in Figure \ref{Subfig-5AfigB-N} is $\epsilon$. This determines how the $N(\delta^3)$ are glued together.  The result is the pentagonal subdivision of the octahedron in Figure \ref{Subfig-5AfigB-f24} or icosahedron in Figure \ref{Subfig-5AfigB-f60}. 
  
Next we turn to AVC(5A36). The AVC implies the AAD of $\delta\epsilon^3$ is $|^{\alpha}\delta^{\beta}|^{\alpha}\epsilon^{\gamma}|^{\alpha}\epsilon^{\gamma}|^{\alpha}\epsilon^{\gamma}|$. This determines (the angle arrangements of the tiles) \circled{1}, \circled{2}, \circled{3}, \circled{4} in Figure \ref{5AfigC}. Then $\alpha_2\gamma_1\cdots=\alpha_3\gamma_2\cdots=\alpha\beta\gamma$ determines $\beta_5, \beta_6$. Then $\beta_1\cdots=\beta_2\cdots=\alpha\beta\gamma$, and $\alpha, \beta$ non-adjacent determine $\alpha_7, \gamma_5, \alpha_8, \gamma_6$. Then $\beta_5, \gamma_5, \beta_6, \gamma_6$ determine \circled{5}, \circled{6}.

\begin{figure}[h!]
\centering
\begin{tikzpicture}[>=latex,scale=1]

\draw
	(0,-0.5) -- (0,0.5) -- (0.4,0.8) -- (0.8,0.5) -- (0.8,-0.5) -- (0.4,-0.8) -- (0,-0.5) -- (-0.4,-0.8) -- (-0.8,-0.5) -- (-0.8,0.5) -- (-0.4,0.8) -- (0,0.5)
	(0.8,0) -- (-0.8,0) 
	(-0.4,0.8) -- (-0.4,1.3) -- (-1.4,1.3) -- (-1.4,-1.3)
	(-0.8,-0.5) -- (-1.4,-0.5)
	(-0.8,0.5) -- (-2.6,0.5) -- (-2.6,-1.3) -- (1.4,-1.3)
	(-2,0.5) -- (-2,-1.9) -- (1.4,-1.9) -- (1.4,-0.5) -- (0.8,-0.5)
	(0.4,-0.8) -- (0.4,-1.3)
	(-0.4,-0.8) -- (-0.4,-1.9);

\node at (0.15,0.15) {\scriptsize $\delta$};
\node at (0.65,0.15) {\scriptsize $\alpha$};
\node at (0.65,0.45) {\scriptsize $\epsilon$};
\node at (0.4,0.6) {\scriptsize $\gamma$};
\node at (0.15,0.45) {\scriptsize $\beta$};

\node at (-0.65,0.15) {\scriptsize $\gamma$};
\node at (-0.15,0.45) {\scriptsize $\alpha$};
\node at (-0.4,0.6) {\scriptsize $\delta$};
\node at (-0.15,0.15) {\scriptsize $\epsilon$};
\node at (-0.65,0.45) {\scriptsize $\beta$};

\node at (-0.65,-0.15) {\scriptsize $\alpha$};
\node at (-0.15,-0.45) {\scriptsize $\gamma$};
\node at (-0.4,-0.6) {\scriptsize $\beta$};
\node at (-0.15,-0.15) {\scriptsize $\epsilon$};
\node at (-0.65,-0.45) {\scriptsize $\delta$};

\node at (0.65,-0.15) {\scriptsize $\gamma$};
\node at (0.15,-0.15) {\scriptsize $\epsilon$};
\node at (0.15,-0.45) {\scriptsize $\alpha$};
\node at (0.4,-0.6) {\scriptsize $\delta$};
\node at (0.65,-0.45) {\scriptsize $\beta$};

\node at (-0.85,0.65) {\scriptsize $\alpha$};
\node at (-0.55,0.85) {\scriptsize $\delta$};
\node at (-0.55,1.15) {\scriptsize $\beta$};
\node at (-1.25,1.15) {\scriptsize $\gamma$};
\node at (-1.25,0.65) {\scriptsize $\epsilon$};

\node at (-0.85,-0.65) {\scriptsize $\delta$};
\node at (-0.55,-0.85) {\scriptsize $\alpha$};
\node at (-0.55,-1.15) {\scriptsize $\epsilon$};
\node at (-1.25,-1.15) {\scriptsize $\gamma$};
\node at (-1.25,-0.65) {\scriptsize $\beta$};

\node at (0.85,-0.65) {\scriptsize $\alpha$};
\node at (0.55,-0.85) {\scriptsize $\delta$};
\node at (0.55,-1.15) {\scriptsize $\beta$};
\node at (1.25,-1.15) {\scriptsize $\gamma$};
\node at (1.25,-0.65) {\scriptsize $\epsilon$};

\node at (0,-0.7) {\scriptsize $\beta$};
\node at (0.25,-0.85) {\scriptsize $\delta$};
\node at (-0.25,-0.85) {\scriptsize $\gamma$};
\node at (0.25,-1.15) {\scriptsize $\alpha$};
\node at (-0.25,-1.15) {\scriptsize $\epsilon$};

\node at (0,0.7) {\scriptsize $\gamma$};
\node at (-0.25,0.85) {\scriptsize $\delta$};

\node at (-1.25,0.35) {\scriptsize $\epsilon$};
\node at (-1.25,-0.35) {\scriptsize $\alpha$};
\node at (-0.95,0.35) {\scriptsize $\gamma$};
\node at (-0.95,-0.35) {\scriptsize $\delta$};
\node at (-0.95,0) {\scriptsize $\beta$};

\node at (1.25,-1.45) {\scriptsize $\beta$};
\node at (0.4,-1.45) {\scriptsize $\gamma$};
\node at (-0.25,-1.45) {\scriptsize $\epsilon$};
\node at (1.25,-1.75) {\scriptsize $\delta$};
\node at (-0.25,-1.75) {\scriptsize $\alpha$};

\node at (-0.4,-2.05) {\scriptsize $\gamma$};

\node at (-1.85,-1.45) {\scriptsize $\epsilon$};
\node at (-1.4,-1.45) {\scriptsize $\alpha$};
\node at (-0.55,-1.45) {\scriptsize $\delta$};
\node at (-1.85,-1.75) {\scriptsize $\gamma$};
\node at (-0.55,-1.75) {\scriptsize $\beta$};

\node at (-2.15,-1.45) {\scriptsize $\epsilon$};

\node at (-1.55,-1.15) {\scriptsize $\beta$};
\node at (-1.55,0.35) {\scriptsize $\epsilon$};
\node at (-1.55,-0.5) {\scriptsize $\gamma$};
\node at (-1.85,-1.15) {\scriptsize $\delta$};
\node at (-1.85,0.35) {\scriptsize $\alpha$};

\node at (-2,0.65) {\scriptsize $\beta$};
\node at (-1.55,0.65) {\scriptsize $\delta$};

\node at (-2.15,0.35) {\scriptsize $\gamma$};
\node at (-2.45,0.35) {\scriptsize $\beta$};
\node at (-2.45,-0.5) {\scriptsize $\delta$};
\node at (-2.15,-1.15) {\scriptsize $\epsilon$};
\node at (-2.45,-1.15) {\scriptsize $\alpha$};

\node[inner sep=0.5,draw,shape=circle] at (-0.4,-0.3) {\scriptsize $2$};
\node[inner sep=0.5,draw,shape=circle] at (-0.4,0.3) {\scriptsize $1$};
\node[inner sep=0.5,draw,shape=circle] at (0.4,0.3) {\scriptsize $4$};
\node[inner sep=0.5,draw,shape=circle] at (0.4,-0.3) {\scriptsize $3$};
\node[inner sep=0.5,draw,shape=circle] at (-1.2,0) {\scriptsize $5$};
\node[inner sep=0.5,draw,shape=circle] at (0,-1) {\scriptsize $6$};
\node[inner sep=0.5,draw,shape=circle] at (-1,-0.9) {\scriptsize $8$};
\node[inner sep=0.5,draw,shape=circle] at (1,-0.9) {\scriptsize $9$};
\node[inner sep=0,draw,shape=circle] at (0.8,-1.6) {\scriptsize $11$};
\node[inner sep=0,draw,shape=circle] at (-1,-1.6) {\scriptsize $12$};
\node[inner sep=0,draw,shape=circle] at (-1.7,0) {\scriptsize $10$};
\node[inner sep=0,draw,shape=circle] at (-2.3,0) {\scriptsize $13$};
\node[inner sep=0.5,draw,shape=circle] at (-1,0.9) {\scriptsize $7$};

\end{tikzpicture}
\caption{No tiling for AVC(5A36).}
\label{5AfigC}
\end{figure}

Then the AAD $|^{\alpha}\delta^{\beta}|^{\alpha}\delta^{\beta}|^{\alpha}\delta^{\beta}|$ of $\delta_2\delta_5\cdots=\delta_3\delta_6\cdots=\delta^3$ and $\circled{2}, \circled{3}$ determine two $N(\delta^3)$, including \circled{8}, \circled{9}. Then $\alpha_5\beta_8\cdots=\epsilon_6\epsilon_8\cdots=\gamma_8\cdots=\gamma_9\cdots=\alpha\beta\gamma$, and $\alpha, \gamma$ non-adjacent determine \circled{\footnotesize 10}, \circled{\footnotesize 11}, and $\alpha_{12}$. Then $\epsilon_6\epsilon_8\epsilon_{11}\cdots=\delta\epsilon^3$ and $\alpha_{12}$ determine $\circled{\footnotesize 12}$. Then the AAD $|^{\alpha}\delta^{\beta}|^{\alpha}\epsilon^{\gamma}|^{\alpha}\epsilon^{\gamma}|^{\alpha}\epsilon^{\gamma}|$ of $\delta_{10}\epsilon_{12}\cdots=\delta\epsilon^3$ and $\circled{\footnotesize 10}, \circled{\footnotesize 12}$ determine $\circled{\footnotesize 13}$. Then $\alpha_{10}\gamma_{13}\cdots=\alpha\beta\gamma$, and $\epsilon_5\epsilon_{10}\cdots=\delta\epsilon^3$, and $\beta, \epsilon$ non-adjacent determine $\epsilon_7$. Then $\alpha_7, \epsilon_7$ determine $\circled{7}$. Then $\delta_1\delta_7\cdots=\delta^3$ and $\alpha_1\beta_4\cdots=\alpha\beta\gamma$ imply $\gamma, \delta$ adjacent. This is a contradiction.
\end{proof}

The tilings for AVC(5A24) and AVC(5A60) constructed in the proof of Theorem \ref{5Athm} are obtained by glueing copies of the basic pieces $N(\delta^3)$ in Figure \ref{Subfig-5AfigB-N} together. For AVC(5A24), the pieces $N(\delta^3)$ are the triangular faces in the octahedron $P_8$ outlined by the green lines in Figure \ref{2D24-8p}. These green lines correspond to the blue and black lines in Figure \ref{Subfig-5AfigB-f24}. The pentagonal subdivision $PP_8$ means each triangular face of $P_8$ is further divided into three pentagons by the black lines in Figure \ref{2D24-8p}	(corresponding to red lines in Figure \ref{Subfig-5AfigB-f24}). Figure \ref{2D60-20p} similarly describes the pentagonal subdivision $PP_{20}$ of the triangular faces of the icosahedron $P_{20}$.

\begin{figure}[h!]
\centering
\begin{subfigure}[b]{0.22\linewidth}
\centering
\begin{tikzpicture}[>=latex,scale=1]

\pgftext{
	\includegraphics[scale=0.077]{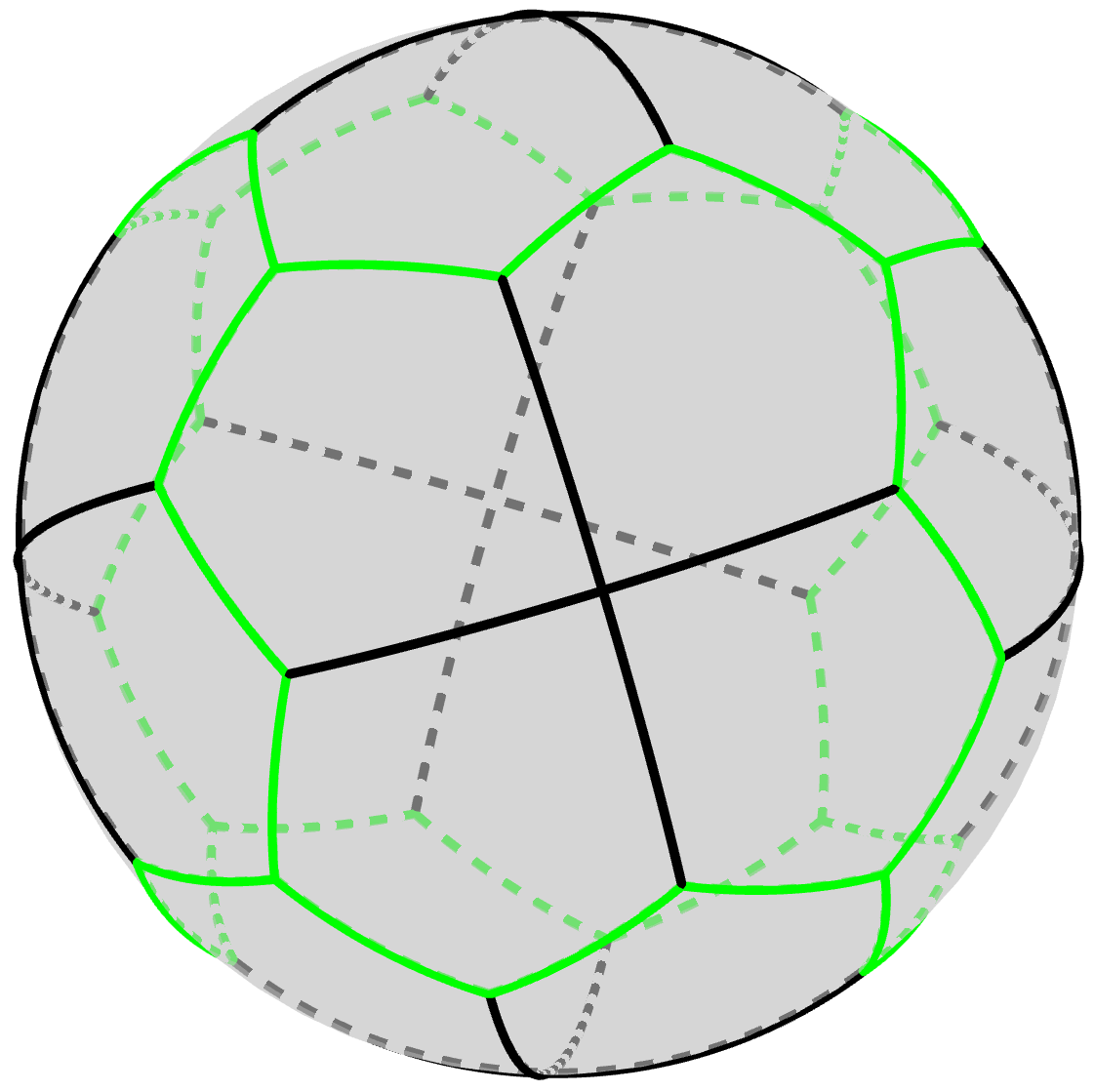}
	};
	
\end{tikzpicture}	
\caption{$PP_6$}
\label{2D24-6p}
\end{subfigure}
\begin{subfigure}[b]{0.22\linewidth}
\centering
\begin{tikzpicture}[>=latex,scale=1]

\pgftext{
	\includegraphics[scale=0.066]{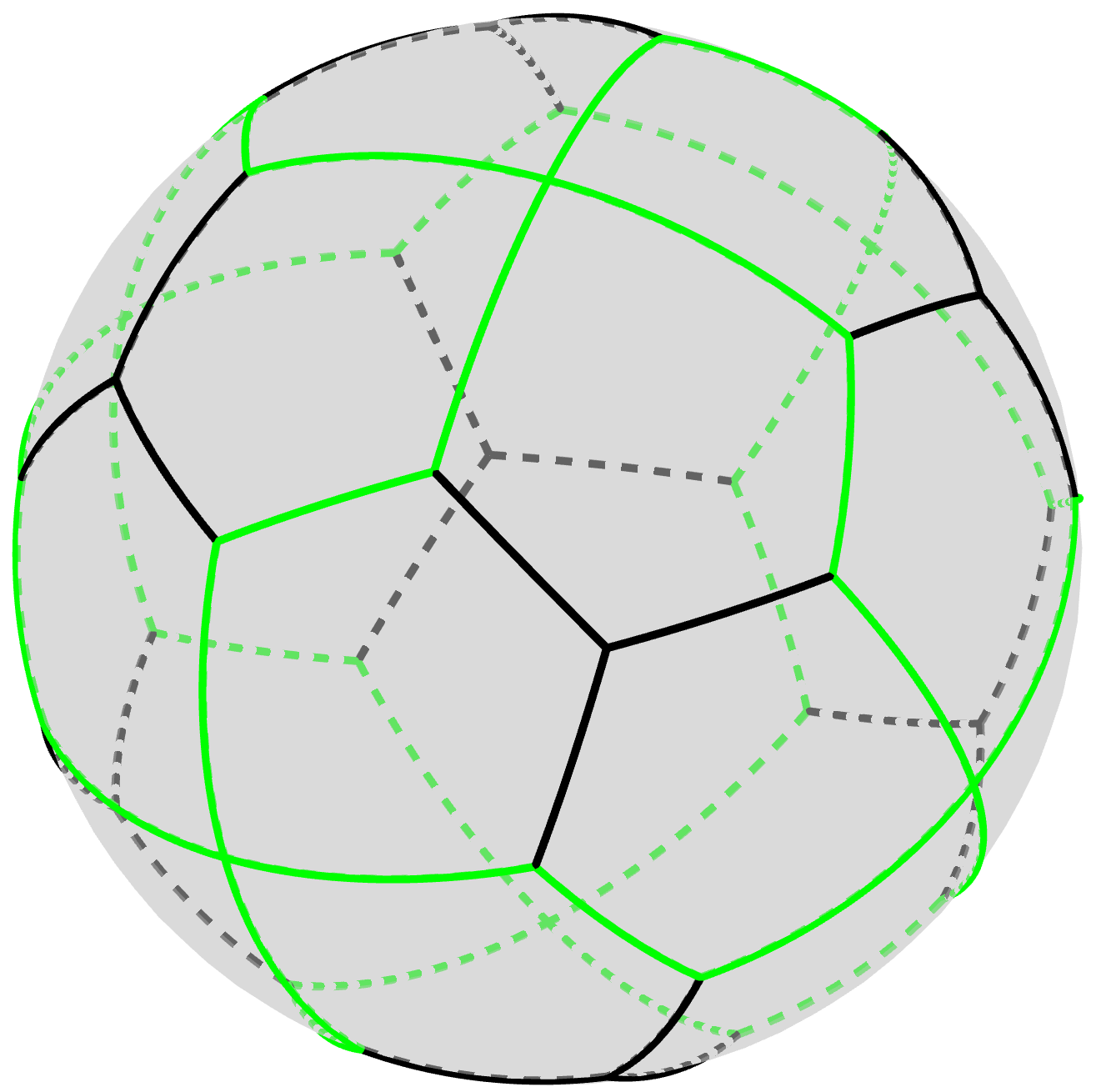}
	};
	
\end{tikzpicture}
\caption{$PP_8$}
\label{2D24-8p}	
\end{subfigure}
\begin{subfigure}[b]{0.22\linewidth}
\centering
\begin{tikzpicture}[>=latex,scale=1]

\pgftext{
	\includegraphics[scale=0.066]{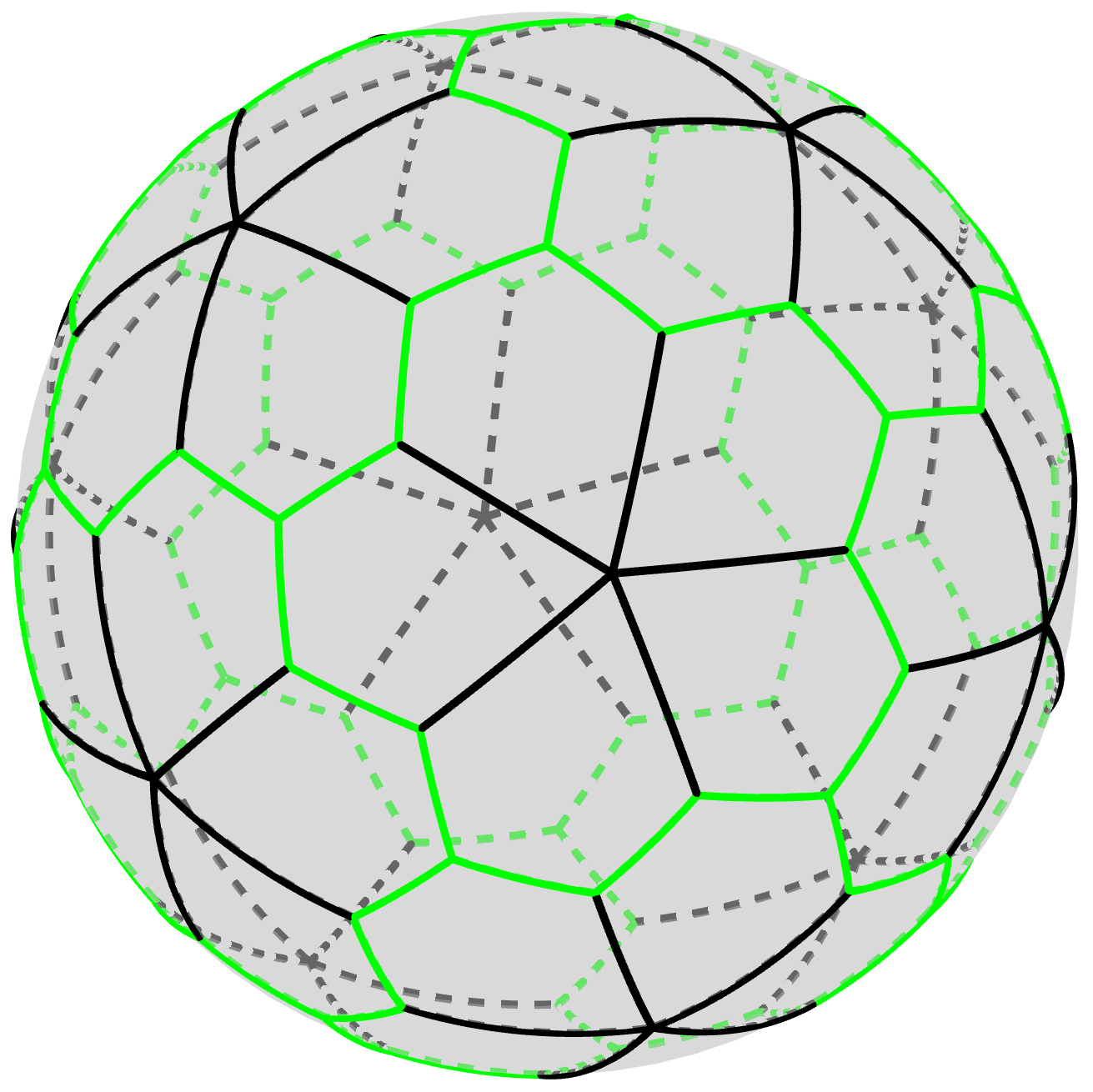}
	};
	
\end{tikzpicture}
\caption{$PP_{12}$}
\label{2D60-12p}	
\end{subfigure}
\begin{subfigure}[b]{0.22\linewidth}
\centering
\begin{tikzpicture}[>=latex,scale=1]

\pgftext{
	\includegraphics[scale=0.068]{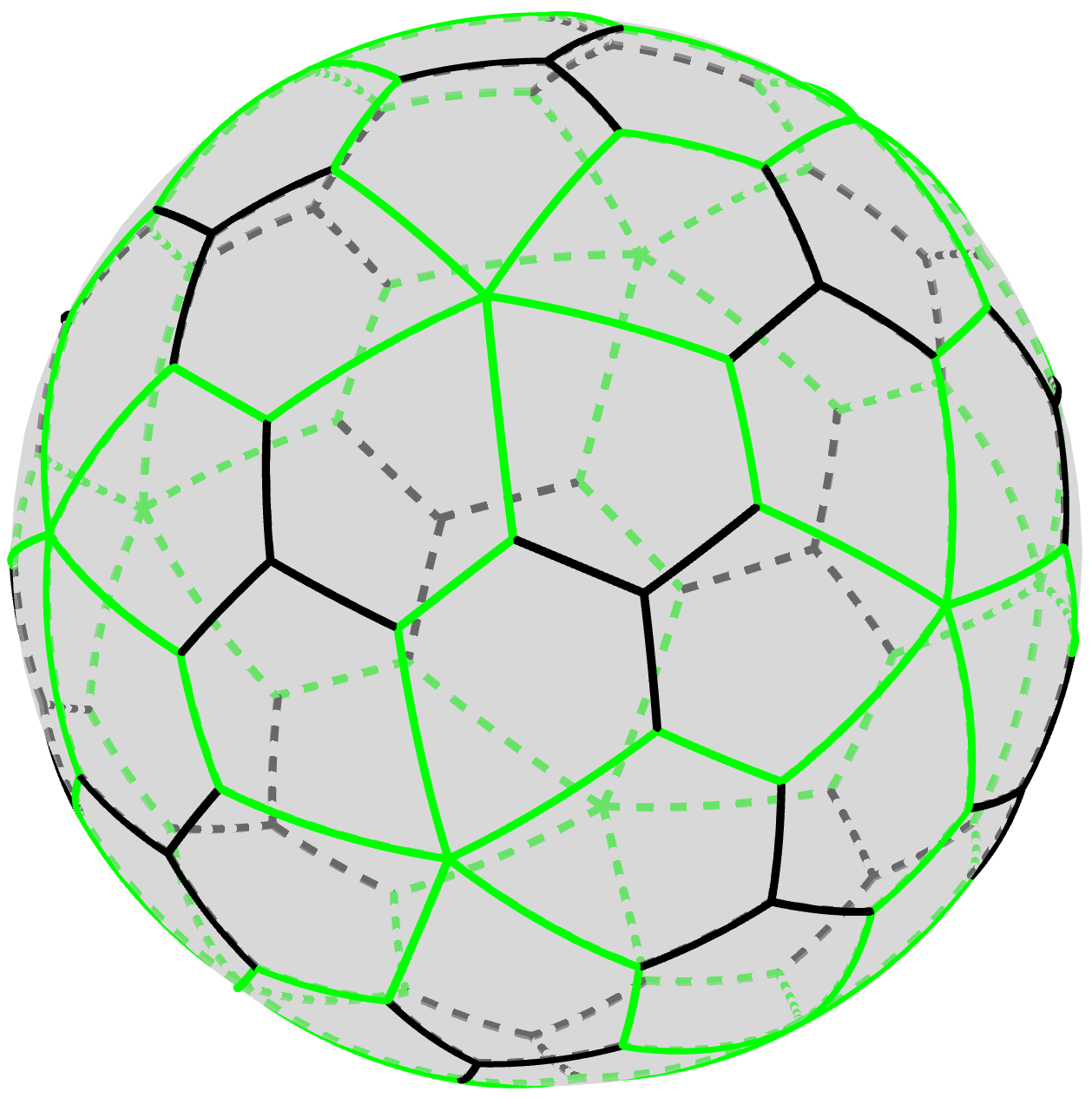}
	};
	
\end{tikzpicture}
\caption{$PP_{20}$}
\label{2D60-20p}	
\end{subfigure}
\caption{Two viewpoints of the same tiling.}
\label{psubdivision2}
\end{figure}

The tilings can also equivalently constructed as the pentagonal subdivisions $PP_6$ and $PP_{12}$ of the cube $P_6$ and the dodecahedron $P_{12}$. The tiling $PP_6$ in Figure \ref{psubdivision} is the same as the tiling $PP_8$ in Figure \ref{Subfig-5AfigB-f24}. Each square face of the cube $P_6$ is further divided into four pentagons according to the neighborhood $N(\epsilon^4)$ of the vertex $\epsilon^4$ in Figure \ref{Subfig-psubdivision-N}. The boundary of the square face $N(\epsilon^4)$ is given by the red and black lines. Figure \ref{2D24-6p} gives the 3d picture of $PP_6$, where the square face is outlined by green lines (and the black lines correspond to blue lines in Figure \ref{Subfig-5AfigB-f24}). Similarly, Figure \ref{2D60-12p} show the division of the pentagonal faces of the dodecahedron $P_{12}$ into five pentagons, as the neighborhood $N(\epsilon^5)$ of the vertex $\epsilon^5$. The result is the pentagonal subdivision tiling $PP_{12}$ and is the same as $PP_{20}$.

\begin{figure}[h!]

\begin{subfigure}[b]{0.45\linewidth}
\centering
\begin{tikzpicture}[>=latex]

\foreach \a in {0,1,2,3}
{
\begin{scope}[rotate=90*\a, thick]

\draw
	(-0.4,1.2) -- (0.4,1.2);

\draw[blue]
	(0,0) -- (1.2,0.4)
	(0.4,1.2) -- ++(0,0.4);

\draw[red]
	(0.4,1.2) -- (1.2,1.2) -- (1.2,0.4)
	(1.2,1.2) -- (1.6,1.6);

\fill (1.2,1.2) circle (0.05);

\begin{scope}[font=\scriptsize]

\node at (1.15,1.35) {$\delta$};
\node at (1.35,1.15) {$\delta$};

\node at (0.55,-1.05) {$\alpha$};
\node at (0.15,-1.05) {$\gamma$};
\node at (1.05,-1.05) {$\delta$};
\node at (-0.4,-1.05) {$\beta$};

\node at (0.4,-1.35) {$\beta$};
\node at (-0.55,-1.35) {$\alpha$};
\node at (-0.25,-1.35) {$\gamma$};

\node at (0.2,-0.1) {$\epsilon$};

\end{scope}

\end{scope}
}

\filldraw[fill=white]
	(0,0) circle (0.06);

\end{tikzpicture}
\caption{$N(\epsilon^4)$}
\label{Subfig-psubdivision-N}
\end{subfigure}
\begin{subfigure}[b]{0.45\linewidth}
\centering
\begin{tikzpicture}[>=latex]

%% tiling

\foreach \a in {0,...,3}
{
\begin{scope}[rotate=90*\a, thick]

\draw
	(0.5,0.9) -- (0,0.6)
	(0.9,0.5) -- (1.4,0.5) 
	(-0.5,1.4) -- (-1.4,1.4)
	(2,-0.5) -- (2,-2);

\draw[blue]
	(0,0) -- (0.6,0)
	(0.5,0.9) -- (0.5,2)
	(1.4,1.4) -- (-0.5,1.4)
	(2,2) -- (2.4,2.4);

\draw[red]
	(0.6,0) -- (0.9,0.5) -- (0.5,0.9)
	(2,1.4) -- (1.4,1.4)
	(2,-0.5) -- (2,2)
	;
	
\fill
	(0.9,0.5) circle (0.06)
	(2,1.4) circle (0.06);

\begin{scope}[font=\scriptsize]

\node at (0.15,0.15) {$\epsilon$};
\node at (0.65,1.25) {$\epsilon$};
\node at (0.35,1.25) {$\epsilon$};
\node at (0.65,1.55) {$\epsilon$};
\node at (0.35,1.55) {$\epsilon$};
\node at (0,2.4) {$\epsilon$};

\node at (-0.35,1) {$\delta$};
\node at (0.7,0.45) {$\delta$};
\node at (0.95,0.65) {$\delta$};
\node at (1.85,1.25) {$\delta$};
\node at (-1.4,2.15) {$\delta$};
\node at (1.85,1.55) {$\delta$};

\node at (0.5,0.15) {$\alpha$};
\node at (2.15,1.95) {$\alpha$};
\node at (0.65,0.95) {$\alpha$};
\node at (-0.35,1.25) {$\alpha$};
\node at (0.35,1.85) {$\alpha$};
\node at (1.4,1.55) {$\alpha$};

\node at (0,0.8) {$\beta$};
\node at (1.25,0.65) {$\beta$};
\node at (0.45,0.7) {$\beta$};
\node at (1.55,1.25) {$\beta$};
\node at (0.5,2.15) {$\beta$};
\node at (1.85,1.85) {$\beta$};

\node at (0.15,0.5) {$\gamma$};
\node at (1.95,2.15) {$\gamma$};
\node at (0.35,1) {$\gamma$};
\node at (1.25,1.25) {$\gamma$};
\node at (-0.5,1.55) {$\gamma$};
\node at (0.65,1.85) {$\gamma$};

\end{scope}

\end{scope}
}

\foreach \a in {0,...,3}
\filldraw[fill=white, rotate=90*\a] 
	(0.5,1.4) circle (0.05)
	(0,0) circle (0.05)
	(2.4,2.4) circle (0.05);

\end{tikzpicture}
\caption{}
\label{Subfig-psubdivision-tiling}
\end{subfigure}
\caption{Pentagonal subdivision $PP_6$ of the cube $P_6$.}
\label{psubdivision}
\end{figure}

Although it makes no difference whether we use $PP_8$ or $PP_6$ in Theorem \ref{5Athm}, in the future theorems, the tilings may sometimes be better described in terms of the triangular faces of $P_8$, and sometimes in terms of the square faces of $P_6$. We will use the more relevant one among $PP_8$ and $PP_6$ (and the same among $PP_{20}$ and $PP_{12}$) in the statements of theorems.

In the second theorem, we will show that tilings for AVC(EMT) are earth map tilings and their rotation modifications. The earth map tilings by geometrically congruent pentagons first appeared in \cite{cly2}, and denoted $E_{\pentagon}2$ and $RE_{\pentagon}2$ in the earlier paper. If we remove the edge length information, then we get the pentagonal subdivision tilings in this paper, in which all tiles are angle congruent. The study of the purely combinatorial aspect of the earth map tilings appeared in \cite{yan}.

\begin{theorem}\label{EMTthm}
The tilings for {\rm AVC(EMT)} are the earth map tilings and their rotation modifications.
\end{theorem}

Figure \ref{emtC} gives 3d rendering of the earth map tiling and its rotation modification. The green line divides the earth map tiling into two hemispheres (denoted ${\mc H}$ in the proof below), and the modification rotates one hemisphere.

We remark that AVC(EMT) assumes $f=0$ mod $4$. Moreover, the rotation modification happens only for $f=4$ mod $8$. 

\begin{figure}[h!]
\centering
\begin{tikzpicture}[>=latex,xscale=-1]

\begin{scope}[shift={(-0.03,0.03)}]

\pgftext{
	\includegraphics[scale=0.096]{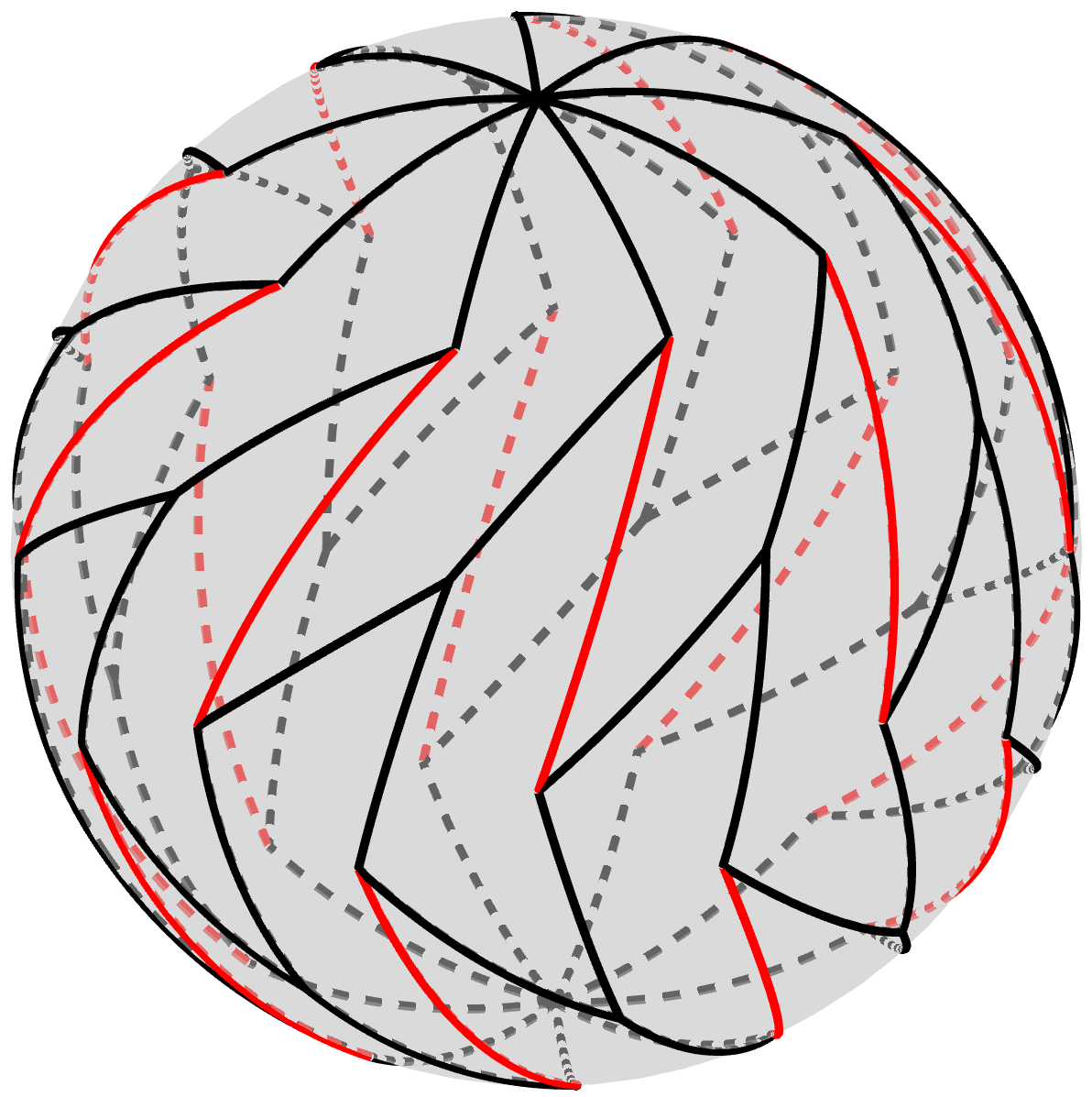}
	};
		
\end{scope}

\begin{scope}[shift={(-5.02,0.02)}]

\pgftext{
	\includegraphics[scale=0.089]{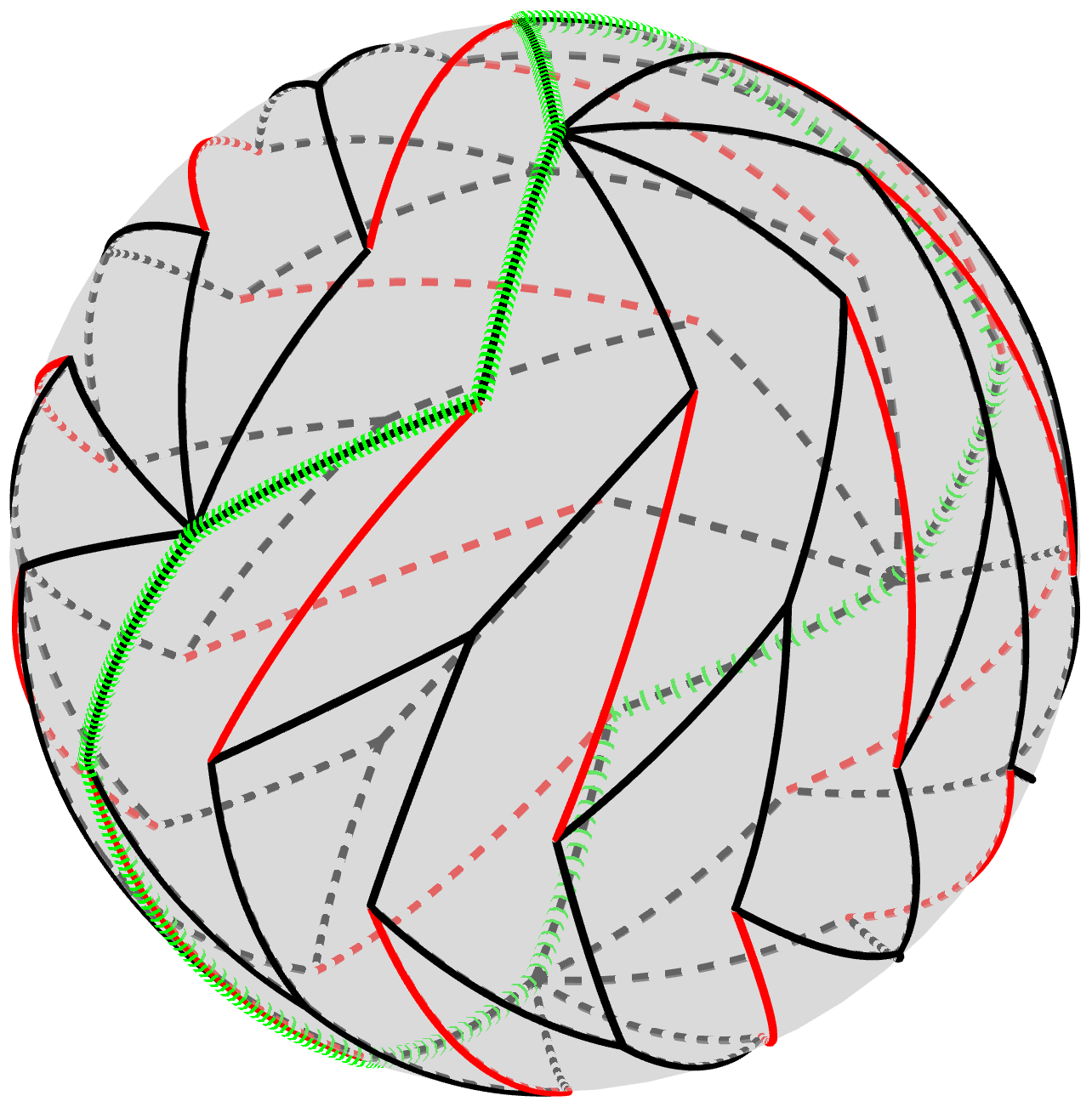}
	};
		
\end{scope}
	
\end{tikzpicture}
\caption{Earth map tiling for $f=36$,  and its rotation modification (the right of green line is rotated).}
\label{emtC}
\end{figure}

\begin{proof}
The AVC is symmetric with respect to the permutation of $\alpha, \beta, \gamma$. Similar to the proof of Theorem \ref{5Athm}, we only need to consider two possible angle arrangements of the pentagon, and the one in Figure \ref{Subfig-5AfigA-T1} does not result in tilings. Therefore, we may assume the pentagon is that of Figure \ref{Subfig-5AfigA-T2}, where $\delta, \epsilon$ are non-adjacent.

Since $\alpha^2\cdots, \beta^2\cdots$ are not vertices, the AAD of consecutive $|\delta|\delta|$ at a vertex is $|^{\alpha}\delta^{\beta}|^{\alpha}\delta^{\beta}|$. This determines $\circled{1}, \circled{2}$ in Figure \ref{Subfig-emtA-timezone}. By $\alpha_2\beta_1\cdots=\alpha\beta\gamma$, and $\epsilon\cdots=\delta\epsilon^2, \delta^{\frac{f+4}{8}}\epsilon, \delta^{\frac{f}{4}}$, and $\gamma, \delta$ non-adjacent, we determine $\circled{3}$. Then by $\alpha_3\cdots=\alpha\beta\gamma$, and $\epsilon_2\epsilon_3\cdots=\delta\epsilon^2$, and $\gamma, \delta$ non-adjacent, we determine $\circled{4}, \gamma_5$. Then by $\delta_3\cdots=\delta\epsilon^2, \delta^{\frac{f+4}{8}}\epsilon, \delta^{\frac{f}{4}}$, and $\gamma, \delta$ non-adjacent, we determine $\circled{5}$. Similarly, by $\alpha_4\gamma_2\cdots=\alpha\beta\gamma$, and $\epsilon_4\cdots=\delta\epsilon^2, \delta^{\frac{f+4}{8}}\epsilon, \delta^{\frac{f}{4}}$, and $\beta, \epsilon$ non-adjacent, we determine $\circled{6}$.

\begin{figure}[h!]
\begin{subfigure}{0.45\linewidth}
\centering
\begin{tikzpicture}[>=latex,scale=1]

\foreach \a in {0,1,2}
\draw[gray!70, line width=2, xshift=1.6*\a cm]
	(-0.8,-0.7) -- (-0.8,-0.3) -- (-0.4,0) -- (-0.4,0.5) -- (-1.2,0.5) -- (-1.6,0.8) -- (-1.6,1.2);

\foreach \a in {0,1}
{
\begin{scope}[xshift=1.6*\a cm]

\draw
	(-0.8,-0.7) -- (-0.8,-0.3) -- (-0.4,0) -- (-0.4,0.5)
	(-0.4,0.5) -- (-1.2,0.5) -- (-1.6,0.8) -- (-1.6,1.2) 
	(-0.4,0) -- (0.4,0) -- (0.4,0.5) -- (0,0.8) -- (0,1.2)
	(0.4,0.5) -- (1.2,0.5) -- (1.2,0) -- (0.8,-0.3) -- (0.8,-0.7)
	;

\draw[red]
	(-0.4,0.5) -- (0,0.8)
	(0.4,0) -- (0.8,-0.3);

\begin{scope}[font=\scriptsize]

\node at (0,0.6) {$\gamma$};
\node at (-0.25,0.45) {$\beta$};
\node at (0.25,0.45) {$\epsilon$};
\node at (-0.25,0.15) {$\delta$};
\node at (0.25,0.15) {$\alpha$};

\node at (-0.15,0.85) {$\beta$};
\node at (-1.45,0.85) {$\alpha$};
\node at (-0.45,0.65) {$\gamma$};
\node at (-1.15,0.65) {$\epsilon$};
\node at (-0.8,1.1) {$\delta$};

\end{scope}

\end{scope}
}

\begin{scope}[font=\scriptsize]

\node at (0.8,-0.1) {$\gamma$};
\node at (0.55,0.05) {$\beta$};
\node at (1.05,0.05) {$\epsilon$};
\node at (0.55,0.35) {$\delta$};
\node at (1.05,0.35) {$\alpha$};

\node at (0.65,-0.35) {$\beta$};
\node at (-0.65,-0.35) {$\alpha$};
\node at (0.35,-0.15) {$\gamma$};
\node at (-0.35,-0.15) {$\epsilon$};
\node at (0,-0.6) {$\delta$};

\node[inner sep=0.5,draw,shape=circle] at (-0.8,0.8) {$1$};
\node[inner sep=0.5,draw,shape=circle] at (0.8,0.8) {$2$};
\node[inner sep=0.5,draw,shape=circle] at (0,0.3) {$3$};
\node[inner sep=0.5,draw,shape=circle] at (0.8,0.2) {$4$};
\node[inner sep=0.5,draw,shape=circle] at (0,-0.3) {$5$};
\node[inner sep=0.5,draw,shape=circle] at (1.6,0.3) {$6$};

\end{scope}

\end{tikzpicture}
\caption{}
\label{Subfig-emtA-timezone}
\end{subfigure}
\begin{subfigure}{0.45\linewidth}
\centering
\begin{tikzpicture}[>=latex,scale=1]

\draw[green!70, line width=2]
	(-0.5,0.7) -- (-0.5,0.4) -- (-0.25,0.2) -- (0.25,0.2) -- (0.25,-0.2) -- (0,-0.4) -- (0,-0.7)
	(3.5,0.7) -- (3.5,0.4) -- (3.75,0.2) -- (3.75,-0.2) -- (3.25,-0.2) -- (3,-0.4) -- (3,-0.7);
	
\node at (3.5,-0.7) {\scriptsize ${\mc H}$};

\draw[<->]
	(-0.5,0.8) -- node[fill=white] {\scriptsize $\delta^{q+1}$} ++(4,0);

\draw[<->]
	(0,-0.8) -- node[fill=white] {\scriptsize $\delta^q$} ++(3,0);

\foreach \a in {0,...,3}
{
\begin{scope}[xshift=\a cm]

\draw
	(-0.5,0.7) -- (-0.5,0.4) -- (-0.25,0.2) -- (0.25,0.2) -- (0.25,-0.2) -- (0.75,-0.2) -- (0.75,0.2) -- (0.5,0.4) -- (0.5,0.7)
	(0.25,-0.2) -- (0,-0.4)	
	;

\draw[red]
	(0.25,0.2) -- (0.5,0.4);

\end{scope}
}
			
\foreach \a in {0,1,2}
{
\begin{scope}[xshift=\a cm]

\draw
	(0,-0.7) -- (0,-0.4) -- (0.25,-0.2) -- (0.75,-0.2)
	(0.75,-0.2) -- (1,-0.4) -- (1,-0.7);

\draw[red]
	(0.75,-0.2) -- (1,-0.4);
	
\end{scope}
}

\end{tikzpicture}
\caption{}
\label{Subfig-emtA-half}
\end{subfigure}
\caption{Time zone, and half earth map tiling.}
\label{emtA}
\end{figure}

If $\delta^{\frac{f}{4}}$ is a vertex, then we may apply the argument in Figure \ref{Subfig-emtA-timezone} to all $\frac{f}{4}$ consecutive $|\delta|\delta|$ in $\delta^{\frac{f}{4}}$. We obtain the earth map tiling by repeating the time zone consisting of $\circled{1}, \circled{3}, \circled{4}, \circled{5}$. Together with the two unlabeled tiles, Figure \ref{Subfig-emtA-timezone} shows two consecutive time zones.

Next, assuming $\delta^{\frac{f}{4}}$ is not a vertex gives $y_2=0$, and we update the AVC
\[
\text{AVC(EMT)}
=\{(8q+4)\alpha\beta\gamma\delta\epsilon\colon (8q+4)\alpha\beta\gamma,
4q\delta\epsilon^2,
4\delta^{q+1}\epsilon\}, \quad
q=\tfrac{f-4}{8}.
\]
Applying Figure \ref{Subfig-emtA-timezone} to the $q$ consecutive $|\delta|\delta|$ in the $\delta^{q+1}$ part of the vertex $\delta^{q+1}\epsilon$ gives the {\em half earth map tiling} ${\mc H}$ in Figure \ref{Subfig-emtA-half}, with $\delta^{q+1}$ and $\delta^q$ at the two ends. Moreover, we have angles $\alpha, \epsilon, \gamma|\beta, \delta|\epsilon, \alpha$ along the left boundary of ${\mc H}$ (left side of $\circled{1}, \circled{3}, \circled{5}$), and $\beta|\gamma, \epsilon, \alpha, \delta|\epsilon, \gamma|\beta$ along the right boundary of ${\mc H}$ (right side of $\circled{2}, \circled{6}, \circled{4}, \circled{5}$).

Figure \ref{emtB} shows the half earth map tiling ${\mc H}$ obtained from the $\delta^{q+1}$ part of the vertex $\delta^{q+1}\epsilon$ (indicated by $\bullet$) as the outside of the disk. Figure \ref{emtA} shows ${\mc H}$ inside out -- the angles on the left of ${\mc H}$ are along the right outside of the disk in Figure \ref{emtB}, and the angles on the right of ${\mc H}$ are along the left outside of the disk.

\begin{figure}[h!]
\begin{subfigure}[b]{0.244\linewidth}
\centering
\begin{tikzpicture}[>=latex, font=\scriptsize]

%% 1

\foreach \a in {0,...,11}
\draw[green!70, line width=2, rotate=30*\a]
	(0:1.3) -- (30:1.3)
	;
		
\foreach \a in {0,...,11}
\draw[rotate=30*\a]
	(0:1.3) -- (30:1.3);

\draw
	(-30:1.3) -- (-30:1.6)
	(210:1.3) -- (210:1.6);
	
\draw[red]
	(120:1.3) -- (120:1.6)
	(240:1.3) -- (240:1.6)
	(0:1.3) -- (0:1.6);
	
\fill (0,1.3) circle (0.05);

\node at (110:1.45) {$\beta_1$};
\node at (127:1.45) {$\gamma_1$};
\node at (150:1.45) {$\epsilon_1$};
\node at (180:1.45) {$\alpha$};
\node at (204:1.45) {$\delta$};
\node at (217:1.45) {$\epsilon$};
\node at (245:1.45) {$\beta$};
\node at (235:1.45) {$\gamma$};

\node at (60:1.45) {$\alpha$};
\node at (30:1.45) {$\epsilon$};
\node at (-7:1.45) {$\beta$};
\node at (6:1.45) {$\gamma$};
\node at (-24:1.45) {$\delta$};
\node at (-35:1.45) {$\epsilon$};
\node at (-60:1.45) {$\alpha$};

\node at (90:1.45) {$\delta^{q+1}$};
\node at (-90:1.45) {$\delta^q$};

\node at (-50:1.8) {${\mc H}$};

\begin{scope}[rotate=240]

\draw
	(30:1.3) -- (30:0.9)
	(150:1.3) -- (150:0.9);

\draw[red]
	(180:1.3) -- (180:0.9)
	(60:1.3) -- (60:0.9)
	(-60:1.3) -- (-60:0.9);

\node at (120:1.1) {$\alpha$};
\node at (142:1.1) {$\epsilon$};
\node at (158:1.1) {$\delta$};
\node at (173:1.1) {$\beta$};
\node at (188:1.1) {$\gamma$};
\node at (210:1.1) {$\epsilon$};
\node at (240:1.1) {$\alpha$};

\node at (69:1.1) {$\beta$};
\node at (51:1.1) {$\gamma$};
\node at (37:1.1) {$\epsilon$};
\node at (22:1.1) {$\delta$};
\node at (0:1.1) {$\alpha$};
\node at (-30:1.1) {$\epsilon$};
\node at (-69:1.1) {$\beta$};
\node at (-52:1.1) {$\gamma$};

\node at (-95:0.95) {$\delta^{q+1}$};
\node at (90:1.1) {$\delta^q$};

\end{scope}

\end{tikzpicture}
\caption{}
\label{Subfig-emtB-1}
\end{subfigure}
\begin{subfigure}[b]{0.244\linewidth}
\centering
\begin{tikzpicture}[>=latex, font=\scriptsize]

%% 2

\foreach \a in {0,...,11}
\draw[green!70, line width=2, rotate=30*\a]
	(0:1.3) -- (30:1.3)
	;
	
\foreach \a in {0,...,11}
\draw[rotate=30*\a]
	(0:1.3) -- (30:1.3);

\draw
	(-30:1.3) -- (-30:1.6)
	(210:1.3) -- (210:1.6);
	
\draw[red]
	(120:1.3) -- (120:1.6)
	(240:1.3) -- (240:1.6)
	(0:1.3) -- (0:1.6);
	
\fill (0,1.3) circle (0.05);

\node at (113:1.45) {$\beta$};
\node at (127:1.45) {$\gamma$};
\node at (150:1.45) {$\epsilon_1$};
\node at (180:1.45) {$\alpha_1$};
\node at (202:1.45) {$\delta_1$};
\node at (217:1.45) {$\epsilon$};
\node at (245:1.45) {$\beta$};
\node at (235:1.45) {$\gamma$};

\node at (60:1.45) {$\alpha$};
\node at (30:1.45) {$\epsilon_2$};
\node at (-7:1.45) {$\beta_1$};
\node at (6:1.45) {$\gamma_1$};
\node at (-24:1.45) {$\delta$};
\node at (-35:1.45) {$\epsilon$};
\node at (-60:1.45) {$\alpha$};

\node at (90:1.45) {$\delta^{q+1}$};
\node at (-90:1.45) {$\delta^q$};

\node at (-50:1.8) {${\mc H}$};

\begin{scope}[rotate=120]

\draw
	(30:1.3) -- (30:0.9)
	(150:1.3) -- (150:0.9);

\draw[red]
	(180:1.3) -- (180:0.9)
	(60:1.3) -- (60:0.9)
	(-60:1.3) -- (-60:0.9);

\node at (120:1.1) {$\alpha$};
\node at (142:1.1) {$\epsilon$};
\node at (158:1.1) {$\delta$};
\node at (171:1.1) {$\beta$};
\node at (188:1.1) {$\gamma$};
\node at (210:1.1) {$\epsilon$};
\node at (240:1.1) {$\alpha$};

\node at (69:1.1) {$\beta$};
\node at (53:1.1) {$\gamma$};
\node at (37:1.1) {$\epsilon$};
\node at (22:1.1) {$\delta$};
\node at (0:1.1) {$\alpha$};
\node at (-30:1.1) {$\epsilon$};
\node at (-69:1.1) {$\beta$};
\node at (-52:1.1) {$\gamma$};

\node at (-90:0.95) {$\delta^{q+1}$};
\node at (90:1.05) {$\delta^q$};
	
\end{scope}

\end{tikzpicture}
\caption{}
\label{Subfig-emtB-2}
\end{subfigure}
\begin{subfigure}[b]{0.244\linewidth}
\centering
\begin{tikzpicture}[>=latex, font=\scriptsize]

%% 3

\foreach \a in {0,...,11}
\draw[green!70, line width=2, rotate=30*\a]
	(0:1.3) -- (30:1.3)
	;
	
\foreach \a in {0,...,11}
\draw[rotate=30*\a]
	(0:1.3) -- (30:1.3);

\draw
	(-30:1.3) -- (-30:1.6)
	(210:1.3) -- (210:1.6);
	
\draw[red]
	(120:1.3) -- (120:1.6)
	(240:1.3) -- (240:1.6)
	(0:1.3) -- (0:1.6);
	
\fill (0,1.3) circle (0.05);

\node at (110:1.45) {$\beta_1$};
\node at (127:1.45) {$\gamma_1$};
\node at (150:1.45) {$\epsilon_1$};
\node at (180:1.45) {$\alpha$};
\node at (204:1.45) {$\delta$};
\node at (217:1.45) {$\epsilon$};
\node at (245:1.45) {$\beta$};
\node at (235:1.45) {$\gamma$};

\node at (60:1.45) {$\alpha$};
\node at (30:1.45) {$\epsilon$};
\node at (-7:1.45) {$\beta$};
\node at (6:1.45) {$\gamma$};
\node at (-24:1.45) {$\delta$};
\node at (-35:1.45) {$\epsilon$};
\node at (-60:1.45) {$\alpha$};

\node at (90:1.45) {$\delta^{q+1}$};
\node at (-90:1.45) {$\delta^q$};

\node at (-50:1.8) {${\mc H}$};

\node at (0,0) {${\mc H}$};

\draw
	(30:1.3) -- (30:0.9)
	(150:1.3) -- (150:0.9);

\draw[red]
	(180:1.3) -- (180:0.9)
	(60:1.3) -- (60:0.9)
	(-60:1.3) -- (-60:0.9);
	
\node at (120:1.1) {$\alpha$};
\node at (142:1.1) {$\epsilon$};
\node at (158:1.1) {$\delta$};
\node at (173:1.1) {$\beta$};
\node at (188:1.1) {$\gamma$};
\node at (210:1.1) {$\epsilon$};
\node at (240:1.1) {$\alpha$};

\node at (67:1.1) {$\beta$};
\node at (53:1.1) {$\gamma$};
\node at (37:1.1) {$\epsilon$};
\node at (22:1.1) {$\delta$};
\node at (0:1.1) {$\alpha$};
\node at (-30:1.1) {$\epsilon$};
\node at (-69:1.1) {$\beta$};
\node at (-52:1.1) {$\gamma$};

\node at (-90:1.05) {$\delta^{q+1}$};
\node at (90:1.1) {$\delta^q$};

\draw[gray, ->]
	(-90:0.6) arc (-90:150:0.6);
\draw[gray, ->]
	(-90:0.8) arc (-90:30:0.8);

\end{tikzpicture}
\caption{}
\label{Subfig-emtB-3}
\end{subfigure}
\begin{subfigure}[b]{0.244\linewidth}
\centering
\begin{tikzpicture}[>=latex]

%% 4

\foreach \a in {0,...,11}
\draw[green!70, line width=2, rotate=30*\a]
	(0:1.3) -- (30:1.3)
	;

\foreach \a in {0,...,11}
\draw[rotate=30*\a]
	(0:1.3) -- (30:1.3);

\foreach \a in {0,1,2}
{
\begin{scope}[rotate=120*\a, font=\scriptsize]

\node at (30:1.45) {$\epsilon$};
\node at (60:1.45) {$\alpha$};
\node at (90:1.45) {$\bar{\epsilon}$};
\node at (120:1.45) {$\bar{\alpha}$};

\node at (30:1.1) {$\bar{\epsilon}$};
\node at (60:1.1) {$\bar{\alpha}$};
\node at (90:1.1) {$\epsilon$};
\node at (120:1.1) {$\alpha$};

\end{scope}
}

\end{tikzpicture}
\caption{}
\label{Subfig-emtB-4}
\end{subfigure}
\caption{Earth map tiling, and rotation modification.}
\label{emtB}
\end{figure}

Keeping in mind $\epsilon_1\cdots=\delta\epsilon^2, \delta^{q+1}\epsilon$, the tilings in Figure \ref{emtB} are determined as follows. 

In Figure \ref{Subfig-emtB-1}, we have $\epsilon_1\cdots=\delta^{q+1}\epsilon$. The $\delta^{q+1}$ part of it induces another half earth map tiling that fills the inside disk. By $\beta_1\gamma_1\cdots=\alpha\beta\gamma$, the angles along the boundary of the interior half earth map tiling are given as indicated. 

In Figures \ref{Subfig-emtB-2}, \ref{Subfig-emtB-3}, we have $\epsilon_1\cdots=\delta\epsilon^2$, and the two pictures show the two possible locations of the inside $\delta, \epsilon$ at the vertex. In Figure \ref{Subfig-emtB-2}, by $\alpha_1\cdots=\alpha\beta\gamma$ and $\beta, \epsilon$ non-adjacent, we determine the $\beta, \gamma$ at $\alpha_1\cdots$. Then by $\delta_1\cdots=\delta\epsilon^2, \delta^{q+1}\epsilon$, and $\beta, \epsilon$ non-adjacent, we get $\delta_1\cdots=\delta^{q+1}\epsilon$. The $\delta^{q+1}$ part of the vertex determines a half earth map tiling, ending with $\delta^q$ at $\epsilon_2\cdots$. Then the $\delta^{q+1}$ part of $\epsilon_2\cdots=\delta^q\epsilon\cdots=\delta^{q+1}\epsilon$ determines a half earth map tiling that fills the inside disk. By $\beta_1\gamma_1\cdots=\alpha\beta\gamma$, the angles along the boundary of the interior half earth map tilings are given as indicated. 

In Figure \ref{Subfig-emtB-3}, the locations of the inside $\delta, \epsilon$ at $\epsilon_1\cdots$ are different from those in Figure \ref{Subfig-emtB-2}. By the inside $\epsilon$ and $\beta_1\gamma_1\cdots=\alpha\beta\gamma$, we determine a tile in the inner disk. The $\delta$ angle of the tile is at $\bullet$. Then the vertex $\bullet$ is $\delta\delta^{q+1}\cdots=\delta^{2q+1}=\delta^{\frac{f}{4}}$. This gives the earth map tiling, which is the union of two half earth map tilings ${\mc H}$ in the picture.

The tilings in Figures \ref{Subfig-emtB-1}, \ref{Subfig-emtB-2} are obtained by rotating the inside half of the earth map tiling in Figure \ref{Subfig-emtB-3} by $240^{\circ}$ and $120^{\circ}$. They are actually the same, given by the rotation modification of the earth map tiling. 
\end{proof}

Figure \ref{Subfig-emtB-4} gives an interpretation of the rotation modification. We indicate the angle sum values along the two sides of the boundary between the two half earth map tilings, with $\bar{\alpha}=2\pi-\alpha$ and $\bar{\epsilon}=2\pi-\epsilon$. Then it is clear that rotating one half earth map tiling by $120^{\circ}$ or $240^{\circ}$ still gives a tiling. We also note that, among three outside $\epsilon$, two are actually the angle $\epsilon$ of the pentagon, and one is the $\delta^q$ (of value $\epsilon$) part of a vertex. What we proved is that one of the three $\epsilon$ angles must be matched with an interior $\delta^{q+1}$ (of value $\bar{\epsilon}$). If the two angles $\epsilon$ are matched with $\delta^{q+1}$, then we get rotation modifications. If $\delta^p$ is matched with $\delta^{q+1}$, then we get the earth map tiling.

\section{Tiling Reductions}
\label{Sec-Reductions}

The tilings in Section \ref{2free} allow free continuous choice of two angles. By choosing some angles to be of equal value (a concept that will later be generalized), we get reductions of the tilings. In this section, we examine whether the AVC for a reduction admits new tilings other than the standard ones in Section \ref{2free}.

We first consider the most extreme reduction $\alpha=\beta=\gamma=\delta$ of AVC(5A24) and AVC(5A60), which become the following:
\begin{align*}
\text{AVC(2D24)}
&=\{24\alpha^4\beta \colon 32\alpha^3,6\beta^4\}, \;
\alpha=\tfrac{2}{3}\pi, \;
\beta=\tfrac{1}{2}\pi;   \\
\text{AVC(2D60)}
&=\{60\alpha^4\beta \colon 80\alpha^3,12\beta^5\}, \;
\alpha=\tfrac{2}{3}\pi, \;
\beta=\tfrac{2}{5}\pi.
\end{align*}
The reduction may be interpreted as changing $\alpha, \beta, \gamma, \delta, \epsilon$ to $\alpha, \alpha, \alpha, \alpha, \beta$. Applying the change to the pentagonal subdivisions in Theorem \ref{5Athm} gives the 2D reductions of the pentagonal subdivision tilings for AVC(2D24) and AVC(2D60). In this reduction tiling, $\beta$ appears only at $\circ=\beta^4/\beta^5$ in Figure \ref{5AfigB}, and $\alpha$ appears at all other vertices. Figures \ref{2D24-6p} and \ref{2D60-12p} give the 3d renderings of the tilings, with all $\beta$ concentrated at the meeting places of four or five black lines, and all the remaining angles are $\alpha$. It turns out that they are the only tilings for AVC(2D24) and AVC(2D60).

\begin{theorem}\label{2Dthm}
The tilings for {\rm AVC(2D24)} and {\rm AVC(2D60)} are the {\rm 2D} reductions of $PP_6$ and $PP_{12}$.
\end{theorem}

\begin{proof}
Figure \ref{2DfigA} illustrates three tiles $\circled{1}, \circled{2}, \circled{3}$ containing three consecutive $|\beta|\beta|\beta|$ at $\beta^4$ or $\beta^5$. The angles in the tiles are arranged as indicated. By $\alpha\cdots=\alpha^3$, we determine three $\alpha$ angles in $\circled{4}, \circled{6}$, and two $\alpha$ angles in $\circled{5}, \circled{7}$. Without loss of generality, we may assume $\beta_4$ is located as indicated. By $\beta_4\cdots=\beta^4/\beta^5$, we determine $\circled{4}, \circled{5}$. Then by $\alpha\cdots=\alpha^3$ and three existing $\alpha$ angles of $\circled{6}$, we know the four $\alpha$ angles in $\circled{6}$, which determines $\circled{6}$. Then by $\beta_6\cdots=\beta^4/\beta^5$, we determine $\circled{7}$.

\begin{figure}[htp]
\centering
\begin{tikzpicture}[>=latex, font=\scriptsize]

\foreach \a in {0,...,3}
\draw[green!70, line width=2, rotate=90*\a]
	(0.9,0.5) -- (0.5,0.9) -- (0,0.6) -- (-0.5,0.9) -- (-0.5,1.4) -- (-1.4,1.4) -- (-1.4,2) -- (2,2) -- (2,1.4);

\foreach \a in {0,...,3}
{
\begin{scope}[gray!70, rotate=90*\a]

\draw
	(0,0) -- (0.6,0) -- (0.9,0.5) -- (0.5,0.9) -- (0,0.6)
	(0.9,-0.5) -- (2,-0.5)
	(0.9,0.5) -- (1.4,0.5)
	(2,1.4) -- (-1.4,1.4)
	(-2,2) -- (2,2) -- (2.4,2.4);

\fill
	(0.9,0.5) circle (0.05);

\node at (0.15,0.15) {$\beta$};
\node at (0.65,1.25) {$\beta$};
\node at (0.35,1.25) {$\beta$};
\node at (0.65,1.55) {$\beta$};
\node at (0.35,1.55) {$\beta$};
\node at (0,2.4) {$\beta$};

\node at (-0.35,1) {$\alpha$};
\node at (0.7,0.45) {$\alpha$};
\node at (0.95,0.65) {$\alpha$};
\node at (1.85,1.25) {$\alpha$};
\node at (-1.4,2.15) {$\alpha$};
\node at (1.85,1.55) {$\alpha$};

\node at (0.5,0.15) {$\alpha$};
\node at (2.15,1.95) {$\alpha$};
\node at (0.65,0.95) {$\alpha$};
\node at (-0.35,1.25) {$\alpha$};
\node at (0.35,1.85) {$\alpha$};
\node at (1.4,1.55) {$\alpha$};

\node at (0,0.8) {$\alpha$};
\node at (1.25,0.65) {$\alpha$};
\node at (0.45,0.7) {$\alpha$};
\node at (1.55,1.25) {$\alpha$};
\node at (0.5,2.15) {$\alpha$};
\node at (1.85,1.85) {$\alpha$};

\node at (0.15,0.5) {$\alpha$};
\node at (1.95,2.15) {$\alpha$};
\node at (0.35,1) {$\alpha$};
\node at (1.25,1.25) {$\alpha$};
\node at (-0.5,1.55) {$\alpha$};
\node at (0.65,1.85) {$\alpha$};

\end{scope}
}

\foreach \a in {0,...,3}
\filldraw[fill=white, gray!70, rotate=90*\a] 
	(0.5,1.4) circle (0.05)
	(0,0) circle (0.05)
	(2.4,2.4) circle (0.05);

\foreach \a in {-1,0}
{
\begin{scope}[rotate=90*\a]

\draw
	(-0.5,0.9) -- (-0.5,1.4) -- (1.4,1.4) -- (1.4,0.5) -- (0.9,0.5)
	(0.5,0.9) -- (0.5,1.4);

\fill
	(0,0) circle (0.05)
	(0.5,1.4) circle (0.05);

\node at (0.35,1) {$\alpha$};
\node at (-0.35,1) {$\alpha$};
\node at (0,0.8) {$\alpha$};
\node at (-0.35,1.25) {$\alpha$};
\node at (0.35,1.25) {$\beta$};

\node at (0.65,0.95) {$\alpha$};
\node at (0.95,0.65) {$\alpha$};
\node at (1.25,1.25) {$\alpha$};
\node at (1.25,0.65) {$\alpha$};
\node at (0.65,1.25) {$\beta$};

\end{scope}
}

\foreach \a in {-1,0,1}
{
\begin{scope}[rotate=90*\a]

\draw
	(0,0) -- (0.6,0) -- (0.9,0.5) -- (0.5,0.9) -- (0,0.6) -- (0,0);
	
\node at (0.5,0.15) {$\alpha$};
\node at (0.15,0.5) {$\alpha$};
\node at (0.45,0.7) {$\alpha$};
\node at (0.7,0.45) {$\alpha$};
\node at (0.15,0.15) {$\beta$};

\end{scope}
}

\node[inner sep=0.5,draw,shape=circle] at (-0.43,0.43) {$1$};
\node[inner sep=0.5,draw,shape=circle] at (0.43,0.43) {$2$};
\node[inner sep=0.5,draw,shape=circle] at (0.43,-0.43) {$3$};
\node[inner sep=0.5,draw,shape=circle] at (0,1.15) {$4$};
\node[inner sep=0.5,draw,shape=circle] at (1,1) {$5$};
\node[inner sep=0.5,draw,shape=circle] at (1.15,0) {$6$};
\node[inner sep=0.5,draw,shape=circle] at (1,-1) {$7$};

\end{tikzpicture}
\caption{Tiling for AVC(2D24).}
\label{2DfigA}
\end{figure}

The process of obtaining $\beta_6\beta_7\cdots=\beta^4/\beta^5$ from $\beta_4\beta_5\cdots=\beta^4/\beta^5$ can be repeated for all $|\beta|\beta|\beta|$ in the initial $\beta_1\beta_2\beta_3\cdots=\beta^4/\beta^5$, and we determine two layers of tiles. Moreover, we obtain four $\beta^4$ or five $\beta^5$ around the boundary of the second layer. Then the argument that started with the initial $\beta^4/\beta^5$ can be repeated at the new $\beta^4/\beta^5$ along the boundary, until obtaining the 2D reductions of $PP_6$ and $PP_{12}$. The gray part of Figure \ref{2DfigA} is the 2D reduction of $PP_6$. 
\end{proof}

We remark that this proof -- as well as all subsequent proofs -- do not use specific angle values. Therefore, the conclusion of Theorem \ref{2Dthm} is also valid for any AVC that can be reduced to AVC(2D24/60). For example, by ignoring the distinction between $\alpha, \beta, \gamma, \delta$ in AVC(5A24/60), the proof of Theorem \ref{2Dthm} is still valid in the context of Theorem \ref{5Athm}. What remains to be investigated for Theorem \ref{5Athm} is the way $\alpha, \beta, \gamma, \delta$ are arranged in each tile.

The observation leads to the natural question about how much the distinction of angles contributes to the features of tilings. Although the angle value is the most common way of distinguishing angles, there are many other ways of distinguishing angles (even with the same value), such as lengths of edges bounding the angles in \cite{cly2,wy1,wy2}. However, in the context of this work, we may also treat distinct angles to be the same, by ignoring the distinction and assigning the same label. 

Therefore, we abandon the usual convention that $\alpha=\beta$ means that $\alpha$ and $\beta$ have the same value. Then $\alpha, \beta, \gamma, \dots$ are merely symbols for the angles (or less confusingly, corners) of the pentagon, and all tiles are congruent in the sense that the corners of all tiles are labeled in the same way. 

Now {\em reduction} means corners with distinct labels are reduced to the same label. We introduce the new notation $\alpha\doteq\beta$ to mean that the corners $\alpha$ and $\beta$ are not distinguished. In the reverse direction, {\em splitting} means corners with the same labels (say $\alpha, \alpha$) are changed to distinct labels (say $\alpha, \beta$), and we need to care about the distinction in constructing the tilings.

In the rest of the section, we discuss an extensive list of the reductions of AVC(5A24/60). In Section \ref{36tiles}, we consider some reductions of AVC(5A36). 

\medskip

\noindent{\bf Four distinct angles}

\medskip

By assuming two angles are equal in AVC(5A24/60), we get AVCs with four distinct angles. Given the symmetry between $\alpha, \beta, \gamma$ in AVC(5A24/60), up to permutations, Figure \ref{two_equal} lists all the reductions.

\begin{figure}[h!]
\centering
\begin{tikzpicture}[>=latex]

\begin{scope}[xshift=-3cm]

\foreach \a in {0,...,4}
\draw[rotate=72*\a]
	(18:0.6) -- (90:0.6);
	
\node at (90:0.4) {\scriptsize $\alpha$};
\node at (162:0.4) {\scriptsize $\delta$};
\node at (18:0.4) {\scriptsize $\epsilon$};
\node at (234:0.4) {\scriptsize $\beta$};
\node at (-54:0.4) {\scriptsize $\gamma$};

\end{scope}

\draw[very thick, ->]
	(-2,0) -- ++(1,0);

\foreach \a in {0,1,2}
\foreach \b in {-1,0,1}
\foreach \x in {0,...,4}
\draw[shift={(2*\a cm, -1.4*\b cm)}, rotate=72*\x]
	(18:0.6) -- (90:0.6);

\foreach \a in {0,1,2}
{
\begin{scope}[xshift=2*\a cm]

\node[yshift=1.4cm] at (162:0.4) {\scriptsize $\gamma$};
\node[yshift=1.4cm] at (18:0.4) {\scriptsize $\delta$};

\node at (162:0.4) {\scriptsize $\alpha$};
\node at (18:0.4) {\scriptsize $\delta$};

\node[yshift=-1.4cm] at (162:0.4) {\scriptsize $\delta$};
\node[yshift=-1.4cm] at (18:0.4) {\scriptsize $\alpha$};

\end{scope}
}

%% 1

\begin{scope}[yshift=1.4cm]

\node at (90:0.4) {\scriptsize $\beta$};
\node at (234:0.4) {\scriptsize $\alpha$};
\node at (-54:0.4) {\scriptsize $\alpha$};
\node at (0,0) {\scriptsize $4{\rm A}_1$};

\begin{scope}[xshift=2 cm]

\node at (90:0.4) {\scriptsize $\alpha$};
\node at (234:0.4) {\scriptsize $\beta$};
\node at (-54:0.4) {\scriptsize $\alpha$};
\node at (0,0) {\scriptsize $4{\rm A}_2$};

\end{scope}

\begin{scope}[xshift=4 cm]

\node at (90:0.4) {\scriptsize $\alpha$};
\node at (234:0.4) {\scriptsize $\alpha$};
\node at (-54:0.4) {\scriptsize $\beta$};
\node at (0,0) {\scriptsize $4{\rm A}_3$};

\end{scope}

\end{scope}

%% 2

\node at (90:0.4) {\scriptsize $\alpha$};
\node at (234:0.4) {\scriptsize $\gamma$};
\node at (-54:0.4) {\scriptsize $\beta$};
\node at (0,0) {\scriptsize $4{\rm D}_1$};

\begin{scope}[xshift=2 cm]

\node at (90:0.4) {\scriptsize $\beta$};
\node at (234:0.4) {\scriptsize $\alpha$};
\node at (-54:0.4) {\scriptsize $\gamma$};
\node at (0,0) {\scriptsize $4{\rm D}_2$};

\end{scope}

\begin{scope}[xshift=4 cm]

\node at (90:0.4) {\scriptsize $\beta$};
\node at (234:0.4) {\scriptsize $\gamma$};
\node at (-54:0.4) {\scriptsize $\alpha$};
\node at (0,0) {\scriptsize $4{\rm D}_3$};

\end{scope}

%% 3

\begin{scope}[yshift=-1.4cm]

\node at (90:0.4) {\scriptsize $\alpha$};
\node at (234:0.4) {\scriptsize $\beta$};
\node at (-54:0.4) {\scriptsize $\gamma$};
\node at (0,0) {\scriptsize $4{\rm E}_1$};

\begin{scope}[xshift=2 cm]

\node at (90:0.4) {\scriptsize $\beta$};
\node at (234:0.4) {\scriptsize $\alpha$};
\node at (-54:0.4) {\scriptsize $\gamma$};
\node at (0,0) {\scriptsize $4{\rm E}_2$};

\end{scope}

\begin{scope}[xshift=4 cm]

\node at (90:0.4) {\scriptsize $\beta$};
\node at (234:0.4) {\scriptsize $\gamma$};
\node at (-54:0.4) {\scriptsize $\alpha$};
\node at (0,0) {\scriptsize $4{\rm E}_3$};

\end{scope}

\end{scope}

\end{tikzpicture}
\caption{Reductions to four distinct angles.}
\label{two_equal}
\end{figure}

The first row assumes two of $\alpha, \beta, \gamma$ are equal. We have the corresponding changes of labels:
\begin{description}
\item[${\rm 4A_1}$]  $\beta\doteq\gamma$: $\alpha, \beta, \gamma, \delta, \epsilon\to \beta, \alpha, \alpha, \gamma, \delta$.
\item[${\rm 4A_2}$]  $\alpha\doteq\gamma$: $\alpha, \beta, \gamma, \delta, \epsilon\to \alpha, \beta, \alpha, \gamma, \delta$.
\item[${\rm 4A_3}$] $\alpha\doteq\beta$: $\alpha, \beta, \gamma, \delta, \epsilon\to \alpha, \alpha, \beta, \gamma, \delta$.
\end{description}
The change of labels in the second reduction ${\rm 4A_2}$ is obtained as follows 
\[
\alpha, \beta, \gamma, \delta, \epsilon
\to \alpha, \beta, \alpha, \delta, \epsilon 
\to \alpha, \beta, \alpha, \gamma, \delta.
\]
The first $\to$ simply implements $\alpha\doteq\gamma$. The second $\to$ relabels the remaining $\beta, \delta, \epsilon$ in (tighter) lexicographic order. The change for ${\rm 4A_3}$ follows the same logic. The change for ${\rm 4A_1}$ is 
\[
\alpha, \beta, \gamma, \delta, \epsilon
\to \alpha, \beta, \beta, \delta, \epsilon 
\to \alpha, \beta, \beta, \gamma, \delta 
\to \beta, \alpha, \alpha, \gamma, \delta.
\]
The first two $\to$ follow the same logic as before, and the third $\to$ exchanges $\alpha$ and $\beta$.

With careful relabeling of the symbols (especially the further change for ${\rm 4A_1}$), all three reductions reduce AVC(5A24/60) to
\begin{align*}
{\text{AVC(4A24)}}
&=\{24\alpha^2\beta\gamma\delta
\colon 
24\alpha^2\beta,
8\gamma^3,
6\delta^4 \};   \\
\text{AVC(4A60)}
&=\{
60\alpha^2\beta\gamma\delta 
\colon 
60\alpha^2\beta,
20\gamma^3,
12\delta^5 \}.
\end{align*}
Even though the angle sum equations give $\alpha+2\beta=2\pi$, $\gamma=\tfrac{2}{3}\pi$, and $\delta=\tfrac{1}{2}\pi$ or $\tfrac{2}{5}\pi$, the angle values are irrelevant to our discussion, and are not assumed.

The second row of Figure \ref{two_equal} assumes that one of $\alpha, \beta, \gamma$ equals $\delta$. We have the corresponding changes of labels:
\begin{description}
\item[${\rm 4D_1}$] $\alpha\doteq\delta$: $\alpha, \beta, \gamma, \delta, \epsilon\to \alpha, \gamma, \beta, \alpha, \delta$.
\item[${\rm 4D_2}$] $\beta\doteq\delta$: $\alpha, \beta, \gamma, \delta, \epsilon\to \beta, \alpha, \gamma, \alpha, \delta$.
\item[${\rm 4D_3}$] $\gamma\doteq\delta$: $\alpha, \beta, \gamma, \delta, \epsilon\to \beta, \gamma, \alpha, \alpha, \delta$.
\end{description}
They reduce AVC(5A24/60) to
\begin{align*}
\text{AVC(4D24)}
&=\{
24\alpha^2\beta\gamma\delta 
\colon 
24\alpha\beta\gamma,
8\alpha^3,
6\delta^4 \}; \\
\text{AVC(4D60)}
&=\{
60\alpha^2\beta\gamma\delta 
\colon 
60\alpha\beta\gamma,
20\alpha^3,
12\delta^5 \}.
\end{align*}

The third row of Figure \ref{two_equal} assumes one of $\alpha, \beta, \gamma$ equals $\epsilon$. We have the corresponding changes of labels:
\begin{description}
\item[${\rm 4E_1}$] $\alpha\doteq\epsilon$: $\alpha, \beta, \gamma, \delta, \epsilon\to \alpha, \beta, \gamma, \delta, \alpha$.
\item[${\rm 4E_2}$] $\beta\doteq\epsilon$: $\alpha, \beta, \gamma, \delta, \epsilon\to \beta, \alpha, \gamma, \delta, \alpha$.
\item[${\rm 4E_3}$] $\gamma\doteq\epsilon$: $\alpha, \beta, \gamma, \delta, \epsilon\to \beta, \gamma, \alpha, \delta, \alpha$.
\end{description}
They reduce AVC(5A24/60) to 
\begin{align*}
\text{AVC(4E24)}
&=\{
24\alpha^2\beta\gamma\delta 
\colon 
24\alpha\beta\gamma,
8\delta^3,
6\alpha^4 \};  \\
\text{AVC(4E60)}
&=\{
60\alpha^2\beta\gamma\delta 
\colon 
60\alpha\beta\gamma,
20\delta^3,
12\alpha^5 \}.
\end{align*}

There are six possible angle arrangements for the pentagon $\alpha^2\beta\gamma\delta$, as shown in Figure \ref{2abcd_arrangement}. Since AVC(4D24/60) and AVC(4E24/60) are symmetric with respect to the exchange of $\beta, \gamma$, for tilings with these AVCs, we only need to consider the first four angle arrangements in Figure \ref{2abcd_arrangement}.

\begin{figure}[htp]
\centering
\begin{subfigure}[b]{0.15\linewidth}
\centering
\begin{tikzpicture}

\foreach \a in {0,...,4}{
\draw[rotate=72*\a]
	(18:0.6) -- (90:0.6);
}

\node[] at (90:0.4) {\scriptsize $\delta$};

\node at (162:0.4) {\scriptsize $\alpha$};
\node at (18:0.4) {\scriptsize $\alpha$};
\node at (234:0.4) {\scriptsize $\beta$};
\node at (-54:0.4) {\scriptsize $\gamma$};

\end{tikzpicture}
\caption{}
\label{Subfig-2abcd_arrangement-1}
\end{subfigure}
\begin{subfigure}[b]{0.15\linewidth}
\centering
\begin{tikzpicture}

\foreach \a in {0,...,4}{
\draw[rotate=72*\a]
	(18:0.6) -- (90:0.6);
}

\node[] at (90:0.4) {\scriptsize $\delta$};

\begin{scope}%[xshift=2 cm]

\node at (162:0.4) {\scriptsize $\beta$};
\node at (18:0.4) {\scriptsize $\gamma$};
\node at (234:0.4) {\scriptsize $\alpha$};
\node at (-54:0.4) {\scriptsize $\alpha$};

\end{scope}

\end{tikzpicture}
\caption{}
\label{Subfig-2abcd_arrangement-2}
\end{subfigure}
\begin{subfigure}[b]{0.15\linewidth}
\centering
\begin{tikzpicture}

\foreach \a in {0,...,4}{
\draw[rotate=72*\a]
	(18:0.6) -- (90:0.6);
}

\node[] at (90:0.4) {\scriptsize $\delta$};

\begin{scope}%[xshift=4 cm]

\node at (162:0.4) {\scriptsize $\alpha$};
\node at (18:0.4) {\scriptsize $\beta$};
\node at (234:0.4) {\scriptsize $\gamma$};
\node at (-54:0.4) {\scriptsize $\alpha$};

\end{scope}

\end{tikzpicture}
\caption{}
\label{Subfig-2abcd_arrangement-3}
\end{subfigure}
\begin{subfigure}[b]{0.15\linewidth}
\centering
\begin{tikzpicture}

\foreach \a in {0,...,4}{
\draw[rotate=72*\a]
	(18:0.6) -- (90:0.6);
}

\node[] at (90:0.4) {\scriptsize $\delta$};

\begin{scope}%[xshift=6 cm]

\node at (162:0.4) {\scriptsize $\alpha$};
\node at (18:0.4) {\scriptsize $\beta$};
\node at (234:0.4) {\scriptsize $\alpha$};
\node at (-54:0.4) {\scriptsize $\gamma$};

\end{scope}

\end{tikzpicture}
\caption{}
\label{Subfig-2abcd_arrangement-4}
\end{subfigure}
\begin{subfigure}[b]{0.15\linewidth}
\centering
\begin{tikzpicture}

\foreach \a in {0,...,4}{
\draw[rotate=72*\a]
	(18:0.6) -- (90:0.6);
}

\node[] at (90:0.4) {\scriptsize $\delta$};

\begin{scope}%[xshift=4 cm]

\node at (162:0.4) {\scriptsize $\alpha$};
\node at (18:0.4) {\scriptsize $\gamma$};
\node at (234:0.4) {\scriptsize $\beta$};
\node at (-54:0.4) {\scriptsize $\alpha$};

\end{scope}

\end{tikzpicture}
\caption{}
\label{Subfig-2abcd_arrangement-5}
\end{subfigure}
\begin{subfigure}[b]{0.15\linewidth}
\centering
\begin{tikzpicture}

\foreach \a in {0,...,4}{
\draw[rotate=72*\a]
	(18:0.6) -- (90:0.6);
}

\node[] at (90:0.4) {\scriptsize $\delta$};

\begin{scope}%[xshift=6 cm]

\node at (162:0.4) {\scriptsize $\alpha$};
\node at (18:0.4) {\scriptsize $\gamma$};
\node at (234:0.4) {\scriptsize $\alpha$};
\node at (-54:0.4) {\scriptsize $\beta$};

\end{scope}

\end{tikzpicture}
\caption{}
\label{Subfig-2abcd_arrangement-6}
\end{subfigure}
\caption{Angle arrangements for $\alpha^2\beta\gamma\delta$.}
\label{2abcd_arrangement}
\end{figure}

\medskip

\noindent{\bf Three distinct angles}

\medskip

We get three distinct angles by assuming three angles are equal, or assuming two pairs of angles are equal. 

If $\alpha, \beta, \gamma$ are equal, then we get
\begin{description}
\item[${\rm 3A}$] $\alpha\doteq\beta\doteq\gamma$: $\alpha, \beta, \gamma, \delta, \epsilon\to \alpha, \alpha, \alpha, \beta, \gamma$.
\end{description}
It reduces AVC(5A24/60) to
\begin{align*}
\text{AVC(3A24)}
&=\{
24\alpha^3\beta\gamma 
\colon 
24\alpha^3,
8\beta^3,
6\gamma^4 \};   \\
\text{AVC(3A60)}
&=\{
60\alpha^3\beta\gamma 
\colon 
60\alpha^3,
20\beta^3,
12\gamma^5 \}. 
\end{align*}

If two of $\alpha, \beta, \gamma$ equal $\delta$, then we get
\begin{description}
\item[${\rm 3B_1}$]  $\beta\doteq\gamma\doteq\delta$: $\alpha, \beta, \gamma, \delta, \epsilon\to \beta, \alpha, \alpha, \alpha, \gamma$.
\item[${\rm 3B_2}$]  $\alpha\doteq\gamma\doteq\delta$: $\alpha, \beta, \gamma, \delta, \epsilon\to \alpha, \beta, \alpha, \alpha, \gamma$.
\item[${\rm 3B_3}$] $\alpha\doteq\beta\doteq\delta$: $\alpha, \beta, \gamma, \delta, \epsilon\to \alpha, \alpha, \beta, \alpha, \gamma$.
\end{description}
They reduce AVC(5A24/60) to
\begin{align*}
\text{AVC(3B24)}
&=\{
24\alpha^3\beta\gamma 
\colon 
24\alpha^2\beta,
8\alpha^3,
6\gamma^4 \}; \\
\text{AVC(3B60)}
&=\{
60\alpha^3\beta\gamma 
\colon 
60\alpha^2\beta,
20\alpha^3,
12\gamma^5 \}. 
\end{align*}

If two of $\alpha, \beta, \gamma$ equal $\epsilon$, then we get
\begin{description}
\item[${\rm 3C_1}$] $\beta\doteq\gamma\doteq\epsilon$: $\alpha, \beta, \gamma, \delta, \epsilon\to \beta, \alpha, \alpha, \gamma, \alpha$.
\item[${\rm 3C_2}$] $\alpha\doteq\gamma\doteq\epsilon$: $\alpha, \beta, \gamma, \delta, \epsilon\to \alpha, \beta, \alpha, \gamma, \alpha$.
\item[${\rm 3C_3}$] $\alpha\doteq\beta\doteq\epsilon$: $\alpha, \beta, \gamma, \delta, \epsilon\to \alpha, \alpha, \beta, \gamma, \alpha$.
\end{description}
They reduce AVC(5A24/60) to
\begin{align*}
\text{AVC(3C24)}
&=\{
24\alpha^3\beta\gamma \colon 
24\alpha^2\beta,
8\gamma^3,
6\alpha^4 \};   \\
\text{AVC(3C60)}
&=\{
60\alpha^3\beta\gamma \colon 
60\alpha^2\beta,
20\gamma^3,
12\alpha^5 \}. 
\end{align*}

There are two possible angle arrangements for the pentagon $\alpha^3\beta\gamma$, as shown in Figure \ref{3abc_arrangement}. 

\begin{figure}[h!]
\centering
\begin{subfigure}[b]{0.15\linewidth}
\centering
\begin{tikzpicture}

\foreach \a in {0,...,4}
\draw[rotate=72*\a]
	(18:0.6) -- (90:0.6);
	
\node at (90:0.4) {\scriptsize $\alpha$};
\node at (162:0.4) {\scriptsize $\alpha$};
\node at (18:0.4) {\scriptsize $\alpha$};
\node at (234:0.4) {\scriptsize $\beta$};
\node at (-54:0.4) {\scriptsize $\gamma$};

\end{tikzpicture}
\caption{}
\label{Subfig-3abc_arrangement-1}
\end{subfigure}
\begin{subfigure}[b]{0.15\linewidth}
\centering
\begin{tikzpicture}

\raisebox{0.5ex}{
\foreach \a in {0,...,4}
\draw[rotate=72*\a]
	(18:0.6) -- (90:0.6);

\node at (90:0.4) {\scriptsize $\alpha$};
\node at (162:0.4) {\scriptsize $\beta$};
\node at (18:0.4) {\scriptsize $\gamma$};
\node at (234:0.4) {\scriptsize $\alpha$};
\node at (-54:0.4) {\scriptsize $\alpha$};
}
\end{tikzpicture}
\caption{}
\label{Subfig-3abc_arrangement-2}
\end{subfigure}
\caption{Angle arrangements for $\alpha^3\beta\gamma$.}
\label{3abc_arrangement}
\end{figure}

If two of $\alpha, \beta, \gamma$ are equal, and the third equals $\delta$, then we get:
\begin{description}
\item[${\rm 3D_1}$]  $\beta\doteq\gamma$ and $\alpha\doteq\delta$: $\alpha, \beta, \gamma, \delta, \epsilon\to \beta, \alpha, \alpha, \beta, \gamma$.
\item[${\rm 3D_2}$]  $\alpha\doteq\gamma$ and $\beta\doteq\delta$: $\alpha, \beta, \gamma, \delta, \epsilon\to \alpha, \beta, \alpha, \beta, \gamma$.
\item[${\rm 3D_3}$] $\alpha\doteq\beta$ and $\gamma\doteq\delta$: $\alpha, \beta, \gamma, \delta, \epsilon\to \alpha, \alpha, \beta, \beta, \gamma$.
\end{description}
They reduce AVC(5A24/60) to
\begin{align*}
\text{AVC(3D24)}
&=\{
24\alpha^2\beta^2\gamma 
\colon 
24\alpha^2\beta,
8\beta^3,
6\gamma^4 \};  \\
\text{AVC(3D60)} 
&=\{
60\alpha^2\beta^2\gamma 
\colon 
60\alpha^2\beta,
20\beta^3,
12\gamma^5 \}. 
\end{align*}

If two of $\alpha, \beta, \gamma$ are equal, and the third equals $\epsilon$, then we get 
\begin{description}
\item[${\rm 3E_1}$]  $\beta\doteq\gamma$ and $\alpha\doteq\epsilon$: $\alpha, \beta, \gamma, \delta, \epsilon\to \beta, \alpha, \alpha, \gamma, \beta$.
\item[${\rm 3E_2}$]  $\alpha\doteq\gamma$ and $\beta\doteq\epsilon$: $\alpha, \beta, \gamma, \delta, \epsilon\to \alpha, \beta, \alpha, \gamma, \beta$.
\item[${\rm 3E_3}$] $\alpha\doteq\beta$ and $\gamma\doteq\epsilon$: $\alpha, \beta, \gamma, \delta, \epsilon\to \alpha, \alpha, \beta, \gamma, \beta$.
\end{description}
They reduce AVC(5A24/60) to
\begin{align*}
\text{AVC(3E24)}
&=\{
24\alpha^2\beta^2\gamma 
\colon 
24\alpha^2\beta,
8\gamma^3,
6\beta^4 \}; \\
\text{AVC(3E60)}
&=\{
60\alpha^2\beta^2\gamma 
\colon 
60\alpha^2\beta,
20\gamma^3,
12\beta^5 \}.
\end{align*}

There are four possible angle arrangements for the pentagon $\alpha^2\beta^2\gamma$, given by Figure \ref{2a2bc_arrangement}. 

\begin{figure}[h!]
\centering
\begin{subfigure}[b]{0.15\linewidth}
\centering
\begin{tikzpicture}

\foreach \a in {0,...,4}
\draw[rotate=72*\a]
	(18:0.6) -- (90:0.6);

\node[] at (90:0.4) {\scriptsize $\gamma$};

\node at (162:0.4) {\scriptsize $\alpha$};
\node at (18:0.4) {\scriptsize $\alpha$};
\node at (234:0.4) {\scriptsize $\beta$};
\node at (-54:0.4) {\scriptsize $\beta$};

\end{tikzpicture}
\caption{}
\label{Subfig-2a2bc_arrangement-1}
\end{subfigure}
\begin{subfigure}[b]{0.15\linewidth}
\centering
\begin{tikzpicture}

\foreach \a in {0,...,4}
\draw[rotate=72*\a]
	(18:0.6) -- (90:0.6);

\node[] at (90:0.4) {\scriptsize $\gamma$};

\node at (162:0.4) {\scriptsize $\beta$};
\node at (18:0.4) {\scriptsize $\beta$};
\node at (234:0.4) {\scriptsize $\alpha$};
\node at (-54:0.4) {\scriptsize $\alpha$};

\end{tikzpicture}
\caption{}
\label{Subfig-2a2bc_arrangement-2}
\end{subfigure}
\begin{subfigure}[b]{0.15\linewidth}
\centering
\begin{tikzpicture}

\foreach \a in {0,...,4}
\draw[rotate=72*\a]
	(18:0.6) -- (90:0.6);

\node[] at (90:0.4) {\scriptsize $\gamma$};

\node at (162:0.4) {\scriptsize $\alpha$};
\node at (18:0.4) {\scriptsize $\beta$};
\node at (234:0.4) {\scriptsize $\beta$};
\node at (-54:0.4) {\scriptsize $\alpha$};

\end{tikzpicture}
\caption{}
\label{Subfig-2a2bc_arrangement-3}
\end{subfigure}
\begin{subfigure}[b]{0.15\linewidth}
\centering
\begin{tikzpicture}

\foreach \a in {0,...,4}
\draw[rotate=72*\a]
	(18:0.6) -- (90:0.6);

\node[] at (90:0.4) {\scriptsize $\gamma$};

\node at (162:0.4) {\scriptsize $\alpha$};
\node at (18:0.4) {\scriptsize $\beta$};
\node at (234:0.4) {\scriptsize $\alpha$};
\node at (-54:0.4) {\scriptsize $\beta$};

\end{tikzpicture}
\caption{}
\label{Subfig-2a2bc_arrangement-4}
\end{subfigure}
\caption{Angle arrangements for $\alpha^2\beta^2\gamma$.}
\label{2a2bc_arrangement}
\end{figure}

The various reductions of 5A are also related by reductions. For example, the reduction $\alpha,\beta,\gamma,\delta\to \alpha,\alpha,\beta,\gamma$ from 4A to 3A is compatible with the reductions from 5A to 4A and from 5A to 3A. In fact, the compatibility uniquely determines the reduction from 4A to 3A.

\begin{figure}[htp]
\centering
\begin{tikzpicture}[>=latex]

\node at (0,0) { 5A};
\node at (2,1) { 4A};
\node at (2,0) { 4E};
\node at (2,-1) { 4D};
\node at (4,2) { 3A};
\node at (4,1) { 3E};
\node at (4,0) { 3D};
\node at (4,-1) { 3C};
\node at (4,-2) { 3B};
\node at (6,0) { 2D};

\draw[->]
	(0.5,0) -- ++(1,0);
\draw[->]
	(0.5,0.3) -- ++(1,0.5);
\draw[->]
	(0.5,-0.3) -- ++(1,-0.5);

\draw[->]
	(2.5,1.3) -- ++(1,0.5);
\draw[->]
	(2.5,1) -- ++(1,0);
\draw[->]
	(2.5,0.7) -- ++(1,-0.5);
	
\draw[->]
	(2.5,-0.3) -- ++(1,-0.5);
\draw[->]
	(2.5,0.3) -- ++(1,0.5);
\draw[->]
	(2.5,-1.3) -- ++(1,-0.5);
\draw[->]
	(2.5,-0.7) -- ++(1,0.5);

\draw[->]
	(4.5,1.8) -- ++(1,-1.5);
\draw[->]
	(4.5,0) -- ++(1,0);
\draw[->]
	(4.5,-1.8) -- ++(1,1.5);
							
\end{tikzpicture}
\caption{Reduction relations for the reductions of AVC(5A24/60).}
\label{Fig-reduction}
\end{figure}

Figure \ref{Fig-reduction} shows all the reduction relations.  The following are the details of the reductions used in the proofs of the later theorems
\begin{itemize}
\item $\text{3A}\to \text{2D}$ (Theorem \ref{3Athm}): $\alpha,\alpha,\alpha,\beta,\gamma\to \alpha,\alpha,\alpha,\alpha,\beta$. 
\item $\text{3B}\to \text{2D}$ (Theorem \ref{3Bthm}): $\alpha,\alpha,\alpha,\beta,\gamma\to \alpha,\alpha,\alpha,\alpha,\beta$. 
\item $\text{3D}\to \text{2D}$ (Theorem \ref{3Dthm}): $\alpha,\alpha,\beta,\beta,\gamma\to \alpha,\alpha,\alpha,\alpha,\beta$. 
\item $\text{4A}\to \text{3A}$ (Theorem \ref{4Athm}): $\alpha,\alpha,\beta,\gamma,\delta\to \alpha,\alpha,\alpha,\beta,\gamma$.
\item $\text{4D}\to \text{3D}$ (Theorem \ref{4Dthm}): $\alpha,\alpha,\beta,\gamma,\delta\to \beta,\beta,\alpha,\alpha,\gamma$.
\item $\text{4E}\to \text{3E}$ (Theorem \ref{4Ethm}): $\alpha,\alpha,\beta,\gamma,\delta\to \beta,\beta,\alpha,\alpha,\gamma$.
\end{itemize}
The splitting is the reverse process. For example, to get 4D tilings from 3D tilings, we need to change $\gamma$ to $\delta$, change two $\beta$ to two $\alpha$, and change two $\alpha$ to one $\beta$ and one $\gamma$.

\section{Tilings with Three Distinct Angles}

\subsection{Tilings Reducible to 2D}

According to Figure \ref{Fig-reduction}, the reductions 3A, 3B, 3D can be further reduced to 2D. We obtain tilings for the three reductions by splitting the tiling in Theorem \ref{2Dthm}.

\begin{theorem}\label{3Athm}
The tilings for 
\begin{align*}
\text{\rm AVC(3A24)}
&=\{
24\alpha^3\beta\gamma 
\colon 
24\alpha^3,
8\beta^3,
6\gamma^4 \};   \\
\text{\rm AVC(3A60)}
&=\{
60\alpha^3\beta\gamma 
\colon 
60\alpha^3,
20\beta^3,
12\gamma^5 \},
\end{align*}
are the {\rm 3A} reductions of $PP_6$ and $PP_{12}$.
\end{theorem}

\begin{proof}
The pentagon has two possible angle arrangements, given by Figure \ref{3abc_arrangement}. If $\beta, \gamma$ are adjacent, as in Figure \ref{Subfig-3abc_arrangement-1}, then the angles around the vertex $\beta^3$ is one of the two cases $|^{\alpha}\beta^{\gamma}|^{\alpha}\beta^{\gamma}|^{\alpha}\beta^{\gamma}|$ and $|^{\alpha}\beta^{\gamma}|^{\alpha}\beta^{\gamma}|^{\gamma}\beta^{\alpha}|$ as in Figure \ref{Subfig-3AFigA-be3}. In both cases, we get a vertex $\alpha\gamma\cdots$. As this $\alpha\gamma\cdots$ is not in AVC(3A24/60), we get a contradiction.

\begin{figure}[htp]
\centering
\begin{subfigure}[t]{0.25\linewidth}
\centering
\begin{tikzpicture}[>=latex]

%% 1

\foreach \a in {0,1,2}
\foreach \b in {1,-1}
{
\begin{scope}[yshift=0.3cm+1.3*\b cm, rotate=120*\a]

\draw
	(0,0) -- (90:0.6) -- (50:1) -- (10:1) -- (-30:0.6)
	;

\node at (0,-0.2) {\scriptsize $\beta$};

\end{scope}
}

\foreach \a in {0,1,2}
{
\begin{scope}[yshift=1.6 cm, rotate=120*\a]

\node at (0.33,-0.38) {\scriptsize $\alpha$};
\node at (-0.33,-0.38) {\scriptsize $\gamma$};

\end{scope}
}

\node at (0,0.3) {\footnotesize $|^{\alpha}\beta^{\gamma}|^{\alpha}\beta^{\gamma}|^{\alpha}\beta^{\gamma}|$};

\begin{scope}[yshift=-1 cm]

\node at (-0.33,-0.38) {\scriptsize $\alpha$};
\node at (0.33,-0.38) {\scriptsize $\gamma$};

\node at (-0.5,-0.1) {\scriptsize $\alpha$};
\node at (0.5,-0.1) {\scriptsize $\gamma$};

\node at (0.18,0.47) {\scriptsize $\alpha$};
\node at (-0.18,0.47) {\scriptsize $\gamma$};

\node at (0,-1.3) {\footnotesize $|^{\alpha}\beta^{\gamma}|^{\alpha}\beta^{\gamma}|^{\gamma}\beta^{\alpha}|$};

\end{scope}

\end{tikzpicture}
\caption{}
\label{Subfig-3AFigA-be3}
\end{subfigure}
\begin{subfigure}[t]{0.5\linewidth}
\centering
\begin{tikzpicture}[>=latex]

%% 2

\begin{scope}[xshift=4.5cm]

\foreach \a in {0,...,3}
{
\begin{scope}[rotate=90*\a]

\draw
	(0,0) -- (0.6,0) -- (0.9,0.5) -- (0.5,0.9) -- (0,0.6)
	(0.9,-0.5) -- (2,-0.5)
	(0.9,0.5) -- (1.4,0.5)
	(2,1.4) -- (-1.4,1.4)
	(-2,2) -- (2,2) -- (2.4,2.4);

\fill
	(-0.5,0.9) circle (0.05)
	(2,1.4) circle (0.05);
	
\node at (0.15,0.15) {\scriptsize $\gamma$};
\node at (0.65,1.25) {\scriptsize $\gamma$};
\node at (0.35,1.25) {\scriptsize $\gamma$};
\node at (0.65,1.55) {\scriptsize $\gamma$};
\node at (0.35,1.55) {\scriptsize $\gamma$};
\node at (0,2.4) {\scriptsize $\gamma$};

\node at (0.5,0.15) {\scriptsize $\alpha$};
\node at (0.15,0.5) {\scriptsize $\alpha$};
\node at (0.35,1) {\scriptsize $\alpha$};
\node at (0,0.8) {\scriptsize $\alpha$};
\node at (0.65,0.95) {\scriptsize $\alpha$};
\node at (1.25,0.65) {\scriptsize $\alpha$};
\node at (-0.35,1.25) {\scriptsize $\alpha$};
\node at (1.25,1.25) {\scriptsize $\alpha$};
\node at (-0.5,1.55) {\scriptsize $\alpha$};
\node at (1.4,1.55) {\scriptsize $\alpha$};
\node at (0.35,1.85) {\scriptsize $\alpha$};
\node at (0.65,1.85) {\scriptsize $\alpha$};
\node at (1.95,2.15) {\scriptsize $\alpha$};
\node at (2.15,1.95) {\scriptsize $\alpha$};
\node at (0.45,0.7) {\scriptsize $\alpha$};
\node at (1.55,1.25) {\scriptsize $\alpha$};
\node at (2.15,-0.5) {\scriptsize $\alpha$};

\node at (0.7,0.45) {\scriptsize $\beta$};
\node at (-0.35,1) {\scriptsize $\beta$};
\node at (0.95,0.65) {\scriptsize $\beta$};
\node at (1.85,1.25) {\scriptsize $\beta$};
\node at (1.85,1.55) {\scriptsize $\beta$};
\node at (2.15,1.4) {\scriptsize $\beta$};

\end{scope}
}

\foreach \a in {0,...,3}
{
\filldraw[fill=white, rotate=90*\a] 
	(0.5,1.4) circle (0.05)
	(0,0) circle (0.05)
	(2.4,2.4) circle (0.05);
}

\end{scope}

\end{tikzpicture}
\caption{}
\label{Subfig-3AFigA-f24}
\end{subfigure}
\caption{Tiling for AVC(3A24).}
\label{3AFigA}
\end{figure}

Therefore, $\beta, \gamma$ are non-adjacent in the pentagon, which means $\gamma$ is adjacent to two $\alpha$, as in Figure \ref{Subfig-3abc_arrangement-2}. As explained at the end of Section \ref{Sec-Reductions}, the splitting from AVC(2D24/60) to AVC(3A24/60) means that, in the 2D tilings in Theorem \ref{2Dthm}, we change $\beta$ to $\gamma$, and change one of the four $\alpha$ angles to $\beta$. Moreover, we need to keep $\beta, \gamma$ non-adjacent in all the tiles.

For AVC(3A24), we change all the $\beta$ angles in Figure \ref{2DfigA} to get all the $\gamma$ angles in Figure \ref{Subfig-3AFigA-f24}. Then we use $\gamma$ adjacent to two $\alpha$, and $\alpha\cdots=\alpha^3$ to get all the $\alpha$ angles in the figure. Then we know the $\gamma$ and three $\alpha$ angles for each tile. Therefore, the remaining angles are $\beta$, yielding the 3A reduction of the $PP_6$ in Figure \ref{Subfig-3AFigA-f24}. The same argument applies to AVC(3A60).
\end{proof}

\begin{theorem}\label{3Bthm}
The tilings for
\begin{align*}
\text{\rm AVC(3B24)}
&=\{
24\alpha^3\beta\gamma 
\colon 
24\alpha^2\beta,
8\alpha^3,
6\gamma^4 \}; \\
\text{\rm AVC(3B60)}
&=\{
60\alpha^3\beta\gamma 
\colon 
60\alpha^2\beta,
20\alpha^3,
12\gamma^5 \},
\end{align*}
are the $3{\rm B}_1$, $3{\rm B}_2$, $3{\rm B}_3$ reductions of $PP_6$ and $PP_{12}$, and the following modifications:
\begin{enumerate}
\item If $\beta, \gamma$ are adjacent in the pentagon, then each $N(\gamma^4/\gamma^5)$ can be independently oriented.
\item If $\beta, \gamma$ are non-adjacent, then each vertex $\bullet$ can be independently configured like any in Figure \ref{Subfig-3BFigB-2}.
\end{enumerate}
\end{theorem}

The tilings for AVC(3B24) are shown in Figure \ref{3BFigA} ($\beta, \gamma$ adjacent) and Figure \ref{Subfig-3BFigB-1} ($\beta, \gamma$ non-adjacent). In the first tiling, each $N(\gamma^4)$ has an orientation given by $\alpha\to\beta$ along the boundary of $N(\gamma^4)$, and the orientations can be independently chosen. Figures \ref{3BFigAb} and \ref{3BFigAc} shows the 3d pictures, in which the front $N(\gamma^4)$ have different orientations, and the other $N(\gamma^4)$ have the same orientations. In the second tiling, the angles around $\bullet$ vertices can be any in Figure \ref{Subfig-3BFigB-2}. This means either all three $\alpha$ at $\bullet$, or two $\alpha$ and one $\beta$ at $\bullet$.

Figure \ref{3BFigA} shows the $3{\rm B}_1$, $3{\rm B}_3$ reductions of $PP_6$, and their modifications. The orientation modifications are in one-to-one correspondence with the assignment of $(+)$ and $(-)$ to the faces of the cube and the dodecahedron (treating mirror image assignments as distinct). The total numbers of assignments are $10$ for $P_6$ and $96$ for $P_{12}$. 

Figure \ref{Subfig-3BFigB-1} shows the $3{\rm B}_2$ reduction of $PP_6$, and their modifications. The total number of such tilings is $2836$. The reduction of $PP_{12}$ is similar, however, counting the total number of tilings in this case is computationally non-trivial and we leave it as an open problem.

\begin{proof}
The 2D tilings in Theorem \ref{2Dthm} are the unions of $N(\beta^4/\beta^5)$. As explained at the end of Section \ref{Sec-Reductions}, the splitting from AVC(2D24/60) to AVC(3B24/60) means changing $\beta$ to $\gamma$, and changing one of the four $\alpha$ angles to $\beta$. We apply the splitting to each $N(\beta^4)$.

The pentagon has two possible angle arrangements, given by Figure \ref{3abc_arrangement}. If $\beta, \gamma$ are adjacent, as in Figure \ref{Subfig-3abc_arrangement-1}, by $\beta^2\cdots$ not in AVC(3B24), the AAD of $\gamma^4$ is $|^{\alpha}\gamma^{\beta}|^{\alpha}\gamma^{\beta}|^{\alpha}\gamma^{\beta}|^{\alpha}\gamma^{\beta}|$. This means that $N(\gamma^4)$ is given by the center four tiles in Figure \ref{3BFigAa}, or their flips. Equivalently, $N(\gamma^4)$ can take any of the two orientations given by the direction of $\alpha\to\beta$. Therefore, the tiling is determined by the independent choices of the orientations in the six instances of $N(\gamma^4)$. Figures \ref{3BFigAb} and \ref{3BFigAc} are 3d pictures of two versions of the tiling, in which the small red triangles indicate all the $\beta$ angles. The only difference between the two versions is the orientation of the front $N(\gamma^4)$.

\begin{figure}[htp]
\centering
\begin{subfigure}[b]{0.4\linewidth}
\centering
\begin{tikzpicture}[>=latex,scale=1]

\foreach \a in {0,...,3}
\draw[green!70, line width=2, rotate=90*\a]
	(0.9,0.5) -- (0.5,0.9) -- (0,0.6) -- (-0.5,0.9) -- (-0.5,1.4) -- (-1.4,1.4) -- (-1.4,2) -- (2,2) -- (2,1.4);
	
\draw[->]
	(210:0.35) arc (210:-30:0.35);

\foreach \a in {0,...,3}
{
\begin{scope}[
rotate=90*\a]

\draw
	(0,0) -- (0.6,0) -- (0.9,0.5) -- (0.5,0.9) -- (0,0.6)
	(0.9,-0.5) -- (2,-0.5)
	(0.9,0.5) -- (1.4,0.5)
	(2,1.4) -- (-1.4,1.4)
	(-2,2) -- (2,2) -- (2.4,2.4);

\node at (0.15,0.15) {\scriptsize $\gamma$};
\node at (0.65,1.25) {\scriptsize $\gamma$};
\node at (0.35,1.25) {\scriptsize $\gamma$};
\node at (0.65,1.55) {\scriptsize $\gamma$};
\node at (0.35,1.55) {\scriptsize $\gamma$};
\node at (0,2.4) {\scriptsize $\gamma$};
	
\end{scope}

\begin{scope}[rotate=90*\a]

\draw[shift={(2.4cm,2.4cm)}, ->]
	(180:0.3) arc (180:270:0.3);
	
\node at (-0.35,1) {\scriptsize $\alpha$};
\node at (0,0.8) {\scriptsize $\alpha$};
\node at (0.95,0.65) {\scriptsize $\alpha$};
\node at (1.25,0.65) {\scriptsize $\alpha$};

\node at (0.5,0.15) {\scriptsize $\alpha$};
\node at (0.15,0.5) {\scriptsize $\beta$};
\node at (0.45,0.7) {\scriptsize $\alpha$};
\node at (0.7,0.45) {\scriptsize $\alpha$};

\node at (1.95,2.15) {\scriptsize $\alpha$};
\node at (0.5,2.15) {\scriptsize $\alpha$};
\node at (-1.4,2.15) {\scriptsize $\alpha$};
\node at (2.15,1.95) {\scriptsize $\beta$};

\node at (1.85,1.85) {\scriptsize $\alpha$};
\node at (1.85,1.55) {\scriptsize $\alpha$};
\node at (1.85,1.25) {\scriptsize $\alpha$};
\node at (1.55,1.25) {\scriptsize $\alpha$};

\end{scope}
}

\foreach \a in {0,2}
{
\begin{scope}[rotate=90*\a]

\draw[shift={(0.5cm,1.4cm)}, ->]
	(-30:0.35) arc (-30:210:0.35);

\draw[shift={(1.4cm,-0.5cm)}, ->]
	(-120:0.35) arc (-120:120:0.35);

\node at (0.35,1) {\scriptsize $\alpha$};
\node at (0.65,0.95) {\scriptsize $\beta$};
\node at (-0.35,1.25) {\scriptsize $\beta$};
\node at (1.25,1.25) {\scriptsize $\alpha$};
\node at (-0.5,1.55) {\scriptsize $\alpha$};
\node at (1.4,1.55) {\scriptsize $\beta$};
\node at (0.35,1.85) {\scriptsize $\beta$};
\node at (0.65,1.85) {\scriptsize $\alpha$};

\node at (1,-0.35) {\scriptsize $\alpha$};
\node at (0.95,-0.65) {\scriptsize $\beta$};
\node at (1.25,0.35) {\scriptsize $\beta$};
\node at (1.25,-1.25) {\scriptsize $\alpha$};
\node at (1.55,0.5) {\scriptsize $\alpha$};
\node at (1.55,-1.4) {\scriptsize $\beta$};
\node at (1.85,-0.35) {\scriptsize $\beta$};
\node at (1.85,-0.65) {\scriptsize $\alpha$};

\end{scope}
}

\end{tikzpicture}	
\caption{}
\label{3BFigAa}
\end{subfigure}
\begin{subfigure}[b]{0.25\linewidth}
\centering
\begin{tikzpicture}[>=latex,scale=1]

\raisebox{1.1cm}{

\pgftext{
	\includegraphics[scale=0.071]{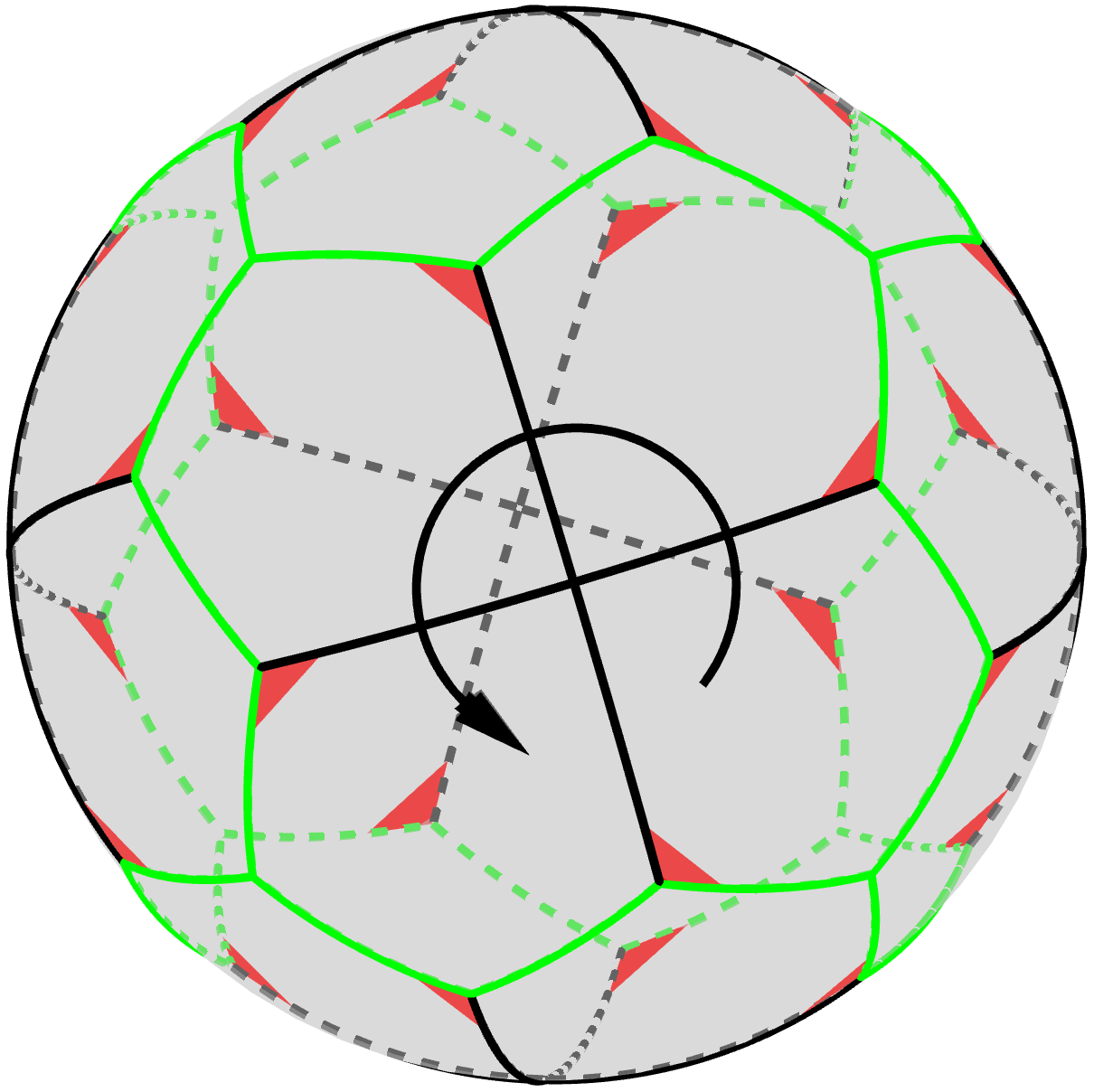}
	};
	
}

\end{tikzpicture}	
\caption{}
\label{3BFigAb}
\end{subfigure}
\begin{subfigure}[b]{0.25\linewidth}
\centering
\begin{tikzpicture}[>=latex,scale=1]

\raisebox{1.1cm}{

\pgftext{
	\includegraphics[scale=0.064]{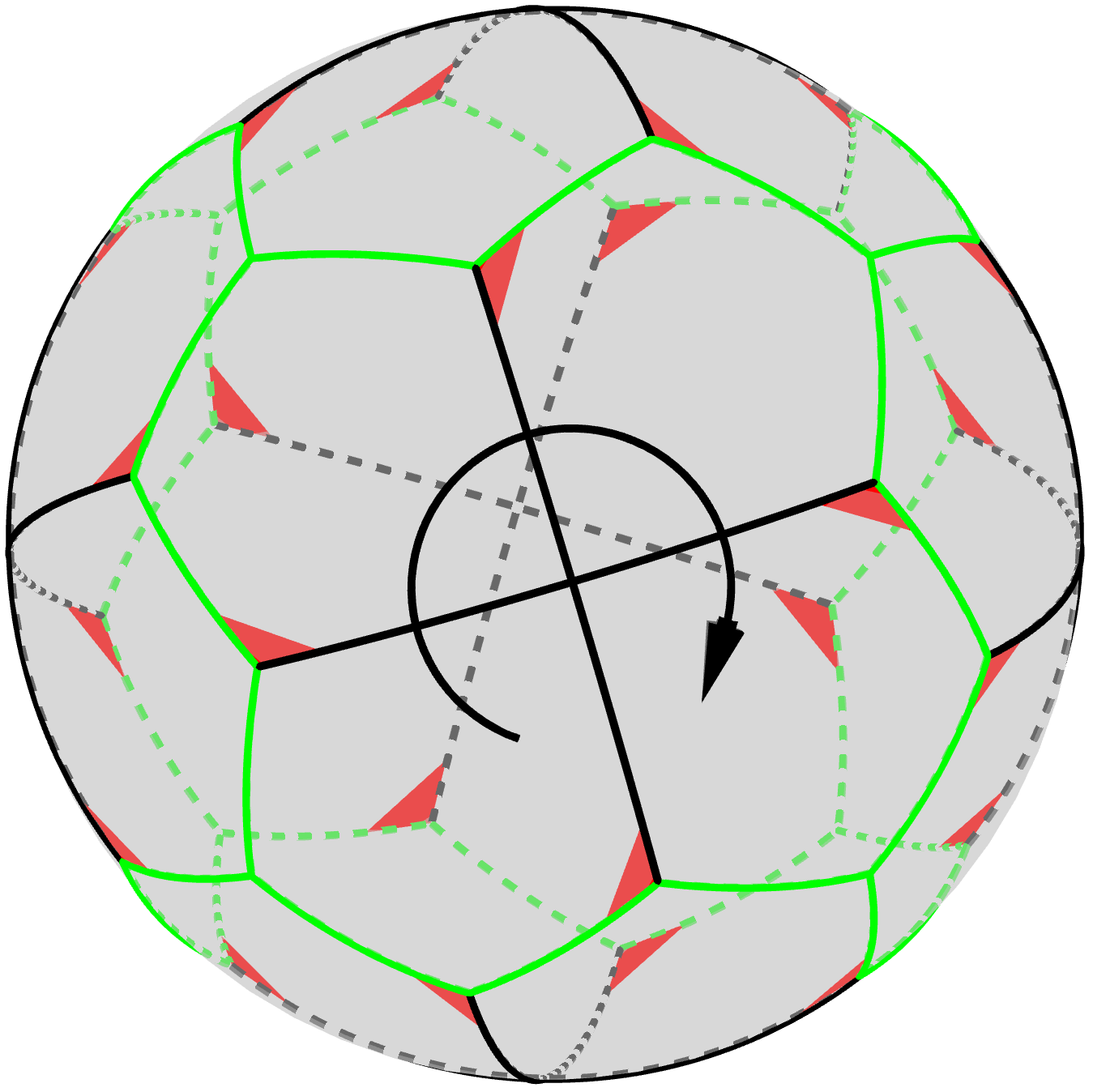}
	};

}

\end{tikzpicture}	
\caption{}
\label{3BFigAc}
\end{subfigure}
\caption{Tiling for AVC(3B24), $\beta, \gamma$ adjacent. In the 3d pictures, the red corners are $\beta$, and $\gamma$ are concentrated at the meeting places of four black lines.}
\label{3BFigA}
\end{figure}

If $\beta, \gamma$ are non-adjacent, as in Figure \ref{Subfig-3abc_arrangement-2}, then $\gamma$ is adjacent to two $\alpha$. Therefore, $N(\gamma^4)$ is given by the center four tiles in Figure \ref{Subfig-3BFigB-1}, with one $\alpha$ and one $\beta$ at the ends of each thick gray line. It remains to assign $\beta$ to the ends of the thick gray lines. Since the only vertex involving $\beta$ is $\alpha^2\beta$, the only constraint is that there are at most one $\beta$ at each vertex. This means the three $\beta$ angles around each $\bullet$ vertex must be like those in Figure \ref{Subfig-3BFigB-2}. 

\begin{figure}[htp]
\centering
\begin{subfigure}[t]{0.4\linewidth}
\centering
\begin{tikzpicture}[>=latex]

\foreach \a in {0,...,3}
{
\begin{scope}[rotate=90*\a]

\draw
	(0,0) -- (0.6,0) -- (0.9,0.5) -- (0.5,0.9) -- (0,0.6)
	(0.9,-0.5) -- (2,-0.5)
	(0.9,0.5) -- (1.4,0.5)
	(2,1.4) -- (-1.4,1.4)
	(-2,2) -- (2,2) -- (2.4,2.4);

\node at (0.15,0.15) {\scriptsize $\gamma$};
\node at (0.65,1.25) {\scriptsize $\gamma$};
\node at (0.35,1.25) {\scriptsize $\gamma$};
\node at (0.65,1.55) {\scriptsize $\gamma$};
\node at (0.35,1.55) {\scriptsize $\gamma$};
\node at (0,2.4) {\scriptsize $\gamma$};
	
\end{scope}

}

\foreach \a in {0,...,3}
{
\begin{scope}[rotate=90*\a]

\fill
	(-0.5,0.9) circle (0.05)
	(2,1.4) circle (0.05);
	
\draw[gray, very thick]
	(-1.35,1.5) -- ++(0,0.4)
	(-1.3,2.05) -- ++(1.8,0)
	(-1.5,1.95) -- ++(-0.4,0)
	(-0.55,1.3) -- ++(0,-0.3)
	(0.82,0.5) -- (0.5,0.82)
	(0.9,0.4) -- (0.67,0);

\node at (0.5,0.15) {\scriptsize $\alpha$};
\node at (0.15,0.5) {\scriptsize $\alpha$};
\node at (0.35,1) {\scriptsize $\alpha$};
\node at (0.65,0.95) {\scriptsize $\alpha$};
\node at (-0.35,1.25) {\scriptsize $\alpha$};
\node at (1.25,1.25) {\scriptsize $\alpha$};
\node at (-0.5,1.55) {\scriptsize $\alpha$};
\node at (1.4,1.55) {\scriptsize $\alpha$};
\node at (0.35,1.85) {\scriptsize $\alpha$};
\node at (0.65,1.85) {\scriptsize $\alpha$};
\node at (1.95,2.15) {\scriptsize $\alpha$};
\node at (2.15,1.95) {\scriptsize $\alpha$};
	
\end{scope}
}

\foreach \a in {0,...,3}
\filldraw[fill=white, 
rotate=90*\a] 
	(0.5,1.4) circle (0.05)
	(0,0) circle (0.05)
	(2.4,2.4) circle (0.05);

\end{tikzpicture}
\caption{}
\label{Subfig-3BFigB-1}
\end{subfigure}
\begin{subfigure}[t]{0.4\linewidth}
\centering
\begin{tikzpicture}[>=latex]

\raisebox{10ex}{

\foreach \a in {0,1,2}
\foreach \b in {1,-1}
{
\begin{scope}[xshift=\b cm, rotate=120*\a]

\draw
	(0,0) -- (0,0.7);
	
\draw[gray, very thick]
	(-0.05,0.1) -- (-0.05,0.7);
\fill
	(0,0) circle (0.07);

\end{scope}
}

\foreach \a in {0,1,2}
{
\begin{scope}[xshift=-1 cm, rotate=120*\a]

\node at (-0.2,0.6) {\scriptsize $\beta$};
\node at (-0.2,0.12) {\scriptsize $\alpha$};

\end{scope}
}

\begin{scope}[xshift=1 cm]

\node at (-0.2,0.12) {\scriptsize $\beta$};
\node at (-0.2,0.6) {\scriptsize $\alpha$};
\node at (0.2,0.12) {\scriptsize $\alpha$};
\node at (-0.43,-0.48) {\scriptsize $\beta$};
\node at (0.57,-0.1) {\scriptsize $\beta$};
\node at (0,-0.2) {\scriptsize $\alpha$};

\end{scope}

}

\end{tikzpicture}
\caption{}
\label{Subfig-3BFigB-2}
\end{subfigure}
\caption{Tiling for AVC(3B24), $\beta, \gamma$ non-adjacent.}
\label{3BFigB}
\end{figure}

The discussion of the splitting from AVC(2D60) to AVC(3B60) is similar.  
\end{proof}

\begin{theorem}\label{3Dthm}
The tilings for 
\begin{align*}
\text{\rm AVC(3D24)}
&=\{
24\alpha^2\beta^2\gamma 
\colon 
24\alpha^2\beta,
8\beta^3,
6\gamma^4 \},  \\
\text{\rm AVC(3D60)} 
&=\{
60\alpha^2\beta^2\gamma 
\colon 
60\alpha^2\beta,
20\beta^3,
12\gamma^5 \}, 
\end{align*}
are the $3{\rm D}_1$, $3{\rm D}_2$, $3{\rm D}_3$ reductions of $PP_6$ and $PP_{12}$, and the two tilings in Figures \ref{Subfig-3DFigE-T1} and \ref{3DFigF}.
\end{theorem}

The tilings in Figures \ref{Subfig-3DFigE-T1} and \ref{3DFigF} are only for AVC(3D24), in which the pentagons  are the same as the ones for $3{\rm D}_3$ and $3{\rm D}_1$. 

\begin{proof}
The splitting from AVC(2D24/60) to AVC(3D24/60) means changing $\beta$ to $\gamma$, and changing four $\alpha$ angles to two $\alpha$ angles and two $\beta$ angles, doing so according to the arrangements of angles in Figure \ref{2a2bc_arrangement}. 

For the pentagons in Figures \ref{Subfig-2a2bc_arrangement-1} and \ref{Subfig-2a2bc_arrangement-2}, the splittings are unique for each tile. The first splitting gives the $3{\rm D}_2$ reductions of $PP_6$ and $PP_{12}$. The second splitting gives tilings with $\alpha\beta^2$, which is not in AVC(3D24/60). Therefore, the second splitting is not a feasible tiling for AVC(3D24/60).

The splittings according to Figures \ref{Subfig-2a2bc_arrangement-3}, \ref{Subfig-2a2bc_arrangement-4} are not unique. In fact, each 2D tile has two ways of changing four $\alpha$ into two $\alpha$ angles and two $\beta$ angles. To handle the flexibility, we study the splitting that takes $N(\beta^4/\beta^5)$ to $N(\gamma^4/\gamma^5)$. 

Figure \ref{3DFigC} shows all the possibilities for the splitting of a part of Figure \ref{2DfigA} according to the pentagon in Figure \ref{Subfig-2a2bc_arrangement-3}. It shows that, around a vertex $\gamma^4$, the angle arrangements of $\circled{1}, \circled{2}$ and the choice of $\gamma_4$ determine the other tiles. Specifically, when $\gamma_1|\gamma_2$ is $\gamma^{\alpha}|^{\alpha}\gamma, \gamma^{\beta}|^{\beta}\gamma, \gamma^{\alpha}|^{\beta}\gamma, \gamma^{\beta}|^{\alpha}\gamma$, we find that the corresponding $\gamma_2|\gamma_3$ is $\gamma^{\beta}|^{\beta}\gamma, \gamma^{\alpha}|^{\alpha}\gamma, \gamma^{\alpha}|^{\beta}\gamma, \gamma^{\beta}|^{\alpha}\gamma$, respectively. Moreover, the AAD of $\gamma_1|\gamma_2$ determines the AADs of $\gamma_4|\gamma_5$ and $\gamma_6|\gamma_7$.

\begin{figure}[htp]
\centering
\begin{subfigure}{0.225\linewidth} % 1
\centering
\begin{tikzpicture}[>=latex]

\draw[green!70, line width=2]
	(-0.9,0.5) -- (-0.5,0.9) -- (0,0.6) -- (0.5,0.9) -- (0.9,0.5) -- (0.6,0) -- (0.9,-0.5) -- (0.5,-0.9) -- (0,-0.6)
	(0.5,-0.9) -- (0.5,-1.4)
	(-0.5,1.4) -- (-0.5,0.9)
	(0.9,0.5) -- (1.4,0.5);

\begin{scope}[]

\foreach \a in {-1,0,1}
{
\begin{scope}[rotate=90*\a]

\draw
	(0,0) -- (0.6,0) -- (0.9,0.5) -- (0.5,0.9) -- (0,0.6) -- (0,0);

\node at (0.15,0.15) {\scriptsize $\gamma$};

\end{scope}
}
	
\foreach \a in {-1,0}
{
\begin{scope}[rotate=90*\a]

\draw
	(-0.5,0.9) -- (-0.5,1.4) -- (1.4,1.4) -- (1.4,0.5) -- (0.9,0.5)
	(0.5,0.9) -- (0.5,1.4);

\node at (0.35,1.25) {\scriptsize $\gamma$};
\node at (0.65,1.25) {\scriptsize $\gamma$};
	
\end{scope}
}

\end{scope}

\filldraw[fill=white]
	(0.05,0.05) -- ++(-0.1,0) -- ++(0,-0.1) -- ++(0.1,0) -- cycle
	(0.55,1.45) -- ++(-0.1,0) -- ++(0,-0.1) -- ++(0.1,0) -- cycle
	(1.4,-0.5) circle (0.05)
;	

\node at (0.35,1) {\scriptsize $\alpha$};
\node at (-0.35,1) {\scriptsize $\alpha$};
\node at (0,0.8) {\scriptsize $\beta$};
\node at (-0.35,1.25) {\scriptsize $\beta$};

\node at (1,0.35) {\scriptsize $\alpha$};
\node at (1,-0.35) {\scriptsize $\alpha$};
\node at (0.8,0) {\scriptsize $\beta$};
\node at (1.25,0.35) {\scriptsize $\beta$};

\node at (-0.5,0.15) {\scriptsize $\beta$};
\node at (-0.15,0.5) {\scriptsize $\alpha$};
\node at (-0.5,0.7) {\scriptsize $\beta$};
\node at (-0.7,0.45) {\scriptsize $\alpha$};

\node at (0.5,-0.15) {\scriptsize $\beta$};
\node at (0.15,-0.5) {\scriptsize $\alpha$};
\node at (0.5,-0.7) {\scriptsize $\beta$};
\node at (0.7,-0.45) {\scriptsize $\alpha$};

\node at (0.65,0.95) {\scriptsize $\alpha$};
\node at (0.95,0.65) {\scriptsize $\beta$};
\node at (1.25,1.25) {\scriptsize $\beta$};
\node at (1.25,0.65) {\scriptsize $\alpha$};

\node at (0.5,0.15) {\scriptsize $\beta$};
\node at (0.15,0.5) {\scriptsize $\alpha$};
\node at (0.45,0.7) {\scriptsize $\beta$};
\node at (0.7,0.45) {\scriptsize $\alpha$};

\node at (0.65,-0.95) {\scriptsize $\alpha$};
\node at (0.95,-0.65) {\scriptsize $\beta$};
\node at (1.25,-1.25) {\scriptsize $\alpha$};
\node at (0.65,-1.25) {\scriptsize $\beta$};

\node[inner sep=0.5,draw,shape=circle] at (-0.43,0.43) {\scriptsize $1$};
\node[inner sep=0.5,draw,shape=circle] at (0.43,0.43) {\scriptsize $2$};
\node[inner sep=0.5,draw,shape=circle] at (0.43,-0.43) {\scriptsize $3$};
\node[inner sep=0.5,draw,shape=circle] at (0,1.15) {\scriptsize $4$};
\node[inner sep=0.5,draw,shape=circle] at (1,1) {\scriptsize $5$};
\node[inner sep=0.5,draw,shape=circle] at (1.15,0) {\scriptsize $6$};
\node[inner sep=0.5,draw,shape=circle] at (1,-1) {\scriptsize $7$};
\end{tikzpicture}
\caption{}
\label{Subfig-3DFigC-1}
\end{subfigure}
\begin{subfigure}{0.225\linewidth} % 2
\centering
\begin{tikzpicture}[>=latex]

\draw[green!70, line width=2]
	(-0.9,0.5) -- (-0.5,0.9) -- (0,0.6) -- (0.5,0.9) -- (0.9,0.5) -- (0.6,0) -- (0.9,-0.5) -- (0.5,-0.9) -- (0,-0.6)
	(0.5,-0.9) -- (0.5,-1.4)
	(-0.5,1.4) -- (-0.5,0.9)
	(0.9,0.5) -- (1.4,0.5);

\begin{scope}[]

\foreach \a in {-1,0,1}
{
\begin{scope}[rotate=90*\a]

\draw
	(0,0) -- (0.6,0) -- (0.9,0.5) -- (0.5,0.9) -- (0,0.6) -- (0,0);

\node at (0.15,0.15) {\scriptsize $\gamma$};

\end{scope}
}
	
\foreach \a in {-1,0}
{
\begin{scope}[rotate=90*\a]

\draw
	(-0.5,0.9) -- (-0.5,1.4) -- (1.4,1.4) -- (1.4,0.5) -- (0.9,0.5)
	(0.5,0.9) -- (0.5,1.4);

\node at (0.35,1.25) {\scriptsize $\gamma$};
\node at (0.65,1.25) {\scriptsize $\gamma$};
	
\end{scope}
}

\end{scope}

\foreach \a in {0}
{
\filldraw[fill=white]
	(0.05,0.05) -- ++(-0.1,0) -- ++(0,-0.1) -- ++(0.1,0) -- cycle
	(1.45,-0.45) -- ++(-0.1,0) -- ++(0,-0.1) -- ++(0.1,0) -- cycle
	(0.5,1.4) circle (0.05)
;	
}

\node at (0.35,1) {\scriptsize $\alpha$};
\node at (-0.35,1) {\scriptsize $\alpha$};
\node at (0,0.8) {\scriptsize $\beta$};
\node at (-0.35,1.25) {\scriptsize $\beta$};

\node at (1,0.35) {\scriptsize $\alpha$};
\node at (1,-0.35) {\scriptsize $\alpha$};
\node at (0.8,0) {\scriptsize $\beta$};
\node at (1.25,0.35) {\scriptsize $\beta$};

\node at (-0.5,0.15) {\scriptsize $\alpha$};
\node at (-0.15,0.5) {\scriptsize $\beta$};
\node at (-0.5,0.7) {\scriptsize $\alpha$};
\node at (-0.7,0.45) {\scriptsize $\beta$};

\node at (0.5,-0.15) {\scriptsize $\alpha$};
\node at (0.15,-0.5) {\scriptsize $\beta$};
\node at (0.5,-0.7) {\scriptsize $\alpha$};
\node at (0.7,-0.45) {\scriptsize $\beta$};

\node at (0.65,0.95) {\scriptsize $\beta$};
\node at (0.95,0.65) {\scriptsize $\alpha$};
\node at (1.25,1.25) {\scriptsize $\alpha$};
\node at (1.25,0.65) {\scriptsize $\beta$};

\node at (0.5,0.15) {\scriptsize $\alpha$};
\node at (0.15,0.5) {\scriptsize $\beta$};
\node at (0.45,0.7) {\scriptsize $\alpha$};
\node at (0.7,0.45) {\scriptsize $\beta$};

\node at (0.65,-0.95) {\scriptsize $\beta$};
\node at (0.95,-0.65) {\scriptsize $\alpha$};
\node at (1.25,-1.25) {\scriptsize $\beta$};
\node at (0.65,-1.25) {\scriptsize $\alpha$};

\end{tikzpicture}
\caption{}
\label{Subfig-3DFigC-2}
\end{subfigure}
\begin{subfigure}{0.225\linewidth} % 3
\centering
\begin{tikzpicture}[>=latex]

\draw[green!70, line width=2]
	(-0.9,0.5) -- (-0.5,0.9) -- (0,0.6) -- (0.5,0.9) -- (0.9,0.5) -- (0.6,0) -- (0.9,-0.5) -- (0.5,-0.9) -- (0,-0.6)
	(0.5,-0.9) -- (0.5,-1.4)
	(-0.5,1.4) -- (-0.5,0.9)
	(0.9,0.5) -- (1.4,0.5);

\begin{scope}[]

\foreach \a in {-1,0,1}
{
\begin{scope}[rotate=90*\a]

\draw
	(0,0) -- (0.6,0) -- (0.9,0.5) -- (0.5,0.9) -- (0,0.6) -- (0,0);

\node at (0.15,0.15) {\scriptsize $\gamma$};

\end{scope}
}
	
\foreach \a in {-1,0}
{
\begin{scope}[rotate=90*\a]

\draw
	(-0.5,0.9) -- (-0.5,1.4) -- (1.4,1.4) -- (1.4,0.5) -- (0.9,0.5)
	(0.5,0.9) -- (0.5,1.4);

\node at (0.35,1.25) {\scriptsize $\gamma$};
\node at (0.65,1.25) {\scriptsize $\gamma$};
	
\end{scope}
}

\end{scope}

\filldraw[fill=white]
	(0,0) circle (0.05)
	(0.5,1.4) circle (0.05)
	(1.4,-0.5) circle (0.05)
;

\node at (0.35,1) {\scriptsize $\beta$};
\node at (-0.35,1) {\scriptsize $\beta$};
\node at (0,0.8) {\scriptsize $\alpha$};
\node at (-0.35,1.25) {\scriptsize $\alpha$};

\node at (1,0.35) {\scriptsize $\beta$};
\node at (1,-0.35) {\scriptsize $\beta$};
\node at (0.8,0) {\scriptsize $\alpha$};
\node at (1.25,0.35) {\scriptsize $\alpha$};

\node at (-0.5,0.15) {\scriptsize $\beta$};
\node at (-0.15,0.5) {\scriptsize $\alpha$};
\node at (-0.5,0.7) {\scriptsize $\beta$};
\node at (-0.7,0.45) {\scriptsize $\alpha$};

\node at (0.5,-0.15) {\scriptsize $\beta$};
\node at (0.15,-0.5) {\scriptsize $\alpha$};
\node at (0.5,-0.7) {\scriptsize $\beta$};
\node at (0.7,-0.45) {\scriptsize $\alpha$};

\node at (0.65,0.95) {\scriptsize $\alpha$};
\node at (0.95,0.65) {\scriptsize $\beta$};
\node at (1.25,1.25) {\scriptsize $\beta$};
\node at (1.25,0.65) {\scriptsize $\alpha$};

\node at (0.5,0.15) {\scriptsize $\alpha$};
\node at (0.15,0.5) {\scriptsize $\beta$};
\node at (0.45,0.7) {\scriptsize $\alpha$};
\node at (0.7,0.45) {\scriptsize $\beta$};

\node at (0.65,-0.95) {\scriptsize $\beta$};
\node at (0.95,-0.65) {\scriptsize $\alpha$};
\node at (1.25,-1.25) {\scriptsize $\beta$};
\node at (0.65,-1.25) {\scriptsize $\alpha$};

\end{tikzpicture}
\caption{}
\label{Subfig-3DFigC-3}
\end{subfigure}
\begin{subfigure}{0.225\linewidth} % 4
\centering
\begin{tikzpicture}[>=latex]

\draw[green!70, line width=2]
	(-0.9,0.5) -- (-0.5,0.9) -- (0,0.6) -- (0.5,0.9) -- (0.9,0.5) -- (0.6,0) -- (0.9,-0.5) -- (0.5,-0.9) -- (0,-0.6)
	(0.5,-0.9) -- (0.5,-1.4)
	(-0.5,1.4) -- (-0.5,0.9)
	(0.9,0.5) -- (1.4,0.5);

\begin{scope}[]

\foreach \a in {-1,0,1}
{
\begin{scope}[rotate=90*\a]

\draw
	(0,0) -- (0.6,0) -- (0.9,0.5) -- (0.5,0.9) -- (0,0.6) -- (0,0);

\node at (0.15,0.15) {\scriptsize $\gamma$};

\end{scope}
}
	
\foreach \a in {-1,0}
{
\begin{scope}[rotate=90*\a]

\draw
	(-0.5,0.9) -- (-0.5,1.4) -- (1.4,1.4) -- (1.4,0.5) -- (0.9,0.5)
	(0.5,0.9) -- (0.5,1.4);

\node at (0.35,1.25) {\scriptsize $\gamma$};
\node at (0.65,1.25) {\scriptsize $\gamma$};
	
\end{scope}
}

\end{scope}

\filldraw[fill=white]
	(0.55,1.45) -- ++(-0.1,0) -- ++(0,-0.1) -- ++(0.1,0) -- cycle
	(1.45,-0.45) -- ++(-0.1,0) -- ++(0,-0.1) -- ++(0.1,0) -- cycle
	(0,0) circle (0.05)
;

\node at (0.35,1) {\scriptsize $\beta$};
\node at (-0.35,1) {\scriptsize $\beta$};
\node at (0,0.8) {\scriptsize $\alpha$};
\node at (-0.35,1.25) {\scriptsize $\alpha$};

\node at (1,0.35) {\scriptsize $\beta$};
\node at (1,-0.35) {\scriptsize $\beta$};
\node at (0.8,0) {\scriptsize $\alpha$};
\node at (1.25,0.35) {\scriptsize $\alpha$};

\node at (-0.5,0.15) {\scriptsize $\alpha$};
\node at (-0.15,0.5) {\scriptsize $\beta$};
\node at (-0.5,0.7) {\scriptsize $\alpha$};
\node at (-0.7,0.45) {\scriptsize $\beta$};

\node at (0.5,-0.15) {\scriptsize $\alpha$};
\node at (0.15,-0.5) {\scriptsize $\beta$};
\node at (0.5,-0.7) {\scriptsize $\alpha$};
\node at (0.7,-0.45) {\scriptsize $\beta$};

\node at (0.65,0.95) {\scriptsize $\beta$};
\node at (0.95,0.65) {\scriptsize $\alpha$};
\node at (1.25,1.25) {\scriptsize $\alpha$};
\node at (1.25,0.65) {\scriptsize $\beta$};

\node at (0.5,0.15) {\scriptsize $\beta$};
\node at (0.15,0.5) {\scriptsize $\alpha$};
\node at (0.45,0.7) {\scriptsize $\beta$};
\node at (0.7,0.45) {\scriptsize $\alpha$};

\node at (0.65,-0.95) {\scriptsize $\alpha$};
\node at (0.95,-0.65) {\scriptsize $\beta$};
\node at (1.25,-1.25) {\scriptsize $\alpha$};
\node at (0.65,-1.25) {\scriptsize $\beta$};
\end{tikzpicture}
\caption{}
\label{Subfig-3DFigC-4}
\end{subfigure}

\caption{Splitting of Figure \ref{2DfigA} according to Figure \ref{Subfig-2a2bc_arrangement-3}.}
\label{3DFigC}
\end{figure}

For AVC(3D24), the way $\gamma_1|\gamma_2$ determines $\gamma_2|\gamma_3$ implies that $\gamma^4$ has two possible AADs $|^{\alpha}\gamma^{\beta}|^{\beta}\gamma^{\alpha}|^{\alpha}\gamma^{\beta}|^{\beta}\gamma^{\alpha}|$ or $|^{\alpha}\gamma^{\beta}|^{\alpha}\gamma^{\beta}|^{\alpha}\gamma^{\beta}|^{\alpha}\gamma^{\beta}|$, yielding two possible subtilings: $N_{\ssq}(\gamma^4)$ in Figures \ref{Subfig-3DFigC-1}, \ref{Subfig-3DFigC-2}, and $N_{\circ}(\gamma^4)$ in Figures \ref{Subfig-3DFigC-3}, \ref{Subfig-3DFigC-4}. Moreover, Figures \ref{Subfig-3DFigC-1}, \ref{Subfig-3DFigC-2} imply that one $N_{\ssq}(\gamma^4)$ is surrounded by two $N_{\ssq}(\gamma^4)$ and two $N_{\circ}(\gamma^4)$ (Figure \ref{Subfig-3DFigD-1}). Figures \ref{Subfig-3DFigC-3}, \ref{Subfig-3DFigC-4} imply that one $N_{\circ}(\gamma^4)$ is surrounded either by four $N_{\circ}(\gamma^4)$ (Figure \ref{Subfig-3DFigD-3}), or by four $N_{\ssq}(\gamma^4)$ (Figure \ref{Subfig-3DFigD-2}).

\begin{figure}[htp]
\centering
\begin{subfigure}{0.25\linewidth}
\centering
\begin{tikzpicture}[>=latex]

%% 1

\foreach \a in {0,1,2,3}
\draw[green!70, line width=2, rotate=90*\a]
	(0.9,0.5) -- (0.5,0.9) -- (0,0.6) -- (-0.5,0.9) -- (-0.5,1.4)
;

\foreach \a in {0,1,2,3}
{
\begin{scope}[rotate=90*\a]

\draw
	(0,0) -- (0.6,0) -- (0.9,0.5) -- (0.5,0.9) -- (0,0.6)
	(0.5,0.9) -- (0.5,1.4)
	(-0.5,0.9) -- (-0.5,1.4)
	(1.4,1.4) -- (1.4,-1.4);

\node at (0.15,0.15) {\scriptsize $\gamma$};
\node at (0.35,1.25) {\scriptsize $\gamma$};
\node at (0.65,1.25) {\scriptsize $\gamma$};

\end{scope}
}

\foreach \a in {0,...,3}
{
\begin{scope}[rotate=90*\a]

\node at (0.35,1) {\scriptsize $\alpha$};
\node at (-0.35,1) {\scriptsize $\alpha$};
\node at (0,0.8) {\scriptsize $\beta$};
\node at (-0.35,1.25) {\scriptsize $\beta$};	

\end{scope}
}

\foreach \a in {0,1}
{
\begin{scope}[rotate=180*\a]

\node at (0.5,0.15) {\scriptsize $\alpha$};
\node at (0.15,0.5) {\scriptsize $\beta$};
\node at (0.5,0.7) {\scriptsize $\alpha$};
\node at (0.7,0.45) {\scriptsize $\beta$};

\node at (0.5,-0.15) {\scriptsize $\alpha$};
\node at (0.15,-0.5) {\scriptsize $\beta$};
\node at (0.5,-0.7) {\scriptsize $\alpha$};
\node at (0.7,-0.45) {\scriptsize $\beta$};

\node at (0.65,0.95) {\scriptsize $\beta$};
\node at (0.95,0.65) {\scriptsize $\alpha$};
\node at (1.25,0.65) {\scriptsize $\beta$};
\node at (1.25,1.25) {\scriptsize $\alpha$};

\node at (-0.65,0.95) {\scriptsize $\beta$};
\node at (-0.95,0.65) {\scriptsize $\alpha$};
\node at (-0.65,1.25) {\scriptsize $\alpha$};
\node at (-1.25,1.25) {\scriptsize $\beta$};

\filldraw[fill=white]
	(0.5,1.4) circle (0.05)
	(1.45,-0.45) -- ++(-0.1,0) -- ++(0,-0.1) -- ++(0.1,0) -- cycle
	(0.05,0.05) -- ++(-0.1,0) -- ++(0,-0.1) -- ++(0.1,0) -- cycle
;
	
\end{scope}
}

\end{tikzpicture}
\caption{}
\label{Subfig-3DFigD-1}
\end{subfigure}
\begin{subfigure}{0.25\linewidth}
\centering
\begin{tikzpicture}[>=latex]

%% 2

\foreach \a in {0,1,2,3}
\draw[green!70, line width=2, rotate=90*\a]
	(0.9,0.5) -- (0.5,0.9) -- (0,0.6) -- (-0.5,0.9) -- (-0.5,1.4)
;

\foreach \a in {0,1,2,3}
{
\begin{scope}[rotate=90*\a]

\draw
	(0,0) -- (0.6,0) -- (0.9,0.5) -- (0.5,0.9) -- (0,0.6)
	(0.5,0.9) -- (0.5,1.4)
	(-0.5,0.9) -- (-0.5,1.4)
	(1.4,1.4) -- (1.4,-1.4);

\node at (0.15,0.15) {\scriptsize $\gamma$};
\node at (0.35,1.25) {\scriptsize $\gamma$};
\node at (0.65,1.25) {\scriptsize $\gamma$};

\end{scope}
}

\begin{scope}[]

\foreach \a in {0,...,3}
{
\begin{scope}[rotate=90*\a]

\node at (0.5,0.15) {\scriptsize $\alpha$};
\node at (0.15,0.5) {\scriptsize $\beta$};
\node at (0.5,0.7) {\scriptsize $\alpha$};
\node at (0.7,0.45) {\scriptsize $\beta$};

\node at (0.35,1) {\scriptsize $\beta$};
\node at (-0.35,1) {\scriptsize $\beta$};
\node at (0,0.8) {\scriptsize $\alpha$};
\node at (-0.35,1.25) {\scriptsize $\alpha$};

\filldraw[fill=white]
	(0.5,1.4) circle (0.05)
	(0,0) circle (0.05);
	
\end{scope}
}

\foreach \a in {0,1}
{
\begin{scope}[rotate=180*\a]

\node at (0.65,0.95) {\scriptsize $\alpha$};
\node at (0.95,0.65) {\scriptsize $\beta$};
\node at (1.25,0.65) {\scriptsize $\alpha$};
\node at (1.25,1.25) {\scriptsize $\beta$};

\node at (-0.65,0.95) {\scriptsize $\beta$};
\node at (-0.95,0.65) {\scriptsize $\alpha$};
\node at (-0.65,1.25) {\scriptsize $\alpha$};
\node at (-1.25,1.25) {\scriptsize $\beta$};

\end{scope}
}

\end{scope}

\end{tikzpicture}
\caption{}
\label{Subfig-3DFigD-3}
\end{subfigure}
\begin{subfigure}{0.25\linewidth}
\centering
\begin{tikzpicture}[>=latex]

%% 3

\foreach \a in {0,1,2,3}
\draw[green!70, line width=2, rotate=90*\a]
	(0.9,0.5) -- (0.5,0.9) -- (0,0.6) -- (-0.5,0.9) -- (-0.5,1.4)
;

\foreach \a in {0,1,2,3}
{
\begin{scope}[rotate=90*\a]

\draw
	(0,0) -- (0.6,0) -- (0.9,0.5) -- (0.5,0.9) -- (0,0.6)
	(0.5,0.9) -- (0.5,1.4)
	(-0.5,0.9) -- (-0.5,1.4)
	(1.4,1.4) -- (1.4,-1.4);

\node at (0.15,0.15) {\scriptsize $\gamma$};
\node at (0.35,1.25) {\scriptsize $\gamma$};
\node at (0.65,1.25) {\scriptsize $\gamma$};

\end{scope}
}

\begin{scope}[]

\foreach \a in {0,...,3}
{
\begin{scope}[rotate=90*\a]

\node at (0.5,0.15) {\scriptsize $\beta$};
\node at (0.15,0.5) {\scriptsize $\alpha$};
\node at (0.5,0.7) {\scriptsize $\beta$};
\node at (0.7,0.45) {\scriptsize $\alpha$};

\node at (0.35,1) {\scriptsize $\beta$};
\node at (-0.35,1) {\scriptsize $\beta$};
\node at (0,0.8) {\scriptsize $\alpha$};
\node at (-0.35,1.25) {\scriptsize $\alpha$};

\node at (-0.65,0.95) {\scriptsize $\alpha$};
\node at (-0.95,0.65) {\scriptsize $\beta$};
\node at (-0.65,1.25) {\scriptsize $\beta$};
\node at (-1.25,1.25) {\scriptsize $\alpha$};	

\filldraw[fill=white]
	(1.45,-0.45) -- ++(-0.1,0) -- ++(0,-0.1) -- ++(0.1,0) -- cycle
	(0,0) circle (0.05)
;
	
\end{scope}
}

\end{scope}

\end{tikzpicture}
\caption{}
\label{Subfig-3DFigD-2}
\end{subfigure}
\caption{Relation between neighboring $N(\gamma^4)$. }
\label{3DFigD}
\end{figure}

If Figure \ref{Subfig-3DFigD-3} is part of the tiling, then we cannot have Figure \ref{Subfig-3DFigD-1} or \ref{Subfig-3DFigD-2} in the tiling. Therefore, the tiling consists of six $N_{\circ}(\gamma^4)$. The schematics of the tiling is given by the upper of Figure \ref{Subfig-3DFigE-1}, and the tiling is the $3{\rm D}_3$ reduction of $PP_6$.

\begin{figure}[htp]
\centering
\begin{subfigure}[t]{0.2\linewidth}
\centering
\begin{tikzpicture}[>=latex, scale=0.8]

\foreach \y in {1,-1}
{
\begin{scope}[shift={(0,1.5*\y)}]

\draw
	(0,0) -- (1.6,-0.6) -- (3,0) -- (3,1.6) -- (1.4,2.2) -- (0,1.6) -- (0,0)
	(0,0) -- (0,1.6) -- (1.6,1) -- (1.6,-0.6)
	(3,1.6) -- (1.6,1);

\draw[dashed]
	(0,0) -- (1.4,0.6) -- (3,0)
	(1.4,0.6) -- (1.4,2.2);

\end{scope}
}

\begin{scope}[shift={(0,1.5)}]
	
\filldraw[fill=white]
	(1.5,0) circle (0.07)
	(1.5,1.6) circle (0.07)
	(0.8,0.5) circle (0.07)
	(2.2,1.1) circle (0.07)
	(0.7,1.1) circle (0.07)
	(2.3,0.5) circle (0.07);

\end{scope}

\begin{scope}[shift={(0,-1.5)}]

\filldraw[fill=white]
	(1.5,0) circle (0.07)
	(1.5,1.6) circle (0.07)
	(0.87,0.57) -- ++(-0.14,0) -- ++(0,-0.14) -- ++(0.14,0) -- cycle
	(2.27,1.17) -- ++(-0.14,0) -- ++(0,-0.14) -- ++(0.14,0) -- cycle
	(0.77,1.17) -- ++(-0.14,0) -- ++(0,-0.14) -- ++(0.14,0) -- cycle
	(2.37,0.57) -- ++(-0.14,0) -- ++(0,-0.14) -- ++(0.14,0) -- cycle;
	
\end{scope}	

\end{tikzpicture}
\caption{}
\label{Subfig-3DFigE-1}
\end{subfigure}
\begin{subfigure}[t]{0.38\linewidth}
\centering
\begin{tikzpicture}[>=latex, scale=1]	

%% tiling

\foreach \a in {0,...,3}
{
\begin{scope}[rotate=90*\a]

\draw
	(0,0) -- (0.6,0) -- (0.9,0.5) -- (0.5,0.9) -- (0,0.6)
	(0.9,-0.5) -- (2,-0.5)
	(0.9,0.5) -- (1.4,0.5)
	(2,1.4) -- (-1.4,1.4)
	(-2,2) -- (2,2) -- (2.4,2.4);

\filldraw[fill=white]
	(0,0) circle (0.05);

\node at (0.15,0.15) {\scriptsize $\gamma$};
\node at (0.35,1.25) {\scriptsize $\gamma$};
\node at (0.65,1.25) {\scriptsize $\gamma$};
\node at (0.35,1.55) {\scriptsize $\gamma$};
\node at (0.65,1.55) {\scriptsize $\gamma$};
\node at (0,2.4) {\scriptsize $\gamma$};

\end{scope}
}

\foreach \a in {0,...,3}
\foreach \b in {0,1}
\filldraw[fill=white, 
rotate=90*\a] 
	(0.55,1.45) -- ++(-0.1,0) -- ++(0,-0.1) -- ++(0.1,0) -- cycle
	(2.4,2.4) circle (0.05);

%% 1

\foreach \a in {0,...,3}
{
\begin{scope}[rotate=90*\a]

\node at (0.5,0.15) {\scriptsize $\beta$};
\node at (0.15,0.5) {\scriptsize $\alpha$};
\node at (0.5,0.7) {\scriptsize $\beta$};
\node at (0.7,0.5) {\scriptsize $\alpha$};

\node at (0.35,1) {\scriptsize $\beta$};
\node at (-0.35,1) {\scriptsize $\beta$};
\node at (0,0.8) {\scriptsize $\alpha$};
\node at (-0.35,1.25) {\scriptsize $\alpha$};

\node at (-0.65,0.95) {\scriptsize $\alpha$};
\node at (-0.95,0.65) {\scriptsize $\beta$};
\node at (-0.65,1.25) {\scriptsize $\beta$};
\node at (-1.25,1.25) {\scriptsize $\alpha$};	

\node at (0.65,1.85) {\scriptsize $\beta$};
\node at (1.4,1.55) {\scriptsize $\alpha$};
\node at (1.85,1.55) {\scriptsize $\beta$};
\node at (1.85,1.85) {\scriptsize $\alpha$};

\node at (-0.5,1.55) {\scriptsize $\alpha$};
\node at (0.35,1.85) {\scriptsize $\beta$};
\node at (-1.25,1.55) {\scriptsize $\beta$};
\node at (-1.25,1.85) {\scriptsize $\alpha$};

\node at (1.95,2.15) {\scriptsize $\alpha$};
\node at (2.15,1.95) {\scriptsize $\beta$};
\node at (0.5,2.15) {\scriptsize $\beta$};
\node at (-1.4,2.15) {\scriptsize $\alpha$};

\end{scope}
}

\end{tikzpicture}
\caption{Figure \ref{Subfig-2a2bc_arrangement-3}}
\label{Subfig-3DFigE-T1}
\end{subfigure}
\begin{subfigure}[t]{0.38\linewidth}
\centering
\begin{tikzpicture}[>=latex, scale=1]

\foreach \a in {0,...,3}
{
\begin{scope}[
rotate=90*\a]

\draw
	(0,0) -- (0.6,0) -- (0.9,0.5) -- (0.5,0.9) -- (0,0.6)
	(0.9,-0.5) -- (2,-0.5)
	(0.9,0.5) -- (1.4,0.5)
	(2,1.4) -- (-1.4,1.4)
	(-2,2) -- (2,2) -- (2.4,2.4);

\node at (0.15,0.15) {\scriptsize $\gamma$};
\node at (0.35,1.25) {\scriptsize $\gamma$};
\node at (0.65,1.25) {\scriptsize $\gamma$};
\node at (0.35,1.55) {\scriptsize $\gamma$};
\node at (0.65,1.55) {\scriptsize $\gamma$};
\node at (0,2.4) {\scriptsize $\gamma$};

\end{scope}
}

\foreach \a in {0,...,3}
\filldraw[fill=white, 
rotate=90*\a] 
	(0.55,1.45) -- ++(-0.1,0) -- ++(0,-0.1) -- ++(0.1,0) -- cycle
	(2.4,2.4) circle (0.05);

\foreach \a in {0,...,3}
{
\begin{scope}[rotate=90*\a]

\node at (0.5,0.15) {\scriptsize $\alpha$};
\node at (0.15,0.5) {\scriptsize $\beta$};
\node at (0.5,0.7) {\scriptsize $\beta$};
\node at (0.7,0.5) {\scriptsize $\alpha$};

\node at (0.35,1) {\scriptsize $\alpha$};
\node at (-0.35,1) {\scriptsize $\beta$};
\node at (0,0.8) {\scriptsize $\alpha$};
\node at (-0.35,1.25) {\scriptsize $\beta$};

\node at (-0.65,0.95) {\scriptsize $\alpha$};
\node at (-0.95,0.65) {\scriptsize $\alpha$};
\node at (-0.65,1.25) {\scriptsize $\beta$};
\node at (-1.25,1.25) {\scriptsize $\beta$};	

\node at (0.65,1.85) {\scriptsize $\alpha$};
\node at (1.4,1.55) {\scriptsize $\beta$};
\node at (1.85,1.55) {\scriptsize $\beta$};
\node at (1.85,1.85) {\scriptsize $\alpha$};

\node at (-0.5,1.55) {\scriptsize $\beta$};
\node at (0.35,1.85) {\scriptsize $\alpha$};
\node at (-1.25,1.55) {\scriptsize $\beta$};
\node at (-1.25,1.85) {\scriptsize $\alpha$};

\node at (1.95,2.15) {\scriptsize $\beta$};
\node at (2.15,1.95) {\scriptsize $\alpha$};
\node at (0.5,2.15) {\scriptsize $\beta$};
\node at (-1.4,2.15) {\scriptsize $\alpha$};

\end{scope}
}

\end{tikzpicture}
\caption{Figure \ref{Subfig-2a2bc_arrangement-4}}
\label{Subfig-3DFigE-3}
\end{subfigure}
\caption{Tilings for AVC(3D24), not pentagonal subdivision.}
\label{3DFigE}
\end{figure}

What remains is the tiling with a mixture of $N_{\circ}(\gamma^4)$ and $N_{\ssq}(\gamma^4)$, glued together according to Figures \ref{Subfig-3DFigD-1} and \ref{Subfig-3DFigD-2}. The tiling is given by Figure \ref{Subfig-3DFigE-T1}, and the schematics of the tiling is given by the lower of Figure \ref{Subfig-3DFigE-1}, and the 3d picture of the tiling is given by Figure \ref{3DFigFc}. 

For AVC(3D60), there is no analogue of $|^{\alpha}\gamma^{\beta}|^{\beta}\gamma^{\alpha}|^{\alpha}\gamma^{\beta}|^{\beta}\gamma^{\alpha}|$ for the AAD of $\gamma^5$. Therefore, the AAD of $\gamma\cdots=\gamma^5$ can only be $|^{\alpha}\gamma^{\beta}|^{\alpha}\gamma^{\beta}|^{\alpha}\gamma^{\beta}|^{\alpha}\gamma^{\beta}|^{\alpha}\gamma^{\beta}|$. This means only the $\gamma^5$ version of Figure \ref{Subfig-3DFigC-3} can happen, and the tiling is the $3{\rm D}_3$ reduction of $PP_{12}$. 

The analysis of splitting according to Figure \ref{Subfig-2a2bc_arrangement-4} is similar. We get the $3{\rm D}_1$ reduction of $PP_{12}$, and the tiling in Figure \ref{Subfig-3DFigE-3}. The schematics of the tilings are also given by Figure \ref{Subfig-3DFigE-1}, and the 3d picture of the tiling in Figure \ref{Subfig-3DFigE-3} is given by Figure \ref{3DFigFd}. 
\end{proof}

\begin{figure}[h!]
\centering
\begin{subfigure}[b]{0.3\linewidth}
\centering
\begin{tikzpicture}[>=latex,scale=1]

\pgftext{
	\includegraphics[scale=0.08]{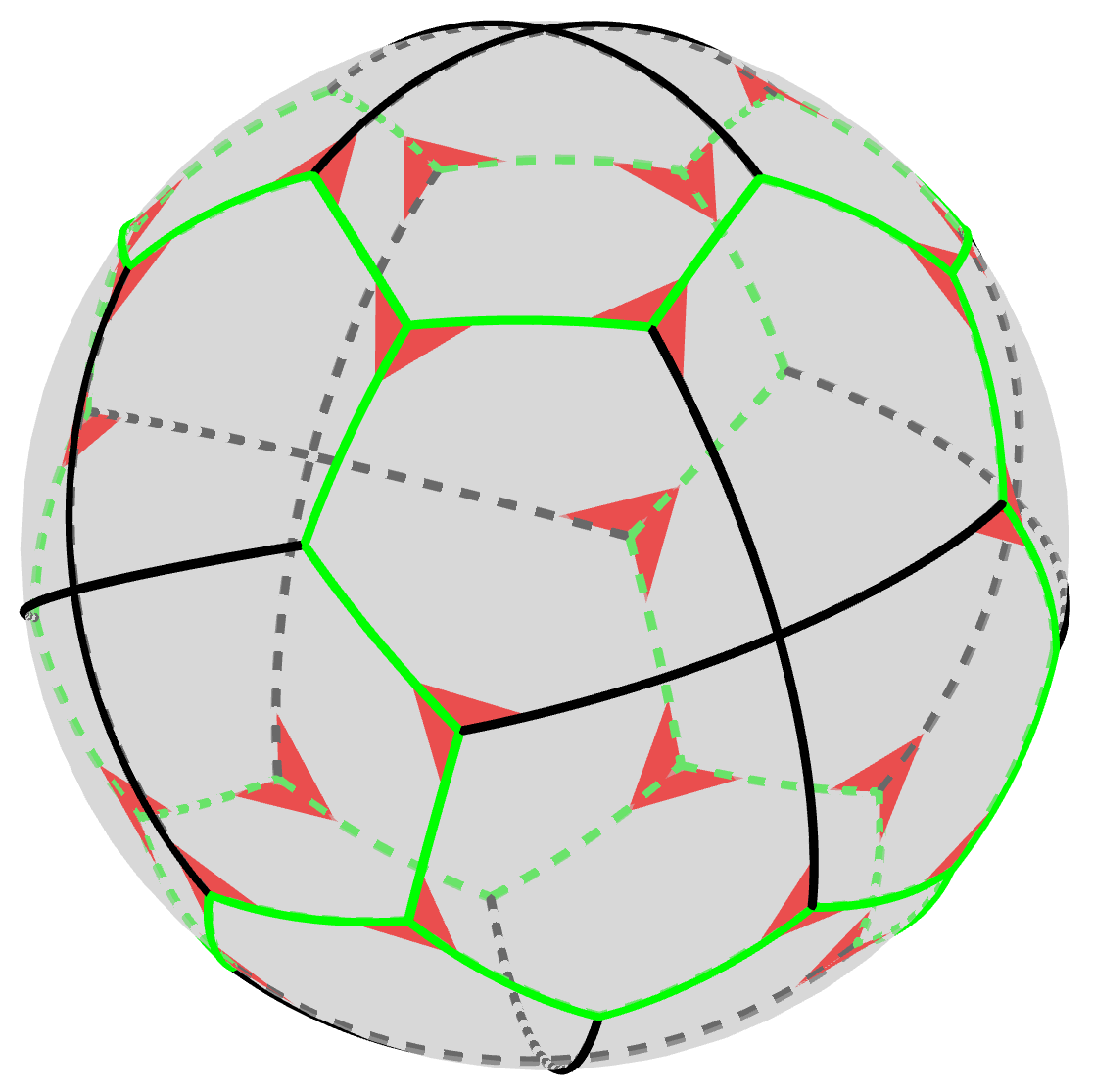}
	};

\end{tikzpicture}	
\caption{}
\label{3DFigFc}	
\end{subfigure}
\begin{subfigure}[b]{0.3\linewidth}
\centering
\begin{tikzpicture}[>=latex,scale=1]

\pgftext{
	\includegraphics[scale=0.077]{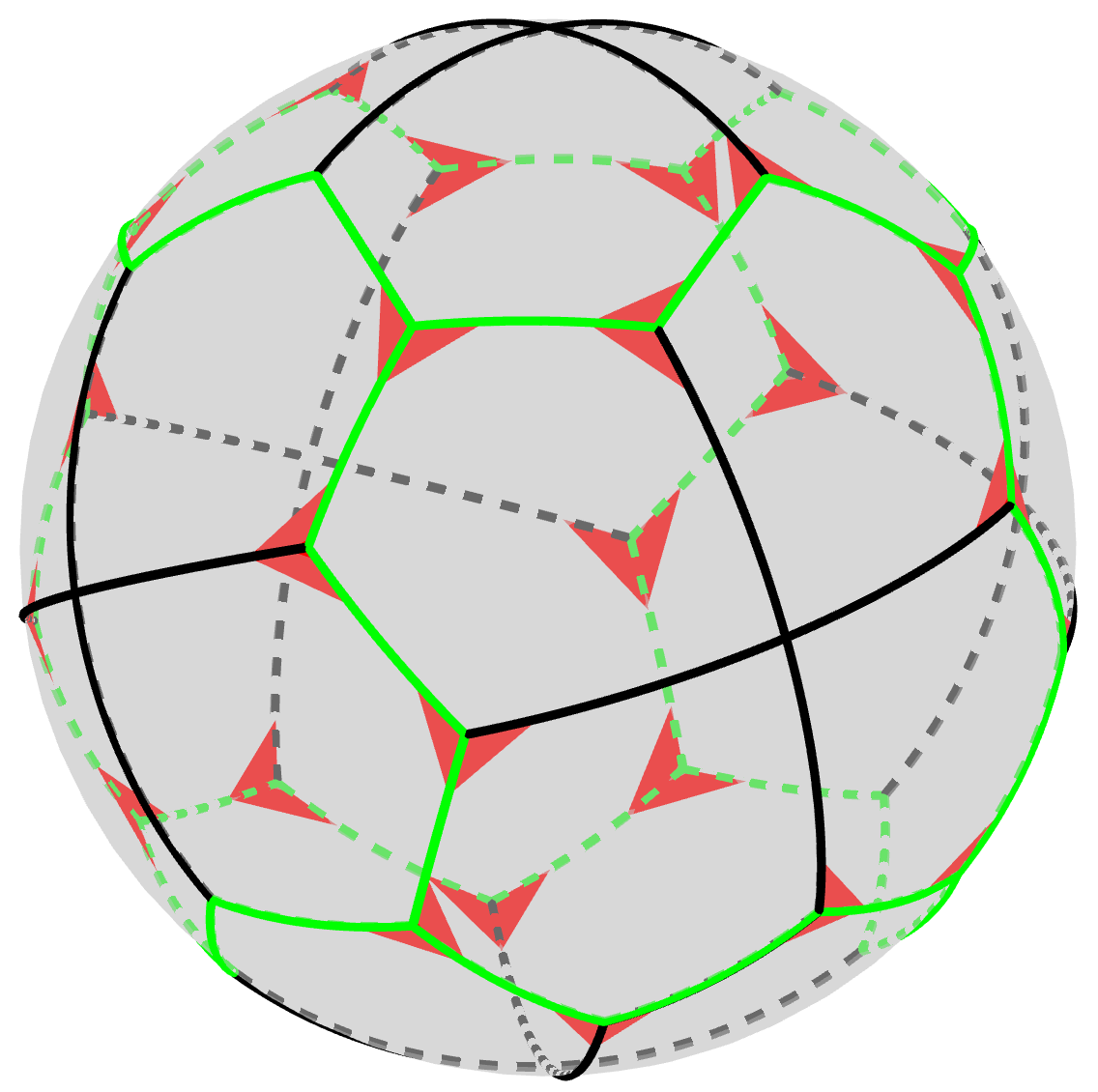}
	};
		
\end{tikzpicture}
\caption{}
\label{3DFigFd}	
\end{subfigure}
\caption{Tilings for AVC(3D24) that are not pentagonal subdivision. The red corners are $\alpha$, and $\gamma$ are concentrated at the meeting places of four black lines.}
\label{3DFigF}
\end{figure}

\subsection{Tilings not Reducible to 2D}

According to Figure \ref{Fig-reduction}, the reductions 3C, 3E cannot be further reduced to 2D. We can no longer use Theorem \ref{2Dthm}, and need to find tilings directly.

\begin{theorem}\label{3Ethm}
The tilings for 
\begin{align*}
\text{\rm AVC(3E24)}
&=\{
24\alpha^2\beta^2\gamma 
\colon 
24\alpha^2\beta,
8\gamma^3,
6\beta^4 \}, \\
\text{\rm AVC(3E60)}
&=\{
60\alpha^2\beta^2\gamma 
\colon 
60\alpha^2\beta,
20\gamma^3,
12\beta^5 \},
\end{align*}
are the $3{\rm E}_1$, $3{\rm E}_2$, $3{\rm E}_3$ reductions of $PP_8$ and $PP_{20}$, and the two tilings in Figures \ref{3EFigD3} and \ref{3EFigE}. 
\end{theorem}

The two tilings in Figures \ref{3EFigD3} and \ref{3EFigE} use the same pentagon as in the reduction $3{\rm E}_1$, but are not pentagonal subdivisions of Platonic solids. We will discuss the underlying structure and the symmetry of these tilings after the proof.

\begin{proof}
The pentagon $\alpha^2\beta^2\gamma$ has the four possible angle arrangements in Figure \ref{2a2bc_arrangement}. 

\medskip

\noindent {\bf Pentagon in Figure \ref{Subfig-2a2bc_arrangement-2}}

\medskip

Figure \ref{Subfig-3EFigA-1} shows three tiles $\circled{1}, \circled{2}, \circled{3}$ around a vertex $\gamma^3$. Then $\circled{1}, \circled{2}$ share $\beta^2\cdots=\beta^4/\beta^5$, determining one $\beta_4$ in $\circled{4}$. Since $\alpha\cdots=\alpha^2\beta$, the tiles $\circled{1}, \circled{4}$ share an $\alpha^2\beta$. By the one $\beta_4$ already obtained, and two $\beta$ in \circled{4} being non-adjacent, the angles at $\alpha^2\beta$ shared by $\circled{1}, \circled{4}$ are arranged as indicated. For the same reason, we also get $\circled{5}$, and $\circled{1}, \circled{5}$ share a vertex $\alpha\cdots=\alpha^2\beta$ as indicated. Then we find two $\beta$ adjacent in a tile, a contradiction.

\begin{figure}[h!]
\centering
\begin{subfigure}[t]{0.225\linewidth}
\centering
\begin{tikzpicture}[>=latex,scale=1]

%% 1

\foreach \a in {-1,1}
{
\begin{scope}[xscale=\a]

\draw
	(0,0.8) -- (1.3,0.8) -- (1.6,0.4) -- (1.3,0) -- (0,0) -- (0,-0.5) -- (0.4,-0.8) -- (0.8,-0.5) -- (0.8,0.8);

\node at (0,0.15) {\scriptsize $\gamma$}; 
\node at (0.15,-0.15) {\scriptsize $\gamma$}; 
\node at (0.65,0.15) {\scriptsize $\beta$};
\node at (0.65,0.65) {\scriptsize $\alpha$};
\node at (0.95,0.15) {\scriptsize $\beta$};
\node at (0.95,0.65) {\scriptsize $\alpha$};
\node at (0.8,0.95) {\scriptsize $\beta$};
\node at (0.65,-0.15) {\scriptsize $\beta$};
\node at (0.65,-0.45) {\scriptsize $\alpha$};
\node at (0.15,-0.45) {\scriptsize $\beta$};
\node at (0.4,-0.6) {\scriptsize $\alpha$};

\end{scope}
}

\node[inner sep=0.5,draw,shape=circle] at (0,0.5) {\scriptsize $1$};
\node[inner sep=0.5,draw,shape=circle] at (-0.4,-0.3) {\scriptsize $2$};
\node[inner sep=0.5,draw,shape=circle] at (0.4,-0.3) {\scriptsize $3$};
\node[inner sep=0.5,draw,shape=circle] at (-1.15,0.4) {\scriptsize $4$};
\node[inner sep=0.5,draw,shape=circle] at (1.15,0.4) {\scriptsize $5$};

\end{tikzpicture}
\caption{}
\label{Subfig-3EFigA-1}
\end{subfigure}
\begin{subfigure}[t]{0.225\linewidth}
\centering
\begin{tikzpicture}[>=latex,scale=1]

%% 2

\begin{scope}[]

\foreach \a in {-1,1}
\draw[xscale=\a]
	(0.4,0.5) -- (0.4,0) -- (0,-0.3) -- (0,-0.8) -- (0.8,-0.8) -- (0.8,-0.3) -- (0.4,0)
	(0.8,-0.3) -- (1.2,0) -- (1.2,0.5) -- (0,0.5);

\filldraw[fill=white]
	(0,-0.8) circle (0.05)
	(0.4,0.5) circle (0.05);
	
\node at (-0.4,-0.2) {\scriptsize $\gamma$};
\node at (-0.15,-0.35) {\scriptsize $\alpha$};
\node at (-0.65,-0.35) {\scriptsize $\alpha$};
\node at (-0.15,-0.65) {\scriptsize $\beta$};
\node at (-0.65,-0.65) {\scriptsize $\beta$};

\node at (0.4,-0.2) {\scriptsize $\alpha$};
\node at (0.15,-0.35) {\scriptsize $\beta$};
\node at (0.65,-0.35) {\scriptsize $\gamma$};
\node at (0.15,-0.65) {\scriptsize $\beta$};
\node at (0.65,-0.65) {\scriptsize $\alpha$};

\node at (-0.55,0.05) {\scriptsize $\gamma$};
\node at (-1.05,0.05) {\scriptsize $\beta$};
\node at (-0.55,0.35) {\scriptsize $\alpha$};
\node at (-1.05,0.35) {\scriptsize $\beta$};
\node at (-0.8,-0.1) {\scriptsize $\alpha$};

\node at (0.55,0.05) {\scriptsize $\alpha$};
\node at (1.05,0.05) {\scriptsize $\alpha$};
\node at (0.55,0.35) {\scriptsize $\beta$};
\node at (1.05,0.35) {\scriptsize $\beta$};
\node at (0.8,-0.1) {\scriptsize $\gamma$};

\node at (-0.25,0.05) {\scriptsize $\gamma$};
\node at (0.25,0.05) {\scriptsize $\beta$};
\node at (-0.25,0.35) {\scriptsize $\alpha$};
\node at (0.25,0.35) {\scriptsize $\beta$};
\node at (0,-0.1) {\scriptsize $\alpha$};

\node at (-0.95,-0.35) {\scriptsize $\beta$};
\node at (0.95,-0.35) {\scriptsize $\gamma$};

\node[inner sep=0.5,draw,shape=circle] at (-0.4,-0.5) {\scriptsize $1$};
\node[inner sep=0.5,draw,shape=circle] at (0.4,-0.5) {\scriptsize $2$};
\node[inner sep=0.5,draw,shape=circle] at (-0.8,0.2) {\scriptsize $3$};
\node[inner sep=0.5,draw,shape=circle] at (0,0.2) {\scriptsize $4$};
\node[inner sep=0.5,draw,shape=circle] at (0.8,0.2) {\scriptsize $5$};

\end{scope}

\end{tikzpicture}
\caption{}
\label{Subfig-3EFigA-2}
\end{subfigure}
\begin{subfigure}[t]{0.225\linewidth}
\centering
\begin{tikzpicture}[>=latex,scale=1]

%% 3

\begin{scope}[shift={(6.7cm,-0.3cm)}]

\foreach \a in {-1,1}
{
\begin{scope}[xscale=\a]

\draw
	(0,-0.5) -- (0.4,-0.8) -- (0.8,-0.5) -- (0.8,0.5) -- (0.4,0.8) -- (0,0.5) -- (0,0) -- (0.8,0)
	(0.8,-0.5) -- (1.4,-0.5) -- (1.4,1.3) -- (0,1.3)
	(0.8,0.5) -- (1.4,0.5)
	(0.4,0.8) -- (0.4,1.3);

\draw[densely dashed]
	(0,0) -- (0,-0.5);

\node at (0.15,0.15) {\scriptsize $\beta$};
\node at (0.15,-0.15) {\scriptsize $\beta$};
\node at (0.65,0.15) {\scriptsize $\alpha$};
\node at (0.65,-0.15) {\scriptsize $\alpha$};
\node at (0.65,0.45) {\scriptsize $\beta$};
\node at (0.15,0.45) {\scriptsize $\gamma$};
\node at (0.4,0.6) {\scriptsize $\alpha$};

\node at (0,0.7) {\scriptsize $\gamma$};
\node at (0.25,0.85) {\scriptsize $\alpha$};
\node at (0.55,0.85) {\scriptsize $\beta$};
\node at (0.85,0.65) {\scriptsize $\alpha$};
\node at (0.95,0.35) {\scriptsize $\alpha$};
\node at (0.95,0) {\scriptsize $\beta$};

\end{scope}
}

\node[inner sep=0.5,draw,shape=circle] at (-0.4,0.3) {\scriptsize $1$};
\node[inner sep=0.5,draw,shape=circle] at (0.4,0.3) {\scriptsize $2$};
\node[inner sep=0.5,draw,shape=circle] at (0,1) {\scriptsize $3$};
\node[inner sep=0.5,draw,shape=circle] at (-1.2,0) {\scriptsize $4$};
\node[inner sep=0.5,draw,shape=circle] at (-1,0.9) {\scriptsize $5$};

\end{scope}

\end{tikzpicture}
\caption{}
\label{Subfig-3EFigA-3}
\end{subfigure}
\begin{subfigure}[t]{0.225\linewidth}
\centering
\begin{tikzpicture}[>=latex,scale=1]

%% 4

\begin{scope}[]

\foreach \a in {-1,1}
\draw[xscale=\a]
	(0.4,0.5) -- (0.4,0) -- (0,-0.3) -- (0,-0.8) -- (0.8,-0.8) -- (0.8,-0.3) -- (0.4,0)
	(0.8,-0.3) -- (1.2,0) -- (1.2,0.5) -- (0,0.5);

\filldraw[fill=white]
	(0,-0.8) circle (0.05)
	(0.4,0.5) circle (0.05);
	
\node at (-0.4,-0.2) {\scriptsize $\gamma$};
\node at (-0.15,-0.35) {\scriptsize $\alpha$};
\node at (-0.65,-0.35) {\scriptsize $\beta$};
\node at (-0.15,-0.65) {\scriptsize $\beta$};
\node at (-0.65,-0.65) {\scriptsize $\alpha$};

\node at (0.4,-0.2) {\scriptsize $\beta$};
\node at (0.15,-0.35) {\scriptsize $\alpha$};
\node at (0.65,-0.35) {\scriptsize $\gamma$};
\node at (0.15,-0.65) {\scriptsize $\beta$};
\node at (0.65,-0.65) {\scriptsize $\alpha$};

\node at (-0.55,0.05) {\scriptsize $\gamma$};
\node at (-1.05,0.05) {\scriptsize $\beta$};
\node at (-0.55,0.35) {\scriptsize $\beta$};
\node at (-1.05,0.35) {\scriptsize $\alpha$};
\node at (-0.8,-0.1) {\scriptsize $\alpha$};

\node at (0.55,0.05) {\scriptsize $\alpha$};
\node at (1.05,0.05) {\scriptsize $\beta$};
\node at (0.55,0.35) {\scriptsize $\beta$};
\node at (1.05,0.35) {\scriptsize $\alpha$};
\node at (0.8,-0.1) {\scriptsize $\gamma$};

\node at (-0.25,0.05) {\scriptsize $\gamma$};
\node at (0.25,0.05) {\scriptsize $\alpha$};
\node at (-0.25,0.35) {\scriptsize $\alpha$};
\node at (0.25,0.35) {\scriptsize $\beta$};
\node at (0,-0.1) {\scriptsize $\beta$};

\node at (-0.95,-0.35) {\scriptsize $\alpha$};
\node at (0.95,-0.35) {\scriptsize $\gamma$};

\node[inner sep=0.5,draw,shape=circle] at (-0.4,-0.5) {\scriptsize $1$};
\node[inner sep=0.5,draw,shape=circle] at (0.4,-0.5) {\scriptsize $2$};
\node[inner sep=0.5,draw,shape=circle] at (-0.8,0.2) {\scriptsize $3$};
\node[inner sep=0.5,draw,shape=circle] at (0,0.2) {\scriptsize $4$};
\node[inner sep=0.5,draw,shape=circle] at (0.8,0.2) {\scriptsize $5$};

\end{scope}

\end{tikzpicture}
\caption{}
\label{Subfig-3EFigA-4}
\end{subfigure}
\caption{Tilings for AVC(3E24/60), Figures \ref{Subfig-2a2bc_arrangement-1}, \ref{Subfig-2a2bc_arrangement-2}, \ref{Subfig-2a2bc_arrangement-3}.}
\label{3EFigA} 
\end{figure}

\medskip

\noindent {\bf Pentagon in Figure \ref{Subfig-2a2bc_arrangement-1}}

\medskip

Figure \ref{Subfig-3EFigA-2} shows two of these tiles, $\circled{1}, \circled{2}$, at a vertex $\beta^4/\beta^5$. We may assume that $\circled{1}$ has angles arranged as shown. Then $\gamma_1\cdots=\gamma^3$ determines $\circled{3}, \circled{4}$. Then $\circled{1}, \circled{4}$ share a vertex $\alpha^2\cdots=\alpha^2\beta$ which determines one $\beta_2$. Then we know both $\beta_2$, and determine $\circled{2}$. Then $\gamma_2\cdots=\gamma^3$ determines $\circled{5}$. Hence $\circled{4}, \circled{5}$ share a vertex $\beta^2\cdots=\beta^4/\beta^5$. 

We have just proven that $\circled{1}$ determines $\circled{2}$ and induces a new vertex $\beta^4/\beta^5$ shared by $\circled{4}, \circled{5}$. The same deduction on the other consecutive $|\beta|\beta|$ at the initial vertex $\beta^4/\beta^5$ determines two layers of tiles around the vertex. Moreover, we get  four $\beta^4$ or five $\beta^5$ along the boundary of the two layers. Repeating the deduction on these new $\beta^4/\beta^5$ constructs the $3{\rm E}_3$ reductions of $P_8$ and $P_{20}$.

\medskip

\noindent {\bf Pentagon in Figure \ref{Subfig-2a2bc_arrangement-3}}

\medskip

We note that one $\alpha$ and one $\beta$ are adjacent to $\gamma$, and the other $\alpha$ and the other $\beta$ are non-adjacent to $\gamma$. In case we know which is which, we denote (in the discussion but not in the picture) $\alpha, \beta$ non-adjacent to $\gamma$ by $\alpha', \beta'$.

Consider the pentagon in Figure \ref{Subfig-2a2bc_arrangement-3}. We first prove that the $\beta$ adjacent to $\gamma$ (this is $\beta$, not $\beta'$) cannot be part of $\beta^4/\beta^5$. We denote the property by $\beta\cdots=\alpha^2\beta$.

We prove $\beta\cdots=\alpha^2\beta$ by contradiction. Assume $\beta_1\cdots=\beta^4/\beta^5$ in the tile \circled{1} in Figure \ref{Subfig-3EFigA-3} (the dashed line means the vertex can be $\beta^5$). Then $\gamma_1\cdots=\gamma^3$ and $\beta_1\cdots=\beta^4/\beta^5$ determine $\circled{2}$ and $\gamma_3$. 

Since $\alpha'_1\cdots=\alpha^2\beta$ ($\alpha'_1$ non-adjacent to $\gamma_1$), and $\beta_1\cdots=\beta^4/\beta^5$, and two $\beta$ non-adjacent in a tile, we get one $\beta_4$. Then this $\beta_4$ and two $\beta$ non-adjacent imply $\beta_1'\cdots\ne\beta^4/\beta^5$. Therefore, $\beta_1'\cdots=\alpha^2\beta$. This determines one $\alpha_5$. Then this $\alpha_5$, and $\alpha'_1\cdots=\alpha^2\beta$, and two $\alpha$ non-adjacent determines one $\alpha_3$ next to \circled{1}. 

For the same reason, we get another $\alpha_3$ next to \circled{2}. Now we get two $\alpha$ adjacent to $\gamma$ in a tile. This is a contradiction, and proves $\beta\cdots=\alpha^2\beta$. 

By (the property) $\beta\cdots=\alpha^2\beta$, we know the two $\beta$ at the vertex are actually $\beta'$ (non-adjacent to $\gamma$). Then we may assume the angles of $\circled{1}$ are arranged as indicated. Then $\beta_1\cdots=\alpha^2\beta$ and $\gamma_1\cdots=\gamma^3$ determine $\circled{3}, \gamma_4$. Then $\beta_3\cdots=\alpha^2\beta$ and $\gamma_4$ determine $\circled{4}$. Then $\alpha_1\beta_4\cdots=\alpha^2\beta$ and $\beta_2'$ determine $\circled{2}$. Then $\alpha'_4\beta_2\cdots=\alpha^2\beta$ and $\gamma_2\cdots=\gamma^3$ determine $\circled{5}$. 

Figure \ref{Subfig-3EFigA-4} shows that $\circled{1}$ determines $\circled{2}, \circled{3}, \circled{4}, \circled{5}$. As in the case of Figure \ref{Subfig-3EFigA-2}, the repeated application of this fact gives the $3{\rm E}_2$ reductions of $PP_8$ and $PP_{20}$. 

\medskip

\noindent {\bf Pentagon in Figure \ref{Subfig-2a2bc_arrangement-4}}

\medskip

Similar to the previous case, we may still denote $\alpha, \beta$ non-adjacent to $\gamma$ by $\alpha', \beta'$.

Suppose we have the AAD $\alpha^{\gamma}|^{\gamma}\alpha$ at an $\alpha^2\beta$, indicated by thick lines in Figure \ref{3EFigD}. Then we determine $\circled{1}, \circled{2}$, and one $\beta$ in $\circled{3}$, in Figures \ref{3EFigD} and \ref{3EFigE}. Then $\alpha'_1\cdots=\alpha'_2\cdots=\alpha^2\beta$ imply that the two angles adjacent to the one $\beta$ in $\circled{3}$ are not $\gamma$. Hence we may assume that the angles of $\circled{3}$ are arranged as shown. Then $\alpha'_2\beta_3\cdots=\alpha^2\beta$ and $\gamma_3\cdots=\gamma^3$ determine $\circled{4}$. Moreover, $\alpha'_4\beta'_2\cdots=\alpha^2\beta$ gives one $\alpha$ in $\circled{5}$.

If $\beta_2\cdots=\alpha^2\beta$, then we get the second $\alpha$ in $\circled{5}$ and one $\alpha$ in $\circled{6}$ in Figure \ref{3EFigD2}. The one $\alpha$ in $\circled{6}$ and $\gamma_1\gamma_2\cdots=\gamma^3$ determine $\circled{6}$. Then the two $\alpha$ in $\circled{5}$ imply that $\gamma_5$ is at either $\beta'_4\cdots$ or $\alpha'_6\cdots$, contradicting $\gamma_5\cdots=\gamma^3$. 

Therefore, $\beta_2\cdots=\beta^4$. Combined with the one $\alpha$ in \circled{6} we already know, we determine \circled{6}. In Figure \ref{3EFigD3}, we denote the four tiles $\circled{1},\circled{2},\circled{3},\circled{6}$ deduced from $\alpha^{\gamma}|^{\gamma}\alpha$ by $N(\alpha^{\gamma}|^{\gamma}\alpha)$. We note that $\alpha^{\gamma}|^{\gamma}\alpha$ gives $\circled{1},\circled{2}$ that is vertically symmetric in Figure \ref{3EFigD2}. There are two choices for the angles in \circled{3}, and we just proved the choice of \circled{3} in the picture determines \circled{6} and the whole $N(\alpha^{\gamma}|^{\gamma}\alpha)$. The other choice for \circled{3} is the vertical flip of the one in the picture, and determines the vertical flip of the \circled{6} in the picture. Therefore, $\circled{1},\circled{2},\circled{6}$ also determines $N(\alpha^{\gamma}|^{\gamma}\alpha)$.

\begin{figure}[h!]
\centering
\begin{subfigure}[b]{0.45\linewidth}
\centering
\begin{tikzpicture}[>=latex]

\foreach \a in {0,1,2}
{
\begin{scope}[rotate=120*\a, thick]

\draw
	(0,0) -- (30:0.5) 
	(90:1.4) -- (50:1.4) -- (10:1.4) 
	(30:1.9) -- (-10:1.9) -- (-50:1.9)
	(110:1.9) -- (90:2.4) -- (50:2.4) -- (10:2.4) -- (-30:2.4)
	(10:2.4) -- (10:2.8);

\draw[blue]
	(30:1.9) -- (70:1.9) -- (90:1.4) -- (130:1.4);

\draw[green]
	(90:2.4) -- (50:2.4) -- (10:2.4) -- (-30:2.4)
	(30:0.5) -- (60:0.9) -- (120:0.9) -- (150:0.5);

\draw[red]
	(0:0.9) -- (10:1.4) -- (30:1.9) -- (50:2.4);
		
\fill
	(0,0) circle (0.06)
	(10:2.8) circle (0.06);

\draw[line width=1.5]
	(50:1.4) -- (60:0.9)
	(110:1.9) -- (90:2.4);
	
\foreach \a in {0,1,2}
\filldraw[fill=white, rotate=120*\a]
	(30:1.9) circle (0.05)
	(10:1.4) circle (0.05);

\begin{scope}[font=\tiny]

\node at (90:0.2) {$\gamma$};
\node at (63:0.72) {$\beta$};
\node at (117:0.72) {$\alpha$};
\node at (47:0.43) {$\beta$};
\node at (133:0.43) {$\alpha$};

\node at (115:1) {$\alpha$};
\node at (65:1) {$\alpha$};
\node at (59:1.23) {$\gamma$};
\node at (121:1.23) {$\beta$};
\node at (90:1.25) {$\beta$};

\node at (133:1.23) {$\beta$};
\node at (47:1.23) {$\gamma$};
\node at (130:0.9) {$\beta$};
\node at (50:0.9) {$\alpha$};
\node at (30:0.65) {$\alpha$};

\node at (69:1.75) {$\alpha$};
\node at (31:1.75) {$\beta$};
\node at (80:1.45) {$\alpha$};
\node at (20:1.45) {$\beta$};
\node at (50:1.55) {$\gamma$};

\node at (92:1.54) {$\alpha$};
\node at (128:1.54) {$\beta$};
\node at (80:1.72) {$\alpha$};
\node at (140:1.72) {$\beta$};
\node at (110:1.75) {$\gamma$};

\node at (101:1.95) {$\gamma$};
\node at (39:1.95) {$\beta$};
\node at (70:2) {$\beta$};
\node at (89:2.25) {$\alpha$};
\node at (51:2.25) {$\alpha$};

\node at (100:2.2) {$\alpha$};
\node at (40:2.2) {$\beta$};
\node at (112:2.03) {$\gamma$};
\node at (28:2.03) {$\beta$};
\node at (10:2.27) {$\alpha$};

\node at (13:2.5) {$\alpha$};
\node at (7:2.5) {$\beta$};
\node at (50:2.5) {$\alpha$};
\node at (90:2.5) {$\beta$};
\node at (70:2.6) {$\gamma$};

\end{scope}

\end{scope}
}

\begin{scope}[font=\tiny]

\node[inner sep=0.5,draw,shape=circle] at (90:1) {1};
\node[inner sep=0.5,draw,shape=circle] at (30:1) {2};
\node[inner sep=0.5,draw,shape=circle] at (90:0.5) {3};
\node[inner sep=0.5,draw,shape=circle] at (-30:0.5) {4};
\node[inner sep=0.5,draw,shape=circle] at (-30:1) {5};
\node[inner sep=0.5,draw,shape=circle] at (60:1.55) {6};
\node[inner sep=0.5,draw,shape=circle] at (-90:1) {8};
\node[inner sep=0.5,draw,shape=circle] at (-60:1.55) {9};
\node[inner sep=0.5,draw,shape=circle] at (-10:1.5) {7};
\node[inner sep=0,draw,shape=circle] at (-40:2.05) {11};
\node[inner sep=0,draw,shape=circle] at (10:2) {10};
\node[inner sep=0,draw,shape=circle] at (-40:2.6) {12};

\end{scope}

\end{tikzpicture}
\caption{}
\label{3EFigD1}
\end{subfigure}
\begin{subfigure}[b]{0.3\linewidth}
\centering
\begin{tikzpicture}[>=latex]

\raisebox{7.5ex}{

\begin{scope}[]

\foreach \a in {-1,0}
{
\begin{scope}[rotate=120*\a]

\draw
	(0,0) -- (30:0.5) -- (60:0.9) -- (120:0.9) -- (150:0.5) -- (0,0);

\node at (90:0.2) {\tiny $\gamma$};
\node at (63:0.72) {\tiny $\beta$};
\node at (117:0.72) {\tiny $\alpha$};
\node at (47:0.43) {\tiny $\beta$};
\node at (133:0.43) {\tiny $\alpha$};

\end{scope}
}

\draw
	(0:0.9) -- (10:1.4) -- (50:1.4)
	(120:0.9) -- (130:1.4) -- (90:1.4) -- (50:1.4) -- (60:0.9)
	(90:1.4) -- (70:1.9) -- (10:2.2) -- (10:1.4) -- (50:1.4)
	(-60:0.9) -- (-70:1.4) -- (10:2.2);

\node at (13:1.23) {\tiny $\beta$};
\node at (47:1.23) {\tiny $\gamma$};
\node at (10:0.9) {\tiny $\beta$};
\node at (50:0.9) {\tiny $\alpha$};
\node at (30:0.65) {\tiny $\alpha$};

\node at (115:1) {\tiny $\alpha$};
\node at (65:1) {\tiny $\alpha$};
\node at (59:1.23) {\tiny $\gamma$};
\node at (121:1.23) {\tiny $\beta$};
\node at (90:1.25) {\tiny $\beta$};

\node at (69:1.75) {\tiny $\beta$};
\node at (13:2) {\tiny $\alpha$};
\node at (80:1.45) {\tiny $\beta$};
\node at (14:1.5) {\tiny $\alpha$};
\node at (50:1.55) {\tiny $\gamma$};

\node at (-5:1) {\tiny $\alpha$};
\node at (6:1.4) {\tiny $\alpha$};

\node[inner sep=0.5,draw,shape=circle] at (90:1) {\tiny 1};
\node[inner sep=0.5,draw,shape=circle] at (30:1) {\tiny 2};
\node[inner sep=0.5,draw,shape=circle] at (90:0.5) {\tiny 3};
\node[inner sep=0.5,draw,shape=circle] at (-30:0.5) {\tiny 4};
\node[inner sep=0.5,draw,shape=circle] at (-30:1) {\tiny 5};
\node[inner sep=0.5,draw,shape=circle] at (60:1.55) {\tiny 6};

\fill
	(0,0) circle (0.06);

\draw[line width=1.5]
	(50:1.4) -- (60:0.9);

\end{scope}
}

\end{tikzpicture}
\caption{}
\label{3EFigD2}
\end{subfigure}
\begin{subfigure}[b]{0.2\linewidth}
\centering
\begin{tikzpicture}[>=latex]

\raisebox{9ex}{

\fill (-1,-0.5) circle (0.07);

\foreach \a in {1,-1}
\draw[xscale=\a]
	(0,-0.8) -- (0.4,-0.5) -- (0.4,0.5) -- (0,0.8) 
	(0.4,-0.5) -- (1,-0.5) -- (1,0.5) -- (0.4,0.5)
	(0,0) -- (0.4,0);

\draw[line width=1.5]
	(0.4,0) -- (-0.4,0);

\foreach \y in {1,-1}
{
\begin{scope}[yscale=\y]

\node at (-0.25,0.15) {\scriptsize $\alpha$};
\node at (0.25,0.15) {\scriptsize $\gamma$};
\node at (0.25,0.45) {\scriptsize $\beta$};
\node at (-0.25,0.45) {\scriptsize $\alpha$};
\node at (0,0.6) {\scriptsize $\beta$};

\end{scope}
}

\node at (0.55,0.35) {\scriptsize $\alpha$};
\node at (0.85,0.35) {\scriptsize $\alpha$};
\node at (0.55,0) {\scriptsize $\gamma$};	
\node at (0.55,-0.35) {\scriptsize $\beta$};
\node at (0.85,-0.35) {\scriptsize $\beta$};

\node at (-0.55,0.35) {\scriptsize $\alpha$};
\node at (-0.85,0.35) {\scriptsize $\alpha$};
\node at (-0.55,0) {\scriptsize $\beta$};	
\node at (-0.55,-0.35) {\scriptsize $\beta$};
\node at (-0.85,-0.35) {\scriptsize $\gamma$};

\node[inner sep=0.5,draw,shape=circle] at (0,0.3) {\tiny $1$};
\node[inner sep=0.5,draw,shape=circle] at (0,-0.3) {\tiny $2$};
\node[inner sep=0.5,draw,shape=circle] at (-0.8,0) {\tiny $3$};
\node[inner sep=0.5,draw,shape=circle] at (0.8,0) {\tiny $6$};

\node at (0,-1.2) {$N(\alpha^{\gamma}|^{\gamma}\alpha)$};

}

\end{tikzpicture}
\caption{}
\label{3EFigD3}
\end{subfigure}
\caption{Tiling for AVC(3E24), not pentagonal subdivision.}
\label{3EFigD} 
\end{figure}

Now we concentrate on AVC(3E24). In Figure \ref{3EFigD3}, by $\beta_2\cdots=\beta^4$ and the one $\alpha$ in \circled{5} we already know, we determine $\circled{5}, \circled{6}$ and one $\beta$ in $\circled{7}$. Then $\alpha_5\beta'_4\cdots=\alpha^2\beta$ and $\gamma_5\cdots=\gamma^3$ determine $\circled{8}$ and $\gamma_9$. Then $\circled{5}, \circled{8}, \circled{4}$ is of the same pattern as $\circled{1}, \circled{2}, \circled{3}$, and Therefore, determines one $N(\alpha^{\gamma}|^{\gamma}\alpha)$, including $\circled{9}$. 

The one $\beta$ in $\circled{7}$ we already know and $\alpha_9\beta_5\cdots=\alpha^2\beta$ determine $\circled{7}$. By $\gamma_7\cdots=\gamma^3$, and $\beta'_6\beta_7\cdots=\beta^4$, and $\alpha_7\alpha'_9\cdots=\alpha^2\beta$, we determine $\circled{\footnotesize 10},\circled{\footnotesize 11}$. Then $\circled{\footnotesize 11}, \circled{\footnotesize 10}, \circled{7}$ is of the same pattern as $\circled{1}, \circled{2}, \circled{6}$, and Therefore, determines one $N(\alpha^{\gamma}|^{\gamma}\alpha)$, including $\circled{\footnotesize 12}$. 

We start from $N(\alpha^{\gamma}|^{\gamma}\alpha)$ consisting of $\circled{1}, \circled{2}, \circled{3},\circled{6}$, and deduce two $N(\alpha^{\gamma}|^{\gamma}\alpha)$ consisting of $\circled{5}, \circled{8},\circled{4},\circled{9}$, and of $ \circled{\footnotesize 11},\circled{\footnotesize 10}, \circled{7},\circled{\footnotesize 12}$. We repeat the argument to $\circled{5}, \circled{8},\circled{4},\circled{9}$ in place of $\circled{1}, \circled{2}, \circled{3},\circled{6}$, and then repeat the process again. Then we get the whole tiling in Figure \ref{3EFigD1}. 

A 3d rendering of the tiling for AVC(3E24) that is not the non-pentagonal subdivision is given by Figure \ref{3E24FigB}. Since the tiling cannot be by congruent polygons with straight edges, and we try to keep all the angles to have correct values, some edges must be curved. Moreover, we can see the symmetry of the tiling is $S_3$, the symmetry group  over three elements. 

\begin{figure}[htp]
\centering
\begin{subfigure}[b]{0.4\linewidth}
\centering
\begin{tikzpicture}[>=latex,scale=1]

\pgftext{
	\includegraphics[scale=0.068]{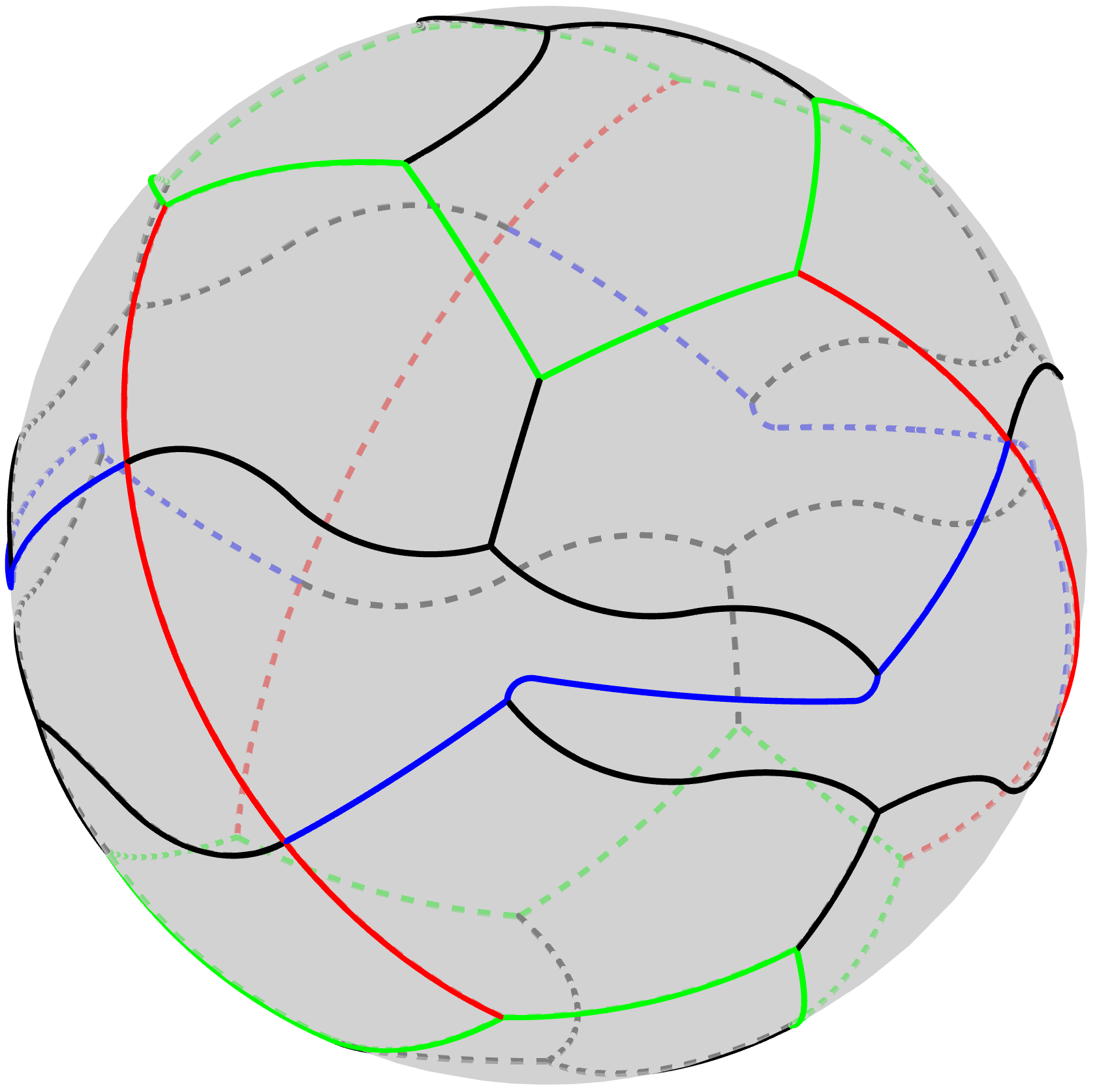}
	};

\end{tikzpicture}	
\caption{AVC(3E24)}
\label{3E24FigB}
\end{subfigure}
\begin{subfigure}[b]{0.4\linewidth}
\centering
\begin{tikzpicture}[>=latex,scale=1]

\pgftext{
	\includegraphics[scale=0.064]{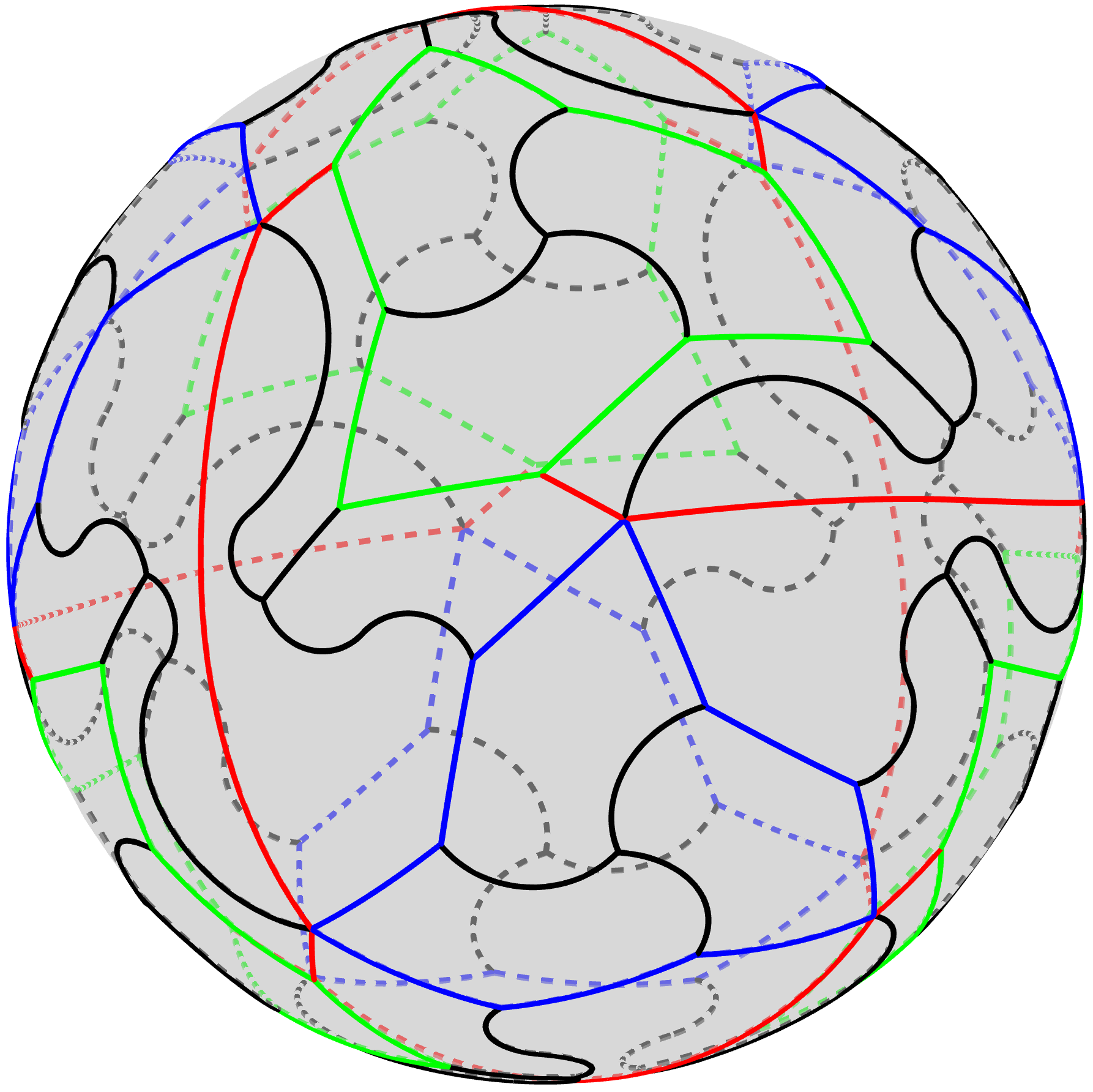}
	};

\end{tikzpicture}	
\caption{AVC(3E60)}
\label{3E60FigB}
\end{subfigure}
\caption{3d renderings of for AVC(3E24) and AVC(3E60) that are not pentagonal subdivision. The angle values are faithful.}
\label{3EFigB}
\end{figure}

Next we consider AVC(3E60). We have $\beta_2\cdots=\beta^5$ in Figure \ref{3EFigE}. Then we may determine $\circled{1}, \dots, \circled{9}$ and their counterparts under rotations in Figure \ref{3EFigE}, similar to the argument for AVC(3E24). The difference is the one $\beta$ in $\circled{\footnotesize 10}$,  between $\circled{6}, \circled{7}$. 

By $\alpha_7\alpha'_9\cdots=\alpha^2\beta$ and $\gamma_7\cdots=\gamma^3$, we determine $\circled{\footnotesize 11}$ and $\gamma_{12}$. Then $\beta'_9\beta'_{11}\cdots=\beta^5$. By the rotational symmetry, this implies $\beta'_6\cdots=\beta^5$. Then we know both $\beta$ in $\circled{\footnotesize 10}$, and one $\beta$ in $\circled{\footnotesize 13},\circled{\footnotesize 17}$. Since the other angles in $\circled{\footnotesize 10}$ are not $\beta$, we get $\beta_7\cdots=\alpha^2\beta$. This, and the two $\beta$ in $\circled{\footnotesize 10}$, and $\gamma_{12}$ determine $\circled{\footnotesize 10},\circled{\footnotesize 12}$. Then $\gamma_{10}\cdots=\gamma^3$, and the one $\beta$ in $\circled{\footnotesize 13}$ we already know, and $\alpha_{10}\alpha'_{12}\cdots=\alpha^2\beta$ determine $\circled{\footnotesize 13}, \circled{\footnotesize 14}$. Since  $\circled{\footnotesize 14}, \circled{\footnotesize 13}, \circled{\footnotesize 10}$ is of the same pattern as $\circled{1}, \circled{2}, \circled{6}$, they determine one $N(\alpha^{\gamma}|^{\gamma}\alpha)$, including $\circled{\footnotesize 15}$. 

\begin{figure}[htp]
\centering
\begin{tikzpicture}[>=latex, scale=0.8]

\foreach \a in {0,1,2}
{
\begin{scope}[rotate=120*\a, thick]

\draw
	(0,0) -- (30:0.6)
	(10:1.8) -- (50:1.8) -- (90:1.8)
	(65:2.5) -- (95:2.5) -- (125:2.5)
	(35:2.5) -- (15:3.2) -- (-5:3.2)
	(75:3.2) -- (95:2.5)
	(50:4.7) -- (20:4.7) -- (-10:4.7)
	(25:3.8) -- (20:4.7)
	(55:3.2) -- (60:4) -- (75:5.4) -- (95:4) -- (95:6.2) to[out=5, in=120] (75:5.4) to[out=30, in=65] (-25:6.2)
	(83:5.97) -- (81:6.8);
	
\draw[green]
	(30:0.6) -- (60:1.1) -- (120:1.1) -- (150:0.6)
	(-10:5.4) -- (-10:4) -- (10:4) -- (25:3.8) -- (35:3.4) -- (45:3.8) -- (50:4.7) -- (52.6:5.4) to[out=-45, in=85] (-10:5.4);

\draw[blue]
	(-85:2.5) -- (-55:2.5)  -- (-30:1.8) -- (10:1.8) -- (5:2.5) -- (-5:3.2) -- (-25:4) -- (-45:3.2) -- (-65:3.2) -- (-85:2.5)
	(95:6.2) to[out=5, in=120] (75:5.4) to[out=30, in=65] (-25:6.2);

\draw[red]
	(0:1.1) -- (10:1.8) -- (35:2.5) -- (35:3.4)
	(110:4) -- (95:4) -- (75:5.4) -- (52.6:5.4);

\draw[line width=1.5]
	(60:1.1) -- (50:1.8)
	(15:3.2) -- (10:4)
	(45:3.8) -- (60:4)
	(-25:5.4) -- (-10:5.4);

\fill 
	(0,0) circle (0.07)
	(-25:2.5) circle (0.07)
	(20:4.7) circle (0.07) 
	(81:6.8) circle (0.07);
			
\end{scope}

}

\foreach \a in {0,1,2}
\filldraw[fill=white, rotate=120*\a]
	(10:1.8) circle (0.07)
	(35:2.5) circle (0.07) 
	(-25:4) circle (0.07)
	(75:5.4) circle (0.07);

\begin{scope}[rotate=60, xscale=-1]

\foreach \a in {0,1,2}
{
\begin{scope}[rotate=120*\a, font=\tiny]

\node at (0,-0.2) {$\gamma$};
\node at (0.17,0.5) {$\beta$};
\node at (-0.17,0.5) {$\alpha$};
\node at (0.5,0.7) {$\beta$};
\node at (-0.5,0.7) {$\alpha$};

\node at (0,0.8) {$\alpha$};
\node at (0.37,1.05) {$\alpha$};
\node at (-0.37,1.05) {$\beta$};
\node at (0.42,1.5) {$\gamma$};
\node at (-0.42,1.5) {$\beta$};

\node at (0,-1.6) {$\beta$};
\node at (0.45,-1.1) {$\alpha$};
\node at (-0.45,-1.1) {$\alpha$};
\node at (0.8,-1.35) {$\gamma$};
\node at (-0.8,-1.35) {$\beta$};

\begin{scope}[rotate=-20]

\node at (0,2) {$\gamma$};
\node at (0.85,1.67) {$\alpha$};
\node at (-0.85,1.67) {$\beta$};
\node at (0.55,2.25) {$\alpha$};
\node at (-0.55,2.2) {$\beta$};

\end{scope}

\begin{scope}[rotate=-80]

\node at (-0.6,2.2) {$\gamma$};
\node at (0.45,1.9) {$\beta$};
\node at (0.45,2.25) {$\beta$};
\node at (-0.6,1.85) {$\alpha$};
\node at (-1.3,1.87) {$\alpha$};

\end{scope}

\begin{scope}[rotate=-35]

\node at (0,2.7) {$\beta$};
\node at (-0.85,2.45) {$\beta$};
\node at (0.85,2.45) {$\gamma$};
\node at (-0.45,2.95) {$\alpha$};
\node at (0.45,2.95) {$\alpha$};

\end{scope}

\node at (-0.55,1.95) {$\beta$};
\node at (-0.1,2.45) {$\beta$};
\node at (-0.95,2.35) {$\alpha$};
\node at (-1.45,2.6) {$\alpha$};
\node at (-0.85,2.85) {$\gamma$};

\begin{scope}[rotate=-80]

\node at (-0.6,2.6) {$\gamma$};
\node at (0.4,2.55) {$\alpha$};
\node at (0.15,3.1) {$\alpha$};
\node at (-0.95,3.6) {$\beta$};
\node at (-1.5,2.75) {$\beta$};

\end{scope}

\begin{scope}[rotate=-5]

\node at (0.2,2.8) {$\beta$};
\node at (0.2,3.3) {$\alpha$};
\node at (0.65,3.55) {$\alpha$};
\node at (1.35,3.5) {$\gamma$};
\node at (1,3.15) {$\beta$};

\node at (-0.2,2.8) {$\beta$};
\node at (-0.2,3.3) {$\beta$};
\node at (-0.65,3.55) {$\alpha$};
\node at (-1.25,3.5) {$\alpha$};
\node at (-1,3.15) {$\gamma$};

\node at (0,3.65) {$\alpha$};
\node at (0.5,3.85) {$\beta$};
\node at (-0.5,3.85) {$\alpha$};
\node at (0.85,4.35) {$\beta$};
\node at (-0.85,4.35) {$\gamma$};

\node at (-1,3.9) {$\beta$};
\node at (-1.3,4.25) {$\gamma$};
\node at (-1.75,3.75) {$\beta$};
\node at (-2.8,3.1) {$\alpha$};
\node at (-3,3.3) {$\alpha$};

\node at (-1.2,4.75) {$\gamma$};
\node at (1.13,4.7) {$\alpha$};
\node at (1.35,5) {$\alpha$};
\node at (-3.3,3.55) {$\beta$};
\node at (-3.55,3.8) {$\beta$};

\node at (1,3.9) {$\alpha$};
\node at (1.4,4.5) {$\alpha$};
\node at (1.65,4.9) {$\beta$}; %
\node at (1.7,3.85) {$\gamma$};
\node at (3.05,4.15) {$\beta$};

\node at (-3.05,2.3) {$\beta$};
\node at (-2.2,2.55) {$\beta$};
\node at (-1.75,3.35) {$\alpha$};
\node at (-2.6,2.75) {$\alpha$};
\node at (-1.4,3.05) {$\gamma$};

\node at (-3.5,2.3) {$\beta$};
\node at (-3.1,2.8) {$\beta$};
\node at (-3.5,3.2) {$\alpha$};
\node at (-3.8,3.5) {$\alpha$};
\node at (-4.4,2.8) {$\gamma$};

\node at (3.25,3.85) {$\beta$};
\node at (3.25,2.3) {$\beta$};
\node at (2.1,2.65) {$\alpha$};
\node at (1.45,3.05) {$\alpha$};
\node at (1.8,3.45) {$\gamma$};

\node at (-5.1,3.2) {$\beta$};
\node at (3.2,4.5) {$\beta$};
\node at (-3.95,3.95) {$\alpha$};
\node at (1.7,5.3) {$\alpha$};
\node at (-4.75,2.95) {$\gamma$};

\node at (3.65,3.95) {$\beta$};
\node at (3.65,2.3) {$\beta$};
\node at (4.35,3.8) {$\alpha$};
\node at (5.05,3.15) {$\alpha$};
\node at (4.55,2.85) {$\gamma$};

\node at (3.65,4.35) {$\beta$};
\node at (4.35,4.25) {$\beta$};
\node at (4.7,4.1) {$\alpha$};
\node at (5.55,3.2) {$\alpha$};
\node at (-1.5,6.3) {$\gamma$};

\end{scope}

\end{scope}
}

\begin{scope}[font=\tiny]

\node[inner sep=0.5,draw,shape=circle] at (150:1.25) {1};
\node[inner sep=0.5,draw,shape=circle] at (210:1.3) {2};
\node[inner sep=0.5,draw,shape=circle] at (178:2) {6};
\node[inner sep=0.5,draw,shape=circle] at (150:0.6) {3};
\node[inner sep=0.5,draw,shape=circle] at (-90:0.6) {4};
\node[inner sep=0.5,draw,shape=circle] at (-90:1.25) {5};
\node[inner sep=0.5,draw,shape=circle] at (330:1.3) {8};
\node[inner sep=0.5,draw,shape=circle] at (-62:2) {9};
\node[inner sep=0.5,draw,shape=circle] at (255:2) {7};
\node[inner sep=0,draw,shape=circle] at (-77:2.7) {11};
\node[inner sep=0,draw,shape=circle] at (222:2.5) {10};
\node[inner sep=0,draw,shape=circle] at (-95:3.2) {12};
\node[inner sep=0,draw,shape=circle] at (215:3.2) {13};
\node[inner sep=0,draw,shape=circle] at (239:3.55) {14};
\node[inner sep=0,draw,shape=circle] at (237:4.25) {15};

\end{scope}

\end{scope}

\end{tikzpicture}
\caption{Tiling for AVC(3E60), not pentagonal subdivision.}
\label{3EFigE} 
\end{figure}

We know the initial $N(\alpha^{\gamma}|^{\gamma}\alpha)$ consisting of $\circled{1}, \circled{2}, \circled{3}, \circled{6}$ determines $\circled{4}, \dots, \circled{\footnotesize 15}$ and their counterparts under rotations in Figure \ref{3EFigE}. Applying the argument to the $N(\alpha^{\gamma}|^{\gamma}\alpha)$ consisting of $\circled{\footnotesize 14}, \circled{\footnotesize 13}, \circled{\footnotesize 15},\circled{\footnotesize 10}$ in place of the initial $N(\alpha^{\gamma}|^{\gamma}\alpha)$, and repeat the process again, we get the whole tiling in Figure \ref{3EFigE}. 

A 3d rendering of the tiling for AVC(3E60) that is not the non-pentagonal subdivision is given by Figure \ref{3E60FigB}. The green triangles and blue triangles are $N(\gamma^3)$ and are congruent. The centers of the four green triangles are the vertices of a regular tetrahedron, and the centers of the four blue triangles are the vertices of the dual regular tetrahedron. Therefore the symmetry of the tiling is the same as the orientable symmetries of the tetrahedron, which is $A_4$, the alternating group over four elements.

The tiling in Figures \ref{3EFigD} and \ref{3EFigE} are constructed based on the assumption of the existence of the AAD $\alpha^{\gamma}|^{\gamma}\alpha$. Now we assume the tiling has no $\alpha^{\gamma}|^{\gamma}\alpha$. Then we get $\gamma\cdots=\gamma^3=|^{\alpha}\gamma^{\beta}|^{\alpha}\gamma^{\beta}|^{\alpha}\gamma^{\beta}|$. This further implies no $\beta^{\gamma}|^{\gamma}\beta$ in the tiling. Combined with no $\alpha\gamma\cdots$ and $\beta\gamma\cdots$ in the AVC, we know the AAD of a vertex $\beta^4 / \beta^5$ is a combination of $|^{\alpha}\beta^{\beta}|$. 
 
In Figure \ref{Subfig-3EFigF-P}, we see that $|^{\alpha}\beta^{\beta}|^{\beta}\beta^{\alpha}|$ at the bottom $\circ$ implies $\beta^2\cdots=\beta^4/\beta^5$ at the top $\circ$. Then the top $\circ$ is a combination of $|^{\alpha}\beta^{\beta}|$. Since a pentagon has only two $\beta$, this implies two $\alpha$ angles adjacent to the top $\circ$. Then we have non-adjacent $\alpha$ angles in a pentagon, a contradiction.

\begin{figure}[h!]
\centering
\begin{subfigure}[b]{0.4\linewidth}
\centering
\begin{tikzpicture}[>=latex]

\begin{scope}[shift={(-3cm,-0.15cm)}]

\draw
	(0.5,0.4) -- (-0.5,0.4) -- (-0.8,0) -- (-0.5,-0.4) -- (0.5,-0.4) -- (0.8,0) -- (0.5,0.4)
	(0,-0.8) -- (0,0.8);

\filldraw[fill=white]
	(0,-0.4) circle (0.05)
	(0,0.4) circle (0.05);

\foreach \a in {1,-1}
{
\begin{scope}[xscale=\a]

\node at (0.15,-0.25) {\scriptsize $\beta$};
\node at (0.15,0.25) {\scriptsize $\beta$};
\node at (0.45,-0.25) {\scriptsize $\alpha$};
\node at (0.45,0.25) {\scriptsize $\alpha$};

\end{scope}
}

\end{scope}
\end{tikzpicture}
\caption{}
\label{Subfig-3EFigF-P}
\end{subfigure}
\begin{subfigure}[b]{0.4\linewidth}
\centering
\begin{tikzpicture}[>=latex]

\foreach \a in {-1,1}
\draw[xscale=\a]
	(0.4,0.5) -- (0.4,0) -- (0,-0.3) -- (0,-0.8) -- (0.8,-0.8) -- (0.8,-0.3) -- (0.4,0)
	(0.8,-0.3) -- (1.2,0) -- (1.2,0.5) -- (0,0.5);

\filldraw[fill=white]
	(0,-0.8) circle (0.05)
	(0.4,0.5) circle (0.05);
	
\node at (-0.4,-0.2) {\scriptsize $\gamma$};
\node at (-0.15,-0.35) {\scriptsize $\beta$};
\node at (-0.65,-0.35) {\scriptsize $\alpha$};
\node at (-0.15,-0.65) {\scriptsize $\beta$};
\node at (-0.65,-0.65) {\scriptsize $\alpha$};

\node at (0.4,-0.2) {\scriptsize $\alpha$};
\node at (0.15,-0.35) {\scriptsize $\alpha$};
\node at (0.65,-0.35) {\scriptsize $\gamma$};
\node at (0.15,-0.65) {\scriptsize $\beta$};
\node at (0.65,-0.65) {\scriptsize $\beta$};

\node at (-0.55,0.05) {\scriptsize $\gamma$};
\node at (-1.05,0.05) {\scriptsize $\beta$};
\node at (-0.55,0.35) {\scriptsize $\alpha$};
\node at (-1.05,0.35) {\scriptsize $\alpha$};
\node at (-0.8,-0.1) {\scriptsize $\beta$};

\node at (0.55,0.05) {\scriptsize $\beta$};
\node at (1.05,0.05) {\scriptsize $\alpha$};
\node at (0.55,0.35) {\scriptsize $\beta$};
\node at (1.05,0.35) {\scriptsize $\alpha$};
\node at (0.8,-0.1) {\scriptsize $\gamma$};

\node at (-0.25,0.05) {\scriptsize $\gamma$};
\node at (0.25,0.05) {\scriptsize $\alpha$};
\node at (-0.25,0.35) {\scriptsize $\beta$};
\node at (0.25,0.35) {\scriptsize $\beta$};
\node at (0,-0.1) {\scriptsize $\alpha$};

\node at (-0.95,-0.35) {\scriptsize $\alpha$};
\node at (0.95,-0.35) {\scriptsize $\gamma$};
%\node at (-0.4,0.65) {\scriptsize $\alpha$};

\node[inner sep=0.5,draw,shape=circle] at (-0.4,-0.5) {\scriptsize $1$};
\node[inner sep=0.5,draw,shape=circle] at (0.4,-0.5) {\scriptsize $2$};
\node[inner sep=0.5,draw,shape=circle] at (-0.8,0.2) {\scriptsize $3$};
\node[inner sep=0.5,draw,shape=circle] at (0,0.2) {\scriptsize $4$};
\node[inner sep=0.5,draw,shape=circle] at (0.8,0.2) {\scriptsize $5$};

\end{tikzpicture}
\caption{}
\label{Subfig-3EFigF-N}
\end{subfigure}
\caption{Tiling for AVC(3E60), in case of no $\alpha^{\gamma}|^{\gamma}\alpha$.}
\label{3EFigF} 
\end{figure}

Since $\beta^4 / \beta^5$ is a combination of $|^{\alpha}\beta^{\beta}|$, and $|^{\alpha}\beta^{\beta}|^{\beta}\beta^{\alpha}|$ is not allowed, we get $\beta^4=|^{\alpha}\beta^{\beta}|^{\alpha}\beta^{\beta}|^{\alpha}\beta^{\beta}|^{\alpha}\beta^{\beta}|$ and $\beta^5=|^{\alpha}\beta^{\beta}|^{\alpha}\beta^{\beta}|^{\alpha}\beta^{\beta}|^{\alpha}\beta^{\beta}|^{\alpha}\beta^{\beta}|$. The AAD $|^{\alpha}\beta^{\beta}|^{\alpha}\beta^{\beta}|$ determines $\circled{1}, \circled{2}$ in Figure \ref{Subfig-3EFigF-N}. Then $\gamma_1\cdots=\gamma_2\cdots=\gamma^3=|^{\alpha}\gamma^{\beta}|^{\alpha}\gamma^{\beta}|^{\alpha}\gamma^{\beta}|$ determine $\circled{3}, \circled{4}, \circled{5}$. Repeating the argument at all $|^{\alpha}\beta^{\beta}|^{\alpha}\beta^{\beta}|$ at the initial $\beta^4/\beta_5$ and new $\beta^4/\beta_5$ such as $\beta'_4\beta'_5\cdots=\beta^4/\beta^5$ gives the $3{\rm E}_1$, $3{\rm E}_2$, $3{\rm E}_3$ reductions of $PP_8$ and $PP_{20}$.
\end{proof}

Next we discuss tilings for AVC(3C24/60). Recall that the pentagons in $3$C are the two in Figure \ref{3abc_arrangement}, one (Figure \ref{Subfig-3abc_arrangement-1}) of them has adjacent $\beta, \gamma$ whereas the other (Figure \ref{Subfig-3abc_arrangement-2}) does not. 

\begin{theorem}\label{3Cthm}
If $\beta, \gamma$ are adjacent in the pentagon, then the tilings for 
\begin{align*}
\text{\rm AVC(3C24)}
&=\{
24\alpha^3\beta\gamma \colon 
24\alpha^2\beta,
8\gamma^3,
6\alpha^4 \};   \\
\text{\rm AVC(3C60)}
&=\{
60\alpha^3\beta\gamma \colon 
60\alpha^2\beta,
20\gamma^3,
12\alpha^5 \},
\end{align*}
are ${\rm 3C_1}, {\rm 3C_2}$ reductions of $PP_8$ and $PP_{20}$, and the modifications caused by independently changing the orientations of each $N(\gamma^3)$.
\end{theorem}

The orientation modifications are in one-to-one correspondence with the assignment of $(+)$ and $(-)$ to the triangular faces of the octahedron and the icosahedron (treating mirror image assignments as distinct). The total numbers of assignments are $23$ for $P_8$ and $17824$ for $P_{20}$.

If $\beta$ and $\gamma$ are non-adjacent, which is the case for $3{\rm C}_3$, then the tilings are much more complicated. We will discuss the non-adjacent case after the proof.

\begin{proof} 
The AVCs imply that any tiling is a union of the $N(\gamma^3)$. We only need to find how $N(\gamma^3)$ instances are glued together. 

If $\beta, \gamma$ are adjacent in the pentagon (Figure \ref{Subfig-3abc_arrangement-1}), then $\gamma$ is adjacent to $\alpha$ and $\beta$. 
Since $\beta^2\cdots$ is not a vertex, the AAD of $\gamma^3$ must be $|^{\alpha}\gamma^{\beta}|^{\alpha}\gamma^{\beta}|^{\alpha}\gamma^{\beta}|$. This implies that $N(\gamma^3)$ is $\circled{1}, \circled{2}, \circled{3}$ in Figure \ref{Subfig-3CFigA-N1} or \ref{Subfig-3CFigA-N2}. The two differ only in their orientations. It is clear that, in a tiling for AVC(3C24/60), we may change the orientation of any one $N(\gamma^3)$ and still get a tiling. Therefore, we may write $\alpha\beta$ instead of $\alpha|\beta$ or $\beta|\alpha$ along the boundary of $N(\gamma^3)$, to indicate this property.

\begin{figure}[htp]
\centering
\begin{subfigure}[t]{0.2\linewidth}
\centering
\begin{tikzpicture}[>=latex]

%% 1

\raisebox{0.5cm}{

\foreach \a in {0,1,2}{
\begin{scope}[rotate=120*\a]

\draw
	(0,0) -- (90:0.6) -- (50:1) -- (10:1) -- (-30:0.6);
	
\node at (0,-0.2) {\scriptsize $\gamma$};
\node at (0.33,-0.38) {\scriptsize $\alpha$};
\node at (-0.33,-0.38) {\scriptsize $\beta$};
\node at (0.22,-0.77) {\scriptsize $\alpha$};
\node at (-0.22,-0.77) {\scriptsize $\alpha$};

\fill
	(0,0) circle (0.05);
	
\end{scope}
}

%\foreach \b in {1,-1}
\draw[->]
	(-60:0.33) arc (-60:240:0.33);

}
	
\end{tikzpicture}
\caption{}
\label{Subfig-3CFigA-N1}
\end{subfigure}
\begin{subfigure}[t]{0.2\linewidth}
\centering
\begin{tikzpicture}[>=latex]

%% 2

\raisebox{0.5cm}{

\foreach \a in {0,1,2}{
\begin{scope}[xscale=-1, rotate=120*\a]

\draw
	(0,0) -- (90:0.6) -- (50:1) -- (10:1) -- (-30:0.6);
	
\node at (0,-0.2) {\scriptsize $\gamma$};
\node at (0.33,-0.38) {\scriptsize $\alpha$};
\node at (-0.33,-0.38) {\scriptsize $\beta$};
\node at (0.22,-0.77) {\scriptsize $\alpha$};
\node at (-0.22,-0.77) {\scriptsize $\alpha$};

\fill
	(0,0) circle (0.05);
	
\end{scope}
}

\foreach \b in {1,-1}
\draw[xscale=-1,->]
	(-60:0.33) arc (-60:240:0.33);

}

\end{tikzpicture}
\caption{}
\label{Subfig-3CFigA-N2}
\end{subfigure}
\begin{subfigure}[t]{0.25\linewidth}
\centering
\begin{tikzpicture}[>=latex]

%% 3

\begin{scope}[
%xshift=6cm-1.5*\b cm, xscale=\b
]

\foreach \a in {0,1,2}
{
\begin{scope}[rotate=120*\a]
	
\draw[gray!70]
	(0,0) -- (90:0.76);
	
\draw
	(10:1) -- (50:1) -- (130:1);
	
\node at (13:0.8) {\scriptsize $\alpha$};
\node at (47:0.8) {\scriptsize $\alpha$};
\node at (30:0.2) {\scriptsize $\gamma$};
\node at (90:0.56) {\scriptsize $\alpha\beta$};

\end{scope}
}

\draw[gray!70]
	(-0.64,1.96) -- (-0.64,0.76)
	(0.64,1.36) -- (-0.64,1.36);

\draw
	(0.64,0.76) -- (0.64,1.96) -- (-1.4,1.96) -- (170:1);

\fill
	(0,0) circle (0.05)
	(-0.64,1.36) circle (0.05);

\node at (0.5,0.9) {\scriptsize $\alpha$};
\node at (0.4,1.36) {\scriptsize $\alpha\beta$};
\node at (0,0.9) {\scriptsize $\alpha$};
\node at (-0.64,0.9) {\scriptsize $\alpha\beta$};
\node at (-0.5,1.2) {\scriptsize $\gamma$};

\node at (0.5,1.8) {\scriptsize $\alpha$};
\node at (0,1.8) {\scriptsize $\alpha$};
\node at (-0.64,1.8) {\scriptsize $\alpha\beta$};
\node at (-0.5,1.5) {\scriptsize $\gamma$};

\node at (-0.8,1.36) {\scriptsize $\gamma$};
\node at (-1.2,1.8) {\scriptsize $\alpha$};
\node at (-0.95,0.5) {\scriptsize $\alpha$};

\node[inner sep=0.5,draw,shape=circle] at (0,1.15) {\scriptsize $1$};

\end{scope}

\end{tikzpicture}
\caption{}
\label{Subfig-3CFigA-fit1}
\end{subfigure}
\begin{subfigure}[t]{0.25\linewidth}
\centering
\begin{tikzpicture}[>=latex]

%% 4

\begin{scope}[
%xshift=6cm-1.5*\b cm, 
xscale=-1
]

\foreach \a in {0,1,2}
{
\begin{scope}[rotate=120*\a]
	
\draw[gray!70]
	(0,0) -- (90:0.76);
	
\draw
	(10:1) -- (50:1) -- (130:1);
	
\node at (13:0.8) {\scriptsize $\alpha$};
\node at (47:0.8) {\scriptsize $\alpha$};
\node at (30:0.2) {\scriptsize $\gamma$};
\node at (90:0.56) {\scriptsize $\alpha\beta$};

\end{scope}
}

\draw[gray!70]
	(-0.64,1.96) -- (-0.64,0.76)
	(0.64,1.36) -- (-0.64,1.36);

\draw
	(0.64,0.76) -- (0.64,1.96) -- (-1.4,1.96) -- (170:1);

\fill
	(0,0) circle (0.05)
	(-0.64,1.36) circle (0.05);

\node at (0.5,0.9) {\scriptsize $\alpha$};
\node at (0.4,1.36) {\scriptsize $\alpha\beta$};
\node at (0,0.9) {\scriptsize $\alpha$};
\node at (-0.64,0.9) {\scriptsize $\alpha\beta$};
\node at (-0.5,1.2) {\scriptsize $\gamma$};

\node at (0.5,1.8) {\scriptsize $\alpha$};
\node at (0,1.8) {\scriptsize $\alpha$};
\node at (-0.64,1.8) {\scriptsize $\alpha\beta$};
\node at (-0.5,1.5) {\scriptsize $\gamma$};

\node at (-0.8,1.36) {\scriptsize $\gamma$};
\node at (-1.2,1.8) {\scriptsize $\alpha$};
\node at (-0.95,0.5) {\scriptsize $\alpha$};

\node[inner sep=0.5,draw,shape=circle] at (0,1.15) {\scriptsize $1$};

\end{scope}

\end{tikzpicture}
\caption{}
\label{Subfig-3CFigA-fit2}
\end{subfigure}
\caption{Orientations of $N(\gamma^3)$, and glueing of two neighboring $N(\gamma^3)$.  }
\label{3CFigA}
\end{figure}

We start with $N(\gamma^3)$ as shown in the lower parts of Figures \ref{Subfig-3CFigA-fit1}, \ref{Subfig-3CFigA-fit2}. Then $\alpha\beta\cdots=\alpha^2\beta$ gives one $\alpha_1$. Since $\gamma\cdots=\gamma^3$, we know $\gamma_1$ is not at the boundary of $N(\gamma^3)$. Then we get two possible locations of $\gamma_1$, and further determine the new $N(\gamma^3)$ around $\gamma_1\cdots=\gamma^3$. The new $N(\gamma^3)$ is glued to the original $N(\gamma^3)$ along either the thick part or the dashed part of the boundary, as in Figure \ref{Subfig-3CFigB-N1}.

\begin{figure}[htp]
\centering
\begin{subfigure}[t]{0.2\linewidth}
\centering
\begin{tikzpicture}[>=latex]

\raisebox{1cm}{

\foreach \a in {0,1,2}
{
\begin{scope}[rotate=120*\a]
	
\draw[gray!70]
	(0,0) -- (90:0.76);
	
\draw
	(10:1) -- (50:1) -- (130:1);

\fill (0,0) circle (0.05);
	
\node at (13:0.8) {\scriptsize $\alpha$};
\node at (47:0.8) {\scriptsize $\alpha$};
\node at (30:0.2) {\scriptsize $\gamma$};
\node at (90:0.56) {\scriptsize $\alpha\beta$};

\end{scope}
}

\draw[line width=1.2]
	(50:1.1) -- (130:1.1) -- (170:1.1); 

\draw[dashed, line width=1.2]
	(14:1.2) -- (50:1.25) -- (126:1.2); 

}
	
\end{tikzpicture}
\caption{}
\label{Subfig-3CFigB-N1}
\end{subfigure}
\begin{subfigure}[t]{0.2\linewidth}
\centering
\begin{tikzpicture}[>=latex]

\raisebox{1cm}{

\foreach \a in {0,1,2}
{
\begin{scope}[rotate=120*\a]
	
\draw[gray!70]
	(0,0) -- (90:0.76);
	
\draw
	(10:1) -- (50:1) -- (130:1);

\fill (0,0) circle (0.05);
	
\node at (13:0.8) {\scriptsize $\alpha$};
\node at (47:0.8) {\scriptsize $\alpha$};
\node at (30:0.2) {\scriptsize $\gamma$};
\node at (90:0.56) {\scriptsize $\alpha\beta$};

\end{scope}
}

\foreach \a in {0,1,2}
\draw[line width=1.2, rotate=120*\a]
	(52:1.1) -- (130:1.13) -- (168:1.1); 

}

\end{tikzpicture}
\caption{}
\label{Subfig-3CFigB-N2}
\end{subfigure}
\begin{subfigure}[t]{0.45\linewidth}
\centering
\begin{tikzpicture}[>=latex]

%% 2

\begin{scope}[shift={(7cm,-1cm)}]

\foreach \a in {0,1,2}
\foreach \b in {0,1,2}
{
\begin{scope}[shift={(-30+120*\b:1.2)}, rotate=120*\a]

\draw[gray!70]
	(0,0) -- (0.693,0);
	
\draw
	(90:1.2) -- (-30:1.2);
	
\fill
	(0,0) circle (0.05);

\node at (60:0.2) {\scriptsize $\gamma$};
\node at (-30:0.9) {\scriptsize $\alpha$};
\node at (60:0.5) {\scriptsize $\alpha$};
\node at (0:0.48) {\scriptsize $\alpha\beta$};

\end{scope}
}

\foreach \a in {0,1,2}
{
\begin{scope}[rotate=120*\a]

\draw[gray!70]
	(0,0) -- (-0.693,0);

\fill
	(0,0) circle (0.05);
	
\node at (0:0.2) {\scriptsize $\gamma$};
\node at (30:0.9) {\scriptsize $\alpha$};
\node at (0:0.5) {\scriptsize $\alpha$};
\node at (-60:0.48) {\scriptsize $\alpha\beta$};
\end{scope}
}

\foreach \a in {0,1,2}
\filldraw[fill=white, rotate=120*\a]
	(-90:1.2) circle (0.05)
	(90:2.4) circle (0.05);

\end{scope}

\end{tikzpicture}
\caption{}
\label{Subfig-3CFigB-N3}
\end{subfigure}
\caption{Tiling for AVC(3C24/60), $\beta, \gamma$ adjacent.}
\label{3CFigC}
\end{figure}

The two glueing options happen at each of the three $\alpha\beta$ instances along the boundary of $N(\gamma^3)$. For the three neighboring $N(\gamma^3)$ instances to be glued in compatible way, they must all follow the thick lines as shown in Figure \ref{Subfig-3CFigB-N2}, or all follow the dashed lines. The former gives Figure \ref{Subfig-3CFigB-N3}, where we note that the central $N(\gamma^3)$ is glued to the neighboring $N(\gamma^3)$ also along the thick lines. We repeat the argument and find all the $N(\gamma^3)$ are glued together along the thick lines. The tilings we get are the $3{\rm C}_1$ and $3{\rm C}_2$ reductions of $PP_8$ and $PP_{20}$, and the modifications, by independently changing the orientations of $N(\gamma^3)$. We remark that the $3{\rm C}_1$ and $3{\rm C}_2$ reductions are related by changing all the orientations.

Similarly, if one $N(\gamma^3)$ is glued onto the original one along the dashed lines, then all others are also glued along the dashed lines. Then we get the same tilings. 
\end{proof}

If $\beta, \gamma$ are non-adjacent in the pentagon (Figure \ref{Subfig-3abc_arrangement-2}), then the configuration of $N(\gamma^3)$ is not unique. The situation is not amenable to a reasonably concise analysis. We do however find all the tilings for AVC(3C24) by computer search, and illustrate them in Figure \ref{3CFigB}.

\begin{figure}[htp]
\centering
\begin{tikzpicture}[>=latex]

\foreach \a in {0,...,4}
\draw[gray!70]
	(0,-3*\a) -- ++(12,0);

\foreach \a in {4,5}
\draw[gray!70]
	(0,-3*\a) -- ++(0,-3) -- ++(9,0) -- ++(0,3);
	
\foreach \a in {0,...,3}
\draw[gray!70]
	(3*\a,0) -- ++(0,-6);

\foreach \a in {0,1,2}
\draw[gray!70]
	(6*\a,-6) -- ++(0,-6);
	
\draw[gray!70]
	(12,0) -- ++(0,-3);

\begin{scope}[shift={(10.5 cm, -15cm)}]

\foreach \a in {1,-1}
{
\draw[yshift=0.5*\a cm, line width=1.5, gray!70]
	(-0.5,0) -- (0.5,0) 
;

\draw[yshift=0.5*\a cm]
	(-0.7,0.2) -- (-0.5,0) -- (0.5,0) -- (0.7,0.2)
	(-0.7,-0.2) -- (-0.5,0) -- (0.5,0) -- (0.7,-0.2)
;
	
\begin{scope}[yscale=\a]

\node at (-0.4,0.35) {\scriptsize $\alpha$};
\node at (-0.4,0.65) {\scriptsize $\beta$};
\node at (0.4,0.65) {\scriptsize $\alpha$};
\node at (0.4,0.35) {\scriptsize $\beta$};

\end{scope}

}

\end{scope}

%% 91

\begin{scope}[shift={(1.5 cm, -1.5cm)}, scale=0.5]

\foreach \a in {0,1,2}
{
\begin{scope}[rotate=120*\a]

\draw[very thin]
	(0,0) -- (30:0.5) -- (60:0.9) -- (120:0.9) -- (150:0.5)
	(120:0.9) -- (130:1.4) -- (90:1.4) -- (50:1.4) -- (10:1.4) -- (0:0.9)
	(50:1.4) -- (60:0.9)
	(90:1.4) -- (70:1.9) -- (30:1.9) -- (-10:1.9) -- (-50:1.9)
	(30:1.9) -- (10:1.4)
	(110:1.9) -- (90:2.4) -- (50:2.4) -- (10:2.4) -- (-30:2.4)
	(50:2.4) -- (30:1.9)
	(90:2.4) -- (90:2.8);

\fill
	(0,0) circle (0.06)
	(50:1.4) circle (0.06)
	(-10:1.9) circle (0.06)
	(90:2.8) circle (0.06);

\fill[gray!70]
	(92:1.5) circle (0.06)
	(115:1) circle (0.06)
	(63:0.75) circle (0.06)
	(69:1.75) circle (0.06)
	(51:2.25) circle (0.06)
	(100:2.2) circle (0.06)
	(10:2.5) circle (0.06);
	
\end{scope}
}

\foreach \a in {0,1,2}
\filldraw[fill=white, rotate=120*\a]
	(30:1.9) circle (0.06)
	(10:1.4) circle (0.06);

\end{scope}

%% 39

\begin{scope}[shift={(4.5 cm, -1.5cm)}, scale=0.5]

\foreach \a in {0,1,2}
{
\begin{scope}[rotate=120*\a]

\draw[very thin]
	(0,0) -- (30:0.5) -- (60:0.9) -- (120:0.9) -- (150:0.5)
	(120:0.9) -- (130:1.4) -- (90:1.4) -- (50:1.4) -- (10:1.4) -- (0:0.9)
	(50:1.4) -- (60:0.9)
	(90:1.4) -- (70:1.9) -- (30:1.9) -- (-10:1.9) -- (-50:1.9)
	(30:1.9) -- (10:1.4)
	(110:1.9) -- (130:2.4) -- (90:2.4) -- (50:2.4) -- (10:2.4)
	(50:2.4) -- (70:1.9)
	(90:2.4) -- (90:2.8);

\fill
	(0,0) circle (0.06)
	(50:1.4) circle (0.06)
	(-10:1.9) circle (0.06)
	(90:2.8) circle (0.06);

\fill[gray!70]
	(92:1.5) circle (0.06)
	(115:1) circle (0.06)
	(63:0.75) circle (0.06)
	(31:1.75) circle (0.06)
	(49:2.27) circle (0.06)
	(90:2.27) circle (0.06)
	(10:2.5) circle (0.06);
	
\end{scope}
}

\foreach \a in {0,1,2}
\filldraw[fill=white, rotate=120*\a]
	(70:1.9) circle (0.06)
	(10:1.4) circle (0.06);

\end{scope}

%% 17

\begin{scope}[shift={(7.5 cm, -1.5cm)}, scale=0.4]

\foreach \a in {1,-1}
{
\begin{scope}[scale=\a]

\foreach \b in {1,-1}
\draw[very thin, xscale=\b]
	(0,2.6) -- (0,1.1) -- (0.5,0.7) --  (1,1.1) -- (1.5,0.7) --  (2,1.1) --  (2,2.6)
	(1,1.1) -- (1,1.8)
	(0,1.8) -- (2,1.8)
	(2,1.1) to[out=-40, in=40] (2,-1.1)
	(0.5,0) -- (0.5,0.7)
	(0,0) -- (1.5,0) -- (1.5,0.7);

\draw[very thin]
	(-2,2.6) -- (2,2.6) to[out=0, in=0] (2,-1.8)
	(2,-2.6) -- (2,-3.4);
	
\fill
	(0.5,0.7) circle (0.075)
	(-1,1.1) circle (0.075)
	(2,1.1) circle (0.075)
	(0,2.6) circle (0.075);

\fill[gray!70]
	(-0.35,0.6) circle (0.075)
	(1.35,0.6) circle (0.075)
	(-0.15,1.2) circle (0.075)
	(-1.85,1.2) circle (0.075)
	(0.85,1.65) circle (0.075)
	(1.85,1.95) circle (0.075)
	(-1,1.95) circle (0.075)
	(1.15,1.2) circle (0.075)
	(1.35,-0.15) circle (0.075)
	(1.65,-0.6) circle (0.075)
	(2.15,-2.6) circle (0.075)
	(2.15,2.4) circle (0.075);

\end{scope}
}

\foreach \a in {1,-1}
\filldraw[fill=white, scale=\a]
	(0,1.8) circle (0.075)
	(0.5,0) circle (0.075)
	(2,-1.8) circle (0.075);  

\end{scope}

%% 47

\begin{scope}[shift={(10.75 cm, -1.5cm)}, scale=0.4]

\foreach \a in {1,-1}
{
\begin{scope}[scale=\a]

\draw[very thin]
	(-2,1.1) -- (-1.5,0.7) -- (-1,1.1) -- (-0.5,0.7) -- (0,1.1) -- (0.5,0.7) -- (1,1.1) -- (1.5,0.7) -- (2,1.1) to[out=-40, in=40] (2,-1.1)  
	(0.5,0.7) -- (0.5,-0.7)
	(1.5,0.7) -- (1.5,-0.7)
	(0,0) -- (1.5,0)
	(0,1.1) -- (0,1.8)
	(1,1.1) -- (1,1.8)
	(-1,1.1) -- (-1,1.8)
	(2,1.1) -- (2,1.8) -- (-2,1.8) -- (-2,1.1);

\fill
	(-1,1.1) circle (0.075)
	(0.5,0.7) circle (0.075)
	(2,1.1) circle (0.075);

\fill[gray!70]
	(0.35,-0.6) circle (0.075)
	(-1.35,0.2) circle (0.075)
	(1.35,0.6) circle (0.075)
	(1.65,-0.6) circle (0.075)
	(-0.15,1.2) circle (0.075)
	(1.15,1.2) circle (0.075)
	(-1,1.95) circle (0.075);
		
\end{scope}
}

\draw[very thin]
	(-2,1.8) -- (-2,2.6) -- (-3,2.6) 
	(-3,3.4) -- (-3,-3.4) 
	(-3,1.8) -- (-2,1.1)
	(-3,-1.8) -- (-2,-1.8)
	(-2,2.6) to [out=0,in=140] (1,1.8)
	(-3,-2.6) to [out=0,in=220] (0,-1.8)
	(2,1.8) -- (2,3.4)
	(2,-1.8) -- (2,-3.4);

\fill
	(-2,2.6) circle (0.075)
	(-3,-2.6) circle (0.075);

\fill[gray!70]
	(1.85,1.95) circle (0.075)
	(0.15,1.65) circle (0.075)
	(-0.85,-1.65) circle (0.075)
	(-1.85,1.65) circle (0.075)
	(1.85,-1.2) circle (0.075)
	(-2.85,1.9) circle (0.075)
	(-2,-1.95) circle (0.075)
	(-2.85,-1.65) circle (0.075)
	(-3.15,2.6) circle (0.075)
	(2.15,-1.8) circle (0.075);
	
\foreach \a in {1,-1}
\filldraw[fill=white, yscale=\a]
	(2,3.4) circle (0.075)
	(-3,3.4) circle (0.075);

\filldraw[fill=white]
	(-2,1.1) circle (0.075)
	(1,1.8) circle (0.075)
	(0,-1.8) circle (0.075)
	(0.5,0) circle (0.075)
	(-0.5,0) circle (0.075);
	
\end{scope}

%% 80

\begin{scope}[shift={(1.5 cm, -4.5cm)}, scale=0.5]

\foreach \a in {1,-1}
{
\begin{scope}[scale=\a]

\draw[very thin]
	(0.9,0) -- (-0.9,0)
	(0.9,0.6) -- (-0.9,0.6)
	(0.3,0.6) -- (-0.3,0)
	(0.9,1.8) -- (-0.3,0.6)
	(-0.3,1.2) -- (-0.9,1.8)
	(0.9,1.2) -- (0.9,-1.2) 
	(-1.5,1.8) -- (1.5,1.8) -- (1.5,1.2) -- (-1.5,1.2) -- (-1.5,2.8)
	(1.5,1.2) to[out=-90, in=60] (0.9,-0.6)
	(1.5,1.8) to[out=-60, in=60](1.5,-1.2)
	(0,1.8) -- (0,2.4) -- (-1.5,2.4) 
	(0,2.4) to[out=0, in=120] 
	(2,1.8) to[out=-60, in=60] (1.5,-1.8)
	;

\fill
	(0.3,0.6) circle (0.06)
	(-0.3,1.2) circle (0.06)
	(1.5,1.2) circle (0.06)
	(0,2.4) circle (0.06);

\fill[gray!70]
	(-0.35,0.7) circle (0.06)
	(0.95,1.7) circle (0.06)
	(0,1.7) circle (0.06)
	(-0.8,0.1) circle (0.06)
	(0.3,0.1) circle (0.06)
	(0.8,1.1) circle (0.06)
	(-1.4,1.3) circle (0.06)
	(-0.9,1.9) circle (0.06)
	(1.5,1.9) circle (0.06)
	(1,0.6) circle (0.06)
	(1,-1.1) circle (0.06)
	(1.6,-2.4) circle (0.06);

\end{scope}
}

\foreach \a in {1,-1}
\filldraw[fill=white, scale=\a]
	(0.9,-0.6) circle (0.06)
	(0.3,1.2) circle (0.06)
	(-1.5,1.8) circle (0.06);

\end{scope}

%% 46

\begin{scope}[shift={(4.5 cm, -4.5cm)}, scale=0.5]

\foreach \a in {1,-1}
{
\begin{scope}[scale=\a]

\draw[very thin]
	(0,0) -- (0.3,0.3)
	(-0.3,0.3) -- (0.3,0.9)
	(-0.9,0.3) -- (1.5,0.3)
	(0.9,0.9) -- (-1.5,0.9)
	(0.9,-0.9) -- (0.9,1.5)
	(-0.3,0.9) -- (-0.3,2.1)
	(1.5,1.5) -- (-1.5,1.5) -- (-1.5,-2.1)
	(1.5,1.5) to[out=-60,in=60] (1.5,-1.5)
	(1.93,0) -- (2.4,0)
	(-2.8,-2.1) -- (0.3,-2.1) to[out=0, in=-90]
	(2.4,0) -- (2.4,2.1);

\fill
	(-0.3,0.3) circle (0.06)
	(-1.5,0.9) circle (0.06)
	(0.9,1.5) circle (0.06)
	(2.4,0) circle (0.06);

\fill[gray!70]
	(0.35,0.2) circle (0.06)
	(0.8,0.8) circle (0.06)
	(1.4,0.4) circle (0.06)
	(1,-0.3) circle (0.06)
	(1.85,0) circle (0.06)
	(1.6,2) circle (0.06)
	(-0.2,2) circle (0.06)
	(0.4,-1) circle (0.06)
	(2.4,2.2) circle (0.06)
	(0.8,-0.8) circle (0.06)
	(-0.3,-1) circle (0.06)
	(1.57,-1.57) circle (0.06);
		
\end{scope}
}

\foreach \a in {1,-1}
\filldraw[fill=white, scale=\a]
	(0.9,0.3) circle (0.06)
	(-0.3,1.5) circle (0.06)
	(1.5,1.5) circle (0.06);

\end{scope}

%% 31

\begin{scope}[shift={(7.5 cm, -4.5cm)}, scale=0.5]

\foreach \a in {1,-1}
\draw[gray, very thick, scale=\a]
	(-1.5,1.1) -- (-2.1,1.1)
	(-2.1,2.3) -- (-2.8,2.3);

\foreach \a in {1,-1}
\foreach \b in {1,-1}
\draw[very thin, xscale=\a, yscale=\b]
	(0,1.1) -- (2.1,1.1) 
	(1.5,1.1) -- (1.5,0)
	(1.5,0) -- (1,0) -- (0.5,0.5) -- (0.5,1.7) -- (0,1.9) -- (0,2.3) -- (2.1,2.3) -- (2.1,0)
	(0,0) -- (0,0.3) -- (0.5,0.5)
	(0.5,1.7) -- (2.1,1.7);

\foreach \a in {1,-1}
{
\begin{scope}[very thin, scale=\a]

\draw
	(-2.1,2.3) -- (-2.8,2.3)
	(2.1,2.3) to[out=-60, in=60] (2.1,-1.7);

\fill
	(-0.5,0.5) circle (0.06)
	(-0.5,1.7) circle (0.06)
	(1.5,1.1) circle (0.06)
	(2.1,2.3) circle (0.06);

\fill[gray!70]
	(0.4,0.57) circle (0.06)
	(0.4,1.63) circle (0.06)
	(0.1,0.22) circle (0.06)
	(1.05,0.1) circle (0.06)
	(-1.4,0.1) circle (0.06)
	(-0.1,2.2) circle (0.06)
	(0.1,1.97) circle (0.06)
	(2,1.6) circle (0.06)
	(2.2,1.1) circle (0.06);

\end{scope}
}

\foreach \a in {1,-1}
\filldraw[fill=white, scale=\a]
	(0.5,1.1) circle (0.06)
	(-0.5,1.1) circle (0.06)
	(2.1,-1.7) circle (0.06);

\end{scope}

%% 3.1 and 3.2

\begin{scope}[shift={(1.5 cm, -7.5cm)}, scale=0.4]

\foreach \a in {1,-1}
{

\foreach \b in {0,1}
\draw[very thin, xshift=7.5*\b cm, scale=\a]
	(0,0) -- (0,0.7) -- (0.5,1.1) -- (1,0.7) -- (1.5,1.1) -- (2,0.7) -- (2,0) -- (1,-0.7) -- (0,-0.7)
	(0,0) -- (2,0)
	(1,0.7) -- (1,0)
	(0.5,1.1) -- (0.5,1.8) 
	(1.5,1.1) -- (1,2.6)
	(2,0) -- (2,-0.7) -- (2.8,-1.1)
	(2,0.7) -- (2.8,1.1)
	(2,2.6) -- (2.8,1.1) -- (2.8,-1.1) -- (2,-2.6) -- (2,-3.4)
	(2,-0.7) -- (1,-1.8) -- (-0.5,-1.8) -- (-1,-2.6) -- (-2,-2.6) -- (-2,-3.4)
	(1,-0.7) -- (0.2,-1.8)
	(-1,2.6) -- (-1,1.8)
	(-2,2.6) -- (1.5,2.6);
	
%% 3.1

\begin{scope}[scale=\a]

\fill
	(1,0.7) circle (0.075)
	(-1,0.7) circle (0.075)
	(-1,2.6) circle (0.075)
	(2.8,1.1) circle (0.075);

\fill[gray!70]
	(-1,0.15) circle (0.075)
	(0.15,0.6) circle (0.075)
	(1.85,0.6) circle (0.075)
	(0.65,1.75) circle (0.075)
	(0.35,1.2) circle (0.075)
	(1.65,1.2) circle (0.075)
	(-0.2,1.95) circle (0.075)
	(2.15,-0.6) circle (0.075)
	(0.92,-1.65) circle (0.075)
	(2.6,-1.17) circle (0.075)
	(1.85,2.75) circle (0.075)
	(-2.15,2.65) circle (0.075)
	;

\end{scope}

%% 3.2

\begin{scope}[xshift=7.5cm, scale=\a]

\foreach \a in {1,-1}	
\draw[gray, very thick, scale=\a]
	(0,0.7) -- (-1,0.7)
	(1.5,1.1) -- (2,0.7)
	(1,-2.6) -- (2,-2.6);

\fill
	(1,0) circle (0.075)
	(2,-0.7) circle (0.075)
	(2,2.6) circle (0.075)
	(0.5,1.8) circle (0.075);

\fill[gray!70]
	(0.5,0.9) circle (0.075)
	(1,0.9) circle (0.075)
	(-0.48,1.65) circle (0.075)
	(-0.85,1.95) circle (0.075)
	(-2.95,1.1) circle (0.075)
	(2.65,0.85) circle (0.075)
	;   
		
\end{scope}
}

\foreach \a in {1,-1}
\foreach \b in {0,1}
\filldraw[fill=white, xshift=7.5*\b cm, scale=\a]
	(0,0) circle (0.075)
	(2,0) circle (0.075)
	(1,2.6) circle (0.075)
	(2,3.4) circle (0.075)
	(-2,3.4) circle (0.075);

\end{scope}

%% 5

\begin{scope}[shift={(7.5 cm, -7.5cm)}, scale=0.4]

\foreach \c in {0,1}
{
\begin{scope}[xshift=7.5*\c cm]

\foreach \a in {1,-1}
\foreach \b in {1,-1}
\draw[very thin, yscale=\b, xscale=\a]
	(0,1.8) -- (0,1.1) -- (0.5,0.7) --  (1,1.1) -- (1.5,0.7) --  (2,1.1) --  (2,1.8)
	(1,1.1) -- (1,1.8)
	(0,1.8) -- (2,1.8)
	(2,1.1) to[out=-40, in=40] (2,-1.1)
	(0.5,0) -- (0.5,0.7)
	(0,0) -- (1.5,0) -- (1.5,0.7);

\foreach \a in {1,-1}
\draw[very thin, scale=\a]
	(3,1.8) -- (1,1.8) -- (1,2.6)
	(2,-1.8) to[out=0, in=-90]
	(3,1.8) -- (3,2.6) -- (1,2.6) to[out=180, in=40]
	(-2,1.8)
	(3,2.6) -- (3,3.4);

\end{scope}
}

\foreach \a in {1,-1}
{
\begin{scope}[scale=\a]

\fill
	(0.5,0.7) circle (0.075)
	(-1,1.1) circle (0.075)
	(2,1.1) circle (0.075)
	(1,2.6) circle (0.075);

\fill[gray!70]
	(-0.35,0.6) circle (0.075)
	(1.35,0.6) circle (0.075)
	(-0.15,1.2) circle (0.075)
	(-1.85,1.2) circle (0.075)
	(0.85,1.65) circle (0.075)
	(2,1.95) circle (0.075)
	(-1,1.95) circle (0.075)
	(1.15,1.2) circle (0.075)
	(1.35,-0.15) circle (0.075)
	(1.65,-0.6) circle (0.075)
	(3.15,2.6) circle (0.075)
	(2.85,1.65) circle (0.075);

\end{scope}
}

%% 5.2

\foreach \a in {1,-1}	
\draw[gray, very thick, xshift=7.5cm, scale=\a]
	(0,1.8) -- (-1,1.8)
	(2,1.8) -- (3,1.8);

\foreach \a in {1,-1}
{
\begin{scope}[xshift=7.5cm, scale=\a]

\fill
	(-0.5,0.7) circle (0.075)
	(1,1.1) circle (0.075)
	(-2,1.1) circle (0.075)
	(1,2.6) circle (0.075);

\fill[gray!70]
	(0.35,0.6) circle (0.075)
	(-1.35,0.6) circle (0.075)
	(1.35,0.15) circle (0.075)
	(0.15,1.2) circle (0.075)
	(1.85,1.2) circle (0.075)
	(-1.15,1.2) circle (0.075)
	(3.15,2.6) circle (0.075);

\end{scope}
}

\foreach \a in {1,-1}
\foreach \c in {0,1}
\filldraw[fill=white, xshift=7.5*\c cm, scale=\a]
	(1,1.8) circle (0.075)
	(0.5,0) circle (0.075)
	(2,-1.8) circle (0.075);

\end{scope}

%% 40 and 41

\begin{scope}[shift={(1.5 cm, -10.5cm)}, scale=0.5]

%% both 

\foreach \b in {0,1,2,3}
{
\begin{scope}[xshift=6*\b cm]

\foreach \a in {0,...,8}
\draw[very thin, rotate=40*\a]
	(70:2.4) -- (110:2.4)
	(70:1.8) -- (90:1.5) -- (110:1.8)
	(90:1.5) -- (90:0.9) -- (50:0.9)
	(70:1.8) -- (70:2.4);

\foreach \a in {0,1,2}
\draw[very thin, rotate=120*\a]
	(0,0) -- (90:0.9);
			
\end{scope}
}

\foreach \b in {0,2}
{
\begin{scope}[xshift=6*\b cm]

\foreach \a in {0,...,8}
\draw[green, rotate=40*\a]
	(30:2.4) -- (-10:2.4)
	(50:0.9) -- (10:0.9);

\foreach \a in {0,1,2}
\draw[red, rotate=120*\a]
	(90:0.9) -- (90:1.5) -- (70:1.8) -- (70:2.4)
	(10:0.9) -- (10:1.5) -- (30:1.8) -- (30:2.4);
	
\end{scope}
}

\foreach \b in {0,...,3}
{
\begin{scope}[xshift=6*\b cm]

\foreach \a in {0,1,2}
\filldraw[fill=white, rotate=120*\a]
	(90:0.9) circle (0.06);
	
\fill
	(0,0) circle (0.06);
	
\end{scope}
}
	
\foreach \a in {0,1,2}
\foreach \b in {0,1}
{
\draw[very thin, xshift=6*\b cm, rotate=120*\a]
	(70:2.4) -- (70:3);
\draw[very thin, xshift=12cm+6*\b cm, rotate=120*\a]
	(30:2.4) -- (30:3);	
	
\draw[gray, very thick, xshift=6cm+12*\b cm, rotate=120*\a]
	(10:0.9) -- (50:0.9);

\draw[gray, very thick, xshift=18 cm, rotate=120*\a]
	(70:2.4) -- (110:2.4);

\fill[xshift=12*\b cm, rotate=120*\a]
	(50:1.5) circle (0.06)
	(-10:1.8) circle (0.06);

\fill[xshift=6*\b cm, rotate=120*\a]
	(70:3) circle (0.06);

\fill[xshift=12cm+6*\b cm, rotate=120*\a]
	(30:3) circle (0.06);

\fill[xshift=6cm+12*\b cm, rotate=120*\a]
	(30:1.8) circle (0.06)
	(-30:1.5) circle (0.06);
		
\fill[xshift=12*\b cm, gray!70, rotate=120*\a]
	(123:1) circle (0.06)
	(15:1.4) circle (0.06)
	(50:0.75) circle (0.06)
	(85:1.4) circle (0.06)
	(74:1.87) circle (0.06)
	(26:1.87) circle (0.06);

\fill[gray!70, rotate=120*\a]
	(33:2.25) circle (0.06);

\fill[xshift=12 cm, gray!70, rotate=120*\a]
	(67:2.25) circle (0.06);
	
\fill[xshift=12*\b cm, gray!70, rotate=120*\a]
	(-10:2.5) circle (0.06);

\fill[gray!70, xshift=6cm+12*\b cm, rotate=120*\a]
	(55:1.4) circle (0.06)
	(5:1.4) circle (0.06)
	(66:1.87) circle (0.06)
	(114:1.87) circle (0.06);

\fill[gray!70, xshift=6cm, rotate=120*\a]
	(107:2.25) circle (0.06)
	(30:2.5) circle (0.06);

\filldraw[fill=white, xshift=6*\b cm, rotate=120*\a]
	(70:2.4) circle (0.06);

\filldraw[fill=white, xshift=12cm + 6*\b cm, rotate=120*\a]
	(30:2.4) circle (0.06);
	
}

\end{scope}

%% 22

\begin{scope}[shift={(1.5 cm, -13.5cm)}, scale=0.5]

%% 22.1

\fill
	(-0.3,0) circle (0.06)
	(0.3,-0.6) circle (0.06)
	(0.9,1.2) circle (0.06)
	(-0.9,1.2) circle (0.06)
	(1.5,-1.2) circle (0.06)
	(-1.5,-1.2) circle (0.06)
	(0.3,1.8) circle (0.06)
	(-0.3,-1.8) circle (0.06);

\fill[gray!70]
	(-0.35,0.7) circle (0.06)
	(0.95,1.7) circle (0.06)
	(-0.2,1.3) circle (0.06)
	(-0.4,1.7) circle (0.06)
	(-1.6,1.2) circle (0.06)
	(1.6,1.2) circle (0.06)
	(-1.57,1.87) circle (0.06)
	(1.57,-1.87) circle (0.06)
	(2.4,1.9) circle (0.06)
	(-2.4,-1.9) circle (0.06)
	(0.3,0.7) circle (0.06)
	(-0.8,0.5) circle (0.06)
	(1.4,0.1) circle (0.06)
	(-1.4,0.1) circle (0.06)
	(1,-0.6) circle (0.06)
	(-1,-0.6) circle (0.06)
	(-0.8,-1.1) circle (0.06)
	(0.8,0.5) circle (0.06)
	(-0.35,-0.5) circle (0.06)
	(0.35,-0.1) circle (0.06)
	(-0.95,-1.7) circle (0.06)
	(0.2,-1.3) circle (0.06)
	(0.8,-1.1) circle (0.06)
	(0.4,-1.7) circle (0.06);

%% 22.2

\foreach \a in {1,-1}
{
\begin{scope}[xshift=6cm, scale=\a]
	
\fill
	(0.3,0.6) circle (0.06)
	(0.3,1.8) circle (0.06)
	(1.5,1.2) circle (0.06)
	(-0.9,1.2) circle (0.06);

\fill[gray!70]
	(-0.35,0.7) circle (0.06)
	(-0.8,0.5) circle (0.06)
	(0.8,1.1) circle (0.06)
	(0.3,0.1) circle (0.06)
	(1,0.6) circle (0.06)
	(-0.2,1.3) circle (0.06)
	(0.95,1.7) circle (0.06)
	(1.4,-0.1) circle (0.06)
	(1.57,-1.87) circle (0.06)
	(0.4,-1.7) circle (0.06)
	(-1.6,1.2) circle (0.06)
	(2.4,1.9) circle (0.06);
	
\end{scope}
}

%% 22.3

\foreach \a in {1,-1}
{
\begin{scope}[xshift=12cm, scale=\a]

\draw[gray, very thick]
	(-0.3,0) -- (0.3,0)
	(1.5,0) -- (1.5,1.2);
	
\fill
	(0.9,0.6) circle (0.06)
	(0.9,-0.6) circle (0.06)
	(0.9,1.8) circle (0.06)
	(-1.5,1.8) circle (0.06);

\fill[gray!70]
	(-0.05,0.7) circle (0.06)
	(0.05,0.5) circle (0.06)
	(-0.3,1.1) circle (0.06)
	(-1.4,1.1) circle (0.06)
	(-0.9,1.3) circle (0.06)
	(0.9,1.3) circle (0.06)
	(-0.2,1.7) circle (0.06)
	(0.2,1.9) circle (0.06)
	(2.4,1.9) circle (0.06);
	
\end{scope}
}

%% all

\foreach \a in {1,-1}
\foreach \b in {0,1,2}
{
\begin{scope}[xshift=6*\b cm, scale=\a]

\draw[very thin]
	(1.5,0) -- (-1.5,0)
	(0.9,0.6) -- (-0.9,0.6)
	(0.3,0.6) -- (-0.3,0)
	(0.9,1.8) -- (-0.3,0.6)
	(-0.3,1.2) -- (-0.3,1.8)
	(0.9,1.2) -- (0.9,-1.2) 
	(1.5,1.2) -- (-1.5,1.2)
	(1.5,-1.8) to[out=60, in=-60] (1.5,1.8) -- (1.5,-1.8) -- (-2.8,-1.8)
	(-0.3,-1.8) to[out=-90, in=240] 
	(2,-1.8) to[out=60, in=-90]  
	(2.4,1.8);

\foreach \a in {1,-1}
\filldraw[fill=white, scale=\a]
	(0.3,1.2) circle (0.06)
	(1.5,1.8) circle (0.06)
	(0.9,0) circle (0.06);

\end{scope}
}

\end{scope}

%% 36

\begin{scope}[shift={(1.5 cm, -16.5cm)}, scale=0.5]

\foreach \a in {0,...,3}
\draw[gray, very thick, xshift=12cm, rotate=90*\a]
	(-1.4,1.4) -- (-0.5,1.4);
		
\foreach \a in {0,...,3}
{

\foreach \b in {0,1,2}
\draw[very thin, xshift=6*\b cm, rotate=90*\a]
	(0,0) -- (0.6,0) -- (0.9,0.5) -- (0.5,0.9) -- (0,0.6)
	(0.9,-0.5) -- (2,-0.5)
	(0.9,0.5) -- (1.4,0.5)
	(2,1.4) -- (-1.4,1.4)
	(-2,2) -- (2,2) -- (2.4,2.4);

\fill[rotate=90*\a]
	(0.9,0.5) circle (0.06);
	
\foreach \b in {1,2}
\fill[xshift=6*\b cm, rotate=90*\a]
	(0.5,0.9) circle (0.06);

\foreach \b in {0,1}
\fill[xshift=6*\b cm, rotate=90*\a]
	(-1.4,2) circle (0.06);

\fill[xshift=12cm, rotate=90*\a]
	(0.5,2) circle (0.06);

\foreach \b in {0,1}
\fill[gray!70, xshift=6*\b cm, rotate=90*\a]
	(1.95,2.1) circle (0.06)
	(-0.5,1.5) circle (0.06)
	(0.6,1.9) circle (0.06)
	(1.3,1.3) circle (0.06);
	
\fill[gray!70, rotate=90*\a]
	(0.1,0.53) circle (0.06)
	(0.4,0.97) circle (0.06);

\foreach \b in {1,2}
\fill[gray!70, xshift=6*\b cm, rotate=90*\a]
	(-0.1,0.53) circle (0.06)
	(-0.4,0.97) circle (0.06);

\fill[gray!70, xshift=12cm, rotate=90*\a]
	(-1.95,2.1) circle (0.06)
	(-1.5,1.9) circle (0.06);
					
}

\foreach \a in {0,...,3}
\foreach \b in {0,1,2}
\filldraw[fill=white, xshift=6*\b cm, rotate=90*\a]
	(0,0) circle (0.06)
	(0.5,1.4) circle (0.06)
	(2.4,2.4) circle (0.06);

\end{scope}	

\end{tikzpicture}
\caption{Tilings for AVC(3C24), $\beta, \gamma$ non-adjacent.}
\label{3CFigB}
\end{figure}

In Figure \ref{3CFigB}, $\bullet$ denotes $\gamma^3$, and $\circ$ denotes $\alpha^4$, and {\color{gray!70} $\bullet$} denotes $\beta$, and each gray shaded edge admits one of the angle configurations shown at the bottom right. In other words, $\alpha$ and $\beta$ along a shaded edge may be exchanged and it can be done independently across all the shaded edges. Up to such exchanges there are 21 tilings. Some tilings share the same underlying combinatorial tilings, and there are 13 such combinatorial tilings. The last three tilings in Figure \ref{3CFigB} are variations of pentagonal subdivisions. The first of the three is the $3{\rm C}_3$ reduction of $PP_8$. The other two are given by the flip modifications introduced in \cite{wy2}. 

Figure \ref{3CFigD} gives the 3d rendering of two of the 21 tilings, in which all angles are faithful. They are the first and third of the fourth row of Figure \ref{3CFigB}. The red corners are $\alpha$, and $\gamma$ are concentrated at the meeting places of three black lines. The remaining angles are $\beta=\pi$, which are easy to spot.

\begin{figure}[htp]
\centering
\begin{subfigure}[b]{0.3\linewidth}
\centering
\begin{tikzpicture}[>=latex,scale=1]

\pgftext{
	\includegraphics[scale=0.049]{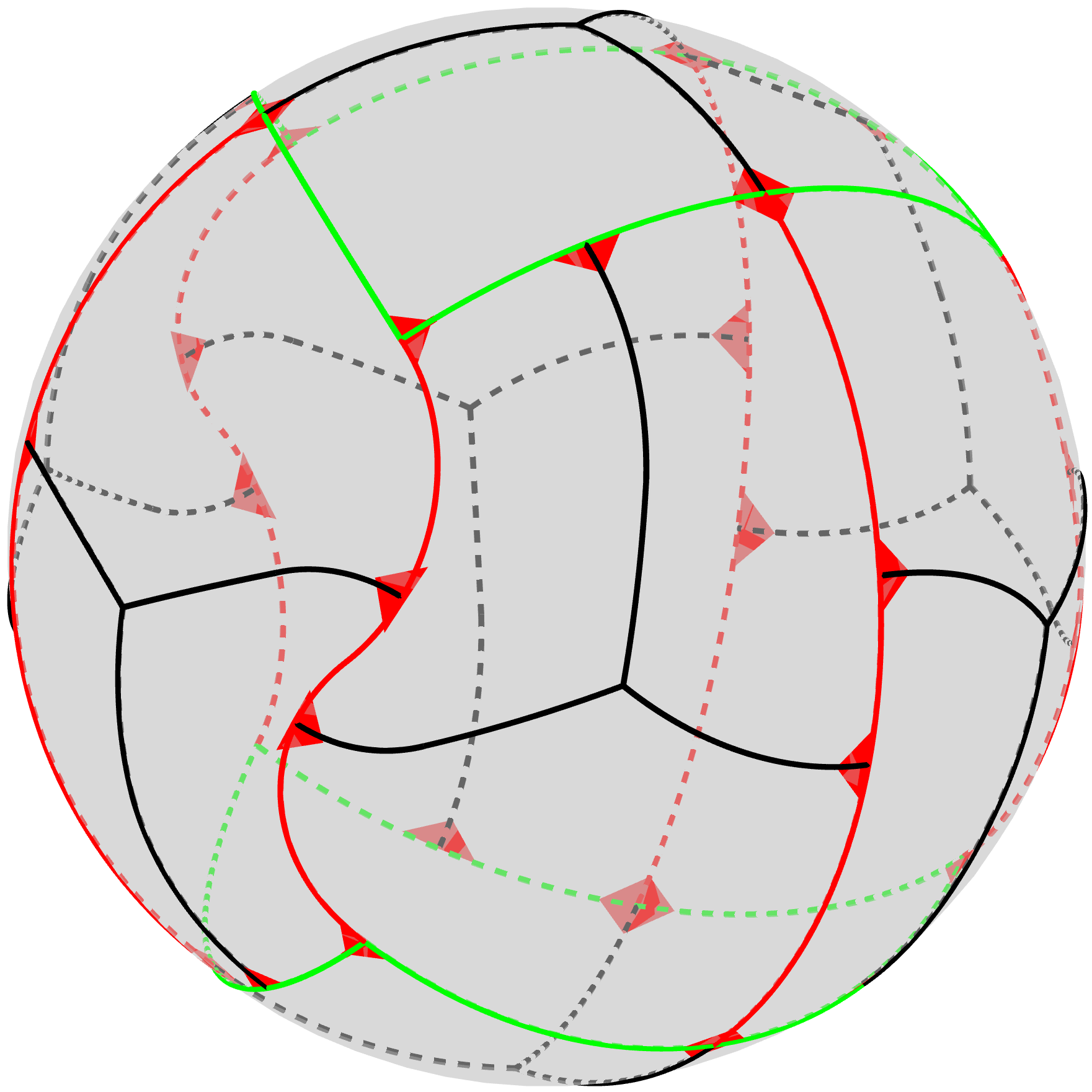}
	};
	
\end{tikzpicture}	
\caption{}
\label{3CFigDA}
\end{subfigure}
\begin{subfigure}[b]{0.3\linewidth}
\centering
\begin{tikzpicture}[>=latex,scale=1]

\pgftext{
	\includegraphics[scale=0.052]{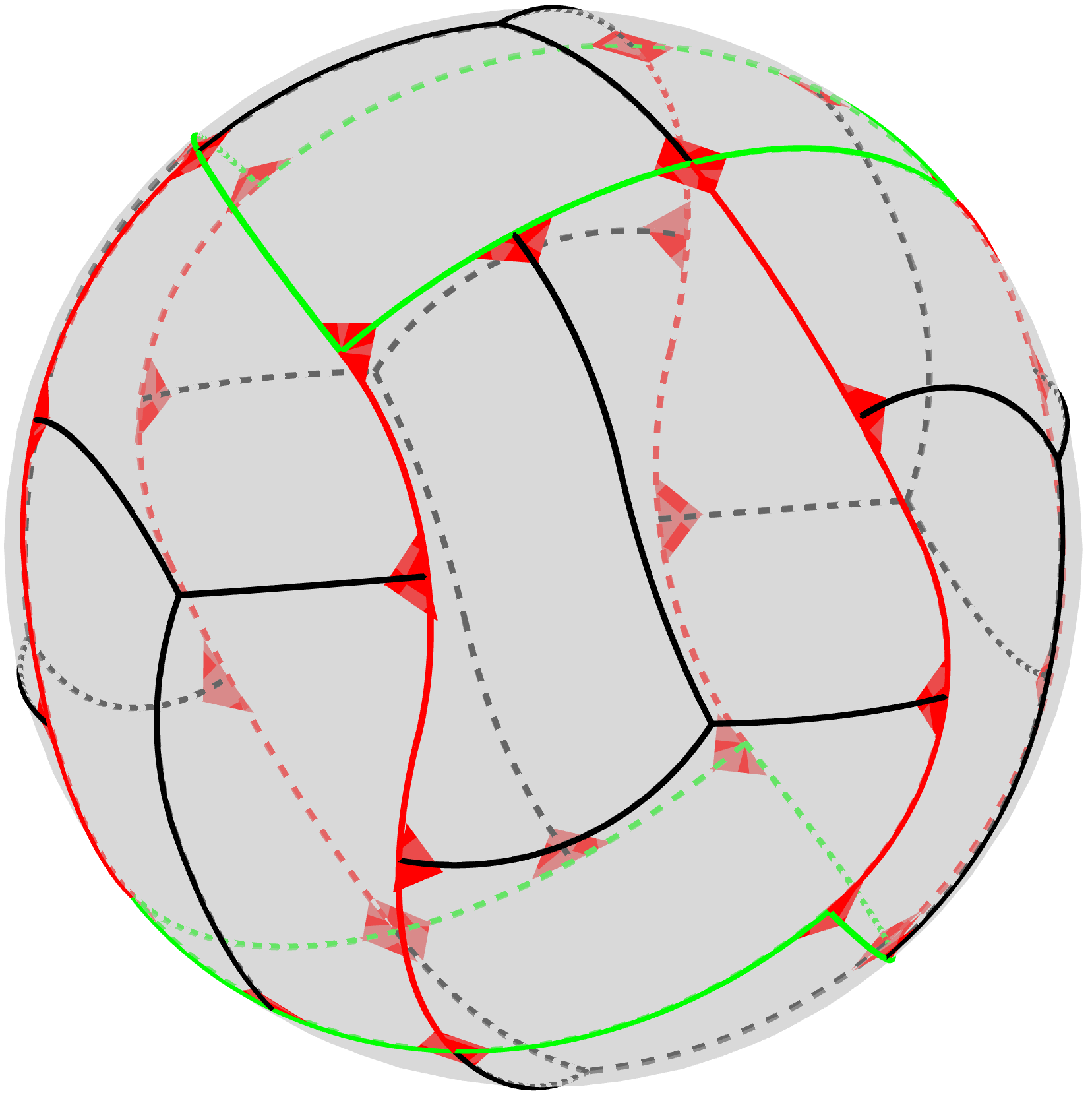}
	};
		
\end{tikzpicture}	
\caption{}
\label{3CFigDB}
\end{subfigure}
\caption{3d rendering of two tilings for AVC(3C24), $\beta, \gamma$ non-adjacent.}
\label{3CFigD}
\end{figure}

Observing that there are already $>10^7$ distinct instances of order $36$ pentagonal polyhedrons \cite{Hasheminezhad2011}, where, to date, the number of order $38$ pentagonal polyhedrons appears to be unknown, determining the number of tilings for AVC(3C60) is currently far beyond our computational capability. In the simplest case, there are five variations of the pentagonal subdivisions of the dodecahedron (icosahedron) as opposed to three (bottom three in Figure \ref{3CFigB}) for AVC(3C24). Moreover, the number of underlying combinatorial tilings is also sizeable.

\section{Tilings with Four Distinct Angles}

\begin{theorem}\label{4Athm}
The tilings for
\begin{align*}
\text{\rm AVC(4A24)}
&=\{24\alpha^2\beta\gamma\delta 
\colon 
24\alpha^2\beta,
8\gamma^3,
6\delta^4 \},   \\
\text{\rm AVC(4A60)}
&=\{
60\alpha^2\beta\gamma\delta 
\colon 
60\alpha^2\beta,
20\gamma^3,
12\delta^5 \},
\end{align*} 
are the ${\rm 4A}_1$, ${\rm 4A}_2$, ${\rm 4A}_3$ reductions of $PP_6$ and $PP_{12}$. 
\end{theorem}

\begin{proof} 
The tilings can be obtained by the splittings from AVC(3A24/60) to AVC(4A24/60). As explained at the end of Section \ref{Sec-Reductions}, this means that, in the tilings in Theorem \ref{3Athm}, which are the 3A reductions of $PP_6$ and $PP_{12}$, we change $\beta, \gamma$ to $\gamma, \delta$, and change one of the three $\alpha$ angles to $\beta$. Since $\beta, \gamma$ are non-adjacent in the 3A reductions of $PP_6$ and $PP_{12}$, we know $\gamma,\delta$ are non-adjacent in the corresponding tilings for AVC(4A24/60). Therefore, among the six possible pentagons in Figure \ref{2abcd_arrangement}, only the pentagons in Figures \ref{Subfig-2abcd_arrangement-1}, \ref{Subfig-2abcd_arrangement-3}, \ref{Subfig-2abcd_arrangement-4} are suitable for AVC(4A24/60). Then we find the splittings from the pentagon in Figure \ref{Subfig-3abc_arrangement-2} to these three pentagons are unique. The uniqueness implies that the tilings for AVC(4A24/60) are the ${\rm 4A}_1$, ${\rm 4A}_2$, ${\rm 4A}_3$ reductions of $PP_6$ and $PP_{12}$. 
\end{proof}

\begin{theorem}\label{4Dthm}
The tilings for 
\begin{align*}
\text{\rm AVC(4D24)}
&=\{
24\alpha^2\beta\gamma\delta 
\colon 
24\alpha\beta\gamma,
8\alpha^3,
6\delta^4 \}, \\
\text{\rm AVC(4D60)}
&=\{
60\alpha^2\beta\gamma\delta 
\colon 
60\alpha\beta\gamma,
20\alpha^3,
12\delta^5 \},
\end{align*}
are the $4{\rm D}_1$, $4{\rm D}_2$, $4{\rm D}_3$ reductions of $PP_6$ and $PP_{12}$, and the modifications of the $4{\rm D}_2$ reduction caused by independently changing the orientations of each $N(\delta^4/\delta^5)$.
\end{theorem}

Since AVC(4D24/60) is symmetric with respect to the exchange of $\beta,\gamma$, the tilings in the theorem are up to their exchange.

\begin{proof}
The tilings can be obtained by the splittings from AVC(3D24/60) to AVC(4D24/60). As explained at the end of Section \ref{Sec-Reductions}, this means that, in the  tilings in Theorem \ref{3Dthm}, we change $\gamma$ to $\delta$, change two $\beta$ to two $\alpha$, and change two $\alpha$ to one $\beta$ and one $\gamma$. 

The pentagons in the tilings for AVC(3D24/60) (in Theorem \ref{3Dthm}) are the ones in Figures \ref{Subfig-2a2bc_arrangement-1}, \ref{Subfig-2a2bc_arrangement-3}, \ref{Subfig-2a2bc_arrangement-4}. Since AVC(4D24/60) is symmetric with respect to the exchange of $\beta,\gamma$, we only need to consider pentagons in Figures \ref{Subfig-2abcd_arrangement-1}, \ref{Subfig-2abcd_arrangement-2}, \ref{Subfig-2abcd_arrangement-3}, \ref{Subfig-2abcd_arrangement-4}. Then we see the splitting takes Figures 
\ref{Subfig-2a2bc_arrangement-1} ($3{\rm D}_2$ reduction), \ref{Subfig-2a2bc_arrangement-3} ($3{\rm D}_3$ reduction), \ref{Subfig-2a2bc_arrangement-4} ($3{\rm D}_1$ reduction) to Figures 
\ref{Subfig-2abcd_arrangement-2} ($4{\rm D}_2$ reduction), \ref{Subfig-2abcd_arrangement-3} ($4{\rm D}_3$ reduction), \ref{Subfig-2abcd_arrangement-4} ($4{\rm D}_1$ reduction), respectively. 

The splittings from Figures \ref{Subfig-2a2bc_arrangement-3} and \ref{Subfig-2a2bc_arrangement-4} are unique. Applying the unique splittings to the $3{\rm D}_3$ and $3{\rm D}_1$ reductions of $PP_6$ and $PP_{12}$, we get the $4{\rm D}_3$ and $4{\rm D}_1$ reductions of $PP_6$ and $PP_{12}$. Applying the unique splittings to the tilings in  Figures \ref{Subfig-3DFigE-T1} and \ref{3DFigF}, we get tilings with vertex $\alpha\gamma^2$. Since the vertex is not in AVC(4D24), these are not tilings for the AVC.

There are two ways to split a pentagon in Figure 
\ref{Subfig-2a2bc_arrangement-1} to a pentagon in Figure 
\ref{Subfig-2abcd_arrangement-2}. We consider the splitting of $N(\gamma^4/\gamma^5)$ to $N(\delta^4/\delta^5)$. Since $\beta^2\cdots$ and $\gamma^2\cdots$ are not vertices in AVC(4D24/60), for the pentagon in Figure 
\ref{Subfig-2abcd_arrangement-2}, the AAD of $\delta^4$ and $\delta^5$ are $|^{\beta}\delta^{\gamma}|^{\beta}\delta^{\gamma}|^{\beta}\delta^{\gamma}|^{\beta}\delta^{\gamma}|$ and $|^{\beta}\delta^{\gamma}|^{\beta}\delta^{\gamma}|^{\beta}\delta^{\gamma}|^{\beta}\delta^{\gamma}|^{\beta}\delta^{\gamma}|$. This implies exactly two ways of splitting $N(\gamma^4/\gamma^5)$ to $N(\delta^4/\delta^5)$. Figure \ref{4DfigA} shows that the two ways of splitting $N(\gamma^4)$ is the same as the choice of orientation (as determined by $\beta\to\gamma$ along the boundary) of $N(\delta^4)$. 

\begin{figure}[htp]
\centering
\begin{tikzpicture}[>=latex]

\foreach \a in {0,1,2,3}
{

\foreach \x in {-1,0,1}
\draw[xshift=3*\x cm, rotate=90*\a]
	(0,0) -- (0.6,0) -- (0.9,0.5) -- (0.5,0.9) -- (0,0.6) -- (0,0);
	
\begin{scope}[rotate=90*\a]

\node at (0.5,0.15) {\scriptsize $\alpha$};
\node at (0.15,0.5) {\scriptsize $\alpha$};
\node at (0.45,0.7) {\scriptsize $\beta$};
\node at (0.7,0.45) {\scriptsize $\beta$};
\node at (0.15,0.15) {\scriptsize $\gamma$};

\end{scope}

\begin{scope}[xshift=3cm, rotate=90*\a]

\node at (0.5,0.15) {\scriptsize $\beta$};
\node at (0.15,0.5) {\scriptsize $\gamma$};
\node at (0.45,0.7) {\scriptsize $\alpha$};
\node at (0.7,0.45) {\scriptsize $\alpha$};
\node at (0.15,0.15) {\scriptsize $\delta$};

\end{scope}

\begin{scope}[xshift=-3cm, rotate=90*\a]

\node at (0.5,0.15) {\scriptsize $\gamma$};
\node at (0.15,0.5) {\scriptsize $\beta$};
\node at (0.45,0.7) {\scriptsize $\alpha$};
\node at (0.7,0.45) {\scriptsize $\alpha$};
\node at (0.15,0.15) {\scriptsize $\delta$};

\end{scope}

}
		
\draw[very thick, ->]
	(-1,0) -- ++(-1,0);
	
\draw[very thick, ->]
	(1,0) -- ++(1,0);

\draw[xshift=3cm, ->]
	(210:0.35) arc (210:-120:0.35);

\draw[xshift=-3cm, <-]
	(210:0.35) arc (210:-120:0.35);
	
\end{tikzpicture}
\caption{Two splittings from $N(\gamma^4)$ to $N(\delta^4)$.}
\label{4DfigA}
\end{figure}

In Theorem \ref{3Dthm}, the tilings for the pentagon in Figure 
\ref{Subfig-2a2bc_arrangement-1} are the $3{\rm D}_2$ reductions of $PP_6$ and $PP_{12}$. Applying the splittings in Figure \ref{4DfigA} to $N(\gamma^4/\gamma^5)$ in these tilings, we get the $4{\rm D}_2$ reductions of $PP_6$ and $PP_{12}$, and we may independently change the orientations of $N(\delta^4/\delta^5)$ in these tilings.
\end{proof}

\begin{theorem}\label{4Ethm}
The tilings for 
\begin{align*}
\text{\rm AVC(4E24)}
&=\{
24\alpha^2\beta\gamma\delta 
\colon 
24\alpha\beta\gamma,
8\delta^3,
6\alpha^4 \};  \\
\text{\rm AVC(4E60)}
&=\{
60\alpha^2\beta\gamma\delta 
\colon 
60\alpha\beta\gamma,
20\delta^3,
12\alpha^5 \},
\end{align*}
are the $4{\rm E}_1$, $4{\rm E}_2$, $4{\rm E}_3$ reductions of $PP_8$ and $PP_{20}$, and the modifications of the $4{\rm E}_3$ reduction caused by independently changing the orientations of $N(\delta^4/\delta^5)$, and the tiling in Figure \ref{4EFigA}.
\end{theorem}

Since AVC(4E24/60) is symmetric with respect to the exchange of $\beta,\gamma$, the tilings in the theorem are up to their exchange.

\begin{proof}
The splitting from AVC(3E24/60) to AVC(4E24/60) is the same as the splitting from AVC(3D24/60) to AVC(4D24/60). The proof is the same as the proof of Theorem \ref{4Dthm}. Therefore, we get the $4{\rm E}_1$, $4{\rm E}_2$, $4{\rm E}_3$ reductions of $PP_8$ and $PP_{20}$, and we may independently change the orientations of $N(\delta^3)$ in the $4{\rm E}_3$ reduction. 

We also need to apply the splitting to the tilings in Figures \ref{3EFigD3} and \ref{3EFigE}. The tiling in Figure \ref{3EFigD3} becomes the tiling in Figure \ref{4EFigA}. Applying the splitting to the tiling in Figure \ref{3EFigE} produces a vertex $\alpha\beta^2$ that is not in AVC(4E60). Therefore, there is no additional tiling for AVC(4E60).
\end{proof}

\begin{figure}[htp]
\centering
\begin{tikzpicture}[>=latex]

\foreach \a in {0,1,2}
{
\begin{scope}[rotate=120*\a]

\draw
	(0,0) -- (30:0.5) -- (60:0.9) -- (120:0.9) -- (150:0.5)
	(120:0.9) -- (130:1.4) -- (90:1.4) -- (50:1.4) -- (10:1.4) -- (0:0.9)
	(50:1.4) -- (60:0.9)
	(90:1.4) -- (70:1.9) -- (30:1.9) -- (-10:1.9) -- (-50:1.9)
	(30:1.9) -- (10:1.4)
	(110:1.9) -- (90:2.4) -- (50:2.4) -- (10:2.4) -- (-30:2.4)
	(50:2.4) -- (30:1.9)
	(10:2.4) -- (10:2.8);

\fill
	(0,0) circle (0.06)
	(50:1.4) circle (0.06)
	(-10:1.9) circle (0.06)
	(10:2.8) circle (0.06)
	(110:1.9) circle (0.05);

\foreach \a in {0,1,2}
\filldraw[fill=white, rotate=120*\a]
	(30:1.9) circle (0.05)
	(10:1.4) circle (0.05);

\node at (90:0.2) {\tiny $\delta$};
\node at (63:0.72) {\tiny $\alpha$};
\node at (117:0.72) {\tiny $\gamma$};
\node at (47:0.43) {\tiny $\alpha$};
\node at (133:0.43) {\tiny $\beta$};

\node at (115:1) {\tiny $\gamma$};
\node at (65:1) {\tiny $\beta$};
\node at (59:1.23) {\tiny $\delta$};
\node at (121:1.23) {\tiny $\alpha$};
\node at (90:1.25) {\tiny $\alpha$};

\node at (133:1.23) {\tiny $\alpha$};
\node at (47:1.23) {\tiny $\delta$};
\node at (130:0.9) {\tiny $\alpha$};
\node at (50:0.9) {\tiny $\beta$};
\node at (30:0.65) {\tiny $\gamma$};

\node at (69:1.75) {\tiny $\gamma$};
\node at (31:1.75) {\tiny $\alpha$};
\node at (80:1.45) {\tiny $\beta$};
\node at (20:1.45) {\tiny $\alpha$};
\node at (50:1.55) {\tiny $\delta$};

\node at (92:1.54) {\tiny $\gamma$};
\node at (128:1.54) {\tiny $\alpha$};
\node at (80:1.72) {\tiny $\beta$};
\node at (140:1.72) {\tiny $\alpha$};
\node at (110:1.75) {\tiny $\delta$};

\node at (101:1.95) {\tiny $\delta$};
\node at (39:1.95) {\tiny $\alpha$};
\node at (70:2) {\tiny $\alpha$};
\node at (89:2.25) {\tiny $\beta$};
\node at (51:2.25) {\tiny $\gamma$};

\node at (100:2.2) {\tiny $\beta$};
\node at (40:2.2) {\tiny $\alpha$};
\node at (112:2.03) {\tiny $\delta$};
\node at (28:2.03) {\tiny $\alpha$};
\node at (10:2.27) {\tiny $\gamma$};

\node at (13:2.5) {\tiny $\beta$};
\node at (7:2.5) {\tiny $\alpha$};
\node at (50:2.5) {\tiny $\gamma$};
\node at (90:2.5) {\tiny $\alpha$};
\node at (70:2.6) {\tiny $\delta$};

\end{scope}
}

\end{tikzpicture}
\caption{Tiling for AVC(4E24), not pentagonal subdivision.}
\label{4EFigA} 
\end{figure}

\section{Reductions of AVC(5A36)}
\label{36tiles}

If we assume that two of $\alpha, \beta, \gamma$ are equal (i.e., not distinguish the two angles), then AVC(5A36) is reduced to 
\[
\text{AVC(4A36)}=
\{36\alpha^2\beta\delta\epsilon\colon 36\alpha^2\beta,8\delta^3,12\delta\epsilon^3\}.
\] 
If we assume that all three $\alpha, \beta, \gamma$ are equal, then AVC(5A36) is reduced to  
\[
\text{AVC(3A36)}=
\{36\alpha^3\delta\epsilon\colon 36\alpha^3,8\delta^3,12\delta\epsilon^3\}.
\]
Assuming $\alpha$ and $\beta$ equal also reduces AVC(4A36) to AVC(3A36). 

We will only consider two reductions of AVC(5A36). Therefore, we keep $\delta,\epsilon$ to see more clearly the relation between tilings for reductions. The main message here is that, although Theorem \ref{5Athm} says there is no tiling if we distinguish $\alpha, \beta, \gamma$, tilings exist if we do not distinguish $\alpha, \beta, \gamma$.

\begin{theorem}\label{3A36thm}
There are five tilings each for {\rm AVC(3A36)} and {\rm AVC(4A36)}, given by Figures \ref{3A36B} and \ref{3A36C}.
\end{theorem}

For AVC(3A36), the pentagon in the tilings is Figure \ref{Subfig-3A36A-2}, and all the unlabelled angles are $\alpha$ (see Figure \ref{Subfig-3A36B-N}). For AVC(4A36), the pentagon in the tilings is Figure \ref{Subfig-4A36A-T1}, and all the unlabelled angles are $\alpha,\beta$, with $\alpha$ bounded by one normal edge and one red edge, and $\beta$ bounded by two normal edges (see Figure \ref{Subfig-4A36A-V}). 

\begin{proof}
The pentagon has two possible angle arrangements, as in Figures \ref{Subfig-3A36A-1}, \ref{Subfig-3A36A-2}. 

For the pentagon in Figure \ref{Subfig-3A36A-1}, the AAD of $\delta^3$ is either $|^{\alpha}\delta^{\epsilon}|^{\alpha}\delta^{\epsilon}|^{\alpha}\delta^{\epsilon}|$ or $|^{\alpha}\delta^{\epsilon}|^{\alpha}\delta^{\epsilon}|^{\epsilon}\delta^{\alpha}|$. This implies that $\alpha\epsilon\cdots$ is a vertex. Since $\alpha\epsilon\cdots$ is not in AVC(3A36), the pentagon is the one in Figure \ref{Subfig-3A36A-2}. 

\begin{figure}[h!]
\centering
\begin{subfigure}[t]{0.15\linewidth}
\centering
\begin{tikzpicture}[>=latex]

\foreach \a in {0,...,4}
\draw[rotate=72*\a]
	(18:0.6) -- (90:0.6);
	
\node at (90:0.4) {\scriptsize $\alpha$};
\node at (162:0.4) {\scriptsize $\alpha$};
\node at (18:0.4) {\scriptsize $\alpha$};
\node at (234:0.4) {\scriptsize $\delta$};
\node at (-54:0.4) {\scriptsize $\epsilon$};

\end{tikzpicture}
\caption{}
\label{Subfig-3A36A-1}
\end{subfigure}
\begin{subfigure}[t]{0.15\linewidth}
\centering
\begin{tikzpicture}[>=latex]

\foreach \a in {0,...,4}
\draw[rotate=72*\a]
	(18:0.6) -- (90:0.6);

\node at (90:0.4) {\scriptsize $\alpha$};
\node at (162:0.4) {\scriptsize $\delta$};
\node at (18:0.4) {\scriptsize $\epsilon$};
\node at (234:0.4) {\scriptsize $\alpha$};
\node at (-54:0.4) {\scriptsize $\alpha$};

\end{tikzpicture}
\caption{}
\label{Subfig-3A36A-2}
\end{subfigure}
\begin{subfigure}[t]{0.15\linewidth}
\centering
\begin{tikzpicture}[>=latex]

\foreach \a in {0,...,4}
\draw[rotate=72*\a]
	(18:0.6) -- (90:0.6);

\node at (162:0.4) {\scriptsize $\delta$};
\node at (18:0.4) {\scriptsize $\epsilon$};
\node at (90:0.4) {\scriptsize $\beta$};
\node at (234:0.4) {\scriptsize $\alpha$};
\node at (-54:0.4) {\scriptsize $\alpha$};

\end{tikzpicture}
\caption{}
\label{Subfig-4A36A-T1}
\end{subfigure}
\begin{subfigure}[t]{0.15\linewidth}
\centering
\begin{tikzpicture}[>=latex]

\foreach \a in {0,...,4}
\draw[rotate=72*\a]
	(18:0.6) -- (90:0.6);

\node at (162:0.4) {\scriptsize $\delta$};
\node at (18:0.4) {\scriptsize $\epsilon$};
\node at (90:0.4) {\scriptsize $\alpha$};
\node at (234:0.4) {\scriptsize $\beta$};
\node at (-54:0.4) {\scriptsize $\alpha$};

\end{tikzpicture}
\caption{}
\label{Subfig-4A36A-T2}
\end{subfigure}
\begin{subfigure}[t]{0.15\linewidth}
\centering
\begin{tikzpicture}[>=latex]

\foreach \a in {0,...,4}
\draw[rotate=72*\a]
	(18:0.6) -- (90:0.6);

\node at (162:0.4) {\scriptsize $\delta$};
\node at (18:0.4) {\scriptsize $\epsilon$};

\node at (90:0.4) {\scriptsize $\alpha$};
\node at (234:0.4) {\scriptsize $\alpha$};
\node at (-54:0.4) {\scriptsize $\beta$};

\end{tikzpicture}
\caption{}
\label{Subfig-4A36A-T3}
\end{subfigure}
\caption{Angle arrangements for AVC(3A36) and AVC(4A36).}
\label{3A36A}
\end{figure}

Since AVC(4A36) is reduced to AVC(3A36) by $\alpha,\alpha,\beta,\delta,\epsilon\to \alpha,\alpha,\alpha,\delta,\epsilon$, and $\delta,\epsilon$ are non-adjacent in tilings for AVC(3A36), we know $\delta,\epsilon$ are also non-adjacent in tilings for AVC(4A36). Therefore the arrangements of $\alpha,\alpha,\beta,\delta,\epsilon$ in pentagon for AVC(4A36) are given by in Figures \ref{Subfig-4A36A-T1}, \ref{Subfig-4A36A-T2}, \ref{Subfig-4A36A-T3}. Then the splitting from AVC(3A36) to AVC(4A36) means changing three $\alpha$ angles in each tile to two $\alpha$ angles and one $\beta$ angle according to one of the three pentagons.

Denote by $\bullet$ the vertices $\delta^3$ and $\delta\epsilon^3$. Since $\delta^3$ and $\delta\epsilon^3$ are the only vertices involving $\delta,\epsilon$, and $\delta,\epsilon$ are non-adjacent, we know each pentagonal tile has exactly two non-adjacent $\bullet$ vertices. By \cite[Theorem 10]{yan2}, this implies the tiling is the simple pentagonal subdivision of a quadrilateral tiling, with $\bullet$ vertices as all the vertices. The simple pentagonal subdivision divides each quadrilateral tile into two pentagonal tiles in compatible way. Figure \ref{Subfig-3A36B-N} shows the subdivision for AVC(3A36) and the pentagon in Figure \ref{Subfig-3A36A-2}. Figure \ref{Subfig-4A36A-V} shows the subdivision for AVC(4A36) and the pentagon in Figure \ref{Subfig-4A36A-T1}. We remark that there are two kinds of quadrilateral tiles. We call the yellow quadrilateral {\em trapezium} because it has angle arrangement $\delta,\delta,\epsilon,\epsilon$. We call the white quadrilateral {\em rhombus} because it has angle arrangement $\delta,\epsilon,\delta,\epsilon$.

\begin{figure}[h!]
\centering
\begin{subfigure}[t]{0.25\linewidth} 
\centering
\begin{tikzpicture}[>=latex, scale=1]

\fill[yellow]
	(-1.2,-0.6) rectangle (0,0.6);
	
\draw
	(-1.2,-0.6) rectangle (1.2,0.6)
	(0,-0.6) -- (0,0.6);

\draw[red]
	(-0.6,-0.6) -- (-0.6,0.6)
	(0,0) -- (1.2,0)
	(0.6,0.6) -- ++(0,0.4)
	(0.6,-0.6) -- ++(0,-0.4)
	(-1.2,0) -- ++(-0.4,0);

\begin{scope}[font=\scriptsize]

\foreach \x/\y in  
	{-0.45/0.45, -0.75/0.45, 0.15/0.15, 1.05/0.15, 0.45/0.75, 0.75/0.75, -1.35/0.15, -0.6/0.75, 0.6/0.45}
\foreach \u in {1,-1}
\node at (\x,\u*\y) {$\alpha$};

\foreach \x in {-0.15, 1.35, -1.05}
\node at (\x,0) {$\alpha$};

\foreach \x/\y in  
	{0.15/0.45, -0.15/0.45, 1.05/-0.45, -1.05/0.45}
\node at (\x,\y) {$\delta$};

\foreach \x/\y in  
	{0.15/-0.45, -0.15/-0.45, -1.05/-0.45, 1.05/0.45}
\node at (\x,\y) {$\epsilon$};

\end{scope}

\end{tikzpicture}
\caption{}
\label{Subfig-3A36B-N}
\end{subfigure}
\begin{subfigure}[t]{0.25\linewidth} 
\centering
\begin{tikzpicture}[>=latex, scale=1]

\fill[yellow]
	(-1.2,-0.6) rectangle (0,0.6);
	
\draw
	(-1.2,-0.6) rectangle (1.2,0.6)
	(0,-0.6) -- (0,0.6);

\draw[red]
	(-0.6,-0.6) -- (-0.6,0.6)
	(0,0) -- (1.2,0)
	(0.6,0.6) -- ++(0,0.4)
	(0.6,-0.6) -- ++(0,-0.4)
	(-1.2,0) -- ++(-0.4,0);

\begin{scope}[font=\scriptsize]

\foreach \x/\y in  
	{-0.45/0.45, -0.75/0.45, 0.15/0.15, 1.05/0.15, 0.45/0.75, 0.75/0.75, -1.35/0.15}
\foreach \u in {1,-1}
\node at (\x,\u*\y) {$\alpha$};

\foreach \x/\y in  
	{-0.6/0.75, 0.6/0.45}
\foreach \u in {1,-1}
\node at (\x,\u*\y) {$\beta$};

\foreach \x in {-0.15, 1.35, -1.05}
\node at (\x,0) {$\beta$};

\foreach \x/\y in  
	{0.15/0.45, -0.15/0.45, 1.05/-0.45, -1.05/0.45}
\node at (\x,\y) {$\delta$};

\foreach \x/\y in  
	{0.15/-0.45, -0.15/-0.45, -1.05/-0.45, 1.05/0.45}
\node at (\x,\y) {$\epsilon$};

\end{scope}

\end{tikzpicture}
\caption{}
\label{Subfig-4A36A-V}
\end{subfigure}
\begin{subfigure}[t]{0.23\linewidth}
\centering
\begin{tikzpicture}[>=latex]

\foreach \b in {1,-1}
{
\begin{scope}[scale=\b]

\draw
	(-1,0.25) -- (-0.5,0.75) -- (0.5,0.75) -- (1,0.25) -- (1,-0.25) -- (-0.5,-0.25) -- (-1,0.25)
	(0.5,0.75) -- (0.5,0.25)
	(-0.5,0.75) -- (0.5,-0.25);

\fill
	(0.5,0.25) circle (0.05)	
	(0.5,-0.25) circle (0.05);

\node at (0.55,0.9) {\scriptsize 3};
\node at (-0.55,0.9) {\scriptsize 3};
\node at (1.15,0.3) {\scriptsize 2};
\node at (1.15,-0.3) {\scriptsize 4};	

\end{scope}
}

\end{tikzpicture}
\caption{}
\label{Subfig-3A36A-3}
\end{subfigure}
\begin{subfigure}[t]{0.23\linewidth}
\centering
\begin{tikzpicture}[>=latex]

\foreach \a in {0,1,2}
{
\begin{scope}[rotate=120*\a]

\draw
	(0,0) -- (0,1) -- (30:1) -- (-30:1)
	(150:0.5) -- (0,0.5) -- (30:0.5) -- (30:1) ;

\fill
	(0,0) circle (0.05)	
	(30:0.5) circle (0.05);
	
\node at (30:1.15) {\scriptsize 3};
\node at (-30:1.15) {\scriptsize 3};

\end{scope}
}

\end{tikzpicture}
\caption{}
\label{Subfig-3A36A-4}
\end{subfigure}
\caption{Quadrilateral disk tilings, and their simple pentagonal subdivisions that give tilings for AVC(3A36) and AVC(4A36).}
\label{3A36D}
\end{figure}

The quadrilateral tiling has $36\div 2=18$ tiles, and $12$ vertices $\delta\epsilon^3$ of degree $4$, and $8$ vertices $\delta^3$ of degree $3$. Since each pentagonal tile has only one $\delta$, we require that no two $\delta^3$ are connected to each other in the quadrilateral tiling. We use computer to find that there are exactly three such quadrilateral tilings. They are constructed by glueing two copies of the disk tilings in Figures \ref{Subfig-3A36A-3}, \ref{Subfig-3A36A-4}. Both disk tilings already have 4 degree $3$ vertices in the interior, indicated by $\bullet$. Since the total number of $\bullet$ vertices is 8, we glue two copies of the disk tilings together, such that all vertices along the boundary of the disk tilings have degree 4. We indicate the degrees of boundary vertices in Figures \ref{Subfig-3A36A-3} and \ref{Subfig-3A36A-4}. The gluing should match a degree $2$ vertex with a degree $4$ vertex, and match a degree $3$ vertex with another degree $3$ vertex. This means Figure \ref{Subfig-3A36A-3} glues to itself in unique way, and Figure \ref{Subfig-3A36A-4} glues to itself in two possible ways. 

The tiling obtained by glueing Figure \ref{Subfig-3A36A-3} to itself is given by the normal lines in Figure \ref{3A36B}, with $\bullet$ and $\circ$ as vertices. The simple pentagonal subdivision is constructed by adding red lines to each quadrilateral tile, such that the neighboring tiles are compatibly divided like those in Figure \ref{Subfig-3A36B-N}. 

Next, we assign $\delta, \epsilon$ at $\bullet$ and $\circ$ vertices, and all vertices at the ends of the dashed edges are $\alpha^3$ for AVC(3A36) and $\alpha^2\beta$ for AVC(4A36). For the pentagons in Figures \ref{Subfig-3A36A-2} and \ref{Subfig-4A36A-T1}, Figures \ref{Subfig-3A36B-N} and \ref{Subfig-4A36A-V} show how one $\delta$ or one $\epsilon$ determines all the angles in a pentagonal tile. In particular, $\delta$ and $\epsilon$ determine each other in any pentagonal tile.

First, we assign $\delta^3$ to all $\bullet$ vertices. Then each $\delta$ determines $\epsilon$ in the corresponding pentagonal tile. We find four $\circ$ vertices where three $\epsilon$ are already determined. Then these vertices are $\delta\epsilon^3$, and the remaining angle at each vertex is $\delta$. These $\delta$ determine more $\epsilon$. In this way, we determine all the black $\delta, \epsilon$ in Figure \ref{3A36B}.

\begin{figure}[h!]
\centering
\begin{subfigure}[t]{0.45\linewidth} %% tiling 1
\centering
\begin{tikzpicture}[>=latex, scale=1]

\foreach \b in {1,-1}
\fill[yellow, scale=\b]
	(0.4,0.3) -- (-0.4,0.3) -- (-0.6,0.8) -- (1.4,0.8) -- (1.4,1.8) -- (0.6,1.3) -- (0.6,0.8);
	
\foreach \b in {1,-1}
{
\begin{scope}[scale=\b]

\draw
	(-0.4,0.3) -- (0.4,0.3) -- (0.4,-0.3)
	(0.4,0.3) -- (0.6,0.8)
	(-0.4,0.3) -- (-0.6,0.8)
	(-0.6,-1.3) -- (1.4,-1.3) -- (1.4,0.8) -- (-2.2,0.8) -- (-1.4,1.8) -- (1.4,1.8) -- (2.2,0.8)  -- (2.2,-0.8) 
	(0.4,0.3) -- (1,0) -- (0.6,-0.8)
	(1.4,0.8) -- (1,0)
	(0.6,0.8) -- (0.6,1.3) -- (1.4,1.8) -- (1.4,0.8)
	;
	
\draw[red]
	(-0.4,0) -- (0.4,0)
	(0,0.3) -- (0,0.8)
	(0.5,0.55) -- (1.2,0.4)
	(0.5,-0.55) -- (0.7,0.15)
	(0.8,-0.4) -- (1.4,0)
	(1,0.8) -- (1,1.55)
	(1.4,-1.3) -- (1.4,-1.8) 
	(-1,0.8) to[out=50, in=170] (0.6,1.05)
	(1.8,-0.8) to[out=80, in=-40] (1.4,1.3)
	(2.2,0.8) -- (2.8,1);

\foreach \x/\y in {1/0, 0.4/-0.3, 0.6/1.3, 2.2/-0.8}
\fill (\x,\y) circle (0.05);

\begin{scope}[font=\tiny]

\foreach \x/\y in {-0.3/0.4, -0.51/0.25, -0.3/0.2, -1.1/0.05, 0.97/0.15, -0.85/0.05, 0.6/1.45, 0.5/1.2, 0.7/1.25, -2.3/0.85, -2.1/0.67, -2/0.9, 0.3/0.4, 1.3/1.6}
\node at (\x,\y) {$\delta$};

\foreach \x/\y in {-0.47/0.7, -0.61/0.6, 0.3/0.2, 0.5/0.15, 0.53/0.35, 0.47/0.7, 0.7/0.9, 1.5/1.9, 1.47/1.55, 1.1/1.7, 1.3/0.9, 1.33/0.5, -1.3/0.9, -1.5/0.9}
\node at (\x,\y) {$\epsilon$};

\foreach \x/\y in {0.67/0.7, -0.6/0.95, 1.52/0.8, -1.3/0.7}
\node[red] at (\x,\y) {$\delta$};

\foreach \x/\y in {0.5/0.9, -0.75/0.7, 1.25/0.7, -1.5/0.7}
\node[red] at (\x,\y) {$\epsilon$};

\end{scope}

\end{scope}
}

\foreach \x/\y in {0.4/0.3, 0.6/0.8, -0.6/0.8, 1.4/0.8, -1.4/0.8, 1.4/1.8}
\foreach \u in {1,-1}
\filldraw[fill=white] (\u*\x,\u*\y) circle (0.05);

\end{tikzpicture}
\caption{}
\label{Subfig-3A36B1}
\end{subfigure}
\begin{subfigure}[t]{0.45\linewidth} %% tiling 2
\centering
\begin{tikzpicture}[>=latex, scale=1]

\foreach \b in {1,-1}
\fill[yellow, scale=\b]
	(0.4,0.3) -- (-0.4,0.3) -- (-0.6,0.8) -- (-1.4,0.8) -- (-1.4,1.3) -- (0.6,1.3) -- (1.4,1.8) -- (2.2,0.8) -- (2.2,-0.8) -- (0.6,-0.8) -- (1,0);
	
\foreach \b in {1,-1}
{
\begin{scope}[scale=\b]

\draw
	(-0.4,0.3) -- (0.4,0.3) -- (0.4,-0.3)
	(0.4,0.3) -- (0.6,0.8)
	(-0.4,0.3) -- (-0.6,0.8)
	(-0.6,-1.3) -- (1.4,-1.3) -- (1.4,0.8) -- (-2.2,0.8) -- (-1.4,1.8) -- (1.4,1.8) -- (2.2,0.8)  -- (2.2,-0.8) 
	(0.4,0.3) -- (1,0) -- (0.6,-0.8)
	(1.4,0.8) -- (1,0)
	(0.6,0.8) -- (0.6,1.3) -- (1.4,1.8) -- (1.4,0.8)
	;
	
\draw[red]
	(-0.4,0) -- (0.4,0)
	(0,0.3) -- (0,0.8)
	(0.5,0.55) -- (1.2,0.4)
	(0.5,-0.55) -- (0.7,0.15)
	(0.8,-0.4) -- (1.4,0)
	(1,0.8) -- (1,1.55)
	(1.4,-1.3) -- (1.4,-1.8) 
	(-1,0.8) to[out=50, in=170] (0.6,1.05)
	(1.8,-0.8) to[out=80, in=-40] (1.4,1.3)
	(2.2,0.8) -- (2.8,1);

\foreach \x/\y in {1/0, 0.4/-0.3, 0.6/1.3, 2.2/-0.8}
\fill (\x,\y) circle (0.05);

\begin{scope}[font=\tiny]

\foreach \x/\y in {-0.3/0.4, -0.51/0.25, -0.3/0.2, -1.1/0.05, 0.97/0.15, -0.85/0.05, 0.6/1.45, 0.5/1.2, 0.7/1.25, -2.3/0.85, -2.1/0.67, -2/0.9, 0.3/0.4, 1.3/1.6}
\node at (\x,\y) {$\delta$};

\foreach \x/\y in {-0.47/0.7, -0.61/0.6, 0.3/0.2, 0.5/0.15, 0.53/0.35, 0.47/0.7, 0.7/0.9, 1.5/1.9, 1.47/1.55, 1.1/1.7, 1.3/0.9, 1.33/0.5, -1.3/0.9, -1.5/0.9}
\node at (\x,\y) {$\epsilon$};

\foreach \x/\y in {0.67/0.7, -0.6/0.95, 1.52/0.8, -1.3/0.7}
\node[red] at (\x,\y) {$\epsilon$};

\foreach \x/\y in {0.5/0.9, -0.75/0.7, 1.25/0.7, -1.5/0.7}
\node[red] at (\x,\y) {$\delta$};

\end{scope}

\end{scope}
}

\foreach \x/\y in {0.4/0.3, 0.6/0.8, -0.6/0.8, 1.4/0.8, -1.4/0.8, 1.4/1.8}
\foreach \u in {1,-1}
\filldraw[fill=white] (\u*\x,\u*\y) circle (0.05);

\end{tikzpicture}
\caption{}
\label{Subfig-3A36B2}
\end{subfigure}
\caption{Two tilings for AVC(3A36) and AVC(4A36).}
\label{3A36B}
\end{figure}

For the remaining undetermined angles at $\circ$ vertices, we assume one of them is red $\delta$. Then we use the fact that $\delta, \epsilon$ determine each other in any pentagonal tile, and $\epsilon^3\cdots=\delta\epsilon^3$, to determine all the remaining red angles. It turns out the choice of one red $\delta$ determines all the remaining red angles. Then we get two tilings in Figures \ref{Subfig-3A36B1} and \ref{Subfig-3A36B2} that are related by the exchange of all red $\delta,\epsilon$. We also use yellow color to indicate trapeziums in the quadrilateral tiling.

Figure \ref{4A36FigAx} gives another drawing of the tiling for AVC(4A36) in Figure \ref{Subfig-3A36B1}. The red corners are $\delta$ and the blue corners are $\epsilon$. For $\alpha=\frac{\pi}{2}$ and $\beta=\pi$, the white quadrilaterals are rhombi with angles $\delta$ and $\epsilon$. We are able to draw 3d rendering of the tiling, such that the eleven quadrilaterals enclosed by the green lines are faithful. Figure \ref{4A36FigAa} shows the five rhombi at the center of Figure \ref{4A36FigAx}. Figure \ref{4A36FigAb} shows the three rhombi at the top or bottom of Figure \ref{4A36FigAx}. Figure \ref{4A36FigAb} shows part of the tiling outside the green region.

\begin{figure}[htp]
\centering
\begin{subfigure}[b]{0.3\linewidth}
\centering
\begin{tikzpicture}[>=latex,scale=1]

\foreach \a in {1,-1}
\fill[yellow, scale=\a]
	(0.52,0.6) to[out=0, in=90] (1.04,-0.9) -- ++(150:0.6)
	(1.04,-1.5) to[out=30, in=-60] (1.5,1) to[out=120, in=30] (-0.52,1.8) -- (0,1.5) to[out=0, in=30] (1.04,-0.9);

\foreach \x/\y/\a/\b in {
	0/0.3/0/360, -0.52/0.6/-90/30, 0.52/0.6/0/360, 0.52/0/90/-90, 0/0.9/90/210, -0.52/1.2/0/360, 0/1.5/-90/0, -0.52/1.8/0/360, -1.04/0.9/-30/90, -1.04/1.5/210/270}
\foreach \u in {1,-1}
\fill[red!50, scale=\u, shift={(\x,\y)}]
	(0,0) -- (\a:0.12) arc (\a:\b:0.12);

\foreach \x/\y/\a/\b in {
	-0.52/0.6/30/270, 0.52/0/90/270, 0/0.9/-150/90, 0/1.5/0/270, -1.04/0.9/90/330, -1.04/1.5/-90/210}
\foreach \u in {1,-1}
\fill[blue!50, scale=\u, shift={(\x,\y)}]
	(0,0) -- (\a:0.12) arc (\a:\b:0.12);
		
\foreach \a in {1,-1}
{
\begin{scope}[scale=\a]

\draw
	(0,0.9) -- (0,0.3) -- ++(-30:0.6) -- ++(210:0.6)
	(-0.52,0.6) -- ++(90:0.6) -- ++(30:0.6) 
	(-0.52,1.2) -- ++(150:0.6)
	(0.52,0.6) to[out=0, in=90] (1.04,-0.9)
	(0,1.5) to[out=0, in=30] (1.04,-0.9)
	(1.04,-1.5) to[out=30, in=-60] (1.5,1) to[out=120, in=30] (-0.52,1.8)
	;

\draw[green, thick]
	(0.52,0) -- ++(90:0.6) -- ++(150:0.6) -- ++(90:0.6) -- ++(150:0.6) -- ++(210:0.6) -- ++(-90:0.6) -- ++(-30:0.6) -- ++(-90:0.6) 
	(0,0.9) -- ++(210:0.6);

\draw[red]
	(-0.78,0.75) to[out=-90, in=150] (-0.52,-0.3)
	(0,1.2) to[out=-20, in=60] (1.03,0)
	(1.18,0.6) -- ++(30:0.47);

\foreach \x/\y/\u in {-0.78/1.65/120, -0.78/1.35/30, -0.52/0.9/150, -0.26/0.75/90, 0/0.6/210, 0.26/0.15/90, 0.26/-0.15/150}
\draw[red]
	(\x,\y) -- ++(\u:0.6);
	
\foreach \x/\y in {0/0.3, 0.52/0.6, -0.52/1.2, -0.52/1.8}
\fill (\x,\y) circle (0.05);

\end{scope}
}

\foreach \x/\y in {0/0.9, 0/1.5, -0.52/0, -0.52/0.6, -1.04/0.9, 1.04/-1.5}
\foreach \u in {1,-1}
\filldraw[fill=white] (\u*\x,\u*\y) circle (0.05);

\end{tikzpicture}	
\caption{}
\label{4A36FigAx}
\end{subfigure}
\begin{subfigure}[b]{0.22\linewidth}
\centering
\begin{tikzpicture}[>=latex,scale=1]

\raisebox{0.7cm}{

\pgftext{
	\includegraphics[scale=0.05]{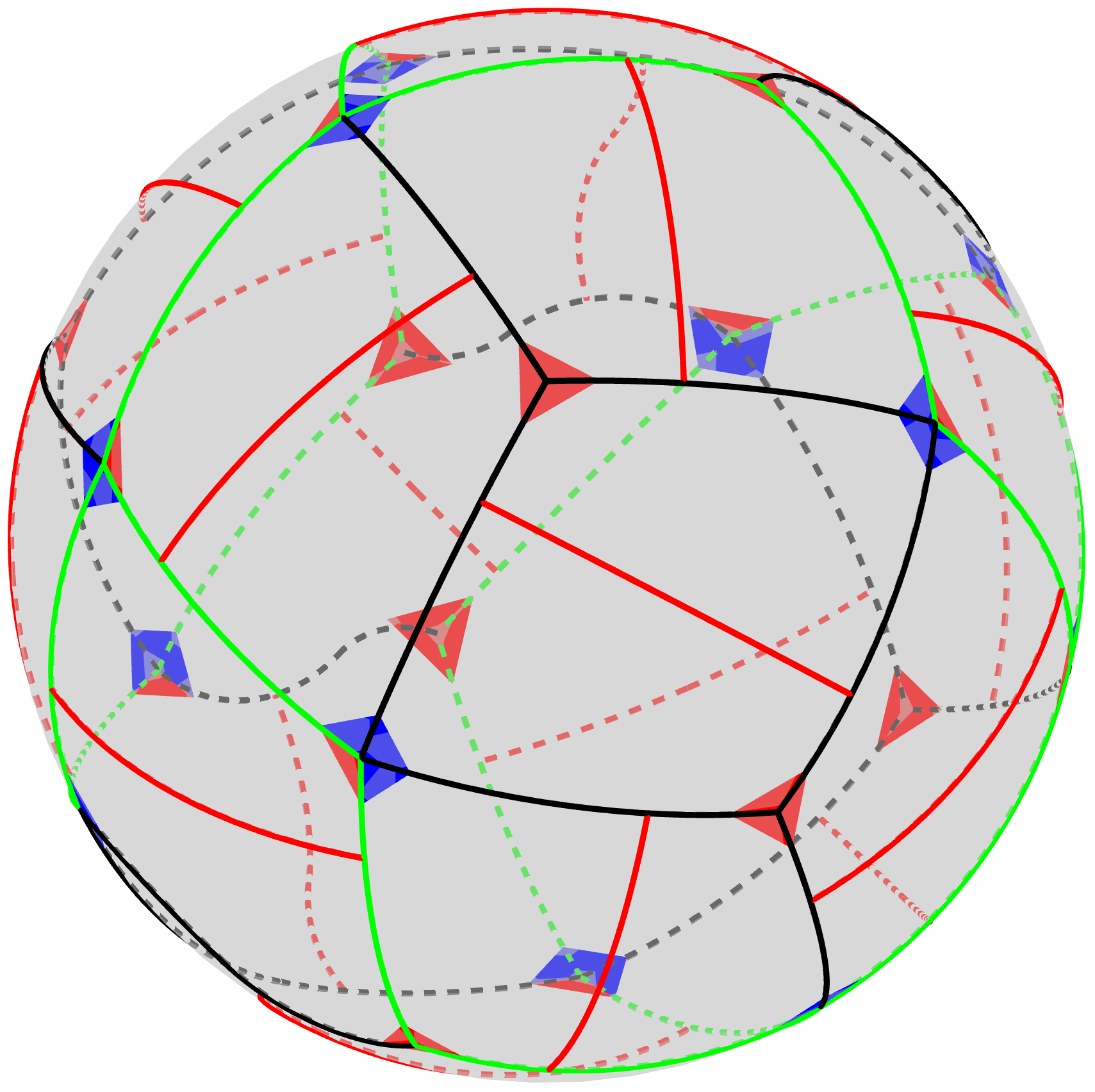}
	};

}

\end{tikzpicture}	
\caption{}
\label{4A36FigAa}
\end{subfigure}
\begin{subfigure}[b]{0.22\linewidth}
\centering
\begin{tikzpicture}[>=latex,scale=1]

\raisebox{0.7cm}{

\pgftext{
	\includegraphics[scale=0.053]{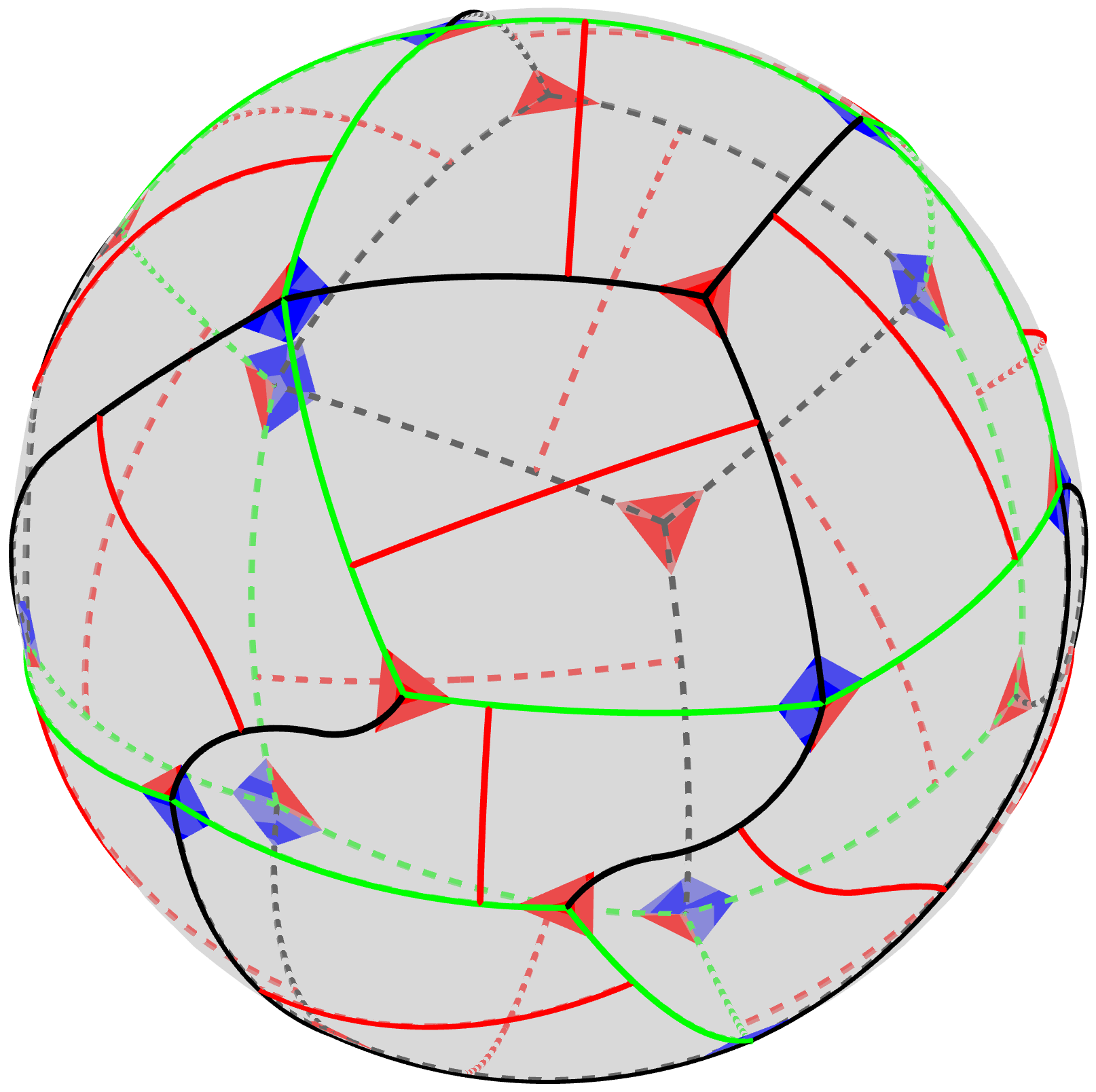}
	};
	
}
	
\end{tikzpicture}	
\caption{}
\label{4A36FigAb}
\end{subfigure}
\begin{subfigure}[b]{0.22\linewidth}
\centering
\begin{tikzpicture}[>=latex,scale=1]

\raisebox{0.7cm}{

\pgftext{
	\includegraphics[scale=0.05]{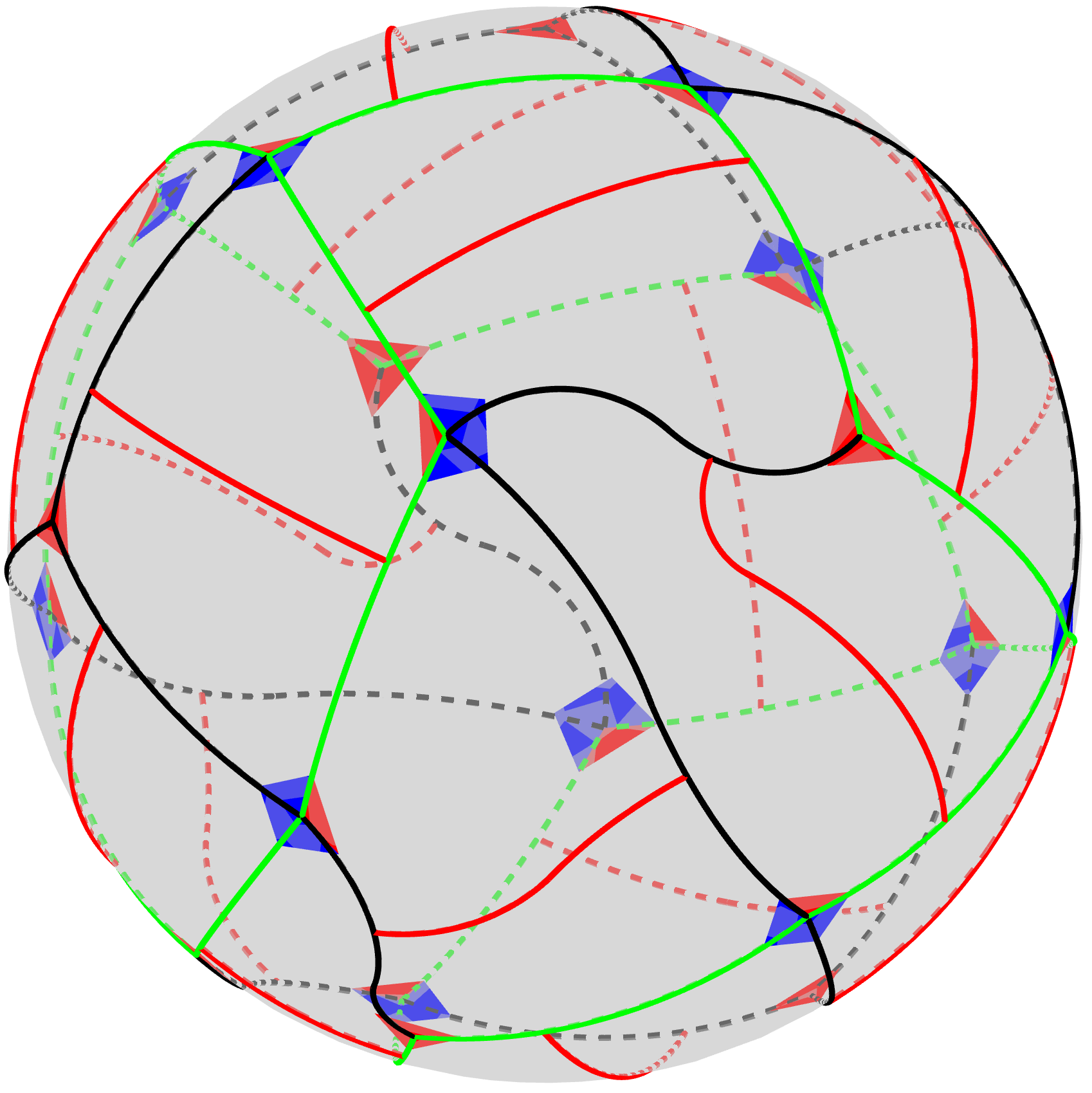}
	};

}

\end{tikzpicture}	
\caption{}
\label{4A36FigAc}
\end{subfigure}
\caption{Tiling in Figure \ref{Subfig-3A36B1} for AVC(4A36), with $\alpha=\frac{\pi}{2}$ and $\beta=\pi$. Red corners are $\delta$ and blue corners are $\epsilon$.}
\label{4A36FigA}
\end{figure} 

The tilings obtained by glueing Figure \ref{Subfig-3A36A-4} to itself are given by the normal lines in Figure \ref{3A36C}, with $\bullet$ and $\circ$ as vertices. We apply the same process of adding red subdivision lines, assigning $\delta$ and $\epsilon$ to $\bullet$ and $\circ$ vertices. We get two tilings in Figures \ref{Subfig-3A36C1} and \ref{Subfig-3A36C3}, and the exchange of all red $\delta, \epsilon$ also gives tilings. Again we use yellow color to indicate trapeziums in the quadrilateral tiling. Then it is easy to see that exchanging all red $\delta,\epsilon$ in Figure \ref{Subfig-3A36C1} gives the tiling in Figure \ref{Subfig-3A36C2}, and exchanging all red $\delta,\epsilon$ in Figure \ref{Subfig-3A36C2} gives an equivalent tiling. In fact, Figure \ref{Subfig-3A36C1} is the union of two identical half sphere tilings ${\mc X}$, and Figure \ref{Subfig-3A36C2} is the union of two identical half sphere tilings ${\mc Y}$, and Figure \ref{Subfig-3A36C3} is the union of ${\mc X}$ and ${\mc Y}$.

\begin{figure}[h!]
\begin{subfigure}[t]{0.5\linewidth} %% tiling 1
\centering
\begin{tikzpicture}[>=latex, scale=1]

\foreach \b in {1,2,3}
\fill[yellow, rotate=120*\b]
	(30:1.6) -- (-30:1.6) -- (-30:0.8) -- (30:0.8) -- (30:2.4) -- (90:2.4) -- (90:1.6) -- (30:1.6);

\foreach \a in {0,...,5}
{
\draw[blue, thick, rotate=60*\a]
	(30:1.6) -- (90:1.6);
\draw[green, thick, rotate=60*\a]
	(30:0.8) -- (90:0.8)
	(30:2.4) -- (90:2.4);
}

\foreach \b in {1,2,3}
{
\begin{scope}[rotate=120*\b]

\draw
	(0,0) -- (90:2.4)
	(30:0.8) -- (30:3);

\draw[red]
	(-30:0.4) -- (60:0.693)
	(30:1.2) -- (90:1.2)
	(0:0.693) -- (0:1.386)
	(30:2) -- (-30:2)
	(60:1.386) -- (60:2.078)
	(30:2.7) to[out=120, in=0] (90:2.8) to[out=180, in=90] (120:2.078);
	
\fill 
	(30:0.8) circle (0.05)
	(0,0) circle (0.05)
	(30:3) circle (0.05)
	(90:2.4) circle (0.05);

\begin{scope}[font=\tiny]

\foreach \a/\r in {30/0.15, 30/0.65, 38/0.9, 22/0.9, 87/2.2, 93/2.2, -87/2.2, 98/0.9, 
93/2.2, 87/2.2, 90/2.55, 90/3}
\node at (\a:\r) {$\delta$};

\foreach \a/\r in {26/1.45, 34/1.7, 100/0.62, 80/0.62, 82/0.9, 95/1.4, 86/1.7,
27/2.2, 27.5/2.5, 32.5/2.5
}
\node at (\a:\r) {$\epsilon$};

\foreach \a/\r in {26/1.7, 85/1.4}
\node[red] at (\a:\r) {$\delta$};

\foreach \a/\r in {34/1.45, 94/1.7}
\node[red] at (\a:\r) {$\epsilon$};

\end{scope}
	
\end{scope}
}

\foreach \a in {0,1,2}
{

\filldraw[fill=white, rotate=120*\a]
	(90:0.8) circle (0.05)
	(90:1.6) circle (0.05)
	(30:1.6) circle (0.05);
		
}	

\end{tikzpicture}
\caption{}
\label{Subfig-3A36C1}
\end{subfigure}
\begin{subfigure}[t]{0.5\linewidth} %% tiling 2
\centering
\begin{tikzpicture}[>=latex, scale=1]

\foreach \b in {0,...,5}
\fill[yellow, rotate=60*\b]
	(30:0.8) -- (30:2.4) -- (90:2.4) -- (90:0.8);
	
\foreach \a in {0,...,5}
{
\draw[blue, thick,  rotate=60*\a]
	(30:1.6) -- (90:1.6);
\draw[green, thick,  rotate=60*\a]
	(30:0.8) -- (90:0.8)
	(30:2.4) -- (90:2.4);
}

\foreach \b in {1,2,3}
{
\begin{scope}[rotate=120*\b]

\draw
	(0,0) -- (90:2.4)
	(30:0.8) -- (30:3);

\draw[red]
	(-30:0.4) -- (60:0.693)
	(30:1.2) -- (90:1.2)
	(0:0.693) -- (0:1.386)
	(30:2) -- (-30:2)
	(60:1.386) -- (60:2.078)
	(30:2.7) to[out=120, in=0] (90:2.8) to[out=180, in=90] (120:2.078);
	
\fill 
	(30:0.8) circle (0.05)
	(0,0) circle (0.05)
	(30:3) circle (0.05)
	(90:2.4) circle (0.05);

\begin{scope}[font=\tiny]

\foreach \a/\r in {30/0.15, 30/0.65, 38/0.9, 22/0.9, 87/2.2, 93/2.2, -87/2.2, 98/0.9, 
93/2.2, 87/2.2, 90/2.55, 90/3}
\node at (\a:\r) {$\delta$};

\foreach \a/\r in {26/1.45, 34/1.7, 100/0.62, 80/0.62, 82/0.9, 95/1.4, 86/1.7,
27/2.2, 27.5/2.5, 32.5/2.5
}
\node at (\a:\r) {$\epsilon$};

\foreach \a/\r in {26/1.7, 85/1.4}
\node[red] at (\a:\r) {$\epsilon$};

\foreach \a/\r in {34/1.45, 94/1.7}
\node[red] at (\a:\r) {$\delta$};

\end{scope}
	
\end{scope}
}

\foreach \a in {0,1,2}
{

\filldraw[fill=white, rotate=120*\a]
	(90:0.8) circle (0.05)
	(90:1.6) circle (0.05)
	(30:1.6) circle (0.05);
		
}	

\end{tikzpicture}
\caption{}
\label{Subfig-3A36C2}
\end{subfigure}
\begin{subfigure}[t]{1\linewidth} %% tiling 3
\centering
\begin{tikzpicture}[>=latex, scale=1]

\foreach \b in {0,1,2}
\fill[yellow, rotate=120*\b]
	(-30:0.8) -- (-30:2.4) -- (30:2.4) -- (90:2.4) -- (90:1.6) -- (30:1.6) -- (30:0.8);

\foreach \a in {0,...,5}
{
\draw[blue, thick,  rotate=60*\a]
	(30:1.6) -- (90:1.6);
\draw[green, thick,  rotate=60*\a]
	(30:0.8) -- (90:0.8)
	(30:2.4) -- (90:2.4);
}

\foreach \b in {1,2,3}
{
\begin{scope}[rotate=120*\b]

\draw
	(0,0) -- (90:2.4)
	(30:0.8) -- (30:3);

\draw[red]
	(-30:0.4) -- (60:0.693)
	(30:1.2) -- (90:1.2)
	(0:0.693) -- (0:1.386)
	(30:2) -- (-30:2)
	(60:1.386) -- (60:2.078)
	(-30:2.7) to[out=240, in=0] (-90:2.8) to[out=180, in=-90] (240:2.078);
	
\fill 
	(30:0.8) circle (0.05)
	(0,0) circle (0.05)
	(90:3) circle (0.05)
	(30:2.4) circle (0.05);

\begin{scope}[font=\tiny]

\foreach \a/\r in {30/0.15, 30/0.65, 38/0.9, 22/0.9, 87/2.2, -87/2.2, 98/0.9, 
33/2.2, 27/2.2, 30/2.55, 30/3}
\node at (\a:\r) {$\delta$};

\foreach \a/\r in {26/1.45, 34/1.7, 100/0.62, 80/0.62, 82/0.9, 95/1.4, 86/1.7,
93/2.2, 87.5/2.5, 92.5/2.5
}
\node at (\a:\r) {$\epsilon$};

\foreach \a/\r in {26/1.7, 85/1.4}
\node[red] at (\a:\r) {$\delta$};

\foreach \a/\r in {34/1.45, 94/1.7}
\node[red] at (\a:\r) {$\epsilon$};

\end{scope}
	
\end{scope}
}

\foreach \a in {0,1,2}
{

\filldraw[fill=white, rotate=120*\a]
	(90:0.8) circle (0.05)
	(90:1.6) circle (0.05)
	(30:1.6) circle (0.05);
		
}	

\end{tikzpicture}
\caption{}
\label{Subfig-3A36C3}
\end{subfigure}
\caption{Three more tilings for AVC(3A36) and AVC(4A36).}
\label{3A36C}
\end{figure}

Figure \ref{4A36FigB} gives 3d rendering of the tilings in Figure2 \ref{Subfig-3A36C1} and \ref{Subfig-3A36C2}, with $\alpha=\frac{\pi}{2}$ and $\beta=\pi$, such that the six quadrilaterals within the central green lines and outside of outer green lines are faithful. Figures \ref{4A36FigBa} and \ref{4A36FigCa} show three such rhombi. Figures \ref{4A36FigBb} and \ref{4A36FigCb} show the tiles on the two sides of the blue equator. If we glue the part of Figure \ref{4A36FigBb} above the equator and the lower part of Figure \ref{4A36FigCb} below the equator, then we get the 3d rendering of the tiling in Figure \ref{Subfig-3A36C3}.

\begin{figure}[htp]
\centering
\begin{subfigure}[b]{0.22\linewidth}
\centering
\begin{tikzpicture}[>=latex,scale=1]

\pgftext{
	\includegraphics[scale=0.055]{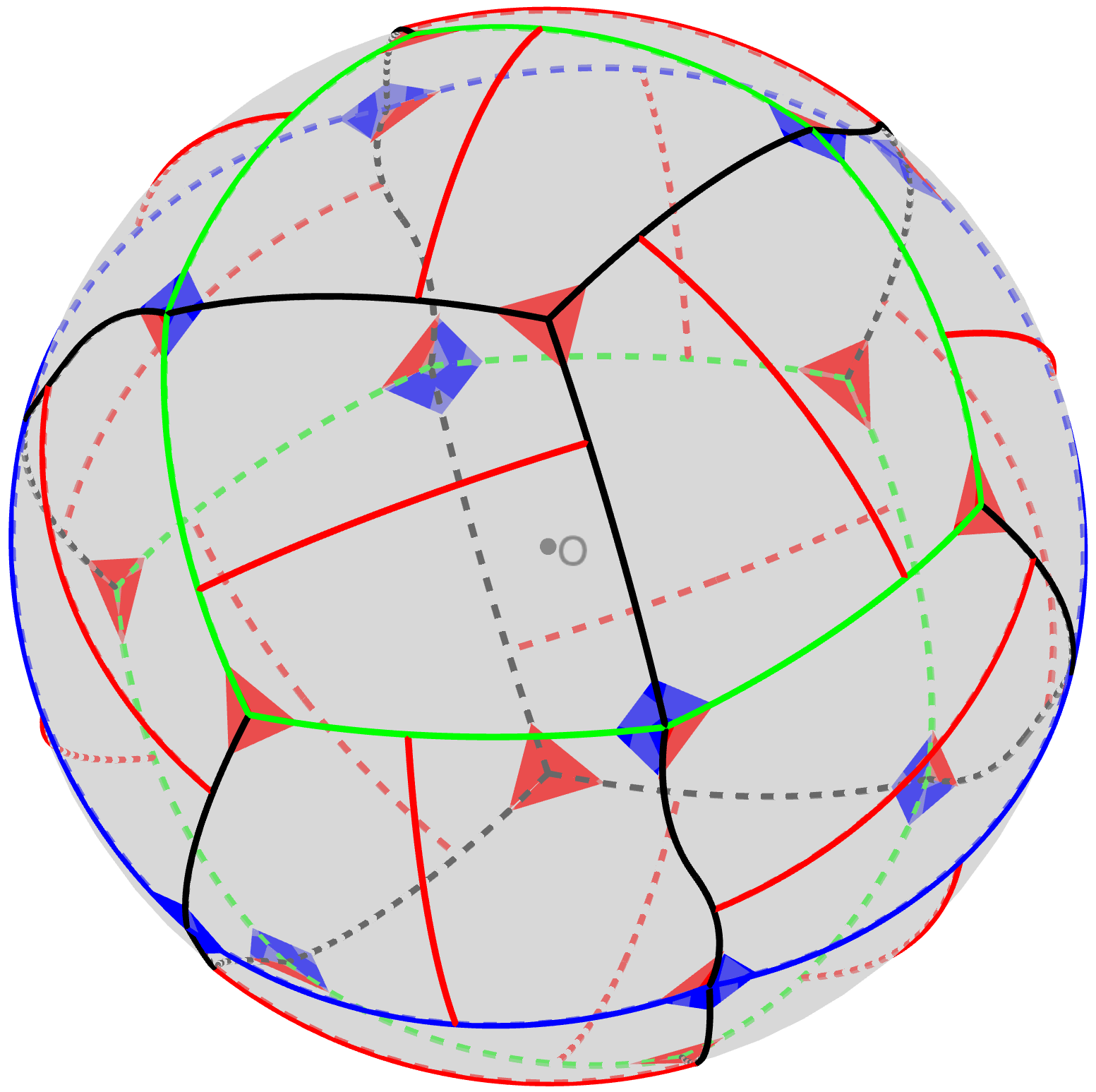}
	};

\end{tikzpicture}	
\caption{Figure \ref{Subfig-3A36C1}}
\label{4A36FigBa}
\end{subfigure}
\begin{subfigure}[b]{0.22\linewidth}
\centering
\begin{tikzpicture}[>=latex,scale=1]

\pgftext{
	\includegraphics[scale=0.055]{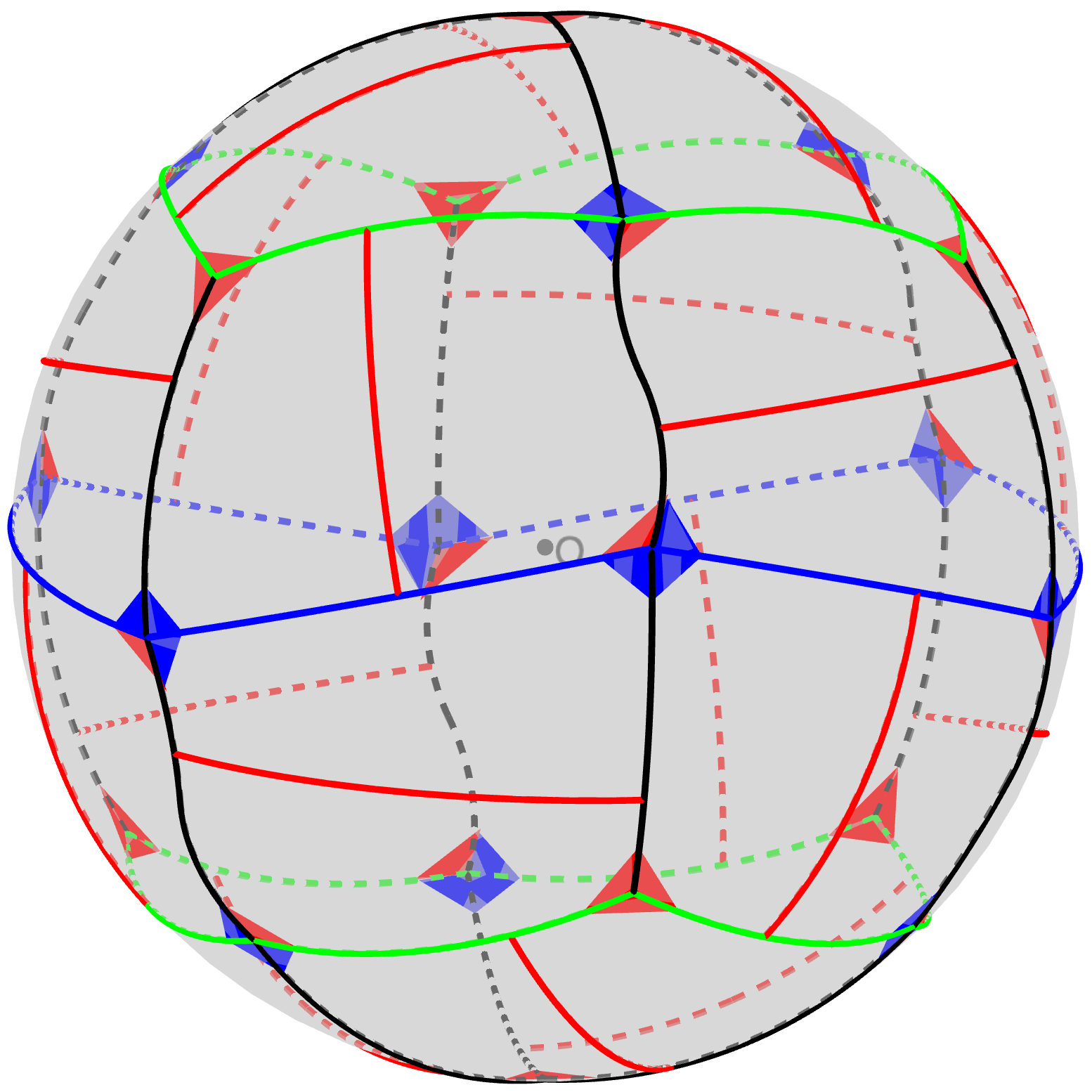}
	};

\end{tikzpicture}	
\caption{Figure \ref{Subfig-3A36C1}}
\label{4A36FigBb}
\end{subfigure}
\begin{subfigure}[b]{0.22\linewidth}
\centering
\begin{tikzpicture}[>=latex,scale=1]

\pgftext{
	\includegraphics[scale=0.057]{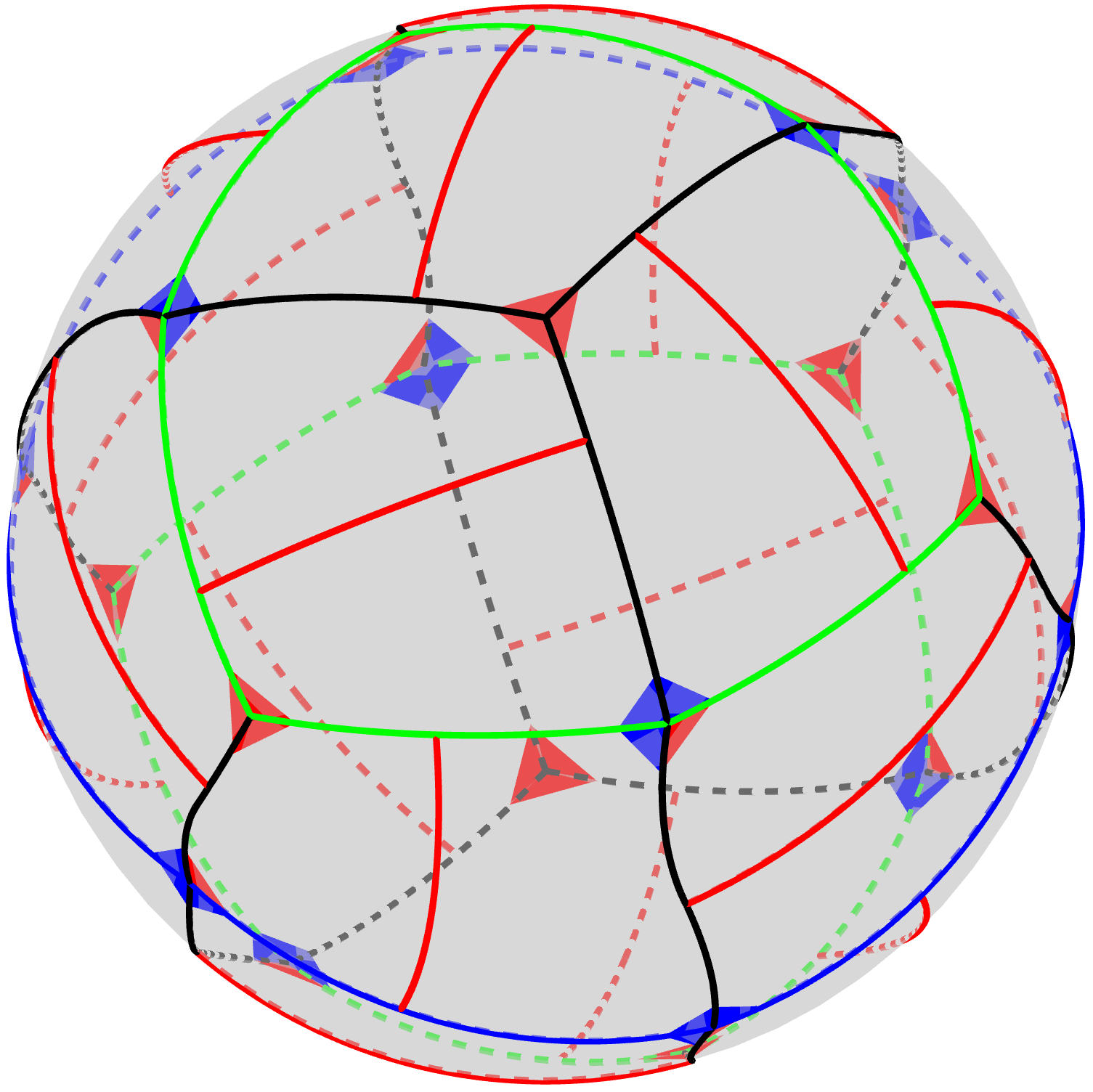}
	};
	
\end{tikzpicture}	
\caption{Figure \ref{Subfig-3A36C2}}
\label{4A36FigCa}
\end{subfigure}
\begin{subfigure}[b]{0.22\linewidth}
\centering
\begin{tikzpicture}[>=latex,scale=1]

\pgftext{
	\includegraphics[scale=0.057]{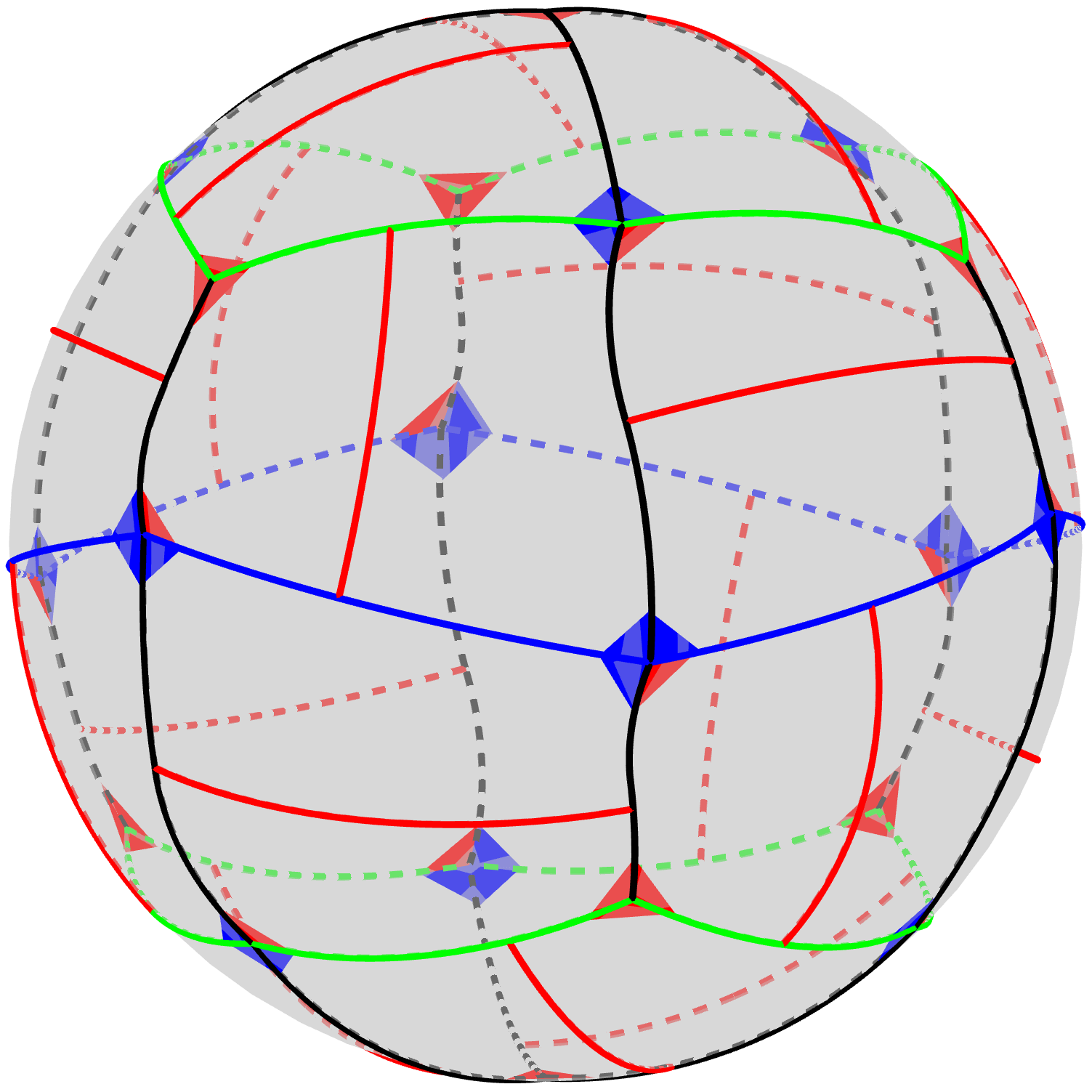}
	};
	
\end{tikzpicture}	
\caption{Figure \ref{Subfig-3A36C2}}
\label{4A36FigCb}
\end{subfigure}
\caption{Tiling in Figure \ref{3A36C} for AVC(4A36), with $\alpha=\frac{\pi}{2}$ and $\beta=\pi$. Red corners are $\delta$ and blue corners are $\epsilon$.}
\label{4A36FigB}
\end{figure}

Finally, we assign $\alpha,\beta$ to the two tilings according to Figures \ref{Subfig-3A36A-2}, \ref{Subfig-4A36A-T1}, \ref{Subfig-4A36A-T2}, \ref{Subfig-4A36A-T3}. In the later two cases, we get a vertex $\alpha\beta^2$, contradicting AVC(4A36). Therefore we get tilings only for the pentagons in Figures \ref{Subfig-3A36A-2} and \ref{Subfig-4A36A-T1}.
\end{proof}

Tilings for AVC(5A36) are obtained by applying the splitting to the tilings for AVC(4A36). Specifically, we change Figure \ref{Subfig-4A36A-T1} to the one in Figure \ref{Subfig-5AfigA-T2}, leaving $\delta$ and $\epsilon$ unchanged. Then we find that the splitting in Figure \ref{Subfig-4A36A-V} always introduces $\alpha\beta^2$. Since this is not in AVC(5A36), we get an alternative proof of no tiling for AVC(5A36) in Theorem \ref{5Athm}. 

Finally, we may consider the more extreme reduction of AVC(5A36) by assuming $\alpha=\beta=\gamma=\delta$
\[
\text{AVC(2D36)}
=\{36\alpha^4\epsilon\colon 44\alpha^3,12\alpha\epsilon^3\}.
\]
This is a further reduction of AVC(3A36) where we assume $\alpha=\delta$, and is comparable for AVC(2D24/60).

The tilings for AVC(3A36) in Figures \ref{3A36B} and \ref{3A36C} reduce to tilings for AVC(2D36). Figure \ref{2D36B1} shows a tiling for AVC(2D36) that is not a reduction. The angle around the gray edges are given Figure \ref{2D36B2}, and the inner $\alpha, \epsilon$ can be exchanged. All other angles are $\alpha$. 

\begin{figure}[h!]
\centering
\begin{subfigure}[t]{0.3\linewidth}
\centering
\begin{tikzpicture}[>=latex]

\foreach \a in {1,-1}
\draw[gray!70, line width =2, scale=\a]
	(0,0) -- (0.3,0)
	(0,0.6) -- (0,1)
	(1.3,0.5) -- (1.3,-0.5)
	(2.1,0) -- (2.5,0);

\foreach \a in {1,-1}
\foreach \b in {1,-1}
{
\begin{scope}[xscale=\a, yscale=\b]

\draw
	(0.3,0.6) -- (0.3,0)
	(0,0) -- (0.9,0) -- (0.9,0.3) -- (0.6,0.6) -- (0,0.6) -- (0,1.8) -- (2.1,1.8) -- (2.1,0)
	(0.6,0.6) -- (0.8,1)
	(0.9,0.3) -- (1.3,0.5) -- (1.7,0.5)
	(0,1) -- (1.3,1)
	(1.3,1.4) -- (1.3,0)
	(0,1.4) -- (1.7,1.4) -- (1.7,0)
	(1.7,1.4) -- (2.1,1.8)
	(1.7,0) -- (2.5,0);

\end{scope}
}

\end{tikzpicture}
\caption{}
\label{2D36B1}
\end{subfigure}
\begin{subfigure}[t]{0.3\linewidth}
\centering
\begin{tikzpicture}[>=latex]

\foreach \a in {1,-1}
{
\begin{scope}[scale=\a]

\draw[gray!70, line width =2]
	(0,0) -- (0.4,0);
	
\draw
	(0,0) -- (0.8,0)
	(0.4,0.4) -- (0.4,-0.4);

\node at (0.25,0.15) {\scriptsize $\epsilon$};
\node at (-0.25,0.15) {\scriptsize $\alpha$};
\node at (0.55,0.15) {\scriptsize $\epsilon$};
\node at (-0.55,0.15) {\scriptsize $\epsilon$};

\end{scope}
}

\end{tikzpicture}
\caption{}
\label{2D36B2}
\end{subfigure}
\begin{subfigure}[t]{0.15\linewidth}
\centering
\begin{tikzpicture}[>=latex]

\draw[gray!70, line width =2]
	(0.3,0) -- (-0.3,0);

\foreach \a in {1,-1}
\foreach \b in {1,-1}
{
\begin{scope}[xscale=\a, yscale=\b]
	
\draw[dashed]
	(0,1) -- (1,1) -- (1,0);

\draw
	(0,0) -- (1,0)
	(0.3,0) -- (0.5,1)
	(0,1) -- ++(0,0.3)
	(1,0.5) -- ++(0.3,0)
	(1,1) -- ++(0.2,0.2);

\end{scope}
}

\end{tikzpicture}
\caption{}
\label{2D36B3}
\end{subfigure}
\caption{One tiling for AVC(2D36).}
\label{2D36B} 
\end{figure}

One way to construct a subset of tilings for AVC(2D36) is by gluing six copies of the patches in Figure \ref{2D36A} together, such that either $\bar{\alpha}$ is matched with $\alpha$, or three $N_i$ meet at $\alpha\alpha\alpha$. Here $\bar{\alpha}$ is either $\alpha^2$ or $\epsilon^3$ along the boundary. We remark that this construction is not exhaustive, as it finds tilings for only $103$ non-isomorphic order $36$ combinatorial tilings corresponding to AVC(2D36). In contrast, a computer search yields that there are altogether $1396$ tilings for AVC(2D36) spread across $295$ non-isomorphic order $36$ combinatorial tilings.

\begin{figure}[h!]
\centering
\begin{tikzpicture}[>=latex,scale=1]

\draw[gray!70, line width=2]
	(-0.8,0) -- (0,0);
	
\foreach \a in {1,-1}
\foreach \x in {0,...,3}
{
\begin{scope}[xshift=3*\x cm, xscale=\a]

\draw
	(0.8,0) -- (0,0) -- (0,0.5) -- (0.4,0.8) -- (0.8,0.5) -- (0.8,-0.5) -- (0.4,-0.8) -- (0,-0.5) -- (0,0);

\node at (0.15,0.15) {\scriptsize $\epsilon$};
\node at (0.65,0.15) {\scriptsize $\alpha$};
\node at (0.65,0.45) {\scriptsize $\alpha$};
\node at (0.4,0.6) {\scriptsize $\alpha$};
\node at (0.15,0.45) {\scriptsize $\alpha$};

\end{scope}
}

\foreach \x in {0,...,3}
{
\begin{scope}[xshift=3*\x cm]

\node at (0.15,-0.15) {\scriptsize $\epsilon$};
\node at (0.65,-0.15) {\scriptsize $\alpha$};
\node at (0.65,-0.45) {\scriptsize $\alpha$};
\node at (0.4,-0.6) {\scriptsize $\alpha$};
\node at (0.15,-0.45) {\scriptsize $\alpha$};

\node[inner sep=0.5,draw,shape=circle] at (-0.4,0.3) {\scriptsize 1};
\node[inner sep=0.5,draw,shape=circle] at (0.4,0.3) {\scriptsize 2};
\node[inner sep=0.5,draw,shape=circle] at (0.4,-0.3) {\scriptsize 3};
\node[inner sep=0.5,draw,shape=circle] at (-0.4,-0.3) {\scriptsize 4};

\end{scope}
}

%% 1

\draw
	(-0.8,0.5) -- (-1.2,0.8) -- (-1.6,0.5) -- (-1.6,-0.5) -- (-1.2,-0.8) -- (-0.8,-0.5)
	(-0.8,0) -- (-1.6,0);

\node at (-0.15,-0.15) {\scriptsize $\alpha$};
\node at (-0.65,-0.15) {\scriptsize $\epsilon$};
\node at (-0.65,-0.45) {\scriptsize $\alpha$};
\node at (-0.4,-0.6) {\scriptsize $\alpha$};
\node at (-0.15,-0.45) {\scriptsize $\alpha$};

\node at (-0.95,-0.15) {\scriptsize $\epsilon$};
\node at (-1.45,-0.15) {\scriptsize $\alpha$};
\node at (-1.45,-0.45) {\scriptsize $\alpha$};
\node at (-1.2,-0.6) {\scriptsize $\alpha$};
\node at (-0.95,-0.45) {\scriptsize $\alpha$};

\node at (-0.95,0.15) {\scriptsize $\epsilon$};
\node at (-1.45,0.15) {\scriptsize $\alpha$};
\node at (-1.45,0.45) {\scriptsize $\alpha$};
\node at (-1.2,0.6) {\scriptsize $\alpha$};
\node at (-0.95,0.45) {\scriptsize $\alpha$};

\node[inner sep=0.5,draw,shape=circle] at (-1.2,-0.3) {\scriptsize 5};
\node[inner sep=0.5,draw,shape=circle] at (-1.2,0.3) {\scriptsize 6};

%% 2 & 3 & 4

\foreach \b in {1,2,3}
{
\begin{scope}[xshift=3*\b cm]

\node at (-0.15,-0.15) {\scriptsize $\alpha$};
\node at (-0.65,-0.15) {\scriptsize $\alpha$};
\node at (-0.65,-0.45) {\scriptsize $\epsilon$};
\node at (-0.4,-0.6) {\scriptsize $\alpha$};
\node at (-0.15,-0.45) {\scriptsize $\alpha$};

\end{scope}
}

\foreach \b in {0,1}
{
\begin{scope}[xshift=3cm+3*\b cm]

\draw
	(-0.4,-0.8) -- (-0.4,-1.3) -- (-1.2,-1.3) -- (-1.2,-0.9) -- (-0.8,-0.5);

\node at (-0.8,-0.7) {\scriptsize $\epsilon$};
\node at (-0.55,-0.85) {\scriptsize $\alpha$};
\node at (-1.05,-0.9) {\scriptsize $\alpha$};
\node at (-0.55,-1.15) {\scriptsize $\alpha$};
\node at (-1.05,-1.15) {\scriptsize $\alpha$};

\node[inner sep=0.5,draw,shape=circle] at (-0.8,-1) {\scriptsize 5};

\end{scope}

\begin{scope}[xshift=6cm+3*\b cm]

\draw
	(-0.8,-0.5) -- (-1.4,-0.5) -- (-1.4,0.5) -- (-0.8,0.5);

\node at (-1.25,0.35) {\scriptsize $\alpha$};
\node at (-1.25,-0.35) {\scriptsize $\alpha$};
\node at (-0.95,0.35) {\scriptsize $\alpha$};
\node at (-0.95,-0.35) {\scriptsize $\epsilon$};
\node at (-0.95,0) {\scriptsize $\alpha$};

\node[inner sep=0.5,draw,shape=circle] at (-1.2,0) {\scriptsize 6};

\end{scope}

}

%% 2

\begin{scope}[xshift=3cm]

\draw
	(-0.8,-0.5) -- (-1.2,-0.1) -- (-1.7,-0.1) -- (-1.7,-0.9) -- (-1.2,-0.9);

\node at (-0.92,-0.2) {\scriptsize $\alpha$};

\node at (-1,-0.5) {\scriptsize $\epsilon$};
\node at (-1.25,-0.25) {\scriptsize $\alpha$};
\node at (-1.25,-0.75) {\scriptsize $\alpha$};
\node at (-1.55,-0.25) {\scriptsize $\alpha$};
\node at (-1.55,-0.75) {\scriptsize $\alpha$};

\node[inner sep=0.5,draw,shape=circle] at (-1.4,-0.5) {\scriptsize 6};

\end{scope}

%% 3

\begin{scope}[xshift=6cm]

\node at (-1.1,-0.6) {\scriptsize $\alpha$};

\end{scope}

%% 4

\begin{scope}[xshift=9cm]

\draw
	(-0.8,-0.5) -- (-0.8,-1.3) -- (-1.7,-1.3) -- (-1.7,-0.9) -- (-1.4,-0.5);

\node at (-0.65,-0.75) {\scriptsize $\alpha$};

\node at (-0.95,-0.65) {\scriptsize $\epsilon$};
\node at (-0.95,-1.15) {\scriptsize $\alpha$};
\node at (-1.55,-1.15) {\scriptsize $\alpha$};
\node at (-1.55,-0.9) {\scriptsize $\alpha$};
\node at (-1.35,-0.65) {\scriptsize $\alpha$};

\node[inner sep=0.5,draw,shape=circle] at (-1.2,-1) {\scriptsize 5};

\end{scope}

%%% boundary

\begin{scope}[shift={(-0.4cm,-2.7cm)}]

\foreach \x in {1,2,3,4}
{
\draw[xshift=-3cm +3*\x cm]
	(0,0) circle (1.15);
	
\node at (-3+3*\x,0) {$N_{\x}$};
}

%% 1

\foreach \a in {1,-1}
\foreach \b in {1,-1}
{
\begin{scope}[xscale=\a, yscale=\b]

\foreach \c in {1,2,4}
\node at (22.5*\c:1) {\scriptsize $\alpha$};

\foreach \c in {0,3}
\node at (22.5*\c:1) {\scriptsize $\bar{\alpha}$};

\end{scope}
}

%% 2, 3, 4

\foreach \a in {1,-1}
\foreach \x in {1,2,3}
{
\begin{scope}[xshift=3*\x cm, xscale=\a]

\node at (22.5:1) {\scriptsize $\bar{\alpha}$};
\node at (45:1) {\scriptsize $\alpha$};
\node at (67.5:1) {\scriptsize $\alpha$};
\node at (90:1) {\scriptsize $\bar{\alpha}$};

\end{scope}
}

%% 2

\begin{scope}[xshift=3 cm, rotate=180]

\node at (0:1) {\scriptsize $\alpha$};
\node at (18:1) {\scriptsize $\alpha$};
\node at (36:1) {\scriptsize $\bar{\alpha}$};
\node at (54:1) {\scriptsize $\bar{\alpha}$};
\node at (72:1) {\scriptsize $\alpha$};
\node at (90:1) {\scriptsize $\alpha$};
\node at (108:1) {\scriptsize $\alpha$};
\node at (126:1) {\scriptsize $\bar{\alpha}$};
\node at (144:1) {\scriptsize $\alpha$};
\node at (162:1) {\scriptsize $\alpha$};
\node at (180:1) {\scriptsize $\bar{\alpha}$};

\end{scope}

%% 3

\begin{scope}[xshift=6 cm, rotate=180]

\node at (0:1) {\scriptsize $\alpha$};
\node at (22.5:1) {\scriptsize $\bar{\alpha}$};
\node at (45:1) {\scriptsize $\alpha$};
\node at (67.5:1) {\scriptsize $\alpha$};
\node at (90:1) {\scriptsize $\bar{\alpha}$};
\node at (112.5:1) {\scriptsize $\alpha$};
\node at (135:1) {\scriptsize $\alpha$};
\node at (157.5:1) {\scriptsize $\alpha$};
\node at (180:1) {\scriptsize $\bar{\alpha}$};

\end{scope}

%% 4

\begin{scope}[xshift=9 cm, rotate=180]

\node at (0:1) {\scriptsize $\alpha$};
\node at (22.5:1) {\scriptsize $\bar{\alpha}$};
\node at (45:1) {\scriptsize $\alpha$};
\node at (67.5:1) {\scriptsize $\bar{\alpha}$};
\node at (90:1) {\scriptsize $\alpha$};
\node at (112.5:1) {\scriptsize $\alpha$};
\node at (135:1) {\scriptsize $\alpha$};
\node at (157.5:1) {\scriptsize $\bar{\alpha}$};
\node at (180:1) {\scriptsize $\alpha$};

\end{scope}

\end{scope}

\end{tikzpicture}
\caption{Four patches for AVC(2D36).}
\label{2D36A}
\end{figure}

We remark that Figure \ref{2D36B1} is the union of six copies of Figure \ref{2D36B3}, which is actually $N_1$ in Figure \ref{2D36A}.

\end{document}